\documentclass[11pt]{amsart}

\usepackage{amsmath,amsthm,indentfirst}
\usepackage{amssymb}
\usepackage{amsfonts}
\usepackage{amscd}
\usepackage[latin1]{inputenc}
\usepackage{hyperref}
\usepackage{ifthen, amsfonts, amssymb, graphicx, srcltx, mathrsfs, xfrac,
amsmath}
\usepackage{enumerate}

\theoremstyle{plain}
\newtheorem*{maintheorem*}{Main Theorem}
\newtheorem*{thmshcurve}{Theorem 6.8}
\newtheorem*{thmstable}{Theorem 5.10}

\newtheorem*{thmdiv}{Theorem 1.2}
\newtheorem*{thmclosed}{Theorem 1.1}
\newtheorem*{thmd*}{Theorem 1.5}
\newtheorem*{thme*}{Theorem 1.6}
\newtheorem*{remark*}{Remark}
\newtheorem*{conjecture*}{Conjecture}
\newtheorem*{prop*}{Proposition}
\newtheorem{thm}{Theorem}[section]
\newtheorem{cor}[thm]{Corollary}
\newtheorem{lem}[thm]{Lemma}
\newtheorem{prop}[thm]{Proposition}
\newtheorem{example}[thm]{Example}
\theoremstyle{definition}

\newtheorem*{proofc*}{Proof of Theorem C}

\newtheorem{conjecture}[thm]{Conjecture}

\newtheorem{definition}[thm]{Definition}

\newtheorem{remark}[thm]{Remark}
\newtheorem{notation}[thm]{Notation}
\newtheorem{claim}[thm]{Claim}

\DeclareMathOperator{\CAT}{CAT}
\DeclareMathOperator{\Teich}{Teich}
\DeclareMathOperator{\Mod}{Mod}

\DeclareMathOperator{\grad}{grad}
\DeclareMathOperator{\WP}{WP}
\DeclareMathOperator{\base}{base}
\DeclareMathOperator{\tw}{tw}

\DeclareMathOperator{\inj}{inj}

\DeclareMathOperator{\diam}{diam}
\DeclareMathOperator{\hull}{hull}
\DeclareMathOperator{\Ax}{Ax}
\DeclareMathOperator{\nonannular}{non-annular}
\DeclareMathOperator{\lift}{lift}

\numberwithin{equation}{section}

\begin{document}


\title[Prescribing the behavior of WP geodesics]{Prescribing the behavior of Weil-Petersson geodesics in the moduli space of Riemann surfaces}

\date{\today}


\author{Babak Modami}
\address{Department of Mathematics, Yale University, 10 Hillhouse Ave, New Haven, CT, 06511}
\email{babak.modami@yale.edu}

\thanks{The author was partially supported by NSF grant DMS-1005973}




\begin{abstract}
We study Weil-Petersson (WP) geodesics with {\it narrow end invariant} and develop techniques to control length-functions and twist parameters along them and prescribe their itinerary in the moduli space of Riemann surfaces. This class of  geodesics is rich enough to provide for examples of closed WP geodesics in the thin part of the moduli space, as well as divergent WP geodesic rays with minimal filling ending lamination. 

Some ingredients of independent interest are the following: A strength version of Wolpert's Geodesic Limit Theorem proved in $\S$\ref{sec : bddwpgeodseg}. The stability of hierarchy resolution paths between narrow pairs of partial markings or laminations in the pants graph proved in $\S$\ref{sec : stablehrpath}. A kind of symbolic coding for laminations in terms of subsurface coefficients presented  in $\S$\ref{sec : preendinv}.
\end{abstract}

\maketitle

\setcounter{tocdepth}{3}
\tableofcontents

\section{Introduction}

The Weil-Petersson (WP) metric on the moduli space of Riemann surfaces is a Riemannian metric with negative sectional curvatures. The metric is incomplete and its sectional curvatures approach to both $0$ and $-\infty$   asymptotic near the completion. These features prevent applying most of the standard techniques available in the study of the global geometry and dynamics of complete Riemannian metrics with negatively pinched curvatures (see e.g. \cite{behaviorneg}, \cite{ebergoedflownegcurv}, \cite[Part 4]{KatokDyn}) to the WP metric. The main theme of this paper and the pioneer work of Brock, Masur and Minsky in \cite{bm},\cite{bmm1},\cite {bmm2} is to apply combinatorial techniques from surface theory to study the global behavior of WP geodesics. 
 
In \cite{bmm1} Brock, Masur and Minsky introduced a notion of {\it ending lamination} for WP geodesic rays. They showed that the ending lamination determines the strong asymptotic class of a WP geodesic ray recurrent to a compact subset of the moduli space. In \cite{bmm2} a more explicit connection between the combinatorics of the ending laminations of a WP geodesic and its behavior was established: A necessary and sufficient combinatorial condition for a WP geodesic to stay in the compact part of the moduli space was proved. In this paper we prove the following two results about the behavior of WP geodesics in the moduli space of Riemann surfaces:
\begin{thm}\label{thm : closed}\textnormal{(Closed geodesics in the thin part)}
Given any compact subset of the moduli space $\mathcal{K}$, there are infinitely many closed Weil-Petersson geodesics which do not intersect $\mathcal{K}$. 
\end{thm}
\begin{thm}\label{thm : div}\textnormal{ (Divergent geodesics)}
Starting from any point in the moduli space there are uncountably many divergent WP geodesic rays with minimal filling ending lamination.
\end{thm}
 The WP volume of the moduli space is finite, so by the Poincar\'{e} recurrence theorem almost every WP geodesic ray is recurrent to a compact subset of the moduli space. However the second theorem above exhibits the abundance of WP geodesic rays divergent in the moduli space. These geodesics diverge in the moduli space by getting closer and closer to a chain of completion strata. 
 \medskip

 Given a WP geodesic $g:(a,b) \to \Teich(S)$ denote the {\it end invariant} associated to its forward trajectory by $\nu^{+}$ and the one associated to its backward trajectory by $\nu^{-}$ ($\nu^{-}$ and $\nu^{+}$ are either laminations or (partial) markings) (see $\S$\ref{subsec : WPendinv}).

 To each essential subsurface $Z \subseteq S$ that is not a three holed sphere, there is an associated {\it subsurface coefficient} denoted by
 $$d_{Z}(\nu^{-},\nu^{+})$$
 which is the distance in the curve complex of $Z$ between the projections of $\nu^{-}$ and $\nu^{+}$ (see $\S$\ref{sec : ccplx}). 

Subsurface coefficients are an analogue of continued fraction expansions which provide for a coding of geodesics on the modular surface which is the moduli space of one hold tori $\mathcal{M}(S_{1,1})$ (see e.g. \cite{modsurf}).  Conjecturally these coefficients provide for extensive information about the behavior of Weil-Petersson geodesics in the moduli space. The following conjecture was proposed in \cite{bmm1}:
 
 \begin{conjecture}\label{conj : coeffsh}
 Let $g$ be a Weil-Petersson geodesic with end invariant $(\nu^{-},\nu^{+})$, we have
 \begin{enumerate}\label{con : shcoeff}
 \item For every $\epsilon>0$ there is an $A>0$, such that if $d_{Z}(\nu^{-}, \nu^{+})>A$ then $\inf_{t}\ell_{\alpha}(g(t))\leq\epsilon$ for every $\alpha\in\partial{Z}$.
  \item For every $A>0$ there is an $\epsilon>0$, such that if $\inf_{t}\ell_{\alpha}(g(t))\leq\epsilon$ then there is a subsurface $Z\subsetneq S$ such that $\alpha\in\partial{Z}$ and $d_{Z}(\nu^{-},\nu^{+})>A$.
 \end{enumerate}
 \end{conjecture} 
  Furthermore, it would be very useful to have an estimate for the length of the time interval that a curve is short (has length less than a given positive $\epsilon$) along a WP geodesic in terms of its end invariant and the associated subsurface coefficients. 
 
 The WP metric exhibits different features in the thick and thin parts of the Teichm\"{u}ller space (see $\S$\ref{sec : WPmetric}). For instance in the thick part the sectional curvatures are all bounded away from both 0 and $-\infty$, while in the thin part the WP metric is almost a product metric with sectional curvatures asypmtotic to both $0$ and $-\infty$. Therefore, to answer questions about the global geometry and dynamics of the WP metric often one needs to determine the {\it itinerary of geodesics}. By itinerary of a geodesic $g$ we mean the thin parts of the Teichm\"{u}ller space that $g$ visits, the order that $g$ visits these parts and the time that $g$ spends in each one of these parts.

 After the work of Masur-Minsky \cite{mm1},\cite{mm2} and Rafi \cite{rshteich},\cite{rteichhyper}, the first part of Conjecture \ref{conj : coeffsh} and a variation of the second part of the conjecture hold for Teichm\"{u}ller geodesics and provides for a complete picture of the itinerary of Teichm\"{u}ller geodesics in the moduli space. These results heavily rely on the explicit description of the Riemann surfaces along Teichm\"{u}ller geodesics in terms of contraction and expansion of measured foliations. But such a description does not exist for Weil-Petersson geodesics. 

The underlying machinery to control the behavior of WP geodesics is the Masur-Minsky machinery of hierarchies and their resolutions in the pants and marking graphs introduced in \cite{mm2}. Given a pair of partial markings or laminations a {\it hierarchy (resolution) paths} $\rho: [m,n]\to P(S)$ ($[m,n]\subseteq\mathbb{Z}\cup\{\pm\infty\}$) is a quasi-geodesic between the pair in the pants graph with certain properties encoded in the pair and their subsurface coefficients. For example corresponding to any subsurface $Z$ with big enough subsurface coefficient, there is a subinterval of $[m,n]$ denoted by $J_{Z}$ so that $\partial{Z}\subset\rho(i)$, for every $i\in J_{Z}$. In Theorem \ref{thm : hrpath} some of the properties are listed. 
  
By a result of Jeff Brock (Theorem \ref{thm : brockqisom}) the Bers pants decompositions along a WP geodesic $g$ trace a quasi-geodesic $Q(g)$ in the pants graph of the surface. When $Q(g)$ and a hierarchy path $\rho$ fellow travel (see Definition \ref{def : fellowtr}) there is a correspondence between the parameters of the hierarchy path and the parameters of the WP geodesic $g$, which roughly speaking is the nearest point correspondence of fellow traveling paths. In this situation we use the hierarchy machinery to determine the itinerary of WP geodesics. For example, let $U_{\epsilon}(\partial{Z})$ be the thin part of the Teichm\"{u}ller space where all of the boundary curves of a subsurface $Z$ are shorter than $\epsilon$. Then we show that $g$ visits $U_{\epsilon}(\partial{Z})$ over a suitably shrunk subinterval of $g$ that corresponds to $J_{Z}$. 
\medskip

An $\it{A-narrow}$ condition on the end invariant $(\nu^{-}, \nu^{+})$ is a constraint on the set of subsurfaces with big subsurface coefficient. More precisely, the pair is $A-$narrow if every non-annular subsurface $Z \subseteq S$ with
$$d_{Z}(\nu^{-}, \nu^{+}) > A$$
is a {\it large subsurface} i.e. the complement of $Z$ consists of only annuli and three holed spheres. 

The $A-$narrow condition on the end invariant implies uniform fellow traveling, depending only on $A$, of the Bers curves along a WP geodesic segment and a hierarchy path between the pair (Theorme \ref{thm : stability}). Heuristically hierarchy paths with narrow end invariant avoid quasi-flats in the pants graph corresponding to separating multi-curves on a surface, and WP geodesics with narrow end invariant avoid asymptotic flats in the completion of the WP metric which correspond to pinching separating multi-curves on the surface. In this paper we develop a control for length-functions and twist parameters along geodesics with narrow end invariant and show that their itinerary mimic combinatorial properties of hierarchy paths.  
In order to prove our main technical results we introduce the following notions:
 
Let $Z \subseteq S$ be an essential subsurface with $d_{Z}(\nu^{-},\nu^{+})$ sufficiently big. We say that $Z$ has {\it $(R,R')-$bounded combinatorics} over a subinterval $[m',n'] \subset J_{Z}$ if for every non-annular subsurface $Y \subsetneq Z$, 
 $$ d_{Y}(\rho(m'),\rho(n'))\leq R,$$
 and for every annular subsurface $A(\gamma)$ with core curve $\gamma$ inside $Z$, 
 $$ d_{\gamma}(\rho(m'),\rho(n'))\leq R'.$$
 This condition is a local version of the bounded combinatorics of the end invariant in \cite{bmm2}. A WP geodesic with bounded combinatorics end invariant stays in the thick part of the moduli space, \cite{bmm2}.  In the direction of Conjecture \ref{conj : coeffsh} we prove:
\begin{thmshcurve}\textnormal{ (Short Curve)} 
Given $A,R,R'>0$ and a sufficiently small $\epsilon>0$, there is a constant $\bar{w}=\bar{w}(A, R, R',\epsilon)$ with the following property. Let $g:[a,b] \to \Teich(S)$ be a WP geodesic segment with $A-$narrow end invariant $(\nu^{-},\nu^{+})$. Let $\rho:[m,n] \to P(S)$ be a hierarchy path between $\nu^{-}$ and $\nu^{+}$. Suppose that $Z$ a large component domain of $\rho$ has $(R,R')-$bounded combinatorics over an interval $[m',n']\subset J_{Z}$. 

If $n'-m' \geq 2\bar{w}$, then for every $\alpha \in \partial{Z}$ we have
$$\ell_{\alpha}(g(t)) \leq  \epsilon$$
for every $t \in [a',b']$, where $a'$ and $b'$ are the corresponding times to $m'+\bar{w}$ and $n'-\bar{w}$, respectively.
\end{thmshcurve}
Here the time correspondence comes from fellow traveling of the Bers curves along a WP geodesic with narrow end invariant $(\nu^{-},\nu^{+})$ and a hierarchy path between $\nu^{-}$ and $\nu^{+}$ (see $\S$\ref{subsec : fellowtr}).

In $\S$\ref{sec : itinerarywpgeod} we use bounded combinatorics intervals to isolate annular subsurface along a hierarchy path. The twist parameter buildup about an {\it isolated annular subsurface} along a WP geodesic is comparable to the one along a fellow traveling hierarchy path. This together with the length-function versus twist parameter controls over uniformly bounded length WP geodesic segments we develop in $\S$\ref{sec : bddwpgeodseg}  are the main technical tools in $\S$\ref{sec : itinerarywpgeod} where we prove the above theorem.  

Using the control we develop on length-functions along WP geodesic segments and by extracting limits of WP geodesic segments with narrow end invariant we construct WP geodesic rays with prescribed itinerary in the moduli space (Theorem \ref{thm : inftyray}). Itineraries of these rays mimic the combinatorial properties of hierarchy paths encoded in the end points and the associated subsurface coefficients. In $\S$\ref{sec : preendinv} we construct pairs of laminations or markings with prescribed list of subsurface coefficients. This is a kind of symbolic coding for laminations in terms of subsurface coefficients, an analogue of continued fraction expansions. The geodesic rays with prescribed itinerary corresponding to these laminations are used in $\S$\ref{sec : examples} to construct examples of WP geodesics with prescribed behaviors in the moduli space. These constructions could be considered as a kind of symbolic coding for WP geodesics. 

The fellow traveling property of hierarchy paths is a crucial part of the combinatorial frame work to control length-functions along WP geodesics. In $\S$\ref{sec : stablehrpath} we prove the following stability result for hierarchy paths in the pants graph of surfaces:   
 \begin{thmstable}\textnormal{ (Stable hierarchy path)}
 Given $A>0$ there is a function $d_{A}$ such that any hierarchy path with $A-$narrow end points is $d_{A}-$stable in the pants graph. 
  \end{thmstable}
 In the above theorem $d_{A}:\mathbb{R}^{\geq 1}\times \mathbb{R}^{\geq 0}\to \mathbb{R}^{\geq 0}$ is the quantifier of the stability (see Definition \ref{def : stable}).

\subsection{Outline of the paper} 

Section \ref{sec : ccplx} is devoted to the background about curve complexes and some important notions and results in the setting of pants graphs and marking graphs. In this section we recall hierarchical structures of pants and marking graphs and their resolutions introduced by Masur and Minsky in \cite{mm2}. The properties of hierarchy resolution paths are listed in Theorem \ref{thm : hrpath}. We also recall the $\Sigma-$hulls and their projections from \cite{bkmm}. In Section \ref{sec : WPmetric} we provide some background about the WP metric and synthetic properties of WP geodesics. 
 
 In Section \ref{sec : bddwpgeodseg} we prove refined versions of some of the key results in \cite{wols} and \cite{bmm2}. The proofs are mainly based on compactness arguments in the WP completion of the Teichm\"{u}ller space. These results give us a kind of control on buildup of Dehn twists versus change of length-functions along uniformly bounded length WP geodesic segments. 
 
 In Section \ref{sec : stablehrpath} we prove that hierarchy paths between narrow pairs are stable. The proof will be through $\Sigma-$hulls and their stability properties. 

 In Section \ref{sec : itinerarywpgeod}, we develop some new techniques to control length-functions and twist parameters along WP geodesics fellow traveling hierarchy paths. 
 
 In Section \ref{sec : preendinv} we construct pairs of markings or laminations with a prescribed list of subsurface coefficients. 
  
In Section \ref{sec : examples} we prove Theorems \ref{thm : closed} and \ref{thm : div}.
\medskip

\noindent{\bf Acknowledgment:} I am so grateful to my thesis advisor Yair Minsky for so many invaluable discussions through which this work evolved. I would like to thank Jeffery Brock who generously contributed some of the ideas in Section \ref{sec : bddwpgeodseg}. I am also grateful to Scott Wolpert for willingly answering my questions about the Weil-Petersson metric. 
 
  \section{Curve complexes and hierarchical structures} \label{sec : ccplx}
  In this paper we use the following notation
  \begin{notation} Given $K \geq 1$ and $C\geq 0$. Let $f,g:X\to \mathbb{R}^{\geq 0}$ be two functions on a set $X$. Then $f \asymp_{K,C}g$ means that for every $x\in X$ we have 
$$\frac{1}{K}g(x)-C \leq f(x) \leq Kg(x)+C$$ 
\end{notation}

\noindent{\bf The curve complex of a surface:} Let $S=S_{g,b}$ be a connected, compact, oriented surface with genus $g$ and $b$ boundary components. The {\it topological type} of $S$ refers to $g$ and $b$ which determines the surface up to homeomorphism.  We define $\xi(S)=3g-3+b$ to be the {\it complexity} of the surface $S$. 
 
 An {\it essential curve} is a closed curve which is not isotopic to a boundary component of $S$ or a point. We do not distinguish between a curve and its isotopy class. The curve complex of $S$, denoted by $\mathcal{C}(S)$, serves to organize the isotopy classes of essential, simple closed curves on $S$. Let $S$ be a surface with $\xi(S)\geq 1$. To the isotopy class of each essential simple closed curve on $S$ is associated a vertex ($0-$simplex) in $\mathcal{C}(S)$. When $\xi(S)>1$, an edge is associated to disjoint pair of isotopy classes of curves. Similarly, a $k-$simplex is associated to any $k+1$ pairwise disjoint isotopy classes of simple closed curves. Two isotopy classes are disjoint if there are curves in each of them which are disjoint on $S$. We denote the $k-$skeleton of $\mathcal{C}(S)$ by $\mathcal{C}_{k}(S)$. When $\xi(S)=1$, $S$ is either a one holed torus or a four holed sphere. Then $1-$simplices (edges) correspond to curves intersecting respectively once and twice, which are the minimum possible intersection number of curves on $S_{1,1}$ and $S_{0,4}$, respectively. 
 
  A {\it multi-curve} is a collection of pair-wise disjoint essential simple closed curves. If $\xi(S)>1$, then a multi-curve consists of the vertices of a simplex in $\mathcal{C}(S)$. 
 
 We equip the curve complex with a distance by making each simplex Euclidean with side lengths $1$, and denote the distance $d_{\mathcal{C}(S)}$ by $d_{S}$. By the main result of  Masur-Minsky in \cite{mm1} $\mathcal{C}(S)$ equipped with $d_{S}$ is a $\delta-$hyperbolic space in the sense of Gromov and $\delta$ depends only on the topological type of $S$. 
 
A subsurface of $S$ is an embedded, connected, compact surface inside $S$. We do not distinguish between a subsurface and its isotopy class. An {\it essential subsurface} is a subsurface $Y\subseteq S$ so that each boundary curve of $Y$ is either an essential curve or a boundary curve of $S$, and $Y$ itself is not a $3-$holed sphere. In this paper unless is specified subsurfaces understood to be essential.
 
 An {\it annular subsurface} is an annulus $Y$ with essential core curve in $S$. The purpose of defining complexes for annuli is to keep track of Dehn twists about their core curves. These complexes are quasi-isometric to $\mathbb{Z}$. Let $Y$ be an annulus with core curve $\alpha$. Let $\widetilde{Y}$ be the annular cover of $S$ to which $Y$ lifts homeomorphically. There is a natural compactification of $\widetilde{Y}$ to a closed annulus $\widehat{Y}$ obtained in the usual way from the compactification of the universal cover $\mathbb{D}^{2}$ (Poincar\'{e} disk) by the closed disk. A vertex of $\mathcal{C}(Y )$ is a path connecting the two boundary components of $\widehat{Y}$ modulo isotopies that fix the endpoints (isotopy classes of arcs relative to the boundary). An edge is associated to two vertices which have representatives with disjoint interiors. The curve complex of an annular subsurface can be equipped with a metric by assigning length $1$ to each edge. We write $A(\alpha)=Y$ and $\mathcal{C}(\alpha)=\mathcal{ C}(Y)$. 
 
Let $Y$ be an essential subsurface. We denote the diameter of a subset $A\subset\mathcal{C}_{0}(Y)$ by $\diam_{Y}(A)$.
  \medskip
 
\noindent{\bf Overlap:}   If there are representatives of the isotopy classes of curves $\alpha$ and $\beta$ which are disjoint, then $\alpha$ and $\beta$ are disjoint. Otherwise, $\alpha$ and $\beta$ {\it overlap} each other which is denoted by $\alpha \pitchfork \beta$. Two multi-curves $\sigma$ and $\sigma'$ are disjoint if any pair of curves $\alpha \in \sigma$ and $\alpha' \in \sigma'$ are disjoint, otherwise $\sigma$ and $\sigma'$ overlap, denoted by $\sigma\pitchfork\sigma'$.

A curve or lamination {\it intersects a subsurface $Y\subset S$ essentially}(overlaps) if it can not be homotoped to the complement of $Y$ inside $S$. Let $\sigma$ be a multi-curve, then $\sigma$ overlaps $Y$, denoted by $\sigma \pitchfork Y$, if at least one of the curves in $\sigma$ overlaps $Y$. Let $Y,Z\subseteq S$. If $\partial{Y}\pitchfork Z$ and $\partial{Z}\pitchfork Y$, we say the subsurfaces $Y$ and $Z$ overlap, and denote it by $Y \pitchfork Z$. 
\medskip

\noindent {\bf Laminations:} Let $S$ be a surface equipped with a complete hyperbolic metric. A {\it geodesic lamination} on $S$ is a closed subset $\lambda$ of $S$ consisting of disjoint, complete, simple geodesics.  Let $\widetilde{S}=\mathbb{D}^{2}$ be the universal cover of $S$. Denote the boundary at infinity of the Poincar\'{e} disk $\mathbb{D}^{2}$ by $\widetilde{S}_{\infty}$. Let 
$$M_{\infty}(S)=(\widetilde{S}_{\infty}\times \widetilde{S}_{\infty}\backslash \Delta)/\sim,$$
 where $\Delta=\{(x,x):x\in S^{1}_{\infty}\}$ and $\sim$ is the equivalence relation generated by $(x,y)\sim(y,x)$. Since the geodesics in $\mathbb{D}^{2}$ are parametrized by points of $M_{\infty}(S)$ the preimage of a geodesic lamination determines a closed subset of $M_{\infty}(S)$ which is invariant under the action of $\pi_{1}(S)$.
We denote the space of geodesic laminations on the surface $S$ equipped with the Hausdorff topology of closed subsets of $M_{\infty}(S)$ by $\mathcal{GL}(S)$. The space $\mathcal{GL}(S)$ is a compact space. For more detail see e.g. \cite[\S I.4]{notesonthurston}. 

A geodesic lamination $\lambda$ can be equipped with transverse measures. A {\it transverse measure} $m$ supported on $\lambda$ is a measure on arcs in $S$ which is invariant under isotopies of the surface $S$ which preserve $\lambda$. The fact that the measure is supported on $\lambda$ means that 
\begin{itemize}
\item an arc $a$ with $a\cap\lambda\neq\emptyset$ and transversal to $\lambda$ has positive measure,
\item an arc $a$ with $a\subset \lambda$ or $a\cap\lambda=\emptyset$ has measure $0$.
\end{itemize}
 The pull back of a transverse measure on $\lambda$ determines a measure on $M_{\infty}(S)$ supported on the preimage of $\lambda$ in $M_{\infty}(S)$. A {\it measured (geodesic) lamination} $\mathcal{L}=(\lambda,m)$ is the pair of a geodesic lamination $\lambda$ and $m$ a transverse measure supported on $\lambda$. We denote the space of measured geodesic laminations of $S$ equipped with the weak$^{*}$ topology by $\mathcal{ML}(S)$. For more detail see \cite{bonlam}, \cite{phtraintr}. 

The group $\mathbb{R}^{+}$ acts on $\mathcal{ML}(S)$ by rescaling of measures, that is $s(\lambda,m)=(\lambda,sm)$. Then the quotient, $\mathcal{PML}(S):=\mathcal{ML}(S)-\{0\}/\mathbb{R}^{+}$ is the space of projective classes of measured laminations. The projective class of a measured lamination $\mathcal{L}$ is denoted by $[\mathcal{L}]$. We equip $\mathcal{PML}(S)$ with the quotient topology of the weak$^{*}$ topology on $\mathcal{ML}(S)$. 

A geodesic lamination $\lambda$ fills $S$ if the connected components of $S\backslash \lambda$ are only ideal polygons or cusped ideal polygons with sides the leaves of $\lambda$. Note that if $\lambda$ is filling, then any simple closed curve and $\lambda$ intersect essentially. The lamination $\lambda$ is {\it minimal} if any half leaf of $\lambda$ is dense in $\lambda$. Given $[\mathcal{L}]\in\mathcal{PML}(S)$, let $|\mathcal{L}|$ be the support of $\mathcal{L}$. Then taking the quotient  
$$\mathcal{PML}(S)/|.|$$
of $\mathcal{PML}(S)$ by forgetting the measure, the {\it ending lamination space}
 $$\mathcal{EL}(S)\subset\mathcal{PML}(S)/|.|$$
 is the image of projective measured laminations with minimal filling support equipped with the quotient topology of the topology of $\mathcal{PML}(S)$. 

The Gromov boundary of a $\delta-$hyperbolic space has a standard topology (see e.g. \cite[$\S$III.H.3]{bhnpc}). Klarreich in \cite{bdrycc} proves that
\begin{thm}\label{thm : ccbdry}
There is a homeomorphism $\Phi$ from the Gromov boundary of $\mathcal{C}(S)$ to $\mathcal{EL}(S)$.
\end{thm}
Furthermore Klarreich describes a relation between the point in the Gromov boundary of $\mathcal{C}(S)$ to which a sequence of curves converges and the accumulation points of the sequence of curves in $\mathcal{PML}(S)$.

\begin{thm}\cite[Theorem 1.4 ]{bdrycc}\label{thm : gobdrycc}
Suppose that $\{\alpha_{i}\}_{i=1}^{\infty}$ is a sequence of vertices in $\mathcal{C}_{0}(S)$ that converges to a point $x$ in the Gromov boundary of $\mathcal{C}(S)$. Then regarding each $\alpha_{i}$ as a projective measured lamination every accumulation point of the sequence $\{\alpha_{i}\}_{i=1}^{\infty}$ in $\mathcal{PML}(S)$ is supported on $\Phi(x)$.
\end{thm}

\begin{remark}
Klarreich states and proves the above theorem in the setting of measured foliations. Using the main result of \cite{fol=lam} the theorem is equivalent to what we stated above.
\end{remark}

 The following notions of subsurface projection and subsurface coefficient are basic in the Masur-Minsky machinery of curve complexes and hierarchical structures on the pants and marking graphs. 
\medskip

\noindent {\bf Subsurface projection:} 
For an essential non-annular subsurface $Y \subseteq S$ define the {\it subsurface projection map}
 $$\pi_{Y}:\mathcal{GL}(S)\to \mathcal{P(C}_{0}(Y)) $$
 where $\mathcal{P(C}_{0}(Y))$ is the power set of $\mathcal{C}_{0}(S)$ as follows: Fix a complete hyperbolic metric on $S$. Given $\lambda \in \mathcal{GL}(S)$ realize $\lambda$ and $\partial{Y}$ geodesically. If $\lambda$ does not intersect $Y$ then $\pi_{Y}(\lambda)=\emptyset$. When $\lambda$ intersects $Y$ consider all of the properly embedded arcs (arcs with end points either on $\partial{Y}$ or at a cusp) and curves in the intersection locus $\lambda\cap Y$ (we ignore infinite leaves of $\lambda$ that are contained completely in the interior of $Y$, except the ones going between two cusps). Identify any two arcs or curves which are isotopic to each other inside $Y$. Through the isotopy the end points of arcs are allowed to move within the boundary of $Y$. Then $\pi_{Y}(\lambda)$ consists of the curves in the boundary of a regular neighborhood of $a\cup \partial{Y}$, for any arc $a$ we obtained above together with all of the closed curves we obtained above. Note that the diameter of $\pi_{Y}(\lambda)$ viewed as a subset of $\mathcal{C}(Y)$ is at most $2$.

Since $\mathcal{C}_{0}(S) \subset \mathcal{GL}(S)$, the above projection restricts to a projection
 $$\pi_{Y}:\mathcal{C}_{0}(S)\to \mathcal{P(C}_{0}(Y)).$$
 
The projection for an annular subsurface $Y$ is defined as follows: If $\lambda \in \mathcal{GL}(S)$ crosses $\alpha$ the core of $Y$ essentially, then the lift of $\lambda$ to $\widehat{Y}$ (the compactified annular cover to which $Y$ lifts homeomorphivcally) has at least one component that connects the two boundaries of $\widehat{Y}$. These components together make up a set of diameter $1$ in $\mathcal{C}(Y )$.  $\pi_{Y}(\lambda)$ is this set. If $\lambda$ does not intersect $Y$ essentially (including the case that $\lambda$ is the core of $Y$) then $\pi_{Y}(\gamma) = \emptyset$. 

The projection of a multi-curve $\sigma$ to a subsurface $Y$ is the union of $\pi_{Y}(\alpha)$ for all $\alpha\in \sigma$. 
\medskip
 
\noindent{\bf Subsurface coefficient:} Let $\lambda$ and $\lambda'$ be two multi-curves or laminations. Let $Y \subseteq S$ be an essential subsurface. We consider the following notion of distance
 \begin{equation}\label{eq : subsurfcoeff}d_{Y}(\lambda, \lambda'):=\min \{ d_{Y}(\gamma,\gamma') : \gamma\in \pi_{Y}(\lambda), \gamma'\in \pi_{Y}(\lambda')\} \end{equation}
 which provides for a useful notion of complexity of  $\lambda$ and $\lambda'$ from the point of view of the subsurface $Y$. We call $d_{Y}(\lambda,\lambda')$ the {\it Y subsurface coefficient} of $\lambda$ and $\lambda'$. 
 
Suppose that $\lambda$ or $\lambda'$ has a component that is minimal and fills the subsurface $Y$. By Theorem \ref{thm : ccbdry} the component determines a point in the Gromov boundary of $\mathcal{C}(Y)$, in particular the $Y$ subsurface projection of $\lambda$ or $\lambda'$ is not defined. In this case for convention we let $d_{Y}(\lambda,\lambda')=\infty$.
 
For an annular subsurface $Y$ with core curve $\alpha$ we denote $\pi_{\alpha}:=\pi_{Y}$ and define the annular subsurface coefficients of geodesic laminations or multi-curves $\lambda$ and $\lambda'$ using the formula (\ref{eq : subsurfcoeff}) and denote it by $d_{\alpha}$.

Subsurface coefficients play the role of continued fraction expansions for coding of laminations (see $\S$\ref{sec : preendinv}).

\begin{lem}\cite[Lemma 2.1]{mm1}\label{lem : di}
Let $W\subseteq S$ be an essential subsurface and $\alpha,\beta\in \mathcal{C}_{0}(S)$. We have 
$$d_{W}(\alpha,\beta)\leq 2i(\alpha,\beta)+1.$$
\end{lem}
\medskip

\noindent{\bf Filling curves:} We say that a collection of curves or arcs $\Gamma\subset S$ fill an essential subsurface $Y\subseteq S$ if any curve in $Y$ intersects $\Gamma$ essentially. 

Let $\alpha,\beta\in\mathcal{C}_{0}(S)$. Suppose that $d_{Y}(\alpha,\beta)> 4$, then $\alpha$ and $\beta$ fill $Y$. To see this, recall the surgery map which assigns to any arc $a$ with end points on $\partial{Y}$ the boundary curves of a regular neighborhood of $a\cup\partial{Y}$. As is shown in \cite[Lemma 2.2]{mm2} this a $2-$Lipschitz map from the arc complex of $Y$ to $\mathcal{PC}_{0}(Y)$. Then since $d_{Y}(\alpha,\beta)>4$, given curves arcs $a$ in $\alpha\cap Y$ and $b$ in $\beta\cap Y$ with end points on $\partial{Y}$, the distance in the arc complex of $Y$ between $a$ and $b$ is at least $3$. But this implies that $a$ and $b$ fill the subsurface $Y$ and therefore $\alpha$ and $\beta$ fill $Y$. For otherwise there is an arc or curve disjoint from $a$ and $b$ implying that the distance of $a$ and $b$ is at most $2$. But this contradicts the lower bound $3$ for the distance of $a$ and $b$ in the arc complex of $Y$.  

A similar argument shows that if $d_{Y}(\alpha,\beta)> 2$, then $\alpha\pitchfork\beta$.
\medskip

\noindent{\bf Hausdorff limit:} Here we provide two useful propositions regarding Hausdorff limits of geodesic laminations.

\begin{prop}\label{prop : hlimwk*lim}
Let $\mathcal{L}_{i}=(\lambda_{i},m_{i})$ be a sequence of measured laminations. Suppose that $\mathcal{L}_{i}$ converge to $\mathcal{L}=(\lambda,m)$ in the weak$^{*}$ topology. Let $\xi$ be the limit of a subsequence of $\lambda_{i}$ (an accumulation point of laminations $\lambda_{i}$) in the Hausdorff topology of $M_{\infty}(S)$. Then $\lambda \subseteq \xi$.  
\end{prop}
\begin{proof} Think of laminations $\lambda_{i}$ and $\lambda$ as closed subset of $M_{\infty}(S)$, and measures $m_{i}$ and $m$ as measures supported on $\lambda_{i}$ and $\lambda$ respectively. Let $x\in \lambda$. Let $I\subset M_{\infty}(S)$ be an open neighborhood of $x$. By the weak$^{*}$ convergence $m_{i}(I)\to m(I)$ as $i\to \infty$. Moreover, $m(I)>0$. Thus for all $i$ sufficiently large $m_{i}(I)>0$. Then $\lambda_{i}\cap I\neq \emptyset$, for otherwise $m_{i}(I)=0$. So there is $x_{i}\in \lambda_{i}$ so that $x_{i}\in I$. Therefore $x$ is in the limit of any convergent subsequence of laminations $\lambda_{i}$ in the Hausdorff  topology of $M_{\infty}(S)$.
\end{proof}

\begin{lem}\label{lem : Hdistdprojbd}
Let $X\subseteq S$ be an essential subsurface and $\lambda$ a multi-curve or lamination that intersects $X$ essentially. There is an $\epsilon>0$ so that if the Hasudorff distance of a curve $\gamma\in \mathcal{C}_{0}(S)$ and $\lambda$ is less than or equal to $\epsilon$, then 
$$\diam_{X}(\pi_{X}(\lambda)\cup\pi_{X}(\gamma))\leq 4.$$
\end{lem}
\begin{proof}
Fix a complete hyperbolic metric on $S$ and realize all the curves and laminations geodesically. Suppose that $X$ is a non-annular subsurface. Let $a$ be an arc in $\lambda\cap X$ with end points on $\partial{X}$. Let $\epsilon$ be so that that the $\epsilon$ neighborhood of $a\cup\partial{X}$ is a regular neighborhood in $X$. Denote the neighborhood by $\mathcal{N}$. Moreover suppose that there is a curve $\alpha$ in the boundary of $\mathcal{N}$ which is essential in the subsurface $X$. Then $\alpha$ is in $\pi_{X}(\lambda)$. Since $\gamma$ is within the $\epsilon$ Hausdorff distance of $\lambda$ the neighborhood $\mathcal{N}$ is also a regular neighborhood of an arc in $\gamma\cap X$ with end points on $\partial{X}$. Thus the essential curve $\alpha$ is in $\pi_{X}(\gamma)$. So $\pi_{X}(\gamma)\cap\pi_{X}(\lambda)\neq \emptyset$. Then the stated bound of the lemma follows from the fact that the diameter of $\pi_{X}(\gamma)$ and $\pi_{X}(\lambda)$ are bounded by $2$. 

Now suppose that $X$ is an annular subsurface. Let $\widehat{X}$ be the compactified annular cover of $S$ to which $X$ lifts homeomorphically. For $\epsilon$ sufficiently small the assumption that $\gamma$ is within the $\epsilon$ Hausdorff distance of $\lambda$ implies that there are lifts of $l$ a leaf of $\lambda$ and $\gamma$ to $\widehat{X}$ so that the interior of the lifts intersect at most once. This implies that $d_{X}(\gamma,\lambda)\leq 2$. Then again the sated bound of the lemma follows from the fact that the diameter of $\pi_{X}(\gamma)$ and $\pi_{X}(\lambda)$ are bounded by $1$.
\end{proof}

As straightforward consequence of Lemma \ref{lem : Hdistdprojbd} is the following proposition.
\begin{prop}\label{prop : Hlimsubsurfcoeff}  Suppose that a sequence of curves $\alpha_{k}$ converges to a lamination $\xi$ in the Hausdorff topology of $M_{\infty}(S)$. Then given an essential subsurface $Z$, for any $\lambda'\in \mathcal{GL}(S)$ we have  
 $$d_{Z}(\lambda',\xi) \asymp_{1,4} d_{Z}(\lambda',\alpha_{k}),$$
  for all $k$ sufficiently large. 
  \end{prop}
 
\subsection{The pants graph and marking graph}\label{subsec : pants-marking}

A {\it pants decomposition} $P$ on a surface $S$ is a maximal collection of pairwise disjoint simple closed curves.  A {\it (partial) marking} $\mu$ consists of a pants decomposition which is  the base of the marking and is denoted by $\base(\mu)$ together with $t_{\alpha}$ a diameter $1$ subset of $\mathcal{C}(\alpha)$  for (some) every $\alpha \in \base{\mu}$. A {\it clean marking} is a marking such that each $t_{\alpha}$ is represented by a curve on $S$ which does not intersect any curve in $P-\alpha$ and has minimal intersection number $1$ or $2$ with $\alpha$ depending on whether the subsurface $S\backslash \{P-\alpha\}$ is a one holed torus or a four holed sphere.
\medskip

\noindent{\bf Pants graph:} Each vertex of the pants graph of $S$ is a pants decomposition of $S$. Each edge in the graph corresponds to two pants decompositions which differ by an elementary move. Two pants decomposition differ by an elementary move if one is obtained from the other one by an elementary move. A pants decomposition $P'$ is obtained from a pants decomposition $P$ by an elementary move if $P'$ is obtained from $P$ by replacing one curve $\alpha\in P$ with a curve $\alpha'$ in the curve complex of $S\backslash\{P-\alpha\}$ with minimal intersection number (1 or 2) with $\alpha$ and fixing the rest of curves in $P$. Assigning length $1$ to each edge defines the distance $d$ on $P(S)$ and makes $P(S)$ a metric graph. 
\medskip 

\noindent{\bf Marking graph:} Each vertex of the marking graph is a marking, each edge corresponds to a pair of markings which differ by an elementary move. An elementary move on a marking roughly speaking is either an elementary move on the base of the marking or is an interchange of a curve in the base and its transversal curve. For more detail see $\S$2.5 of \cite{mm2}. Assigning length one to each edge defines the distance $\tilde{d}$ on the marking graph and makes it a metric graph.
 \medskip

\noindent{\bf Projection of (partial) markings:} Let $\mu$ be a partial marking. Suppose that $Y$ is an essential non-annular subsurface, then $\pi_{Y}(\mu)=\pi_{Y}(\base(\mu))$. The projection of $\mu$ to an annular subsurface with core curve $\alpha$ is the following: If $\alpha\in \base(\mu)$ then $\pi_{\alpha}(\mu)=t_{\alpha}$. Otherwise, $\pi_{\alpha}(\mu)=\pi_{\alpha}(\base(\mu))$. Given two partial markings or laminations $\mu,\mu'$ the subsurface coefficient 
$$d_{Y}(\mu,\mu')$$
is defined by the formula (\ref{eq : subsurfcoeff}). Also in the case that $\mu$ or $\mu'$ has  a component that is minimal and fills $Y$ we let $d_{Y}(\mu,\mu')=\infty$.

\begin{lem}\label{lem : diamproj}
Let $S$ be a surface with $\xi(S)\geq 1$. Let $Y\subseteq S$ be an essential subsurface.  Given a multi-curve or a (partial) marking $\mu$ on $S$ so that $\pi_{Y}(\mu)\neq \emptyset$, we have 
$$\diam_{Y}(\mu)\leq 2.$$
Let $X$ be an essential subsurface. Suppose that $Y\subseteq X$ and $\pi_{X}(\mu)$ contains a curve which overlaps $Y$, then
$$\diam_{Y}(\pi_{Y}(\mu)\cup\pi_{Y}(\pi_{X}(\mu)))$$
is uniformly bounded.
\end{lem}
\begin{proof}
For a multi-curve the first bound of the lemma is \cite[Lemma 2.3]{mm2}. Suppose that $\mu$ is a (partial) marking. Let $\alpha$ be the core curve of $Y$. If $\alpha$ is a curve in $\base(\mu)$, then the bound follows since $\pi_{Y}(\mu)=t_{\alpha}$ and $\diam_{Y}(t_{\alpha})\leq 1$. If $\alpha\not\in\base(\mu)$, then $\pi_{Y}(\mu)=\pi_{Y}(\base(\mu))$ and the bound follows from the bound on the diameter of the projection of multi-curves. The first bound is proved. The second bound is part of \cite[Lemma 2.12]{bkmm}.
\end{proof}

\noindent{\bf Twist parameter:} Let $\alpha\in \mathcal{C}_{0}(S)$. Let $\mu,\mu'$ be two laminations or partial markings. Denote the components of the lift of $\mu$ to the compactification of the annular cover $\widehat{Y}$ going between the two boundary components by $\lift_{\alpha}(\mu)$. The {\it twist parameter} is the subset of $\mathbb{Z}$
$$tw_{\alpha}(\mu,\mu'):=\min\{a.b : a\in \lift_{\alpha}(\mu),b\in \lift_{\alpha}(\mu')\}.$$
where $a.b$ denotes the algebraic intersection number of $a$ and $b$. The arcs $a$ and $b$ are oriented so that intersect the lift of $\alpha$ homotopic to the core of $\widehat{Y}$ in the same direction. Note that the diameter of $tw_{\alpha}(\mu,\mu')$ is at most $2$.

 The twist parameter satisfies
$$|tw_{\alpha}(\mu,\mu')|\asymp d_{\alpha}(\mu,\mu').$$

\noindent{\bf The triangle inequality:} Let $Y\subseteq S$ be a subsurface and $\mu$ be a multi-curve, lamination or a partial marking on surface $S$. Denote by $\diam_{Y}(\mu)$ the diameter of $\pi_{Y}(\mu)$ viewed as a subset of $\mathcal{C}(Y)$.

Let $\mu_{1},\mu_{2},\mu_{3}$ be multi-curves, laminations or partial markings. Suppose that $d_{Y}(\mu_{1},\mu_{2})=d_{Y}(\alpha,\beta)$ and $d_{Y}(\mu_{2},\mu_{3})=d_{Y}(\beta',\gamma)$, where $\alpha\in \pi_{Y}(\mu_{1})$, $\beta,\beta'\in\pi_{Y}(\mu_{2})$ and $\gamma\in\pi_{Y}(\mu_{3})$. Then by the triangle inequality in $\mathcal{C}(Y)$ we have that 
$$d_{Y}(\alpha,\beta)+d_{Y}(\beta,\beta')+d_{Y}(\beta',\gamma)\geq d_{Y}(\alpha,\gamma).$$
 Now since $d_{Y}(\beta,\beta')\leq \diam_{Y}(\mu_{2})$ and $d_{Y}(\alpha,\gamma)\geq d_{Y}(\mu_{1},\mu_{2})$ we conclude that
$$d_{Y}(\mu_{1},\mu_{2})+d_{Y}(\mu_{2},\mu_{3})+\diam_{Y}(\mu_{2})\geq d_{Y}(\mu_{1},\mu_{3}).$$ 
In the interest of brevity in this paper we refer to this inequality also as the triangle inequality.
\medskip

The following theorem plays an important role in the organization of the so called tight geodesics in curve complexes of a surface and its subsurfaces and the inductive construction of hierarchies in the pants and marking graphs of a surface.
 
 \begin{thm}\label {thm : bddgeod}\textnormal{(Bounded Geodesic Image)}\cite[Theorem 3.1]{mm2}
 There exists a constant $G>0$ depending only on the topological type of $S$ with the following property. Let $Y \subsetneq S$ be an essential subsurface and let $g:I \to \mathcal{C}(S)$ be a geodesic such that for every $i \in I$, $\pi_{Y}(g(i)) \neq \emptyset$ i.e. $g(i) \pitchfork Y$. Then we have
 $$\diam_{Y}(\{\pi_{Y}(g(i)) : i \in I\}) \leq G.$$
 \end{thm}
 \medskip
 
 In the rest of this subsection we recall some results in the context of pants and marking graphs which we use often in this paper.
 \medskip

 \noindent{\bf  Distance formula:} Let $A\geq 0$, define the cut-off function $\{.\}_{A}:\mathbb{R}\to \mathbb{R}^{\geq 0}$ by $$\{a\}_{A} = \Big\lbrace 
   \begin{array}{cl}
  a & \text{if}\;a \geq A\\
  0 & \text{if }a < A \\
  \end{array}.$$ 
  Masur and Minsky \cite[Theorem 6.12]{mm2} prove the following quasi distance formula for the pants and marking graph distance.  

There exists a constant  $M_{1}>0$ such that for any $A\geq M_{1}$, there are constants $K\geq1$ and $C\geq 0$, such that the distance between any two pants decompositions $P$ and $Q$ staifies
\begin{equation} \label{eq : dsf} d(P, Q) \asymp_{K,C} \sum_{\substack{ Y\subseteq S\\ \nonannular}}\{ d_{Y}(P, Q) \}_{A}. \end{equation}
 Note that the sum is only over essential non-annular subsurfaces. We call $A$ the threshold constant and say that $K$ and $C$ are the constants corresponding to the threshold constant $A$.
  \begin{remark}
  Given partial markings $\mu$ and $\mu'$ if on the right hand side of (\ref{eq : dsf}) the sum is over all subsurface coefficients, including the annular subsurface coefficients, then (\ref{eq : dsf}) holds for the marking distance of $\mu$ and $\mu'$.
 \end{remark}
  
  \begin{thm}\textnormal{ (Behrstock Inequality)}\cite{beh} \label{thm : behineq}
 There exists a constant $B_{0}>0$ with the property that given two subsurfaces $Y,Z \subseteq S$ with $Y \pitchfork Z$ and a partial marking or lamination $\mu$ such that $\mu\pitchfork Y$ and $\mu\pitchfork Z$ we have
  $$\min\{d_{Y}(\partial{Z},\mu),d_{Z}(\partial{Y},\mu)\} \leq B_{0}.$$
  \end{thm}
 
 We recall Consistency Theorem of Behrstock and Minsky from \cite{bkmm}. There the theorem is stated and proved for markings. It is straightforward to verify that all of their arguments go through for pants decompositions, considering only non-annular subsurfaces (excluding all annular subsurfaces).
 
 \begin{thm} \label {thm : consistency} \textnormal{ (Consistency)}\cite[Theorem 4.3]{bkmm}
 Given $F_{1},F_{2} \geq 1$, there is a constant $F>0$ with the following property. Let $(x_{Y})_{Y \subseteq S}$ be a tuple where $x_{Y} \in\mathcal{C}_{0}(Y)$ for each essential non-annular subsurface $Y\subseteq S$. Suppose that  $(x_{Y})_{Y \subseteq S}$ satisfies the following two consistency conditions
  \begin{enumerate}[1)]
 \item If $Y \pitchfork Z$ then $\min\{d_{Y}(x_{Y}, \partial{Z}),d_{Z}(x_{Z}, \partial{Y})\} \leq F_{1}$, and
 \item If $Y \subseteq Z$ and $d_{Z}(x_{Z},\partial{Y})>F_{2}$, then $d_{Y}(x_{Y},\pi_{Y}(x_{Z})) \leq F_{1}.$ 
 \end{enumerate}
 Then there is $P \in P(S)$ such that for every essential non-annular subsurface $Y\subseteq S$ we have
  $$d_{Y}(P, x_{Y})\leq F.$$
   \end{thm}
  \begin{example}\label{ex : order}
 Given a pants decomposition $P$, consider the tuple $(x_{Y})_{Y\subseteq S}$ where $x_{Y}$ is a vertex in $\pi_{Y}(P)$.  Then by Theorem \ref{thm : behineq} and the second part of Lemma \ref{lem : diamproj} there is a suitable constant $F_{1}$ so that the tuple satisfies Conditions (1) and (2) of Theorem \ref{thm : consistency} for $F_{1}$ and $F_{2}=1$.
 \end{example}
 
 We recall the following relation on subsurfaces of $S$ from \cite{bkmm}.
 \begin{definition}\textnormal{ (Partial order in the pants and marking graph)} \label{def : partialorder} Given $F_{1},F_{2}\geq 1$ with $F_{1}>\max\{F_{2},B_{0},G\}$ ($B_{0}$ is the constant from Theorem \ref{thm : behineq} and $G$ is from Theorem \ref{thm : bddgeod}). Let $(x_{Y})_{Y \subseteq S}$, where $x_{Y}\in \mathcal{C}_{0}(S)$, be a tuple satisfying the consistency conditions in Theorem \ref{thm : consistency} for constants $F_{1}$ and $F_{2}$. We define the following two relations on proper essential subsurfaces of $S$:
  \begin{enumerate}[(i)]
 \item Let $k\in\mathbb{N}$, the relation $Y \prec_{k} Z$ holds, if $Y\pitchfork Z$ and $d_{Y}(x_{Y}, \partial{Z})\geq k(F_{1}+4)$.
 \item  Let $k\in\mathbb{N}$, the relation $Y \ll_{k} Z$ holds, if $Y\pitchfork\partial{Z}$ and $d_{Y}(x_{Y}, \partial{Z})\geq k(F_{1}+4) $ 
  \end{enumerate}
  \end{definition}
Suppose that $Y\pitchfork Z$, then $\partial{Y} \pitchfork Z$ and $\partial{Z}\pitchfork Y$, so it is immediate from the definition that $Y \prec_{k} Z \Longrightarrow Y \ll_{k} Z$, but not the other way round. Moreover, if $k>l$ then $Y \prec_{k} Z \Longrightarrow Y \prec_{l} Z$, and similarly $Y \ll_{k} Z \Longrightarrow Y \ll_{l} Z$.

These notions of partial order are closely related to the {\it partial order of subsurface} in Definition \ref{def : ordersubsurf}. The following theorem provides for some useful transitivity properties of the partial order defined above. 

 \begin{thm}\cite[Lemma 4.4]{bkmm} \label{thm : partialorder}
 
 Let $(x_{Y})_{Y \subseteq S}$ be a tuple satisfying the consistency conditions and consider the partial order defined by it.  Let $U,V$ and $W$ be essential subsurfaces such that $x_{U},x_{V}, x_{W} \neq \emptyset $.  Given an integer $k > 1$ we have 
 \begin{enumerate}
 \item If $U \prec_{k} V$ and $V \ll_{2} W$ then $U \prec_{k-1} W$, also if $\mu$ is a partial marking, and if $ U \prec_{k} V$ and $V \ll_{2} \mu$, then $U \prec_{k-1} \mu$.
\item If $U \ll_{k} V$ and $V \ll_{2} W$ then $U \ll_{k-1} W$, also if $\mu$ is a partial marking, and if $U \ll_{k} V$ and $V \ll_{2} \mu$ then $U \ll_{k-1} \mu$.
\item If $U \pitchfork V$ and both $U \ll_{k} \mu$ and $V \ll_{k} \mu$ for a partial marking $\mu$, then $U$ and $V$ are $\prec_{k-1}-$ordered,  that is, either $U \prec_{k-1} V$ or $V \prec_{k-1} U$.
\end{enumerate}
 \end{thm}
  
 \subsection{Hierarchies and their resolutions in the pants graph and the marking graph} \label{subsec : hierarchy}

The {\it hierarchies of tight geodesics} in curve complexes of subsurfaces of a surface were introduced by Masur and Minsky in \cite{mm2}. See also \cite{elc1}, \cite{bkmm}. Hierarchy paths are quasi-geodesics in the pants and marking graph of a surface, with quantifiers depending only on the topological type of $S$, obtained by resolving hierarchies. Although in most of this paper we use only hierarchy resolution paths and their properties given in Theorem \ref{thm : hrpath}, a good understanding of the structure of hierarchies themselves will be very useful for the reader to follow our arguments. Given a geodesic $h$ in the curve complex of a subsurface $Y$, we refer to $Y$ as the support of $h$ and denote it by $D(h)$. Let $(\mu^{-},\mu^{+})$ be a pair of partial makings or laminations in $S$,  a {\it complete Hierarchy of tight geodesics} $H$ inductively associates to the pair $(\mu^{-},\mu^{+})$ a collections of tight geodesics in the curve complex of $S$ and the curve complexes of the subsurfaces of $S$. A tight geodesic in $\mathcal{C}(Y)$, is a geodesic with the property that any three of its consecutive vertices fill the surface $Y$. Here we list the properties which describe the inductive and layered structure of a hierarchy $H$, for more detail see the references mentioned above.

\begin{enumerate}
\item There is a unique main geodesic $g_{S}$ with $D(g_{S}) = S$, whose endpoints lie on $\base(\mu^{-})$ and $\base(\mu^{+})$.
\item For any geodesic $h\in H$ other than $g_{S}$, there exists another geodesic $k\in H$ such that, for some simplex $v\in k$, $D(h)$ is either a component of $D(k)\backslash v$, or an annulus whose core is a component of $v$. We say that $D(h)$ is a {\it component domain} of $k$.
\item An essential subsurface $Y\subseteq S$ can occur as the component domain of at most one geodesic in $H$.
\end{enumerate}

\noindent{\bf Infinite hierarchies:} An infinite hierarchy is a hierarchy $H$ with associated infinite tight geodesic rays or lines in the curve complexes of subsurfaces. Here each of the end points $\mu^{\pm}$ is a union of minimal, filling laminations supported on disjoint subsurfaces. Each lamination is the point at infinity of an infinite geodesic of the hierarchy in the curve complex of the subsurface supporting the geodesic (see Theorem \ref{thm : ccbdry}). More precisely, we consider laminations $\mu$ such that the restriction of $\mu$ to any subsurface $Y$ in the complement of the closed leaves of $\mu$ in $S$ with $\xi(Y)>1$ is minimal and fills $Y$. In \cite[\S 5]{elc1} the existence of infinite hierarchies between a pair of partial markings or laminations as above $(\mu^{-},\mu^{+})$ with no common infinite leaf is proved.
\medskip
 
\noindent{\bf Hierarchy resolution paths:}  These are path which comprise a set of transitive quasi-geodesics in the pants (marking) graphs of the surfaces $S$ with constants depending only on the topological type of the surface. In $\S$\ref{sec : stablehrpath} we prove a condition for stability of hierarchy resolution paths in the pants graph. 

A pants graph resolution of a hierarchy $H$ is a path in the pants graph denoted by $\rho :[m,n] \to P(S)$, where $[m,n] \subseteq \mathbb{Z}\cup\{\pm \infty\}$. For any $i \in [m,n]$, $\rho(i+1)$ is obtained from $\rho(i)$ by an elementary move ($\rho(i+1)$ and $\rho(i)$ have pants distance $1$). Similarly a marking graph resolution of $H$ is a path of clean markings such that any two consecutive markings differ by an elementary move. Given a hierarchy path $\rho$ we denote 
$$|\rho|=\{\rho(i) : i \in [m,n]\}$$
 which is a subset of the pants graph. A resolution path of a hierarchy consists of {\it slices of the hierarchy}. Each slice is the union of vertices of geodesics in curve complexes of subsurfaces of $S$. The supporting subsurfaces of these geodesics consist a tower of nested component domains of $H$. In the following theorem we list some of the properties of these quasi-geodesics (see also \cite{bmm2}, \cite{bm}) which we will use frequently in this paper. The main feature of these properties is that they are encoded in the subsurface coefficients of the pair. 
\begin{thm} \label {thm : hrpath}\textnormal{ (Properties of pants hierarchy resolution paths)} 

There are positive constants $M_{1},M_{2}, M_{3}$ and $M_{4}$ depending only on the topological type of $S$ so that: Given partial markings or laminations $\mu^{-}$ and $\mu^{+}$ on a surface $S$, there are hierarchy (resolution) paths $\rho:[m,n] \to P(S)$ $([m,n] \subseteq \mathbb{Z}\cup \{\pm \infty\})$ with $\rho(m)=\mu^{-}$ and $\rho(n)=\mu^{+}$ satisfying the following list of properties. 
\begin{enumerate}
\item \label{h : J} Let $Y$ be a component domain of $\rho$. There is a connected interval $J_{Y}\subseteq [m,n]$ and a geodesic $g_{Y} \subset \mathcal{C}(Y)$ such that for each $j \in J_{Y}$,  $\partial{Y} \subset \rho(j)$ and there is a vertex $v\in g_{Y}$ such that $v\in\rho(j)$ $(v=g_{Y}\cap\rho(j))$.
\item \label{h : largelink} An essential non-annular subsurface $Y \subseteq S$ with $d_{Y}(\mu^{-}, \mu^{+})> M_{1}$ is a component domain of $\rho$.
\item (Monotonicity) \label{h : monoton} Let $i,j \in J_{Y}$, and let vertices  $v=\rho(i)\cap g_{Y}$ and $w=\rho(j)\cap g_{Y}$. Then $i\leq j$ if and only if $v\leq w$ as vertices along $g_{Y}$.
\item (Bounded projection) \label{h : lrbddproj} Let $Y$ be a component domain of $\rho$ and let $J_{Y}=[j^{-},j^{+}]$. If $i>j^{+}$ then $d_{Y}(\rho(i),\rho(j^{+})) \leq M_{3}$, and if $i<j^{-}$ then $d_{Y}(\rho(i), \rho(j^{-})) \leq M_{3}$.
\item (Hausdorff distance bound) \label{h : hausd} Let $W\subseteq S$ be a subsurface. Let $\hull_{W}(\mu^{-},\mu^{+})$ be the convex hull of $\pi_{W}(\mu^{-})$ and $\pi_{W}(\mu^{+})$ in $\mathcal{C}(W)$. For any $i \in [m,n]$ there is an $x \in  \hull_{W}(\mu^{-},\mu^{+})$ such that
 $$d_{W}(\rho(i), x)\leq M_{4}.$$ 
 Also for any $x \in  \hull_{W}(\mu^{-},\mu^{+})$ there is an $i \in [m,n]$ such that the above inequality holds. In other words, the Hausdorff distance of $\hull_{W}(\mu^{-},\mu^{+})$ and $\pi_{W}(|\rho|)$ in $\mathcal{C}(W)$ is less than $M_{4}$. Note that here we do not necessarily assume that $W$ is component domain of $\rho$. 
 \item (No backtracking) \label{h : nobacktrack} Let $i,j,k\in [m,n]$ with $i\leq j\leq k$. Then for any subsurafce $Y\subseteq S$, 
 $$d_{Y}(\rho(i),\rho(k))+M_{4}\geq d_{Y}(\rho(i),\rho(j))+d_{Y}(\rho(j),\rho(k)).$$
\end{enumerate}
\end{thm}

In this theorem $J_{Y}\subset [m,n]$ consists of all $j\in [m,n]$ such that $\rho(j)$ is a slice of $H(\mu^{-},\mu^{+})$ containing $(v,g_{Y})$, where $g_{Y}\in H$ is the tight geodesic supported on $Y$ and $v\in g_{Y}$ (see \cite[\S 5]{mm2}). The fact that $J_{Y}$ is an interval was proved in Lemma 4.9 of \cite{elc2}. This explains property (\ref{h : J}). Property (\ref{h : largelink}) is  Lemma 6.2 (Large Link Lemma) in \cite{mm2}. Property (\ref{h : lrbddproj}) is a consequence of the Bounded Geodesic Image Theorem and is established in Lemma 6.1 of \cite{mm2}. Property (\ref{h : hausd}) is a consequence of Lemmas 6.1 and 6.9 in \cite{mm2} and is established in the proof of Lemma 5.14 in \cite{elc1}. Property (\ref{h : monoton}) is a consequence of the definition of slices of a hierarchy and the partial order of them defined using the partial order of pointed geodesics of $H$ in $\S 5$ of \cite{mm2}. Property (\ref{h : nobacktrack}) for a component domain of $\rho$ follows from properties (\ref{h : J}), (\ref{h : monoton}) and (\ref{h : lrbddproj}). For other subsurfaces it follows from properties (\ref{h : largelink}).

\begin{remark}
The constants $M_{1},M_{2},M_{3}$ and $M_{4}$ are related to each other. For example we can choose $M_{1}=2M_{2}$ and $M_{4}=2M_{1}+2M_{3}+2$. But these relations are not playing any role in this paper.
\end{remark}

\begin{remark}
Let $(\mu^{-},\mu^{+})$ be a pair of (partial) markings or laminations. A marking hierarchy resolution path $\tilde{\rho} : [m,n] \to M(S)$ $([m,n]\subseteq\mathbb{Z}\cup\{\pm\infty\})$ between $\mu^{-}$ and $\mu^{+}$ satisfies the same list of properties as pants hierarchy resolution paths, besides in properties (1)-(6) subsurfaces $Y$ and $W$ can be annular subsurfaces as well. 
\end{remark}

Partial order on subsurfaces along hierarchies is introduced in \cite[\S 5]{mm2}. It roughly correspond to the order that the intervals $J_{Y}$ appear along resolutions of the hierarchy. In this paper we mainly need the following weaker version of the partial order. 

Set the constant $M=M_{1}+M_{2}+M_{3}+M_{4}+B_{0}+4$. 

\begin{lem}\label{lem : ordersubsurf}
Given (partial) markings or laminations $\mu,\mu'$. Let $Y,W\subseteq S$ be essential subsurfaces. Suppose that $Y \pitchfork W$, $d_{Y}(\mu,\mu')>4M$ and $d_{W}(\mu,\mu')>4M$. Then in the following either (a) or (b) and not both holds: 
\begin{enumerate}[(a)]
\item$d_{Y}(\mu,\partial{W})\geq 2M$ and $d_{W}(\mu',\partial{Y})\geq 2M$.
\item $d_{W}(\mu,\partial{Y})\geq 2M$ and $d_{Y}(\mu',\partial{W})\geq 2M$.
\end{enumerate}
\end{lem}
\begin{proof}
 First assume that the first inequality of (a) holds. Then by Theorem \ref{thm : behineq}(Behrstock inequality) and since $2M>B_{0}$, we get 
$$d_{W}(\mu,\partial{Y})\leq B_{0}< M.$$  
Moreover by the assumption of the lemma, $d_{W}(\mu,\mu')> 4M$. These two inequalities combined by the triangle inequality give us 
$$d_{W}(\mu',\partial{Y})> 3M-\diam_{W}(\partial{Y})\geq 3M-2>2M,$$
which is the second inequality in (a). Similarly assuming that the second inequality of (a) holds we can show that the first inequality of (a) holds. So the two inequalities in (a) are equivalent. Also similarly we can show that the two inequalities in (b) are equivalent. 

Since $d_{Y}(\mu,\mu')>4M$, by the triangle inequality either $d_{Y}(\mu,\partial{W})\geq 2M$ or $d_{Y}(\mu',\partial{W})\geq 2M$. The inequality $d_{Y}(\mu,\partial{W})\geq 2M$ is the first inequality of (a) which as we saw above implies that $d_{W}(\mu,\partial{Y})\leq B_{0}< M$. Thus the first inequality of (b) does not hold and therefore (b) does not hold.

The inequality $d_{Y}(\mu',\partial{W})\geq 2M$ is the first inequality of (b) and similar to the above argument we can show that when it holds, (a) does not hold.
 \end{proof}
 
 \begin{definition}\label{def : ordersubsurf}\textnormal{ (Order of subsurfaces)}
Given (partial) markings or laminations $\mu,\mu'$. Let $Y$ and $W$ be essential subsurfaces. Suppose that $Y \pitchfork W$, $d_{Y}(\mu,\mu')>4M$ and $d_{W}(\mu,\mu')>4M$. If (a) in Lemma \ref{lem : ordersubsurf} holds we denote $Y<W$. If (b) in the lemma holds we denote $W<Y$.  
\end{definition}

\begin{prop}\label{prop : ordersubsurf} 
The relation $<$ is transitive. Suppose that $W<Y$, $j\in J_{Y}-J_{W}$ and $i \in J_{W}-J_{Y}$. Then $i\leq j$.
\end{prop}
\begin{proof}
To prove the transitivity of $<$, let $W<Y$ and $Y<Z$. Then by Lemma \ref{lem : ordersubsurf}, $d_{Y}(\mu',\partial{W})\geq 2M$ and $d_{Z}(\mu',\partial{Y})\geq 2M$. The second inequality and Theorem \ref{thm : behineq}(Behrstock inequality) imply that $d_{Y}(\mu',\partial{Z})\leq M$. This and the first inequality combined by the triangle inequality imply that 
\begin{equation}\label{eq : dYbZbW}d_{Y}(\partial{W},\partial{Z})\geq M-2.\end{equation} 
By the set up of $M$, $M-2> 2$. Thus $d_{Y}(\partial{W},\partial{Z})> 2$. Then as we saw in $\S$\ref{sec : ccplx}, $\partial{Y}\pitchfork \partial{Z}$. Thus $Y\pitchfork Z$. Moreover, $M-2>B_{0}$ so by the inequality (\ref{eq : dYbZbW}), we also have that $d_{Y}(\partial{W},\partial{Z})>B_{0}$. Then by Behrstock inequality, $d_{W}(\partial{Y},\partial{Z})\leq B_{0}$. Furthermore, since $W<Y$, $d_{Y}(\partial{W},\mu')>2M$, then by the Behrstock inequality $d_{W}(\partial{Y},\mu')\leq B_{0}$. Then by the triangle inequality we have 
\begin{eqnarray*}
d_{W}(\mu,\partial{Z})&\geq& d_{W}(\mu,\mu')-d_{W}(\partial{Y},\mu')-d_{W}(\partial{Y},\partial{Z})\\
&-&\diam_{W}(\partial{Y})-\diam_{W}(\partial{Z})\\
&\geq& 4M-2B_{0}-4>2M.
\end{eqnarray*}
This finishes proving that $W<Z$.

We proceed to prove the second assertion of the proposition. Since $W<Y$, $d_{W}(\partial{Y},\mu)\geq 2M$. So by Behrstock inequality, $d_{Y}(\partial{W},\mu)\leq M$. Moreover, $i\in J_{W}$ so $\partial{W}\subset\rho(i)$ (Theorem \ref{thm : hrpath}(\ref{h : J})). Thus 
\begin{equation}\label{eq : dYrimu}d_{Y}(\rho(i),\mu)\leq M.\end{equation}
Now suppose that $i>j$. Then since $i\in J_{W}-J_{Y}$, $i$ is greater than the right end point of the interval $J_{Y}$. Then by Theorem \ref{thm : hrpath}(\ref{h : lrbddproj}) we have that $d_{Y}(\rho(i),\mu')\leq M$. This inequality and the inequality $d_{Y}(\mu,\mu')>4M$ combined with the triangle inequality imply that 
 $$d_{Y}(\rho(i),\mu)> 3M-2>M,$$ 
 which contradicts the upper bound (\ref{eq : dYrimu}). Thus $i\leq j$ as was desired.
\end{proof}

In this paper we deduce almost all of the properties of hierarchy paths we need from Theorem \ref{thm : hrpath} and Proposition \ref{prop : ordersubsurf}. In a couple of occasions we need some finer properties of hierarchies and their resolutions where we provide a reference.
 
 \subsection{$\Sigma-$hulls and their projections}\label{subsec : sigmahull}

In this section we recall $\Sigma-$hulls in the pants graph and their projections introduced in \cite{bkmm}. Variations of the projection are used in \cite{beh},\cite{bm},\cite{bmm2}. Let $(\mu^{-},\mu^{+})$ be a pair of partial markings or laminations. Let $Y\subseteq S$ be an essential subsurface. Suppose that $\pi_{Y}(\mu^{-})\neq\emptyset$ and $\pi_{Y}(\mu^{+})\neq\emptyset$. Then let $\hull_{Y}(\mu^{-},\mu^{+})$ be the set of all geodesics in $\mathcal{C}(Y)$ that connect $\pi_{Y}(\mu^{-})$ and $\pi_{Y}(\mu^{+})$. Now suppose that either $\pi_{Y}(\mu^{-})=\emptyset$ or $\pi_{Y}(\mu^{+})=\emptyset$. If the restriction of $\mu^{\pm}$ to $Y$ is in $\mathcal{EL}(Y)$ then the restriction $\mu^{\pm}|_{Y}$ determines a point in the Gromov boundary of $\mathcal{C}(Y)$ (Theorem \ref{thm : ccbdry}). Suppose that $\mu^{+}|_{Y}$ is in $\mathcal{EL}(Y)$. Then let $\hull_{Y}(\mu^{-},\mu^{+})$ be the set of geodesics with one end point on $\pi_{Y}(\mu^{-})$ and asymptotic to the point corresponding to $\mu^{+}|_{Y}$ in the Gromov boundary of $\mathcal{C}(Y)$. The definition of $\hull_{Y}(\mu^{-},\mu^{+})$ when $\mu^{-}|_{Y}$ is in $\mathcal{EL}(Y)$, or both $\mu^{-}|_{Y}$ and $\mu^{+}|_{Y}$ are in $\mathcal{EL}(Y)$ is similar.

Note that since $\mathcal{C}(Y)$ is $\delta_{Y}-$hyperbolic all of the geodesics connecting $\pi_{Y}(\mu^{-})$ or $\mu^{-}|_{Y}$ and $\pi_{Y}(\mu^{+})$ or $\mu^{+}|_{Y}$ uniformly fellow travel.

Now given $\epsilon>0$ define the {\it $\Sigma_{\epsilon}-$hull of the pair $(\mu^{-},\mu^{+})$} as follows
\begin{eqnarray*}
\Sigma_{\epsilon}(\mu^{-},\mu^{+}):=\{ P \in P(S):d_{Y}(P,\hull_{Y}(\mu^{-},\mu^{+}))\leq\epsilon\\
\;\text{for every non-annular subsurface}\;\; Y\subseteq S\}.
\end{eqnarray*}

\begin{thm}\cite[Proposition 5.2]{bkmm}\label{thm : projsigmahull}
There is a constant $F>0$ depending only on the topological type of the surface such that for every $\epsilon>F$ there is a coarse (projection) map
$$\Pi : P(S) \to \Sigma_{\epsilon}(\mu^{-},\mu^{+}) $$
with the following properties:
\begin{enumerate}
\item For every non-annular subsurface $Y \subseteq S$ we have
$$d_{Y}(\Pi P, \hull_{Y}(\mu^{-},\mu^{+})) \leq F$$
\item $\Pi\big|_{\Sigma_{\epsilon}(\mu^{-},\mu^{+})}$ is uniformly close to the identity.
\item $\Pi$ is coarse-Lipschitz.
\end{enumerate}
\end{thm}
The set of vertices of the pants graph consists a $1-$dense subset of  $P(S)$. The coarse Lipschitz in part (3) means that $\Pi$ is defined on the set of vertices and is Lipschitz on this set.

 \begin{remark} 
In \cite{bkmm} the notion of $\Sigma-$hulls and their projections are introduced in the context of marking graphs. Moreover Theorem \ref{thm : projsigmahull} is stated and proved in the context of the marking graph. But it is straightforward to verify that their arguments go through in the context of pants graph excluding all annular subsurfaces. 
\end{remark}

The main ingredient of the proof of Theorem \ref{thm : projsigmahull} is the fact that there are positive constants $F_{1}$ and $F_{2}$, depending only on the topological type of the surface, such that the tuple $(x_{Y})_{Y\subseteq S}$, where each $x_{Y} \in \mathcal{C}(Y)$ is a nearest point to $\pi_{Y}(P)$ on $\hull_{Y}(\mu^{-},\mu^{+})$, satisfies the consistency conditions of Consistency Theorem (Theorem \ref{thm : consistency}). Then Consistency Theorem provides a constant $F>0$ and a pants decomposition $\Pi P$ such that 
$$d_{Y}(\Pi P, \hull_{Y}(\mu^{-},\mu^{+}))\leq F$$
 for every essential non-annular subsurface $Y\subseteq S$.


\section{The Weil-Petersson metric and its synthetic properties}\label{sec : WPmetric}

We start with some basic facts about Teichm\"{u}ller theory and the Weil-Petersson (WP) metric, and through out will set up our notation. Let $S=S_{g,b}$ be a surface with genus $g$ and $b$ boundary components. A point in $\Teich(S)$ the {\it Teichm\"{u}ller space} of $S$ is a complete, finite area hyperbolic surface $x$ equipped with a diffeomorphism $h:S\to x$. We say that $h$ is a marking of $x$. Two marked hyperbolic surfaces $h_{1}:S\to x_{1}$ and $h_{2}:S\to x_{2}$ define the same point in $\Teich(S)$ if and only if $h_{2} \circ h_{1}^{-1}:x_{1}\to x_{2}$ is isotopic to an isometry. The {\it mapping class group} of the surface $S$, denoted by $\Mod(S)$, is the group of isotopy classes of orientation preserving self diffeomorphisms of $S$. The group $\Mod(S)$ acts on $\Teich(S)$ by remarking as follows: an element  $f \in \Mod(S)$ maps a marked surface $h:S \to x$ to the marked surface $h\circ f:S \to x$. The quotient $\Teich(S)/\Mod(S)$ is the {\it moduli space} of $S$ denoted by $\mathcal{M}(S)$. Given a point $f:S\to x$ in the Teichm\"{u}ller space we usually drop the marking map and denote it by $x$. We denote the point in the moduli space corresponding to the $\Mod(S)$ orbit of $x\in \Teich(S)$ by $\hat{x}$.

let $\epsilon >0$, the $\epsilon-$thick part of the Teichm\"{u}ller space is the subset of the Teichm\"{u}ller space $\{x \in \Teich(S) : \inj(x) \geq \epsilon\}$.
 Here $\inj(x)$ is the injectivity radius of the hyperbolic surface $x$. The $\epsilon-$thin part of the Teichm\"{u}ller space is the subset $\{x \in \Teich(S) : \inj(x) \leq \epsilon\}$. The $\epsilon-$thick part and $\epsilon-$thin part of the moduli space are defined similarly.
 \medskip
 
 \noindent{\bf Weil-Petersson metric:} Given holomorphic quadratic differentials $\varphi, \psi \in T_{x}^{*}\Teich(S)$ the Weil-Petersson $L^{2}$ co-product is defined by
 $$\mathcal{R}e(\int_{x}\varphi \overline{\psi} \rho^{-2} ),$$
 where $\rho(z)^{2}|dz|^{2}$ is the hyperbolic metric of the marked hyperbolic surface $x$. This co-product induces a norm on Beltrami differentials via the standard pairing of quadratic differentials and measurable Beltrami differentials on $x$ defined by $\int_{x}\varphi \mu$. Any measurable Beltrami differential presents a vector in $T_{x}\Teich(S)$ and the Weil-Petersson metric on the Teichm\"{u}ller space is defined by the polarization of the induced norm. In this paper we study the global behavior of geodesics of this metric.
 
 The Weil-Petersson metric is a Riemannian metric with negative sectional curvatures which is invariant under the action of the mapping class group of the surface. The WP metric is an incomplete metric, however it is geodesically convex. The negative curvature and the geodesic convexity imply that the completion of the Teichm\"{u}ller space equipped with the WP metric $\overline{\Teich(S)}$ is a $\CAT(0)$ space. For background about $\CAT(0)$ spaces see e.g. \cite{bhnpc}.
 
 In the rest of this subsection we recall some properties of the Weil-Petersson metric and its geodesics which will be used in this paper. References for these material are \cite{wols}, \cite{wolb},\cite{wol} see them also for further references. 
\medskip

\noindent{\bf Length-functions:} Let $\alpha \in \mathcal{C}_{0}(S)$ the $\alpha-$length-function
$$\ell_{\alpha}:\Teich(S) \to \mathbb{R}^{+}$$
assigns to $x \in \Teich(S)$ the length of the geodesic representative of $\alpha$ on the marked hyperbolic surface $x$.

The notion of length-function has a natural extension to the space of measured geodesic laminations (see \cite{bonlam}). Let $\mathcal{L}$ be a measured geodesic lamination, we denote the $\mathcal{L}-$length-function by $\ell_{\mathcal{L}}(.)$. 
\medskip

\noindent{\bf Fenchel-Nielsen coordinates:} Let $P$ be a pants decomposition on a surface $S$, a Fenchel-Nielsen (FN) coordinate system corresponding to $P$, $(\ell_{\gamma},\theta_{\gamma})_{\gamma \in P}$ maps $\Teich(S)$ to $\prod_{\gamma \in P}\mathbb{R}_{\gamma}^{+} \times \mathbb{R}_{\gamma}$. The first coordinate of $\mathbb{R}_{\gamma}^{+} \times \mathbb{R}_{\gamma}$ is the $\gamma-$length-function and the second coordinate is a twist parameter about $\gamma$. For more detail see $\S3$ of \cite{buser}. We denote the positive Dehn twist about a curve $\gamma$ by $\mathcal{D}_{\gamma}$ which is defined as follows:  Let $x\in \Teich(S)$. Let $P$ be a pants decomposition with $\gamma \in P$. Fix a FN coordinate system corresponding to $P$. Then $\mathcal{D}_{\gamma}(x)$ is the point in $\Teich(S)$ with all coordinates equal to that of $x$ except $\theta_{\gamma}(\mathcal{D}_{\gamma}(x))=\theta_{\gamma}(x)+2\pi$.
\medskip

\noindent{\bf The Weil-Petersson metric completion of the Teichm\"{u}ller space and the completion strata:} The incompleteness of the Weil-Petersson metric is due to existence of finite length paths in the Teichm\"{u}ller space along which length of a curve converges to zero, \cite{wol}. Masur \cite{maswp} gives a concrete description of the completion as the {\it augmented Teichm\"{u}ller space}. The augmented Teichm\"{u}ller space consists of strata: Let  $\sigma$ be a possibly empty multi-curve, a point in the $\sigma-$stratum is a collection of marked hyperbolic metrics of connected components of $S\backslash \sigma$, where for each curve in $\sigma$ a pair of cusps is introduced. The topology is described via extended Fenchel-Nielsen coordinate systems as follows: Given a pants decomposition $P$, the FN coordinate system maps $\Teich(S)$ to $\prod_{\gamma \in P}\mathbb{R}\times \mathbb{R}^{+}$. We extend the FN coordinate system $(\ell_{\gamma},\theta_{\gamma})_{\gamma \in P}$ to allow length-functions take value $0$ as well. Now take the quotient of $\prod_{\gamma \in P}\mathbb{R}\times \mathbb{R}^{+}$ by identifying $(0,\theta) \sim (0,\theta')$ in each $\mathbb{R}^{+}\times\mathbb{R}$ factor. Let $\sigma \subset P$ then the topology near any point of the $\sigma-$stratum is such that the map defined by the FN coordinate system is a homeomorphism near that point. 
 
 In this topology each stratum $\mathcal{S}(\sigma)$ is the product of the lower dimensional Teichm\"{u}ller spaces of the connected components of $S\backslash \sigma$. 
 \medskip
 
\noindent{\bf Continuity of length-functions:} The following theorem is a consequence of the fact that the topology induced by the Weil-Petersson metric on the Teichm\"{u}ller space and the Chaubaty topology of the Teichm\"{u}ller space coincide. The Chaubaty topology is defined using the fact that each point in $\Teich(S)$ is the conjugacy class of a representation of the fundamental group of the surface $\pi_{1}(S)$ into $PSL_{2}(\mathbb{R})$. For more detail see the beginning of \cite[\S4]{wolb}.

\begin{thm}\textnormal{ (Continuity of length-functions)}\label{thm : continuitylf}
Suppose that a sequence of points $x_{n}\to x$ as $n\to \infty$ in $\overline{\Teich(S)}$. Then for every $\alpha\in\mathcal{C}_{0}(S)$, 
$$\ell_{\alpha}(x_{n})\to \ell_{\alpha}(x)$$
 as $n\to\infty$.
\end{thm}

\noindent{\bf Non-refraction property of completion strata:}
The following {\em non-refraction property} of the Wiel-Petersson completion strata is a consequence of the expansion of the WP metric near the completion strata. This is an expansion to product of the WP metric on a stratum and copies of a model metric on the punctured disk. The expansion first appeared in \cite{maswp} and was improved by Yamada \cite{yamadawp} and further improved by Daskalopoulos-Wentworth \cite{dwwp} and Wolpert \cite{wols}.
\begin{thm}\textnormal{(Non-refraction)}\label{thm : nonrefraction}\cite{dwwp}, \cite{wols} 
Let  $\zeta:[0,T] \to \overline{\Teich(S)}$ be a WP geodesic segment, and let $\sigma^{-}$ and $\sigma^{+}$ be the maximal simplices in $\mathcal{C}(S)$ such that $\zeta(0) \in \mathcal{S}(\sigma^{-})$ and $\zeta(T)\in \mathcal{S}(\sigma^{+})$ then 
$$int(\zeta([0,T])) \subset  \mathcal{S}(\sigma^{-} \cap \sigma^{+}).$$
\end{thm}

\noindent{\bf Closed Weil-Petersson geodesics in the moduli space:} Using the non-refraction property Daskalopoulos-Wentworth and Wolpert show that any pseudo-Anosov element of the mapping class group $f$ has an axis in the Teichm\"{u}ller space equipped with the WP metric. The axis is a bi-infinite WP geodesic $Ax_{f} \subset \Teich(S)$ such that 
$$d_{\WP}(x, fx)=\inf_{y \in \Teich(S)} d_{\WP}(y,fy)$$
for every $x \in \Ax_{f}$. Then the axis of each pseudo-Anosov map projects to a closed geodesic in the moduli space.
\medskip

\noindent{\bf Bers pants decomposition and Bers marking:} By a result of Bers (see e.g. \cite[\S 5]{buser}) given a surface $S$ with negative Euler characteristic, there is a constant $L_{S}>0$ ({\it Bers constant}) depending only on the topological type of $S$ such that any complete finite area hyperbolic metric on $S$ has a pants decomposition ({\it Bers pants decomposition}) with the property that the geodesic representative of any curve in the pants decomposition has length at most $L_{S}$. We call any curve in a Bers pants decomposition  a {\it Bers curve}. By the Collar Lemma there are finitely many Bers curves and therefore finitely many Bers pants decompositions on a complete hyperbolic surface.  Given $x\in \overline{\Teich(S)}$, suppose that $x\in \mathcal{S}(\sigma)$ ($\sigma$ is a possibly empty multi-curve). Then a Bers pants decomposition of $x$ is the union of Bers pants decompositions of each of the connected components of $S\backslash \sigma$ and $\sigma$. A {\it Bers marking} is a (partial) marking obtained from a Bers pants decomposition by adding transversal curves with representatives at $x$ of minimal length.  We denote a Bers marking of $x$ by $\mu(x)$. The partial marking of $x$ does not have any transversal to the curves in $\sigma$.

The following theorem of Brock assigns to any WP geodesic a quasi geodesic in the pants graph. As a result the hierarchies in the pants and marking graphs and their resolutions play an essential role in our study of the global behavior of WP geodesics.
\begin{thm}\textnormal{ (Quasi-isometric model)}\cite{br}\label{thm : brockqisom}
There are constants $K_{\WP}\geq1$ and $C_{\WP}\geq 0$ depending only on the topological type of $S$, such that the coarsely defined map 
 $$Q: \overline{\Teich(S)} \to P(S)$$
which assigns to $x\in \overline{\Teich(S)}$ a Bers pants decomposition $Q(x)$ is a $(K_{\WP},C_{\WP})-$quasi-isometry.
\end{thm}
\medskip

\noindent{\bf Gradient of length-functions:}  Wolpert gives the following estimate for the pairing of the gradients of length-functions:

\begin{lem}\cite{wolb}\label{lem : gradlf}
The WP pairing of length-function gradients of curves $\alpha,\beta$ with disjoint geodesic representatives satisfies
\begin{center}$0<\langle \grad\ell_{\alpha},\grad\ell_{\beta}\rangle- \frac{2}{\pi}\ell_{\alpha}\delta_{\alpha\beta}=O(\ell_{\alpha}^{2}\ell_{\beta}^{2})$ \end{center}
where the constant of the $O$ notation depends only on $c_{0}>0$ with $\ell_{\alpha}, \ell_{\beta} \leq c_{0}$. 
\end{lem}

\begin{cor}\label{cor : gradlfestimate}
Given $c_{0}>0$, there is a two variable function ${\bf d}$ with the following property. Let $l,a\in [0,c_{0}]$ be such that $l>a$. Let $x,x'\in \Teich(S)$ be such that $\ell_{\alpha}(x) \leq l-a$ and  $\ell_{\alpha}(x') \geq l$. Then $d_{\WP}(x',x) \geq {\bf d}(l,a)$. 
\end{cor}

\begin{proof}
By Lemma \ref{lem : gradlf} at $y\in \Teich(S)$ with $\ell_{\alpha}(y)\leq c_{0}$, 
\begin{equation}\label{eq : gradla}||\grad\ell_{\alpha}(y)|| \leq (\frac{2}{\pi}\ell_{\alpha}(y) +O(\ell_{\alpha}(y))^{4})^{1/2}\end{equation} 
where the $O$ notation constant depends only on $c_{0}$. Let $u:[0,T]\to\overline{\Teich(S)}$ be the WP geodesic segment connecting $x$ to $x'$ parametrized by arc-length. Let $t^{*}=\inf\{t\in[0,T]:\ell_{\alpha}(u(t))\geq l\}$. Then $\ell_{\alpha}(u(t)) \leq l$ for every $t \in[0,t^{*}]$. Using this bound and integrating both sides of (\ref{eq : gradla}) we get
\begin{eqnarray*}
a\leq |\ell_{\alpha}(u(t^{*}))-\ell_{\alpha}(u(0))| &\leq& \int_{0}^{t^{*}}||\grad\ell_{\alpha}(u(t))||dt\\ 
&\leq& (\frac{2}{\pi}l+O(l)^{4})^{1/2}t^{*}.
\end{eqnarray*}
Define the function ${\bf d}(l,a)=\frac{a}{\big(\frac{2}{\pi}l+O(l^{4})\big)^{1/2}}$. Then we have that $t^{*}\geq {\bf d}(l,a)$. Now since $d_{\WP}(x,x')\geq t^{*}$ the lemma follows.
 \end{proof}
 Wolpert also gives the following estimate for the distance of a point in the Teichm\"{u}ller space and a completion stratum.  
\begin{prop}\label{prop : diststr}\cite[Corollary 4.10]{wolb}
 Let $x\in \Teich(S)$ and $\sigma$ be a multi-curve, then 
 $$d_{\WP}(x,\mathcal{S}(\sigma))\leq \big( 2\pi\sum_{\alpha\in \sigma}\ell_{\alpha}(x)\big)^{1/2}.$$
\end{prop}

\noindent{\bf Tangent cones of the Weil-Petersson completion of the Teichm\"{u}ller space:} The completion of the Teichm\"{u}ller space with the Weil-Petersson metric is a $\CAT(0)$ space. Assigned to any point $p$ in a $\CAT(0)$ space there is the Alexandrov tangent cone $AC_{p}$ consisting of equivalence classes of geodesic rays $\zeta$ starting at the point $p$. Two geodesics $\zeta$ and $\zeta'$ starting at $p$ are equivalent if their angle at $p$ in the sense of Alexandrov is equal to $0$. For more detail about tangent cones see \cite{bhnpc}.

Let $\sigma$ be a multi-curve on $S$, and $\chi$ be a full marking on $S\backslash \sigma$. Let $p\in\mathcal{S}(\sigma)$. Then for a geodesic ray $\zeta:[0,a) \to \overline{\Teich(S)}$ with $\zeta(0)=p$ let
$$\mathcal{L}(\zeta(t))=(\ell_{\alpha}^{1/2}(\zeta(t)),\ell_{\beta}^{1/2}(\zeta(t)))_{\alpha \in \sigma, \beta \in \chi}.$$
Then define $\Lambda:AC_{p}\to\mathbb{R}_{\geq 0}^{|\sigma|} \times T_{p} \Teich(S\backslash \sigma)$ by
$$\Lambda(\zeta)= (2 \pi)^{1/2}\frac{d}{dt} \Big |_{t=0} \mathcal{L}(\zeta(t)).$$
The following description of the WP Alexandrov tangent cone of the Teichm\"{u}ller space at the point $p\in\mathcal{S}(\sigma)$ is obtained by Wolpert.
\begin{prop}\textnormal{ (The Weil-Petersson tangent cone)}\label{prop : tangentcone}\cite[Theorem 4.18]{wolb}
The map $\Lambda$ from the tangent cone of the WP metric at $p$ to $\mathbb{R}_{\geq 0}^{|\sigma|} \times T_{p} \Teich(S\backslash \sigma)$ is an isometry of tangent cones with restriction of inner products. A WP geodesic $\zeta$ with $\zeta(0)=p$ and root length-function initial derivative $\frac{d}{dt}\big|_{t=0} \ell_{\alpha}^{1/2}(\zeta(t))$ vanishing is contained in the stratum $\{\ell_{\alpha}=0\}$, $\mathcal{S}(\sigma) \subset \{\ell_{\alpha}=0\}$.
\end{prop}


\subsection{End invariant}\label{subsec : WPendinv}

In this subsection we recall the notion of end invariant for WP geodesics introduced by Brock, Masur and Minsky in \cite{bmm2}.

\begin{thm}\textnormal{ (Convexity of length-functions)}\cite{wolb} \label{thm : convlf}
Given $\epsilon>0$ there is a constant $c=c(\epsilon)$ with the following property. Let $g:(a,b) \to \Teich(S)$ be a WP geodesic parametrized by arc-length and $\alpha \in \mathcal{C}_{0}(S)$. If for some $t\in (a,b)$  $\inj(g(t))\geq\epsilon$ ($g(t)$ is a point in the $\epsilon-$thick part of the Teichm\"{u}ller space), then 
\begin{equation} \label{eq : convlf}\ddot{\ell}_{\alpha}(g(t))\geq c  \ell_{\alpha}(g(t)).\end{equation}
 Similar inequality holds for the length of any measured lamination $\mathcal{L}$, 
 $$\ddot{\ell}_{\mathcal{L}}(g(t)) \geq c \ell_{\mathcal{L}}(g(t)).$$

\end{thm}
\begin{remark}
The above estimates are local and only depend on the injectivity radius of the surface $g(t)$.
\end{remark}

\begin{definition}\textnormal{ (Ending measured lamination)}\label{def : endmeaslam}
The  weak$^{*}$ limit in $\mathcal{ML}(S)$ of any weighted sequence of infinitely many distinct Bers curves along a WP geodesic ray $r$ is an ending measured lamination of $r$. 
\end{definition}
 
 In \cite{bmm2} the following notion of {\it ending lamination} for WP geodesic rays is introduced, its existence relies on the convexity of length-functions along WP geodesics and properties of $\CAT(0)$ spaces.  
 
Let $r:[0,a)\to\overline{\Teich(S)}$ be a WP geodesic ray. A curve $\alpha\in\mathcal{C}_{0}(S)$ is pinching along $r$ if $\ell_{\alpha}(r(t))\to 0$ as $t\to a$.

\begin{definition}\textnormal{ (Ending lamination)}\label{def : endlamwp}
The union of pinching curves along a WP geodesic ray $r$ and the geodesic laminations arising as supports of all ending measured laminations of $r$ is the ending lamination of $r$.
\end{definition}

\begin{definition}\textnormal{ (End invariant of Weil-Petersson geodesics)} To each open end of a WP geodesic $g:(a,b) \to \overline{\Teich(S)}$ we associate an end invariant which is a partial marking or a lamination. If the forward trajectory $g|_{[0,b)}$ can be extended to $b$ such that $g(b) \in \overline{ \Teich(S)}$ then the forward end invariant $\nu^{+}(g)$ is any Bers marking $\mu(g(b))$ ( there are finitely many of them). Otherwise, $\nu^{+}(g)$ is the ending lamination of the forward trajectory ray $g|_{[0,b)}$ which was defined above. We define the backward end invariant $\nu^{-}(g)$ similarly by considering the backward trajectory $g|_{(a,0]}$. We call the pair $(\nu^{-},\nu^{+})$ the end invariant of $g$.
\end{definition}
Here we recall two properties of  the ending measured laminations proved in \cite[\S 2]{bmm1}:
\begin{lem}\textnormal{ (Decrease of length-functions along WP geodesic rays)}\label{lem : dlendmeas}
Let $\mathcal{L}$ be any ending measured lamination of a WP geodesic ray $r$, then $\ell_{\mathcal{L}}(r(t))$ is a decreasing function.
\end{lem}
\begin{lem} \label{lem : limraylimmeas} Let $r_{n} \to r$ be a convergent sequence of WP geodesics rays in the WP visual sphere at a point $x\in\Teich(S)$. Let $\mathcal{L}_{n}$ be an ending measured lamination or a weighted pinching curve of $r_{n}$. Then any representative $\mathcal{L} \in \mathcal{ML}(S)$ of the limit $[\mathcal{L}]$ of the projective classes $[\mathcal{L}_{n}]$ in $\mathcal{PML}(S)$ has bounded length along the ray $r$.
\end{lem}

\section{Length-function control along Weil-Petersson geodesic segments}\label{sec : bddwpgeodseg}

In this section we study length-functions and twist parameters along sequences of bounded length WP geodesic segments in the WP completion of the Teichm\"{u}ller space.

In $\S$\ref{subsec : bmmtwlf} we will prove a modified version of Lemma 4.5 in \cite{bmm2} about the buildup of Dehn twists along sequences of uniformly bounded length WP geodesic segments (Theorem \ref{thm : shtw}). Corollaries \ref{cor : shtw} and \ref{cor : twsh} are somewhat quantified versions of this theorem which provide us with a kind of twist parameter versus length-function control along WP geodesic segments. This control plays an important role in $\S$\ref{sec : itinerarywpgeod} where we study the itinerary of WP geodesics fellow traveling hierarchy paths.

The proof of Theorem \ref{thm : shtw} uses  Wolpert's characterization of limits of sequences of uniformly bounded length WP geodesic segments in the Weil-Petersson completion of the Teichm\"{u}ller space. In $\S$\ref{subsec : geodlimit} we state Wolpert's Geodesic Limit Theorem and using suggestions of Jeffrey Brock will give an improved version of this theorem (Theorem \ref{thm : geodlimit}). This improved version is crucial to prove our results in $\S$\ref{subsec : bmmtwlf}.

\subsection{Limits of sequences of uniformly bounded length WP geodesic segments}\label{subsec : geodlimit}
In this subsection we provide a modified version of Wolpert's Geodesic Limit Theorem. Given a multi-curve $\sigma$, denote by $\tw(\sigma)$ the subgroup of $\Mod(S)$ generated by positive Dehn twists about the curves in $\sigma$. Using the non-refraction property of the Weil-Petersson completion strata (Theorem \ref{thm : nonrefraction}) and the fact that the quotient of $U_{\epsilon}(\sigma)$ (the subset of $\Teich(S)$ where all the curves in $\sigma$ have length less than $\epsilon$) by the action of $\tw(\sigma)$ is compact, Wolpert gives the following characterization of the limits of uniformly bounded length WP geodesic segments in the Teichm\"{u}ller space. See also \cite{bmm2}.

\begin{thm}\label {thm : geodlimit'}\cite{wols}
Given $T>0$. Let $\zeta_{n}: [0,T] \to \overline{\Teich(S)}$ be a sequence of unit speed parametrized WP geodesic segments parametrized by arc-length of length $T$ in the WP completion of the Teichm\"{u}ller space. After possibly passing to a subsequence there exist a partition of the interval $[0, T]$ by $0=t_{0}<t_{1}<t_{2}<...<t_{k}<t_{k+1}=T$, multi-curves $\sigma_{0},...,\sigma_{k+1}$ where $\sigma_{0}$ and $\sigma_{k+1}$ are possibly empty, and possibly empty multi-curves $\tau_{i}=\sigma_{i-1}\cap \sigma_{i}$, $i=1,...,k+1$, where $\sigma_{i}\subsetneq \tau_{i}$ for each $1\leq i\leq k$, and a piecewise geodesic 
$$\hat{\zeta}:[0,T]\to \overline{\Teich(S)}$$
with the following properties
\begin{enumerate}
\item $\hat{\zeta}((t_{i-1}, t_{i})) \subset \mathcal{S}(\tau_{i})$, for $i=1,..., k+1$,
\item $\hat{\zeta}(t_{i}) \in \mathcal{S}(\sigma_{i})$, for $i=0,...,k+1$,
\item There are elements $\psi_{n}\in\Mod(S)$, $n\in\mathbb{N}$, and $\mathcal{T}_{i,n} \in \tw(\sigma_{i}- \tau_{i} \cup \tau_{i+1})$, for $i=1,..,k$ and $n\in\mathbb{N}$, such that after possibly passing to a subsequence $\psi_{n}(\zeta_{n}|_{[0,t_{1}]})\to\hat{\zeta}|_{[0,t_{1}]}$, and for each $i=1,..,k$,
$$\mathcal{T}_{i,n} \circ ... \circ \mathcal{T}_{1,n} \circ \psi_{n}(\zeta_{n}|_{[t_{i},t_{i+1}]})\to\hat{\zeta}|_{[t_{i},t_{i+1}]}$$
as $n \to \infty$ in the sense of unit speed parametrized geodesics. For convenience for each $i=0,1,...,k+1$ we define 
\begin{equation} \label{eq : phiin'} \varphi_{i,n}=\mathcal{T}_{i,n} \circ ... \circ \mathcal{T}_{1,n} \circ \psi_{n}.\end{equation}
\item The elements $\psi_{n}$ are either trivial or unbounded and the elements $\mathcal{T}_{i,n}$ are unbounded.
\item The piecewise geodesic $\hat{\zeta}$ is the minimal length path in $\overline{\Teich(S)}$ joining $\hat{\zeta}(0)$ to $\hat{\zeta}(T)$ and intersecting the strata $\mathcal{S}(\sigma_{1}), \mathcal{S}(\sigma_{2}), ..., \mathcal{S}(\sigma_{k})$ in order.
\end{enumerate}
\end{thm}

The following two lemmas which were suggested to us by Jeff Brock help us to considerably improve the above picture of limits of uniformly bounded length WP geodesic segments (see Theorem \ref{thm : geodlimit}). 

\begin{lem}\label{lem : singlestratum}
Given a sequence of WP geodesic segments $\zeta_{n}:[0,T] \to \Teich(S)$, let the multi-curves $\tau_{i},\; i=1,...,k+1,$ be as in Theorem \ref{thm : geodlimit'}. Then $\tau_{1}=...=\tau_{k+1}$. We denote 
\begin{equation}\label{eq : tau}\hat{\tau}\equiv \tau_{i},\;\; i=1,...,k+1. \end{equation}

\end{lem}
\begin{proof} Let the piecewise geodesic path $\hat{\zeta}:[0,T]\to \Teich(S)$, the partition $0=t_{0}<t_{1}<...<t_{k+1}=T$, multi-curves $\sigma_{i},\; i=0,...,k+1$ and $\tau_{i},\;\ i=1,...,k+1$ be as in Theorem \ref{thm : geodlimit'}.
 Let $\delta<\min_{i=1,...,k+1}\frac{t_{i}-t_{i-1}}{2}$. Fix $1\leq i\leq k$. Let $\alpha\in\tau_{i}$. Recall that $\tau_{i}=\sigma_{i-1}\cap\sigma_{i}$, so $\tau_{i}\subseteq \sigma_{i}$. Thus $\alpha\in \sigma_{i}$. 
 
By Theorem \ref{thm : geodlimit'}(5) the concatenation of the geodesic segments $\hat{\zeta}|_{[t_{i}-\delta,t_{i}]}$ and $\hat{\zeta}|_{[t_{i},t_{i}+\delta]}$ is the distance minimizing path in $\overline{\Teich(S)}$ joining $\hat{\zeta}(t_{i}-\delta)$ to $\hat{\zeta}(t_{i}+\delta)$ and intersecting $\mathcal{S}(\sigma_{i})$. Then as Wolpert shows on page 328 of \cite{wolb} the following equality of the one-sided derivatives of the square root of the $\alpha-$length-function (and any other curve in $\sigma_{i}$) holds at $t=t_{i}$,
\begin{equation} \label{eq : lfderatstr}\frac{d}{dt}\Big |_{t=t_{i}^{+}}\ell_{\alpha}^{1/2}(\hat{\zeta}|_{(t_{i},t_{i}+\delta]})= -\frac{d}{dt}\Big|_{t=t_{i}^{-}} \ell_{\alpha}^{1/2}(\hat{\zeta}|_{[t_{i}-\delta,t_{i})}).\end{equation}
 By Theorem \ref{thm : geodlimit'}(1) we have $\hat{\zeta}|_{[t_{i}-\delta,t_{i})} \subset \mathcal{S}(\tau_{i})$, so $\ell_{\alpha}^{1/2}(\hat{\zeta}(t))=0$ for all $t \in [t_{i}-\delta,t_{i})$. Therefore 
$$\frac{d}{dt}|_{t=t_{i}^{-}}\ell_{\alpha}^{1/2}(\hat{\zeta}|_{[t_{i}-\delta,t_{i})})=0.$$
Then by (\ref{eq : lfderatstr}), $\frac{d}{dt}|_{t=t_{i}^{+}}\ell_{\alpha}^{1/2}(\hat{\zeta}|_{(t_{i},t_{i}+\delta]})=0$. So by Proposition \ref{prop : tangentcone}, $\hat{\zeta}((t_{i},t_{i}+\delta]) \subset \overline{\mathcal{S}(\alpha)}$. Moreover, by Theorem \ref{thm : geodlimit'}(1), $\hat{\zeta}((t_{i},t_{i}+\delta])\subset \mathcal{S}(\tau_{i+1})$. The inclusions of the geodesic segment $\hat{\zeta}((t_{i},t_{i}+\delta])$ in the strata $\mathcal{S}(\tau_{i})$ and $\overline{\mathcal{S}(\alpha)}$ imply that $\alpha \in \tau_{i+1}$. This holds for every $\alpha\in \tau_{i}$, so we conclude that $\tau_{i}\subseteq \tau_{i+1}$. Exchanging the role of $\tau_{i}$ and $\tau_{i+1}$ in the above argument we can show that $\tau_{i+1}\subseteq \tau_{i}$. Thus $\tau_{i}=\tau_{i+1}$. 

Since $1\leq i\leq k+1$ was arbitrary we conclude that $\tau_{1}=...=\tau_{k+1}$, as was desired.
\end{proof}

\begin{lem}\label{lem : highdehntwisteachgamma}
Let $\zeta_{n}:[0,T] \to \Teich(S)$ be a sequence of WP geodesic segments parametrized by arc-length. Let the multi-curves $\sigma_{i},\; i=0,...,k+1,$ and $\mathcal{T}_{i,n}\in \tw(\sigma_{i}),\; i=1,...,k,$ be as in Theorem \ref{thm : geodlimit'} and the multi-curve $\hat{\tau}$ be as in (\ref{eq : tau}). Let $1\leq i \leq k$. Suppose that $ \sigma_{i}-\hat{\tau}\neq \emptyset$ and $\gamma \in \sigma_{i}-\hat{\tau}$, then the power of $\mathcal{D}_{\gamma}$ in $\mathcal{T}_{i,n}$ goes to $\infty$ as $n\to \infty$.
\end{lem}

\begin{proof}
 For $i=1,...,k+1$ let $\varphi_{i,n}\in \Mod(S)$ be as in (\ref{eq : phiin'}). For $i=1,...,k$ define multi-curves
$$\sigma_{i,n}=\varphi_{i-1,n}^{-1}(\sigma_{i}) $$
also define the geodesic segments
\begin{center}$\zeta_{i,n}(t)=\varphi_{i-1,n}(\zeta_{n}(t))$ for $t\in [t_{i-1},t_{i+1}]$\end{center}
where $0=t_{0}<t_{1}<...<t_{i}<t_{i+1}<...<t_{k+1}=T$ is the partition from Theorem \ref{thm : geodlimit'}. We claim that 

\begin{claim} \label{claim : zetain}
 There are $\epsilon_{1},\epsilon_{2}>0$ and $\delta>0$ depending only on the sequence $\zeta_{n}$ such that for each $i=1,...,k$ and every $n\in\mathbb{N}$ sufficiently large 
\begin{enumerate}[(i)]
\item $\ell_{\gamma}(\zeta_{i,n}(t_{i}\pm\delta)) \leq \epsilon_{2}$ for every $\gamma \in \sigma_{i}$, and  
\item  The injectivity radius of the points $\zeta_{i,n}(t_{i}\pm\delta)$ outside the collars of the curves in $\hat{\tau}$ is bounded below by $\epsilon_{1}$.
\end{enumerate} 
\end{claim}
Let  $\delta< \min _{i=1,...,k+1} \frac{t_{i}-t_{i-1}}{2}$. Fix $1\leq i\leq k$. By Lemma \ref{lem : singlestratum} the two points $\hat{\zeta}(t_{i} \pm \delta)$ are in the stratum $\mathcal{S}(\hat{\tau})$ and by Theorem \ref{thm : geodlimit'}(2), $\hat{\zeta}(t_{i})\in \mathcal{S}(\sigma_{i})$. Recall that at any point in $\mathcal{S}(\sigma_{i})$ every curve in $\sigma_{i}$ has length $0$, also at any point in $\mathcal{S}(\hat{\tau})$ every curve in $\hat{\tau}$ has length $0$. Then by continuity of length functions there are $\epsilon'_{1},\epsilon'_{2}>0$ such that:
\begin{enumerate}[(i')]
\item  $ \ell_{\gamma}(\hat{\zeta}(t_{i}\pm\delta)) \leq \epsilon'_{2}$ for every $\gamma \in \sigma_{i}$, and 
\item The injectivity radius of the points $\zeta_{i,n}(t_{i}\pm \delta))$ is bounded below by $\epsilon'_{1}$ away from the collars of the curves in $\hat{\tau}$.
\end{enumerate}
By Theorem \ref{thm : geodlimit'}(3) for any $t\in [t_{i-1},t_{i}]$, $\zeta_{i,n}(t)\to \hat{\zeta}(t)$ as $n\to \infty$. Then continuity of length-functions and (i') imply that for $\epsilon_{2}=2\epsilon'_{2}$, (i) holds at $\zeta_{i,n}(t_{i}-\delta)$. Similarly the continuity of length-functions and (ii') imply that for $\epsilon_{1}=\frac{\epsilon'_{1}}{2}$, (ii) holds at  $\zeta_{i,n}(t_{i}-\delta)$. 

Since $\varphi_{i,n}=\mathcal{T}_{i,n} \circ \varphi_{i-1,n}$, by Theorem \ref{thm : geodlimit'}(3), for any $t\in [t_{i},t_{i+1}]$, $\mathcal{T}_{i,n}(\zeta_{i,n}(t))\to\hat{\zeta}(t)$ as $n\to \infty$. Then the continuity of length-functions and the bound (i') imply that 
\begin{equation}\label{eq : lgTinzinti}\ell_{\gamma}(\mathcal{T}_{i,n}(\zeta_{i,n}(t_{i}+\delta)))\leq \epsilon_{2},\end{equation}
for every $\gamma \in \sigma_{i}$ ($\epsilon_{2}=2\epsilon'_{2}$). It follows from Theorem \ref{thm : geodlimit'}(3) and Lemma \ref{lem : singlestratum} that for each $n$, $\mathcal{T}_{i,n}\in \tw (\sigma_{i}-\hat{\tau})$. Thus applying $\mathcal{T}_{i,n}$ does not change the length and the isotopy class of any curve in $\sigma_{i}-\hat{\tau}$ or disjoint from $\sigma_{i}-\hat{\tau}$. Each $\gamma\in\sigma_{i}$ is either in $\sigma_{i}-\hat{\tau}$ or is disjoint from $\sigma_{i}-\hat{\tau}$. Thus the bound (i) at $\zeta_{i,n}(t_{i}+\delta))$ follows from the bound (\ref{eq : lgTinzinti}).
 
As the previous paragraph for any $t\in [t_{i},t_{i+1}]$, $\mathcal{T}_{i,n}(\zeta_{i,n}(t))\to\hat{\zeta}(t)$ as $n\to \infty$. Then the continuity of length functions and the bound (ii') imply that the injectivity radius of the surface $\mathcal{T}_{i,n}(\zeta_{i,n}(t_{i}+\delta)))$ is bounded below by $\epsilon'_{1}$ outside the collars of the curves in $\hat{\tau}$. 
 Decreasing $\delta$ if is necessary, we may assume that $\epsilon'_{2}$ is small enough so that by the Collar Lemma no curve intersecting $\sigma_{i}$ realizes the injectivity radius of the surface $\mathcal{T}_{i,n}(\zeta_{i,n}(t_{i}+\delta))$. Thus the injectivity radius of the surface $\mathcal{T}_{i,n}(\zeta_{i,n}(t_{i}+\delta))$ outside collars of the curves in $\hat{\tau}$ is realized by a curve in $\sigma_{i}$ or disjoint from $\sigma_{i}$. Then the bound (ii) at $\zeta_{i,n}(t_{i}+\delta)$ follows because $\mathcal{T}_{i,n}$ does not change the length of any curve contained in or disjoint from $\sigma_{i}-\hat{\tau}$. The proof of the claim is complete.
\medskip

 We proceed to prove the lemma. Fix $\gamma \in \sigma_{i}-\hat{\tau}$. Let $h_{1}:S\to \zeta_{i,n}(t_{i}-\delta)$ be the marking map of the surface $\zeta_{i,n}(t_{i}-\delta)$. Let $\beta \in \mathcal{C}_{0}(S\backslash \hat{\tau})$ be a curve so that $h_{1}(\beta)$ has minimal length on $\zeta_{i,n}(t_{i}-\delta)$ which has minimal intersection number (1 or 2) with $\gamma$ and does not intersect any curve in $\sigma_{i}-\gamma$ (see Figure \ref{fig : int}). 
 
 \begin{figure}
\centering 
\scalebox{0.2}{\includegraphics{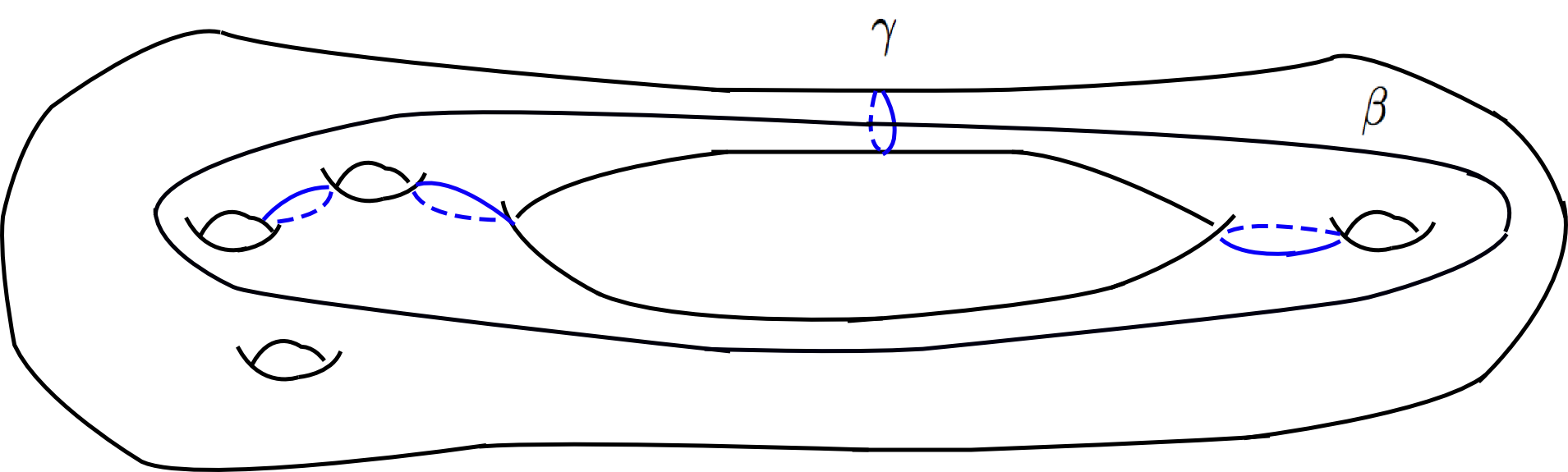}}
\caption{ The curve $\beta\in \mathcal{C}_{0}(S\backslash \hat{\tau})$ is a curve with minimal length on the surface $\zeta_{i,n}(t_{i}-\delta)$ which has minimal intersection number (1 or 2) with $\gamma$ and does not intersect any curve in $\sigma_{i}-\gamma$ ($\sigma_{i}$ consists of the blue curves).}
\label{fig : int}
\end{figure}

 Realize the curves $h_{1}(\beta)$ and $h_{1}(\gamma)$ as geodesics on the surface $\zeta_{i,n}(t_{i}-\delta)$. Denote the collar of $h_{1}(\gamma)$ by $C(h_{1}(\gamma))$. Denote the length of each component of the boundary of the collar by $C$. Furthermore, let $w$ be the width of the collar. Lifting the picture to the universal cover as in Figure \ref{fig : dehn} we can see that the length of $h_{1}(\beta)\cap C(h_{1}(\gamma))$ is bounded above by $w+2C$. Besides, the length of $h_{1}(\beta)$ outside $C(h_{1}(\gamma))$ is bounded above by $d$ the diameter of $\zeta_{i,n}(t_{i}-\delta)$ outside the collars of the curves in $\sigma_{i}$. To see this, let $g_{1}$ be an arc in the collar between the two boundaries of the collar with minimal length. Since diameter of the hyperbolic surface $\zeta_{i,n}(t_{i}-\delta)$ out side the collars of the curves in $\sigma_{i}$ is less than $d$, there is an arc $g_{2}$ in the complement of the collars of the curves in $\sigma_{i}$ that connects the end points of $g_{1}$ on the boundary of the collar and has length less than $d$. The concatenation of $g_{1}$ and $g_{2}$ is a closed curve freely homotopic to $\beta$, so the length of $\beta$ is less than that of $g$. The length of $g_{1}$ is less than the length of $\beta$ intersection with the collar.  Thus the length of $\beta$ out side the collar is less than $d$, as was desired. Then the length of $h_{1}(\gamma)$ is bounded above by $w+2C+d$. By a compactness argument $d$ is bounded above by a constant depending only on a lower bound for the injectivity radius of the surface outside the collars of the curves in $\sigma_{i}$ and $C$. Claim \ref{claim : zetain}(ii) provides the lower bound $\epsilon_{1}$ for the length of $\gamma$ and therefore an upper bound for $C$ and $w$.  Claim \ref{claim : zetain}(i), provides the lower bound $\epsilon_{2}$ for the injectivity radius of the surface outside the collars of the curves in $\hat{\tau}$. Then since $\hat{\tau}\subseteq \sigma_{i}$, in particular we have the lower bound $\epsilon_{2}$ for injectivity radius of the surface outside the collars of the curves in $\sigma_{i}$. Therefore we have an upper bound for $d$. Thus there is an $L>0$ depending only on $\epsilon_{1},\epsilon_{2}$ so that
 \begin{equation} \label{eq : lbn-}\ell_{\beta}(\zeta_{i,n}(t_{i}-\delta)) \leq L. \end{equation}
 Let $h_{2}:S\to \zeta_{i,n}(t_{i}+\delta)$ be the marking of $\zeta_{i,n}(t_{i}+\delta)$. Then $h_{2}(\beta)=h_{1}\circ\mathcal{T}_{i,n}(\beta)$ and $h_{2}(\gamma)=h_{1}(\gamma)$. Denote by $\mathcal{D}_{\gamma}$ the positive Dehn twist about $\gamma$. Let $m_{i}$ be the power of $\mathcal{D}_{\gamma}$ in $\mathcal{T}_{i,n}$. Realize $h_{2}(\gamma)$ and $h_{2}(\mathcal{D}_{\gamma}^{m_{i}}(\beta))$ as geodesics. Lifting the picture to the universal cover as in Figure \ref{fig : dehn} we can see that the length of $\mathcal{D}_{\gamma}^{m_{i}}(\beta)\cap C(h_{2}(\gamma))$ is bounded above by $|m_{i}|\ell_{\gamma}(\zeta_{i,n}(t_{i}+\delta))+w+2C$. Besides, the length of $\mathcal{D}_{\gamma}^{m_{i}}(\beta)$ outside the collar is bounded above by $d$ the diameter of the surface $\zeta_{i,n}(t_{i}+\delta)$ outside the collars of the curves in $\sigma_{i}$. This fact follows from an argument similar to the previous paragraph given to bound the length of  $\beta$ outside the collar. The only difference is that here we consider $g_{1}$ an arc inside the collar between the boundaries of the collar with $m_{i}$ twists about $\gamma$ with minimal length. Then the length of $\mathcal{D}_{\gamma}^{m_{i}}(\beta)$ is bounded above by 
$$|m_{i}|\ell_{\gamma}(\zeta_{i,n}(t_{i}+\delta))+w+2C+d.$$ 
 Similar to the previous paragraph, $w,C$ and $d$ are bounded above by constants depending only on $\epsilon_{1}$ and $\epsilon_{2}$. Moreover assuming that $|m_{i}|\leq N$ for some $N\in \mathbb{N}$, and using the fact that by Claim \ref{claim : zetain}(i), $\ell_{\gamma}(\zeta_{i,n}(t_{i}+\delta))\leq \epsilon_{2}$, we have an upper bound for the first term of the above sum. Thus there is an $L'>0$ depending only on $\epsilon_{1},\epsilon_{2}$ and $N$ so that
\begin{equation} \label{eq : lbn+} \ell_{\beta}(\zeta_{i,n}(t_{i}+\delta)) \leq L'.\end{equation}
 On the other hand, since $\ell_{\gamma}(\zeta_{i,n}(t_{i})) \to 0$ as $n \to \infty$ and $\beta \pitchfork \gamma$, by the Collar Lemma (\cite[\S 4.1] {buser}) we have that
\begin{equation} \label{eq : lbn0}\ell_{\beta}(\zeta_{i,n}(t_{i})) \to \infty \end{equation}
as $n \to \infty$.

 \begin{figure}
\centering 
\scalebox{0.2}{\includegraphics{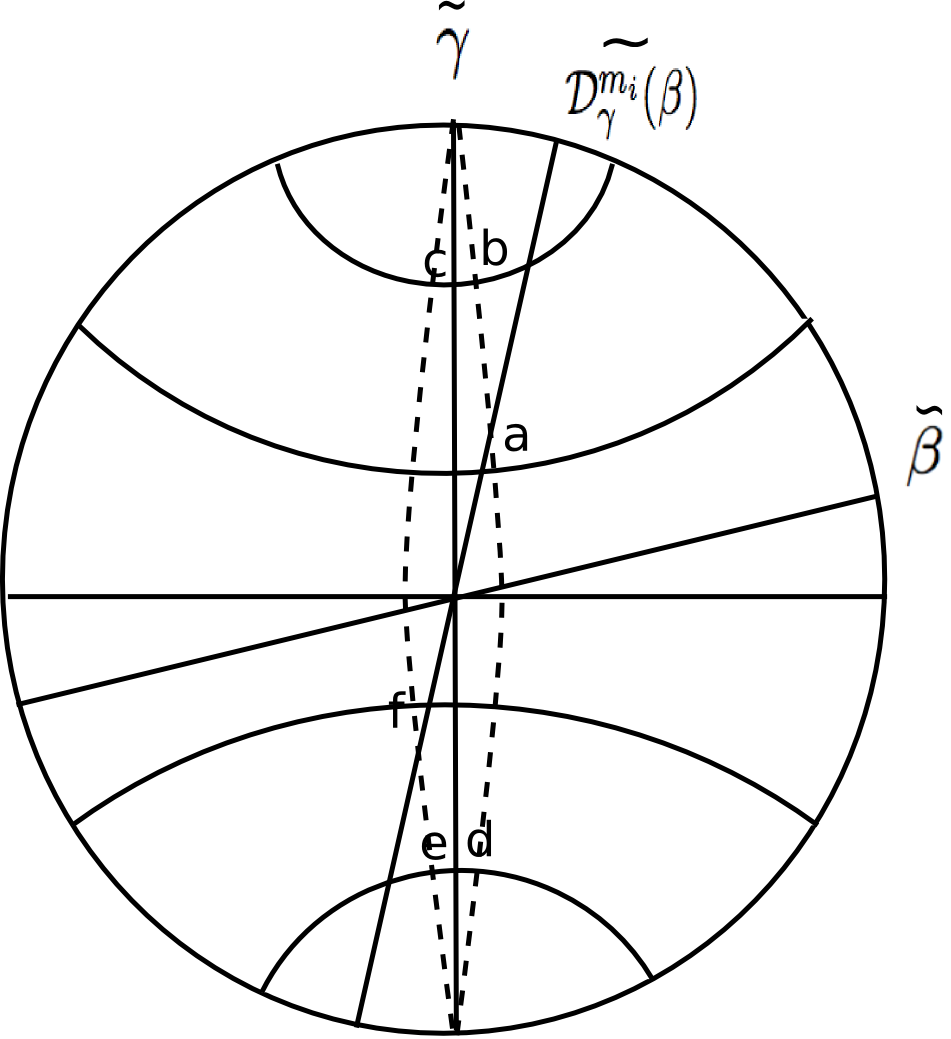}}
\caption{ The lift of the collar of $h_{2}(\gamma)$ and the lift of $h_{2}(\mathcal{D}_{\gamma}^{m_{i}}(\beta))$ to $\mathbb{D}^{2}$ the universal cover of the surface are showen. The length of the intersection of $h_{2}(\mathcal{D}_{\gamma}^{m_{i}}(\beta))$ and the collar of $h_{2}(\gamma)$ is bounded above by the length of the path $abcdef$. The segment $cd$ has length $|m_{i}|\ell_{\gamma}(\zeta_{i,n}(t_{i}+\delta))$. The length of the segments $ab$ and $ef$ are bounded by $C$ the length of each component of the boundary of the collar. The length of the segments $bc$ and $de$ is equal to $w$ the width of the collar.}
\label{fig : dehn}
\end{figure}

The bound (\ref{eq : lbn0}) implies that for any $n$ sufficiently large,
$$\ell_{\beta}(\zeta_{i,n}(t_{i}))> \max\{L,L'\}.$$
 But $t_{i}-\delta<t_{i}<t_{i}+\delta$, so the above bound together with the bounds (\ref{eq : lbn-}) and (\ref{eq : lbn+}) violate the convexity of the $\beta-$length-function along the WP geodesic segment $\zeta_{i,n}$. This contradiction shows that the power of $\mathcal{D}_{\gamma}$ in $\mathcal{T}_{i,n}$ is unbounded. The proof of the lemma is complete.
\end{proof}

Here for the purpose of reference in this paper and in future works we state the following strength version of the picture of the limits of bounded length WP geodesic segments in the completion of the Teichm\"{u}ller space which essentially is the picture from Theorem \ref{thm : geodlimit'} modified to incorporate Lemmas \ref{lem : singlestratum} and \ref{lem : highdehntwisteachgamma}.

\begin{thm} \label{thm : geodlimit}\textnormal{ (Geodesic Limit)}
Given $T>0$. Let $\zeta_{n}: [0,T] \to \overline{\Teich(S)}$ be a sequence of unit speed parametrized WP geodesic segments of length $T$. After possibly passing to a subsequence there exists a partition of the interval $[0, T]$ by $0=t_{0}<t_{1}<...<t_{k}<t_{k+1}=T$, and multi-curves $\sigma_{0},...,\sigma_{k+1}$ where $\sigma_{0}$ and $\sigma_{k+1}$ are possibly empty and the possibly empty multi-curve $\hat{\tau}$ such that $\sigma_{i} \cap \sigma_{i+1}=\hat{\tau}$, for $i=0,...,k$, and a piecewise geodesic 
$$\hat{\zeta}:[0,T]\to \overline{\Teich(S)}$$
with the following properties
\begin{enumerate}
\item \label{gl : mainstr}$\hat{\zeta}((t_{i-1}, t_{i})) \subset \mathcal{S}(\hat{\tau})$, for $i=1,..., k+1$, 
\item \label{gl : hitstr}$\hat{\zeta}(t_{i}) \in \mathcal{S}(\sigma_{i})$, for $i=0,...,k+1$,
\item \label{gl : conv}There are elements $\psi_{n}\in \Mod(S)$ and $\mathcal{T}_{i,n}\in\tw(\sigma_{i}-\hat{\tau})$, for $i=1,..,k$, such that after possibly passing to a subsequence $\psi_{n}(\zeta_{n}|_{[0,t_{1}]})\to\hat{\zeta}|_{[0,t_{1}]}$, and for each $i=1,..,k$, 
$$\mathcal{T}_{i,n} \circ ... \circ \mathcal{T}_{1,n} \circ \psi_{n}(\zeta_{n}|_{[t_{i},t_{i+1}]})\to\hat{\zeta}|_{[t_{i},t_{i+1}]}$$
as $n \to \infty$ in the sense of unit speed parametrized geodesics. For convenience for every $i=0,...,k+1$ we define 
\begin{equation} \label{eq : phiin} \varphi_{i,n}=\mathcal{T}_{i,n} \circ ... \circ \mathcal{T}_{1,n} \circ \psi_{n} \end{equation}
\item \label{gl : unbddpower}The elements $\psi_{n}$ are either trivial or unbounded. Moreover, for any $i=1,...,k$ and $\gamma\in \sigma_{i}$ the power of $\mathcal{D}_{\gamma}$ in the element $\mathcal{T}_{i,n}$ goes to $\infty$ as $n\to \infty$.
\item \label{gl : minlen}The piecewise geodesic $\hat{\zeta}$ is the minimal length path in $\overline{\Teich(S)}$ joining $\hat{\zeta}(0)$ to $\hat{\zeta}(T)$ and intersecting the strata $\mathcal{S}(\sigma_{1}), \mathcal{S}(\sigma_{2}), ..., \mathcal{S}(\sigma_{k})$ in order.
\end{enumerate}
\end{thm}

By Lemma \ref{lem : singlestratum}, $\hat{\tau}\equiv \sigma_{i}\cap \sigma_{i+1}$ for $i=0,...,k$. Part (1) follows from this fact and part (1) of Theorem \ref{thm : geodlimit'}. Part (4) follows from part (4) of Theorem \ref{thm : geodlimit'} and Lemma \ref{lem : highdehntwisteachgamma}.

\subsection{Length-function versus twist parameter control} \label{subsec : bmmtwlf}

 In this subsection we show that, roughly speaking, provided a lower bound for the length of a curve $\gamma$ at the end points of a uniformly bounded length WP geodesic segment $\zeta$, the higher Dehn twist about $\gamma$ forces $\gamma$ to get shorter along $\zeta$ (Corollary \ref{cor : twsh}). Moreover, in Corollary \ref{cor : shtw} we show that the shorter $\gamma$ gets along $\zeta$ the higher Dehn twist builds up about $\gamma$. 

The main technical part of this subsection is the following modification of Lemma 4.5 in \cite{bmm2}. In the following theorem $\mu(x)$ denotes a Bers marking of the point $x\in\overline{\Teich(S)}$ (see $\S$\ref{sec : WPmetric}).
  
\begin{thm}\label{thm : shtw}\textnormal{ (Length-function versus annuler coefficient control)}
Given $\epsilon_{0},T$ and $s$ positive. Let $\zeta_{n}: [0,T_{n}] \to \Teich(S)$ be a sequence of Weil-Petersson geodesic segments parametrized by arc-length of length $2s\leq T_{n}\leq T$. Let $\alpha_{n}$ be a sequence of curves, we have the following

\begin{enumerate}
\item If there are subintervals $J_{n} \subseteq [s,T_{n}-s]$ such that 
\begin{enumerate}
\item $\sup_{t \in J_{n}}\ell_{\alpha_{n}}(\zeta_{n}(t)) \geq \epsilon_{0} $, and
\item $\inf_{t \in J_{n}}\ell_{\alpha_{n}}(\zeta_{n}(t))\to 0$ as $n \to \infty$, 
 \end{enumerate}
 then after possibly passing to a subsequence
$$d_{\alpha_{n}}(\mu(\zeta_{n}(0)),\mu(\zeta_{n}(T_{n})))\to \infty$$ 
as $n \to \infty$.

\item If
\begin{enumerate}
\item $\sup_{t\in[0,T_{n}]}\ell_{\alpha_{n}}(\zeta_{n}(t))\geq \epsilon_{0}$, and
\item $d_{\alpha_{n}}(\mu(\zeta_{n}(0)),\mu(\zeta_{n}(T_{n})))\to \infty$ as $n\to \infty$,
\end{enumerate}
 then after possibly passing to a subsequence
$$\inf_{t \in [0,T_{n}]}\ell_{\alpha_{n}}(\zeta_{n}(t))\to 0$$ 
as $n \to \infty$. 
\end{enumerate}
\end{thm}

\begin{proof}

 Trimming the intervals $[0,T_{n}]$ slightly and changing the parameters $s$ and $T$ we may assume that $T_{n} \equiv T$ for some $T\geq 2s$. After possibly passing to a subsequence by Theorem \ref{thm : geodlimit}, there exist a partition of $[0,T]$ with $0=t_{0}<t_{1}<...<t_{k+1}=T$, multi-curves $\sigma_{i}$ for $i=0,...,k+1$, a multi-curve $\hat{\tau}$, and a piecewise geodesic path 
$$\hat{\zeta}:[0,T]\to \overline{\Teich(S)}$$
so that $\hat{\zeta}([t_{i},t_{i+1}])$ is a geodesic segment in $\overline{\mathcal{S}(\hat{\tau})}$ joining the stratum $\mathcal{S}(\sigma_{i})$ to $\mathcal{S}(\sigma_{i+1})$. Let $\varphi_{i,n}=\mathcal{T}_{i,n}\circ...\circ \mathcal{T}_{i,1}\circ\psi_{n}$ be as in (\ref{eq : phiin}), where $\mathcal{T}_{i,n}\in \tw(\sigma_{i}-\hat{\tau})$ and $\psi_{n}\in \Mod(S)$ is either trivial or unbounded.
\medskip

We start by setting up some notation. For each $i=0,...,k+1$ and $n\in\mathbb{N}$ let
$$\sigma_{i,n}=\varphi_{i,n}^{-1}(\sigma_{i})=\varphi_{i-1,n}^{-1}(\sigma_{i})$$
be the pull backs of $\sigma_{i}$ to the $\zeta_{n}$ picture. For each $i=1,...,k$ and $n\in\mathbb{N}$ let
$$\tau_{i,n}=\varphi_{i-1,n}^{-1}(\hat{\tau}).$$
For each $i = 0,...,k+1$, choose a partial marking $\mu_{i}$ such that
\begin{enumerate}
\item $\sigma_{i} \subset \base(\mu_{i})$, and
\item $\mu_{i}$ restricts to a full marking of each connected component $Y \subseteq S\backslash \hat{\tau}$ with $\xi(Y)\geq 1$.
\end{enumerate}

For each $i=1,...,k+1$ and $n\in\mathbb{N}$ define the pullback marking
$$\mu_{i,n}^{-} =\varphi_{i-1,n}^{-1}(\mu_{i}),$$	
and for each $i=0,...,k$ and $n\in\mathbb{N}$ the pullback marking
$$\mu_{i,n}^{+} =\varphi_{i,n}^{-1}(\mu_{i}).$$	
Let $1\leq i\leq k$ and let $\gamma \in \sigma_{i}-\hat{\tau}$. In the following three claims we will measure the twisting of these markings relative to $\gamma_{n}=\varphi_{i,n}^{-1}(\gamma)$ and prove that 
\begin{equation} \label{eq : twgamnunbdd}d_{\gamma_{n}}(\mu(\zeta_{n}(0)),\mu(\zeta_{n}(T)))\to \infty \end{equation}
as $n \to \infty$. 

\begin{claim}\label{claim : dgmnbd1} 
$ d_{\gamma_{n}}(\mu^{-}_{j,n}, \mu^{+}_{j,n})$ is bounded for any $j=1,...,k$ with $j\neq i$ and $n\in\mathbb{N}$.
\end{claim}
First note that $\base(\mu^{+}_{j,n})$ and $\base(\mu^{-}_{j,n})$ both contain 
$$\sigma_{j,n} = \varphi_{j,n}^{-1}(\sigma_{j}) = \varphi_{j-1,n}^{-1}(\sigma_{j}).$$
Now we verify that
\begin{equation} \label{eq : gaminsig} \gamma_{ n} \notin \sigma_{ j,n},\;\; \text{for any $j\neq i$}.\end{equation}
For otherwise, the length of $\gamma_{n}$ would converge to $0$ at $\zeta_{n}(t_{i})$ and $\zeta_{n}(t_{ j})$, and hence by the convexity of length-functions along WP geodesics on the interval $[t_{i-1},t_{i}]$ or $[t_{i},t_{i+1}]$ (the first if $j < i$ and the second if $ j > i$). So by Theorem \ref{thm : geodlimit}(\ref{gl : mainstr}), $\gamma\in \hat{\tau}$. But this contradicts the choice of $\gamma \in \sigma_{i}-\hat{\tau}$. Thus (\ref{eq : gaminsig}) holds. 

The partial markings $\mu_{ j,n}^{-}$ and $\mu_{j,n}^{+}$ on $S$ restrict to full markings on $S\backslash \sigma_{j,n}$. So by (\ref{eq : gaminsig}) $\gamma_{n}$ intersects $\mu_{ j,n}^{-}$ and $\mu_{ j,n}^{+}$ nontrivially, and therefore $\pi_{\gamma_{n}}(\mu_{j,n}^{\pm})\neq\emptyset$. Let  
$$\widetilde{\mathcal{T}}_{j,n}=\varphi_{j,n}^{-1}\circ \mathcal{T}_{j,n}\circ \varphi_{j,n}.$$
 Note that $\mu_{j,n}^{-}=\widetilde{\mathcal{T}}_{j,n}(\mu_{j,n}^{+})$. Moreover since $\mathcal{T}_{j,n}\in \tw(\sigma_{j})$, $\widetilde{\mathcal{T}}_{j,n}$ is an element of $\tw(\sigma_{j,n})$. So $\mu_{j,n}^{+}$ differs from $\mu_{j,n}^{-}$ by composition of positive Dehn twists about the curves in $\sigma_{j,n}$. Then by the definition of annular subsurface coefficients of partial markings (see $\S$\ref{sec : ccplx}) the subsurface coefficient of a sequence of curves intersecting the markings $\mu^{-}_{j,n}$ and $\mu^{+}_{j,n}$ goes to $\infty$, only if after possibly passing to a subsequence each curve in the sequence is a curve in $\sigma_{j,n}$.  But by (\ref{eq : gaminsig}) it is not the case. The claimed bound follows from this contradiction.
\begin{claim}\label{claim : dgmnbd2}  
$d_{\gamma_{n}}(\mu^{+}_{j,n}, \mu^{-}_{j+1,n})$ is bounded for any $j=1,...,k$ and $n\in\mathbb{N}$.
\end{claim}
The partial markings $\mu_{j}$ and $\mu_{j+1}$ restrict to full markings of $S\backslash \hat{\tau}$, where their marking distance is some finite number. Hence we may connect them with a finite length connected path in the marking graph of $S\backslash\hat{\tau}$. Applying $\varphi_{j,n}^{-1}$ to this path we obtain a path of the same length connecting $\mu_{j,n}^{+}$ to $\mu_{j+1,n}^{-}$ in the marking graph of $S\backslash \tau_{j+1,n}$. Moreover, $\gamma_{n} \notin \tau_{j+1,n}$ (because $\gamma \notin \hat{\tau}$), so all of the markings in the connecting path intersect $\gamma_{n}$ nontrivially. Any two consecutive markings in the path differ by an elementary move and each elementary move increases the $A(\gamma_{n})$ subsurface coefficient by at most one. Thus the claimed bound follows.
 \begin{claim} \label{claim : dgmnubd} 
  \begin{equation} \label{eq : reltwistubd} d_{\gamma_{n}}(\mu^{-}_{i,n},\mu^{+}_{i,n}) \to \infty\;\; \text{as $n \to \infty$}.\end{equation} 
\end{claim}
Note that $\varphi_{i,n}(\mu^{-}_{i,n} ) = \mathcal{T}_{i,n}(\mu_{i})$, so after applying $\varphi_{i,n}$ to all of the curves in the subsurface coefficient in (\ref{eq : reltwistubd}) we get 
$$d_{\gamma}(\mathcal{T}_{i,n}(\mu_{i}), \mu_{i})$$
Now $\mu_{i}$ is a fixed marking which contains $\gamma$ as well as a transversal curve for $\gamma$. By Theorem \ref{thm : geodlimit}(\ref{gl : unbddpower}) $\mathcal{T}_{i,n}$ contains an arbitrarily large power of $\mathcal{D}_{\gamma}$. Then the claimed bound follows.
\medskip

Combining the bounds established in Claims \ref{claim : dgmnbd1}, \ref{claim : dgmnbd2} and \ref{claim : dgmnubd} with the triangle inequality the bound (\ref{eq : twgamnunbdd}) follows. Having this bound in hand we continue by proving our theorem.
\medskip

{\it Proof of part (1):} We show that after possibly passing to a subsequence there is an $1\leq i\leq k$ and a curve $\gamma \in \sigma_{i}-\hat{\tau}$, such that $\alpha_{n}=\varphi_{i,n}^{-1}(\gamma)$. Part (1) then follows from (\ref{eq : twgamnunbdd}).

Let $J_{n}\subset [s,T-s]$ be the subintervals in the statement of part (1). After possibly passing to a subsequence we may assume that the intervals $J_{n}$ converge to an interval $J$. Since each $J_{n}\subset [s,T-s]$ we have that $J\subset [s,T-s]$. By Theorem \ref{thm : geodlimit}(\ref{gl : conv}), for each $j=0,...,k+1$, 
$$\varphi_{j,n}(\zeta_{n}|_{[t_{j},t_{j+1}]})\to\hat{\zeta}|_{[t_{j},t_{j+1}]}$$
as $n\to\infty$. Moreover $\varphi_{j,n}$ is an isometry of the WP metric. So the length of geodesics segments $\zeta_{n}(J_{n})$ converge to the length of $\hat{\zeta}(J)$. Then since $\zeta_{n}$ ($n\in\mathbb{N}$) and $\hat{\zeta}$ are piece-wise geodesics parametrized by arc-length it follows that the length of the intervals $J_{n}$ converge to the length of the interval $J$. Now we show that the lengths of intervals $J_{n}$ are uniformly bounded below. For each $n\in\mathbb{N}$, by 1(a), $\ell_{\alpha_{n}}(\zeta_{n}(t))$ achieves the value $\epsilon_{0}$ in $J_{n}$ and by 1(b), for $n$ sufficiently large 
$$\inf_{t\in J_{n}}\ell_{\alpha_{n}}(\zeta_{n}(t))\leq \frac{\epsilon_{0}}{2}.$$
 Thus by Corollary \ref{cor : gradlfestimate} the length of $J_{n}$ is at least ${\bf d}(\epsilon_{0},\frac{\epsilon_{0}}{2})>0$. This uniform lower bound for the length of $J_{n}$ for all $n$ sufficiently large and the convergence of the length of intervals $J_{n}$ to the length of $J$ imply that the length of $J$ is bounded below by ${\bf d}(\epsilon_{0},\frac{\epsilon_{0}}{2})$. 

For each $n\in\mathbb{N}$, let $t^{*}_{n}\in J_{n}$ be so that $\ell_{\alpha_{n}}(\zeta_{n}(t^{*}_{n}))=\inf_{t \in J_{n}} \ell_{\alpha_{n}}(\zeta_{n}(t))$. There is an $0\leq i\leq k$ so that after possibly passing to a subsequence $t_{n}^{*}$ converge to some $t^{*}\in J\cap [t_{i},t_{i+1}]$. Moreover since $t^{*}_{n}\in J_{n}$ and $J_{n}\subseteq [s,T_{n}-s]$, $t^{*}\neq t_{0}$ and $t_{k+1}$. 

Suppose that $t^{*} \neq t_{i}$ and $t_{i+1}$. In what follows we get a contradiction. By 1(b), $\ell_{\alpha_{n}}(\zeta_{n}(t^{*}_{n}))\to 0$ as $n\to \infty$, so applying $\varphi_{i,n}$ to $\ell_{\alpha_{n}}(\zeta_{n}(t_{n}^{*}))$ we have 
\begin{equation}\label{eq : lphia}\ell_{\varphi_{i,n}(\alpha_{n})}(\varphi_{i,n}(\zeta_{n}(t_{n}^{*})))\to 0\;\text{as}\; n\to \infty.\end{equation} 
Moreover by Theorem \ref{thm : geodlimit}(\ref{gl : conv}), $\varphi_{i,n}(\zeta_{n}(t^{*}_{n}))\to\hat{\zeta}(t^{*})$ as $n\to\infty$. Since $t\neq t_{i}$ and $t_{i+1}$, the point $\hat{\zeta}(t^{*})$ is a marked hyperbolic surface with pinched curves at $\hat{\tau}$. 

Suppose that $\hat{\tau}=\emptyset$. Then by continuity of length functions there is a neighborhood $V$ of $\hat{\zeta}(t^{*})$ where there is a positive lower bound for the length of every curve on any surface in $V$. On the other hand, since $\varphi_{i,n}(\zeta_{n}(t^{*}_{n}))\to \hat{\zeta}(t^{*})$ as $n\to\infty$, for $n$ sufficiently large $\varphi_{i,n}(\zeta_{n}(t_{n}^{*}))$ is in $V$. Then (\ref{eq : lphia}) gives a contradiction to the lower bound for the length of curves on surfaces in $V$.  If $\hat{\tau}\neq\emptyset$, then by continuity of length-functions there is a neighborhood $V$ of $\hat{\zeta}(t^{*})$ and an $\epsilon>0$ so that the only curves on a surface in $V$ with length less than $\epsilon$ are the curves in $\hat{\tau}$. Since $\varphi_{i,n}(\zeta_{n}(t^{*}_{n}))\to\hat{\zeta}(t^{*})$ as $n\to\infty$ for all $n$ sufficiently large $\varphi_{i,n}(\zeta_{n}(t^{*}_{n}))$ is in $V$. Thus by (\ref{eq : lphia}) after possibly passing to a subsequence $\varphi_{i,n}(\alpha_{n})=\beta$ for some $\beta\in\hat{\tau}$.

For each $l=0,...,k+1$, $\hat{\tau}\subseteq \sigma_{l}$, so $\beta\in \sigma_{l}$. Let $0\leq j\leq k+1$. The map $\varphi_{j,n}\circ \varphi_{i,n}^{-1}$ is a composition of powers of positive Dehn twists about curves in $\sigma_{l}$ where $l=j+1,...,i$ if $i>j$ and $l=i+1,...,j$ if $j>i$, and is the identity map if $j=i$ (see (\ref{eq : phiin})). So $\varphi_{j,n}\circ \varphi_{i,n}^{-1}$ preserves $\beta$. Thus $\varphi_{j,n}(\alpha_{n})=\beta$. 

 Given $t\in[0,T]$ we have that $t\in[t_{j},t_{j+1}]$ for some $0\leq j\leq k$. By Theorem \ref{thm : geodlimit}(\ref{gl : conv}), $\varphi_{j,n}(\zeta_{n}(t))\to \hat{\zeta}(t)$ as $n\to \infty$. Moreover, $\ell_{\beta}(\hat{\zeta}(t))\equiv 0$ for all $t\in [0,T]$. Thus the continuity of length-functions implies that $\ell_{\beta}(\varphi_{j,n}(\zeta_{n}(t)))\to 0$ as $n\to \infty$. As we saw above $\varphi_{j,n}(\alpha_{n})=\beta$, so 
 $$\ell_{\varphi_{j,n}(\alpha_{n})}(\varphi_{j,n}(\zeta_{n}(t)))\to 0$$
  as $n\to \infty$, then applying $\varphi_{j,n}^{-1}$ we have that $\ell_{\alpha_{n}}(\zeta_{n}(t))\to 0$ as $n\to \infty$. Thus for any $t\in[0,T]$, $\ell_{\alpha_{n}}(\zeta_{n}(t))\to 0$ as $n\to\infty$. But this contradicts 1(a). 

Therefore, the parameters $t^{*}_{n}$ converge to either $t_{i}$ or $t_{i+1}$. As we saw earlier the end point to which $t^{*}_{n}$ converges is not $t_{0}$ or $t_{k+1}$, so $\sigma_{i}$ and $\sigma_{i+1}$ are not empty. Let $t^{*}_{n}\to t_{i}(t_{i+1})$ as $n\to \infty$. Then $\varphi_{i,n}(\zeta_{n}(t))\to \hat{\zeta}(t_{i})$ ($\varphi_{i,n}(\zeta_{n}(t))\to \hat{\zeta}(t_{i+1})$) as $n\to \infty$. Note that the only curves with length $0$ at $\hat{\zeta}(t_{i}) (\hat{\zeta}(t_{i+1}))$ are the ones in $\sigma_{i}(\sigma_{i+1})$. Thus (\ref{eq : lphia}) and the convergence of length-functions imply that $\alpha_{n}=\varphi_{i,n}^{-1}(\gamma) (\varphi_{i+1,n}^{-1}(\gamma))$ for some $\gamma \in \sigma_{i}-\hat{\tau}(\sigma_{i+1}-\hat{\tau})$, as was desired. 
\medskip


{\it Proof of part (2):} Suppose that there is an $0\leq i\leq k+1$ with $\sigma_{i}\neq \emptyset$ and a sequence $\{n_{r}\}_{r=1}^{\infty}$ so that, $\alpha_{n_{r}}\in \sigma_{i,n_{r}}$. Applying $\varphi_{i,n_{r}}$ to $\ell_{\alpha_{n_{r}}}(\zeta_{n_{r}}(t_{i}))$ we get $\ell_{\varphi_{i,n_{r}}(\alpha_{n_{r}})}(\varphi_{i,n_{r}}(\zeta_{n_{r}}(t_{i})))$. By Theorem \ref{thm : geodlimit}(\ref{gl : conv}) we have that $\varphi_{i,n_{r}}(\zeta_{n_{r}}(t_{i})) \to \hat{\zeta}(t_{i})$ as $r\to\infty$. Since $\alpha_{n_{r}}\in\sigma_{i,n_{r}}$ and $ \sigma_{i,n_{r}}=\varphi_{i,n_{r}}^{-1}(\sigma_{i})$, we have that $\varphi_{i,n_{r}}(\alpha_{n_{r}})\in \sigma_{i}$. Furthermore, the length of every curve in $\sigma_{i}$ is $0$ at $\hat{\zeta}(t_{i})$. So the continuity of length-functions implies that $\ell_{\varphi_{i,n_{r}}(\alpha_{n_{r}})}(\varphi_{i,n_{r}}(\zeta_{n_{r}}(t_{i})))\to 0$ as $r\to\infty$. Therefore $\ell_{\alpha_{n_{r}}}(\zeta_{n_{r}}(t_{i}))\to 0$ as $r\to\infty$. So the proof of part (2) would be complete if we show that there is an $i$ with $\sigma_{i}\neq\emptyset$ and a subsequence $\{n_{r}\}_{r=1}^{\infty}$ so that , $\alpha_{n_{r}}\in\sigma_{i,n_{r}}$. 

Suppose that such a subsequence does not exist. Then for all $n$ sufficiently large and all $i=1,...,k$, $\alpha_{n}\not\in\sigma_{i,n}$. Then $\alpha_{n}$ intersects $\mu_{i,n}^{-}$ for $i=1,...,k+1$ ($\mu_{i,n}^{-}$ is a full marking of $S\backslash\sigma_{i,n}$), and $\mu_{i,n}^{+}$ for $i=0,...,k$ ($\mu_{i,n}^{+}$ is a full marking of $S\backslash\sigma_{i,n}$). For each $i$, let $\widetilde{\mathcal{T}}_{i,n}=\varphi_{i,n}^{-1}\circ \mathcal{T}_{i,n}\circ \varphi_{i,n}$, as before. As we saw earlier $\widetilde{\mathcal{T}}_{i,n}\in\tw(\sigma_{i,n})$ and $\mu_{i,n}^{-}=\widetilde{\mathcal{T}}_{i,n}(\mu_{i,n}^{+})$. So the only annular subsurfaces for which the coefficients of $\mu_{i,n}^{-}$ and $\mu_{i,n}^{+}$ grow as $n\to \infty$ are the ones with core curves in $\sigma_{i,n}$. Thus 
$$d_{\alpha_{n}}(\mu_{i,n}^{-},\mu_{i,n}^{+})$$
 is uniformly bounded for $i=1,...,k$ and all $n$ sufficiently large.  

Moreover, as we saw in the proof of Claim \ref{claim : dgmnbd2} the fact that for each $i=0,...,k+1$ and $n$ sufficiently large, $\alpha_{n}$ intersects $\mu_{i,n}^{+}$ and $\mu_{i,n}^{-}$ implies that 
$$d_{\alpha_{n}}(\mu_{i,n}^{+},\mu_{i,n}^{-})$$
 is uniformly bounded for $i=1,...,k$ and all $n$ sufficiently large. 

Combining the subsurface coefficient bounds from the above two paragraphs by the triangle inequality we conclude that 
$$d_{\alpha_{n}}(\mu_{0,n}^{-},\mu_{k,n}^{+})$$
 is uniformly bounded above for all $n$ sufficiently large. But $\mu_{0,n}^{-}$ is a Bers marking of $\zeta_{n}(0)$ and $\mu_{k,n}^{+}$ is a Bers marking of $\zeta_{n}(T)$, so the above bound contradicts assumption 2(b). 
\end{proof}

We close this section by proving the following two corollaries of Theorem \ref{thm : shtw}. These corollaries provide us with a kind of length-function versus twist parameter bounds over uniformly bounded length WP geodesic segments which often will be used in $\S$\ref{sec : itinerarywpgeod}.

\begin{cor} \label{cor : twsh}\textnormal{ (Large twist $\Longrightarrow$ Short curve)}
Given $T, \epsilon_{0}$ and $\epsilon<\epsilon_{0}$ positive, there is an $N\in\mathbb{N}$ with the following property. Let  $\zeta: [0, T'] \to \Teich(S)$ be a WP geodesic segment of length $T'\leq T$ such that
$$\sup_{t \in [0,T']}\ell_{\gamma}(\zeta(t))\geq \epsilon_{0}.$$ 
If $d_{\gamma}(\mu(\zeta(0)), \mu(\zeta(T'))) >N$, then we have
  $$\inf_{t \in [0,T']} \ell_{\gamma}(\zeta(t))\leq \epsilon.$$
\end{cor}

\begin{proof}
The proof is by contradiction. Suppose that the corollary does not hold. Then there is a sequence of WP geodesic segments $\zeta_{n}: [0,T_{n}] \to \Teich(S)$ parametrized by arc-length with lengths $T_{n}\leq T$ and curves $\gamma_{n} \in \mathcal{C}_{0}(S)$, such that 
\begin{enumerate}[(a)]
\item $\sup_{t \in [0,T_{n}]}\ell_{\gamma_{n}}(\zeta_{n}(t))\geq \epsilon_{0}$  for every $n$, 
\item $d_{\gamma_{n}}(\mu(\zeta_{n}(0)), \mu(\zeta_{n}(T_{n}))) \to \infty $ as $n \to \infty$,
\end{enumerate}
and $\inf_{t \in [0,T_{n}]} \ell_{\gamma_{n}}(\zeta_{n}(t))> \epsilon$ for every $n$. But this contradicts Theorem \ref{thm : shtw}(2).
\end{proof}


\begin{cor} \label{cor : shtw}\textnormal{ (Short curve $\Longrightarrow$ Large twist)}

Given $\epsilon_{0},T,s$ positive with $T>2s$ and $N\in\mathbb{N}$, there is an $\epsilon<\epsilon_{0}$ with the following property. Let $\zeta: [0, T'] \to \Teich(S)$ be a WP geodesic segment parametrized by arc-length of length $T'\in [2s,T]$. Let $J \subseteq [s,T'-s]$ be a subinterval. Suppose that for some $\gamma\in\mathcal{C}_{0}(S)$ we have
$$\sup_{t \in [0,T']} \ell_{\gamma}(\zeta(t)) \geq \epsilon_{0}.$$
If $\inf_{t \in J} \ell_{\gamma}(\zeta(t)) \leq \epsilon$, then 
$$d_{\gamma}(\mu(\zeta(0)),\mu(\zeta(T')))>N.$$
\end{cor}

\begin{proof}
The proof is again by contradiction. Suppose that the corollary does not hold. Then there is a sequence of WP geodesic segments $\zeta_{n}: [0,T_{n}] \to \Teich(S)$ parametrized by arc-length of length $2s\leq T_{n}\leq T$ , $\gamma_{n} \in \mathcal{C}(S)$, subintervals $J_{n} \subset [s,T_{n}-s]$ such that 
\begin{enumerate}[(a)]
\item $\sup_{t \in J_{n}} \ell_{\gamma_{n}}(\zeta_{n}(t)) \geq \epsilon_{0}$ for every $n$, 
\item $\inf_{t\in J_{n}}\ell_{\gamma_{n}}(\zeta_{n}(t))\to 0$ as $n \to \infty$,
\end{enumerate}
and $d_{\gamma_{n}}(\mu(\zeta_{n}(0)),\mu(\zeta_{n}(T_{n}))) \leq N $ for every $n$. But this contradicts Theorem \ref{thm : shtw}(1).
\end{proof}


\section{Stable hierarchy paths}\label{sec : stablehrpath}

In this section we show that a certain class of hierarchy paths are {\it stable} in the pants graph of surfaces.

\begin{definition}\textnormal{ ($d-$stable subset)}\label{def : stable}
Given a function $d:\mathbb{R}^{\geq 1} \times \mathbb{R}^{\geq 0} \to \mathbb{R}^{\geq 0}$ a subset $\mathcal{Y}$ of a metric space $\mathcal{X}$ is $d-$stable if  for any $K \geq 1$ and $C \geq 0$ every $(K,C)-$quasi-geodesic $h$ with end points in $\mathcal{Y}$ is contained in the $d(K,C)$ neighborhood of $\mathcal{Y}$. We call the function $d$ the quantifier of the stability.
\end{definition}

Here we summarize some of the results about stability of subsets of pants graph of surfaces: Brock and Masur in \cite{bm} prove that when $\xi(S)=3$ the pants graph of $S$ is strongly relatively hyperbolic with respect to the quasi-flats corresponding to separating curves. The main ingredient of their proof is that given a hierarchy path $\rho:[m,n]\to P(S)$ the subset of the pants graph $X(\rho)=|\rho|\cup_{W} \{P(W)\times P(W^{c})\}$, where $W$ or $W^{c}=S\backslash W$ is a component domain of $\rho$, is a stable subset of the pants graph. Behrstock, Drutu and Mosher in \cite{behthickmetric} study thick metric spaces. These are metric spaces with rank at least $2$ where any two quasi-flats are connected through a chain of quasi-flats with the property that any two consecutive quasi-flats in the chain have coarse intersection of infinite diameter. They show that thick metric spaces fail to be relatively hyperbolic with respect to any collection of quasi-flats. Moreover, they observe that $P(S)$ for $\xi(S)>3$ is a thick metric space and therefore is not relatively hyperbolic with respect to any collection of quasi-flats. In \cite{bmm2} it is proved that hierarchy paths with bounded combinatorics end points are stable. 
\medskip

In this section we show that restriction of the subsurfaces for which a pair of partial markings or laminations have subsurface coefficient bigger than a given $A>0$ to {\it large subsurfaces} implies the stability of any hierarchy path $\rho$ between the pair in the pants graph. We call such a pair $A-$narrow. Heuristically, these hierarchy paths avoid quasi-flats in the pants graph corresponding to separating multi-curves on the surface. 
\medskip

To be able to save considerable amount of work using results in the context of $\Sigma-$hulls (see $\S$\ref{subsec : sigmahull}) and present our results in a more general setting we prove that for any $\epsilon$ sufficiently large the $\Sigma_{\epsilon}-$hull of an $A-$narrow pair is stable (Theorem \ref{thm : contraction}). Then the stability of hierarchy paths between the $A-$narrow pair follows from the fact that the Hausdorff distance of a hierarchy path between an $A-$narrow pair and the $\Sigma_{\epsilon}-$hull of the pair is bounded depending only on $A$ and $\epsilon$. The last fact is proved in Theorem \ref{thm : narrowhull}. 

\subsection{Narrow pairs}\label{subsec : narrowpair}

In this subsection first we introduce the notion of an {\it A-narrow} pair of partial makings or laminations. Then we will show that any hierarchy path between a narrow pair and the $\Sigma_{\epsilon}-$hull of the pair ($\epsilon$ is sufficiently large) have finite Hausdorff distance depending only on $A$ and $\epsilon$. 

\begin{definition}\textnormal{ (Large subsurface)}
An essential subsurface $Z \subseteq S$ is called large if any connected component of $S\backslash Z$ is either an annulus or a three holed sphere. 
\end{definition}

\begin{definition}\textnormal{ ($A-$narrow)}
A pair of partial markings or laminations $(\mu^{-},\mu^{+})$ is called $A-$narrow if every non-annular essential subsurface $Z \subseteq S$ with the property that
$$ d_{Z}(\mu^{-},\mu^{+}) > A$$
is a large subsurface of $S$. 
\end{definition}

Recall the constants $M_{1},M_{2}, M_{3}$ and $M_{4}$ from Theorem \ref{thm : hrpath} and $B_{0}$ from Theorem \ref{thm : behineq}. We fix the constant $M=M_{1}+M_{2}+M_{3}+M_{4}+B_{0}+4$ in this subsection.

\begin{thm}\textnormal{(The $\Sigma-$hull of a narrow pair)} \label{thm : narrowhull}

Given $\epsilon>M$ and $A>4M+2\epsilon+20$, there is a constant $\Delta=\Delta(A,\epsilon)$ with the following property. Given an $A-$narrow pair $(\mu^{-},\mu^{+})$, the Hasudorff distance of the $\Sigma_{\epsilon}-$hull $\Sigma_{\epsilon}(\mu^{-},\mu^{+})$ and any hierarchy path between $\mu^{-}$ and $\mu^{+}$ is less than $\Delta$.

\end{thm}

\begin{proof} 
Let $\rho:[m,n]\to P(S)$ be a hierarchy path between $\mu^{-}$ and $\mu^{+}$. Let $i \in [m,n]$. Let $X \subseteq S$ be an essential non-annular subsurface. By Theorem \ref{thm : hrpath}(\ref{h : hausd}), there is a vertex $y \in \hull_{X}(\mu^{-},\mu^{+})$ such that $d_{X}(\rho(i),y) \leq M_{4}< \epsilon$ (by the assumption of the theorem on the value of $\epsilon$, $\epsilon>M_{4}$). Then by the definition of the $\Sigma_{\epsilon}-$hull we have that $\rho(i)\in \Sigma_{\epsilon}(\mu^{-},\mu^{+})$. This implies that 
$$|\rho| \subset \Sigma_{\epsilon}(\mu^{-},\mu^{+}).$$ 
We proceed to prove that $\Sigma_{\epsilon}(\mu^{-},\mu^{+})$ is contained in a neighborhood of $|\rho|$.

\begin{lem}\label{lem : nhullsubsurfbd}
 Given $\epsilon>M$ and $A>4M+2\epsilon+20$, there is a constant $d=d(A,\epsilon)$ with the property that for any $P \in \Sigma_{\epsilon}(\mu^{-},\mu^{+})$ there is a $j_{P} \in [m,n]$ such that for any essential non-annular subsurface $X \subseteq S$ we have
\begin{equation} \label{eq : dXrjpbd}d_{X}(\rho(j_{P}),P) \leq d.\end{equation}
\end{lem}
\begin{proof} We start by setting up some parameters and intervals along $[m,n]$ and establishing some inequalities which help us to roughly locate $P$ with respect to $\rho$.

Let $X$ be an essential non-annular subsurface. We have $P\in \Sigma_{\epsilon}(\mu^{-},\mu^{+})$, so there is a point $x_{X}\in \hull_{X}(\mu^{-},\mu^{+})$ such that $d_{X}(P,x_{X})\leq \epsilon$. By Theorem \ref{thm : hrpath}(\ref{h : hausd}) there is a vertex $x'_{X}\in \pi_{X}(|\rho|)$ ($|\rho|=\{\rho(i):i\in [m,n]\}$) with $d_{X}(x_{X},x'_{X})\leq M$.  So by the triangle inequality, $d_{X}(x'_{X},P)\leq M+\epsilon$. Let $j\in [m,n]$ be such that $ x'_{Y}\in\pi_{X}(\rho(j))$, then 
\begin{equation}\label{eq : dXrjP}d_{X}(\rho(j),P)\leq M+\epsilon.\end{equation}

 Let $e=3A$. Define the subset of parameters of $\rho$ assigned to $P$ by
 $$\mathcal{E}_{X}(P)=\{i \in [m,n] : d_{X}(P, \rho(i)) \leq e\}.$$
 This subset is non-empty. Because any $j$ as above is in $\mathcal{E}_{X}$. Denote the minimum and maximum of the set $\mathcal{E}_{X}(P)$ by $e_{X}^{-}$ and $e_{X}^{+}$ respectively. Let the subinterval of $[m,n]$,
 $$E_{X}(P)=[e_{X}^{-},e_{X}^{+}].$$
 When there is no ambiguity we drop the reference to $P$ and denote $E_{X}(P)$ by $E_{X}$. Recall the interval $J_{X}$ from Theorem \ref{thm : hrpath}(\ref{h : J}). 

\begin{claim} \label{claim : EXJX}Suppose that $X$ is a component domain of $\rho$. Then $\mathcal{E}_{X}\cap J_{X}\neq\emptyset$. In particular,
\begin{equation} \label{eq : EXJX}E_{X}\cap J_{X}\neq \emptyset. \end{equation}
\end{claim}
Let $j$ be as in (\ref{eq : dXrjP}). Since $\epsilon+M<e$, we have $j\in\mathcal{E}_{X}$. If $j\in J_{X}$ we are done. Otherwise, let $J_{X}=[j_{X}^{-},j_{X}^{+}]$. By (\ref{eq : dXrjP}), $d_{X}(\rho(j),P)\leq\epsilon+M$. Moreover, by Theorem \ref{thm : hrpath}(4), $\pi_{X}(\rho(j))$ is within distance $M$ of either $\pi_{X}(\rho(j_{X}^{-}))$ or $\pi_{X}(\rho(j_{X}^{+}))$. Suppose that $d_{X}(\rho(j_{X}^{-}),\rho(j))\leq M$. Then by the triangle inequality, 
$$d_{X}(\rho(j_{X}^{-}),P)\leq \epsilon+2M+2\leq e.$$
 Thus $j_{X}^{-}\in E_{X}$. If $d_{X}(\rho(j_{X}^{+}),\rho(j))\leq M$ holds, similarly we get that $j_{X}^{+}\in E_{X}$. The claim is proved.

 \begin{claim}\label{claim : JPdist}
Given $D>0$. Let $j_{1},j_{2}\in J_{X}$. Let $j_{1}\leq j\leq j_{2}$. Suppose that $d_{X}(\rho(j_{1}),P)\leq D$ and $d_{X}(\rho(j_{2}),P)\leq D$. Then $d_{X}(\rho(j),P)\leq 2D+5$.
 \end{claim}
 Let the vertices $u=\rho(j_{1})\cap g_{X}$, $v=\rho(j)\cap g_{X}$ and $w=\rho(j_{2})\cap g_{X}$ be as in Theorem \ref{thm : hrpath}(\ref{h : J}). Then by Theorem \ref{thm : hrpath}(\ref{h : monoton}) (Monotonicity), $u\leq v\leq w$ as vertices along $g_{X}$. We have $u\in\rho(j_{1})$ and $\diam_{X}(\rho(j_{1}))\leq 2$, then since $d_{X}(P,\rho(j_{1}))\leq D$, 
 $$d_{X}(P,u)\leq D+\diam_{X}(\rho(j_{1}))\leq D+2.$$ 
Similarly $w\in\rho(j_{2})$ and $\diam_{X}(\rho(j_{2}))\leq 2$, then since $d_{X}(P,\rho(j_{2}))\leq D$, $d_{X}(P,w)\leq D+2$. Then by the triangle inequality 
$$d_{X}(u,w)\leq 2D+4+\diam_{X}(P)\leq 2D+4+2.$$ 
Thus either $d_{X}(u,v)\leq D+3$, or $d_{X}(v,w)\leq D+3$. Suppose that $d_{X}(u,v)\leq D+3$. As we saw above $d_{X}(P,u)\leq D+2$, so by the triangle inequality we get 
$$d_{X}(P,v)\leq 2D+5.$$
Now since $v\in\rho(j)$, the claim follows. Assuming that $d_{W}(v,w)\leq D+3$, similarly we get the bound. The claim is proved.
\medskip

Suppose that $X$ is a component domain of $\rho$. Let $J_{X}=[j_{X}^{-},j_{X}^{+}]$.

Suppose that $j^{-}_{X}\leq e_{X}^{-}\leq e_{X}^{+}\leq j_{X}^{+}$. Since $d_{X}(\rho(e_{X}^{\pm}),P)\leq e$, by Claim \ref{claim : JPdist}, for any $j\in E_{X}$,
$$d_{X}(\rho(j),P)\leq 2e+5.$$

Suppose that $e_{X}^{-}\leq j_{X}^{-}\leq e_{X}^{+}\leq j_{X}^{+}$. By Theorem \ref{thm : hrpath}(4), $d_{X}(\rho(j_{X}^{-}),\rho(e_{X}^{-}))\leq M$. Moreover, $d_{X}(\rho(e_{X}^{-}),P)\leq e$. Thus by the triangle inequality $d_{X}(\rho(j_{X}^{-}),P)\leq M+e+2$. Furthermore, $d_{X}(\rho(e_{X}^{+}),P)\leq e$. Then by Claim \ref{claim : JPdist}, for any $j_{X}^{-}\leq j\leq e_{X}^{+}$, 
$$d_{X}(\rho(j),P)\leq 2M+2e+9.$$
 Moreover by Theorem \ref{thm : hrpath}(4), for any $e^{-}_{X}\leq j\leq j^{-}_{X}$, $d_{X}(\rho(j),\rho(e_{X}^{-}))\leq M$. Also $d_{X}(\rho(e_{X}^{-}),P)\leq e$. Then by the triangle inequality 
 $$d_{X}(\rho(j),P)\leq e+M+2.$$
  By the above two inequalities for any $j\in E_{X}$,
$$d_{X}(\rho(j),P)\leq 2M+2e+9.$$
Assuming that $j_{X}^{-}\leq e_{X}^{-}\leq j_{X}^{+}\leq e_{X}^{+}$ similarly we get the above inequality. 

Suppose that $e_{X}^{-}\leq e_{X}^{+}\leq j_{X}^{-}\leq j_{X}^{+}$. Then for any $j\in E_{X}$ we have $j\leq j_{X}^{-}$. Then by Theorem \ref{thm : hrpath}(4), $d_{X}(\rho(j),\rho(j_{X}^{+}))\leq M$. Moreover $d_{X}(\rho(e_{X}^{+}),P)\leq e$. These two inequalities and $d_{X}(\rho(e_{X}^{+}),\rho(j_{X}^{+}))\leq M$, combined by the triangle inequality give us
$$d_{X}(P,\rho(j))\leq 2M+e+4.$$
Assuming $j_{X}^{-}\leq j_{X}^{+}\leq e_{X}^{-}\leq e_{X}^{+}$ similarly we get the above inequality.

Therefore for any $j\in E_{X}$ we have
\begin{equation}\label{eq : dXrjPcomp'}d_{X}(\rho(j),P)\leq 2e+2M+9.\end{equation}

Now suppose that $X$ is not a component domain of $\rho$. Then by Theorem \ref{thm : hrpath}(\ref{h : largelink}) we have
$$d_{X}(\mu^{-},\mu^{+})\leq M$$
 The above bound and the fact that $\diam_{X}(\mu^{\pm})\leq 2$ imply that the distance between any two points on $\hull_{X}(\mu^{-},\mu^{+})$ is at most $2M+10$. By Theorem \ref{thm : hrpath}(\ref{h : hausd}) the Hausdorff distance of $\pi_{X}(|\rho|)$ and $\hull_{X}(\mu^{-},\mu^{+})$ is bounded above by $M$. So there is a point on the hull within distance $M$ of $\pi_{X}(\rho(j))$. Moreover there is a point on the hull within distance $\epsilon$ of $\pi_{X}(P)$. Therefore 
\begin{equation}\label{eq : dXrjP'}d_{X}(\rho(j),P)\leq \epsilon+3M+10.\end{equation}
\medskip

To prove the lemma it suffices to show that
\begin{equation}\label{eq : intEX}\bigcap_{\substack{X \subseteq S\\ \text{non-annular}}} E_{X} \neq \emptyset.\end{equation}
To see this, let $j_{P} \in \bigcap_{\substack{X \subseteq S\\ \text{non-annular}}} E_{X}$. Then by the inequalities (\ref{eq : dXrjPcomp'}) and (\ref{eq : dXrjP'}) and since $2e+2M+9>\epsilon+3M+10$, we have 
$$d_{X}(\rho(j_{P}),P)\leq 2e+2M+9,$$
 which is the desired bound in (\ref{eq : dXrjpbd}).
 \medskip

Our strategy to prove that (\ref{eq : intEX}) holds is to verify that for any two essential non-annular subsurfaces $Y$ and $W$ we have
\begin{equation}\label{eq : EYEW}E_{Y} \cap E_{W} \neq \emptyset.\end{equation}
Then Helly's Theorem in one dimension (see \cite{helly}) implies that the intersection of all of the intervals $E_{X}$ where $X\subseteq S$ is a non-annular subsurface is nonempty. 
\medskip

If $d_{Y}(\mu^{-},\mu^{+})\leq A$, then the distance of any two points on $\hull_{Y}(\mu^{-},\mu^{+})$ is at most $2A+10$. We have $P\in \Sigma_{\epsilon}(\mu^{-},\mu^{+})$, so there is a vertex $x_{Y}\in \hull_{Y}(\mu^{-},\mu^{+})$ with $d_{Y}(P,x_{Y}) \leq\epsilon$. By Theorem \ref{thm : hrpath}(\ref{h : hausd}) for every $i \in [m,n]$ there is a vertex $y\in \hull_{Y}(\mu^{-},\mu^{+})$ such that $d_{Y}(\rho(i), y) \leq M$. Therefore
$$d_{Y}(\rho(i),P) \leq \epsilon+M+2A+10\leq e.$$
Thus $E_{Y}=[m,n]$, which obviously intersects $E_{W}$. If $d_{W}(\mu^{-},\mu^{+})\leq A$, similarly we can conclude that $E_{W}=[m,n]$, which implies that $E_{W}\cap E_{Y}\neq \emptyset$. Therefore, in the rest of the proof we will assume that 
\begin{eqnarray}\label{eq : dYdW>A}
d_{Y}(\mu^{-},\mu^{+})&>&A,\;\text{and}\\
d_{W}(\mu^{-},\mu^{+})&>&A.\nonumber
\end{eqnarray}
 We consider the following collection of essential subsurfaces:
$$\mathfrak{L}:=\mathfrak{L}_{A}(\mu^{-},\mu^{+})=\{ X\subseteq S\;\;\text{non-annular} : d_{X}(\mu^{-},\mu^{+})> A\}.$$
Note that since $A>M$ any $X\in\mathfrak{L}$ is a component domain of $\rho$. Moreover by (\ref{eq : dYdW>A}), $Y,W\in\mathfrak{L}$.

For each $X \in \mathfrak{L}$ define the parameters
\begin{eqnarray*}
i_{X}^{-}&=&\max\{i \in [m,n] : d_{X}(\rho(i),\mu^{-})\leq M\},\;\text{and}\\
i_{X}^{+}&=&\min\{ i \in [m,n] : d_{X}(\rho(i),\mu^{+})\leq M\}.
\end{eqnarray*}
\begin{claim}\label{claim : IXJX}
Let $J_{X}=[j_{X}^{-},j_{X}^{+}]$. Then $j_{X}^{-}\leq i_{X}^{-}\leq i_{X}^{+}\leq j_{X}^{+}$ and we may write
\begin{equation}\label{eq : IXJX}[i_{X}^{-},i_{X}^{+}]\subseteq J_{X}.\end{equation}
\end{claim}

 By Theorem \ref{thm : hrpath}(\ref{h : lrbddproj}), $d_{X}(\mu^{-},\rho(j_{X}^{-}))\leq M$, then since $i_{X}^{-}$ is the maximal time so that $d_{X}(\rho(i),\mu^{-})\leq M$ we have $i_{X}^{-}\geq j_{X}^{-}$. Moreover, $d_{X}(\mu^{+},\rho(j_{X}^{+}))\leq M$, then since $i_{X}^{+}$ is the minimal time so that $d_{X}(\rho(i),\mu^{+})\leq M$ we have $i_{X}^{+}\leq j_{X}^{+}$. Then since $J_{X}$ is an interval, $i_{X}^{-},i_{X}^{+}\in J_{X}$. Now we show that $i_{X}^{-}\leq i_{X}^{+}$. As we said above $d_{X}(\rho(j_{X}^{-}),\mu^{-})\leq M$. Moreover by the set up of $i_{X}^{-}$, $d_{X}(\rho(i_{X}^{-}),\mu^{-})\leq M$. The last two inequalities combined by the triangle inequality imply that 
\begin{equation} \label{eq : dXri-rj-}d_{X}(\rho(i^{-}_{X}),\rho(j_{X}^{-}))\leq 2M+2. \end{equation}
 Similarly we have that 
\begin{equation}\label{eq : dXri+rj+}  d_{X}(\rho(i^{+}_{X}),\rho(j_{X}^{+}))\leq 2M+2.\end{equation}
We have $X\in\mathfrak{L}$, so $d_{X}(\mu^{-},\mu^{+})> A> 6M+20$. Moreover by Theorem \ref{thm : hrpath}(\ref{h : lrbddproj}), $d_{X}(\mu^{-},\rho(j_{X}^{-}))\leq M$ and $d_{X}(\mu^{+},\rho(j_{X}^{+}))\leq M$. These three inequalities combined by the triangle inequality give us
\begin{equation}\label{eq : dXrj-rj+}d_{X}(\rho(j_{X}^{-}),\rho(j_{X}^{+}))>4M+16.\end{equation}
 Let the vertices $u=\rho(j_{X}^{-})\cap g_{X}$, $w=\rho(j_{X}^{+})\cap g_{X}$, $v=\rho(i_{X}^{-})\cap g_{X}$ and $v'=\rho(i_{X}^{+})\cap g_{X}$ be as in Theorem \ref{thm : hrpath}(\ref{h : J}). We have $u\in\rho(j_{X}^{-})$ and $v\in\rho(i_{X}^{-})$. Moreover by Lemma \ref{lem : diamproj}, $\diam_{X}(\rho(i_{X}^{-}))\leq 2$ and $\diam_{X}(\rho(j_{X}^{-}))\leq 2$. Then by (\ref{eq : dXri-rj-}) we have that 
 $$d_{X}(u,v)\leq d_{X}(\rho(i_{X}^{-}),\rho(j_{X}^{-}))+\diam_{X}(\rho(i_{X}^{-}))+\diam_{X}(\rho(j_{X}^{-}))\leq 2M+6.$$ 
Similarly by (\ref{eq : dXri+rj+}) we have $d_{X}(w,v')\leq 2M+6$, and by (\ref{eq : dXrj-rj+}) we have $d_{X}(u,w)> 4M+16$. Then since $u< w$ as vertices along $g_{X}$ the last three inequalities of $\mathcal{C}(X)$ distances imply that $v< v'$. Then Theorem \ref{thm : hrpath}(\ref{h : monoton}) (Monotonicity) implies that $i_{X}^{-}\leq i_{X}^{+}$. 

 \begin{claim}\label{claim : dXimu}
 Let $X\in \mathfrak{L}$. If $i\in [m,i_{X}^{-}]$, then $d_{X}(\rho(i), \mu^{-}) \leq 2M+5$. If $i\in[i_{X}^{+},n]$, then $d_{X}(\rho(i),\mu^{+})\leq 2M+5$.
 \end{claim}
We  prove the first part of the claim (The proof of the second part is similar). Let $J_{X}=[j_{X}^{-},j_{X}^{+}]$. By Claim \ref{claim : IXJX}, $j_{X}^{-}\leq i_{X}^{-}$. If $i\leq j_{X}^{-}$ then by Theorem \ref{thm : hrpath}(\ref{h : lrbddproj}) we have
$$d_{X}(\mu^{-},\rho(i))\leq M<2M+3,$$
 which is the desired bound. Otherwise, $j_{X}^{-}< i\leq i_{X}^{-}$. Note that $j_{X}^{-},i,i_{X}^{-}\in J_{X}$. Let the vertices $u=\rho(j_{X}^{-})\cap g_{X}$, $v=\rho(i)\cap g_{X}$ and $w=\rho( i_{X}^{-})\cap g_{X}$ be as in  Theorem \ref{thm : hrpath}(\ref{h : J}). Since $u\in\rho(j_{X}^{-})$ and $w\in\rho(i_{X}^{+})$, by (\ref{eq : dXri-rj-}) we have
\begin{equation}\label{eq : dXuw}d_{X}(u,w)\leq 2M+2+\diam_{X}(\rho(j_{X}^{-}))+\diam_{X}(\rho(i_{X}^{-}))\leq 2M+6.\end{equation} 
Moreover, since $j_{X}^{-}< i\leq i_{X}^{-}$, by Theorem \ref{thm : hrpath}(\ref{h : monoton}) (Monotonicity) $u\leq v\leq w$ as vertices along the geodesic $g_{X}\subset \mathcal{C}(X)$. So by the inequality (\ref{eq : dXuw}) either $d_{X}(u,v)\leq M+3$ or $d_{X}(v,w)\leq M+3$. 

First suppose that $d_{X}(u,v)\leq M+3$. Then since $u\in\rho(j_{X}^{-})$ and $v\in\rho(i)$ we have
$$d_{X}(\rho(j_{X}^{-}),\rho(i))\leq M+3.$$
 The above inequality and the inequality $d_{X}(\rho(j_{X}^{-}),\mu^{-})\leq M$ (Theorem \ref{thm : hrpath}(\ref{h : lrbddproj})) combined by the triangle inequality imply that 
$$d_{X}(\rho(i),\mu^{-})\leq 2M+3+\diam_{X}(\rho(j_{X}^{-}))\leq 2M+5$$
as was desired. 

Now suppose that $d_{X}(v,w)\leq M+3$. Then 
$$d_{X}(\rho(i),\rho(j_{X}^{-}))\leq M+3.$$
 The above inequality and the inequality $d_{X}(\mu^{-},\rho(j_{X}^{-}))\leq M$ (Theorem \ref{thm : hrpath}(\ref{h : lrbddproj})) combined by the triangle inequality imply that $d_{Y}(\rho(i),\mu^{-})\leq 2M+5$ as was desired. The proof of the first part of the claim is complete.
\medskip

Given $Y,W \in \mathfrak{L}$ by the $A-$narrow condition either 
\begin{enumerate}
\item $Y \pitchfork W$, 
\item $W \subsetneq Y$, or 
\item $Y \subsetneq W$.
 \end{enumerate}
holds. In what follows we discuss these three cases and in each case verify that (\ref{eq : EYEW}) holds.
\medskip

\noindent{\bf Case 1:} $Y \pitchfork W$.

The subsurfaces $Y$ and $W$ are in $\mathfrak{L}$. Moreover the constant $A> 4M$. So $d_{Y}(\mu^{-},\mu^{+})> 4M$ and $d_{W}(\mu^{-},\mu^{+})> 4M$. Then by Definition \ref{def : ordersubsurf} either $Y<W$ or $W<Y$ (not both). Suppose that $Y<W$ ($W<Y$ can be treated similarly). Then by Definition \ref{def : ordersubsurf} the following two inequalities hold 
\begin{eqnarray}
d_{Y}(\mu^{-},\partial{W})&>&M,\;\text{and}\label{eq : wy1}\\
d_{W}(\mu^{+},\partial{Y})&>&M\label{eq : wy2}.
\end{eqnarray}
We proceed to discuss the following three subcases depending on the range of values of $d_{Y}(P,\mu^{-})$ and $d_{Y}(P,\mu^{+})$. In each case we verify that (\ref{eq : EYEW}) holds. 
\medskip

\noindent {\bf Case 1.1:} The inequality
 \begin{equation}\label{eq : dYPmu-}d_{Y}(P,\mu^{-}) \leq 3M+\epsilon+7\end{equation}  
 holds.

First we show that the following inclusion of intervals holds
\begin{equation}\label{eq : EYsubset}[m,i_{Y}^{-}]\subseteq E_{Y} .\end{equation}
Let $i\in [m,i^{-}_{Y}]$, then by the first part of Claim \ref{claim : dXimu} we have 
$$d_{Y}(\rho(i),\mu^{-})\leq 2M+5.$$
 The above inequality and (\ref{eq : dYPmu-}) combined by the triangle inequality give us 
$$d_{Y}(P,\rho(i)) \leq 5M+\epsilon+12+\diam_{Y}(\mu^{-})\leq 5M+\epsilon+14.$$ 
Moreover, $e=3A>5M+\epsilon+14$, so by the set up of the interval $E_{Y}$ we have $i \in E_{Y}$. So we conclude that (\ref{eq : EYsubset}) holds.

Now we show that the inclusion of intervals
\begin{equation}\label{eq : EWsubset}[m, i_{W}^{-}] \subseteq E_{W} \end{equation}
holds. To see this, note that since $Y \in \mathfrak{L}$ we have 
$$d_{Y}(\mu^{-},\mu^{+})>A=4M+2\epsilon+20.$$
 By (\ref{eq : wy2}) and Theorem \ref{thm : behineq}(Behrstock inequality) we have
 $$d_{Y}(\partial{W},\mu^{+}) \leq M.$$
  The last two inequalities and (\ref{eq : dYPmu-}) combined by the triangle inequality imply that $d_{Y}(P,\partial{W})>\epsilon> M$. Thus by the Behrstock inequality we have 
\begin{equation}\label{eq : dWbdYp}d_{W}(\partial{Y}, P) \leq M.\end{equation}
By (\ref{eq : wy1}) and the Behrstock inequality we have that $d_{W}(\partial{Y},\mu^{-})\leq M$. This inequality and (\ref{eq : dWbdYp}) combined by the triangle inequality give us 
\begin{equation}\label{eq : dWP,nu-}d_{W}(P,\mu^{-}) \leq 2M +2\end{equation}
Let $i\in [m,i_{W}^{-}]$, by the first part of Claim \ref{claim : dXimu}, 
$$d_{W}(\mu^{-},\rho(i))\leq 2M+5.$$
 Combining the above inequality and (\ref{eq : dWP,nu-}) by the triangle inequality we get 
$$d_{W}(P,\rho(i))\leq 4M+9.$$ 
Now since $e>4M+9$, by the set up of the interval $E_{W}$, $i\in E_{W}$. So we conclude that (\ref{eq : EWsubset}) holds. 
\medskip 

The inclusions of intervals (\ref{eq : EYsubset}) and (\ref{eq : EWsubset}) together imply that $E_{W} \cap E_{Y} \neq \emptyset$.
\medskip

\noindent{\bf Case 1.2:} The inequality
\begin{equation}\label{eq : dYPmu}\min\{d_{Y}(P,\mu^{-}),d_{Y}(P,\mu^{+})\}>3M+\epsilon+7\end{equation}
holds.

We show that 
\begin{equation}\label{eq : EYI}E_{Y} \cap [i_{Y}^{-},i_{Y}^{+}] \neq \emptyset.\end{equation}
To see this, note that $P\in \Sigma_{\epsilon}(\mu^{-},\mu^{+})$, then by (\ref{eq : dXrjP}), there is an $i \in [m,n]$ such that 
\begin{equation}\label{eq : dYPi}d_{Y}(P,\rho(i))\leq M+\epsilon.\end{equation}
 Now since $e>M+\epsilon$, by the set up of the interval $E_{Y}$, $i\in E_{Y}$. We proceed to show that $i \in [i_{Y}^{-},i_{Y}^{+}]$. By (\ref{eq : dYPmu}), 
 $$d_{Y}(P,\mu^{-})>3M+\epsilon+7.$$
  Combining the above inequality and (\ref{eq : dYPi}) by the triangle inequality we have 
$$d_{Y}(\rho(i),\mu^{-})>2M+7-\diam_{Y}(\rho(i))\geq 2M+5.$$
 Then by the contrapositive of the first part of Claim \ref{claim : dXimu} we conclude that $i>i_{Y}^{-}$. Furthermore by (\ref{eq : dYPmu}), $d_{Y}(P,\mu^{+})>3M+\epsilon+7$. Combining this inequality and (\ref{eq : dYPi}) by the triangle inequality we obtain 
$$d_{Y}(\rho(i),\mu^{+})>2M+5.$$
 Then by the contrapositive of the second part of Claim \ref{claim : dXimu} we conclude that $i<i_{Y}^{+}$. Thus $i\in[i_{Y}^{-},i_{Y}^{+}]$. The proof of (\ref{eq : EYI}) is complete.

Now we show that the inclusion of intervals
\begin{equation}\label{eq : EWsubset'}[m,i_{W}^{-}]\subseteq E_{W}\end{equation} 
holds. By (\ref{eq : wy2}) and Theorem \ref{thm : behineq}(Behrstock inequality), $d_{Y}(\partial{W},\mu^{+})\leq M$. Moreover, by (\ref{eq : dYPmu}), $d_{Y}(P,\mu^{+})>3M+\epsilon+7$. The last two inequalities combined by the triangle inequality give us 
$$d_{Y}(P,\partial{W}) > 2M+\epsilon+7-\diam_{Y}(\partial{W})>M.$$
 Therefore, by the Behrstock inequality 
\begin{equation}\label{eq : dWPbdY}d_{W}(P,\partial{Y})\leq M.\end{equation}
Having (\ref{eq : dWPbdY}) the rest of the proof of the inclusion of intervals (\ref{eq : EWsubset'}) is similar to the argument given after (\ref{eq : dWbdYp}) to prove the inclusion of intervals (\ref{eq : EWsubset}).

\begin{claim}\label{claim : iY+<iW-}
We have that $i_{Y}^{+} \leq i_{W}^{-}$.
\end{claim}
 By (\ref{eq : wy2}) and the Behrstock inequality, $d_{Y}(\partial{W},\mu^{+})\leq M$. By (\ref{eq : IXJX}) we have $i_{W}^{+} \in J_{W}$. Then by Theorem \ref{thm : hrpath}(\ref{h : J}), $\partial{W}\subset\rho(i_{W}^{+})$. Thus 
 $$d_{Y}(\rho(i_{W}^{-}),\mu^{+})\leq M.$$
  Then since $i_{Y}^{+}$ is the minimal time such that $d_{Y}(\rho(i), \mu^{+})\leq M$ holds, we have that $i_{Y}^{+}\leq i_{W}^{-}$.
  \medskip

By Claim \ref{claim : iY+<iW-}, $[i_{Y}^{-},i_{Y}^{+}]\subseteq [m,i_{W}^{-}]$. Then by (\ref{eq : EWsubset'}), $[i_{Y}^{-},i_{Y}^{+}]\subseteq E_{W}$. Then by (\ref{eq : EYI}) we may conclude that $E_{Y}\cap E_{W}\neq\emptyset$.
\medskip

\noindent {\bf Case 1.3:} The inequality
 \begin{equation}\label{eq : dYPmu+}d_{Y}(P,\mu^{+})\leq 3M+\epsilon+7\end{equation}
 holds.

The inequality (\ref{eq : dYPmu+}) and the second part of Claim \ref{claim : dXimu} using an argument similar to the one for the proof of (\ref{eq : EYsubset}) in Case 1.1 imply that the inclusion of intervals
\begin{equation}\label{eq : EYiY+n}[i_{Y}^{+},n]\subseteq E_{Y},\end{equation}
holds. Let $j\in J_{W}$. By Theorem \ref{thm : hrpath}(\ref{h : J}) we have  $\partial{W}\subset\rho(j)$. By (\ref{eq : wy2}) and the Behrstock inequality, $d_{Y}(\partial{W},\mu^{+})\leq M$. Therfeore $d_{Y}(\rho(j),\mu^{+})\leq M$. Then since $i_{Y}^{+}$ is the minimal parameter such that $d_{Y}(\rho(i),\mu^{+})\leq M$ holds, we have $j\geq i_{Y}^{+}$. Therefore $j\in[i_{Y}^{+},n]$. Thus we have $J_{W}\subseteq [i_{Y}^{+},n]$. Then by (\ref{eq : EYiY+n}), $J_{W} \subseteq E_{Y}$. Moreover, by (\ref{eq : EXJX}), $J_{W} \cap E_{W}\neq \emptyset$. Thus we conclude that $E_{W} \cap E_{Y} \neq \emptyset$.
\medskip

\noindent {\bf Case 2:} $W \subsetneq Y$.
 
 Recall the constant $e=3A$. We consider the following two subcases depending on the value of $d_{Y}(\partial{W},P)$.
\medskip

\noindent {\bf Case 2.1:} $d_{Y}(\partial{W},P) \leq e$.

Let $i \in J_{W}$. By Theorem \ref{thm : hrpath}(\ref{h : J}) $\partial{W} \subset \rho(i)$, so 
$$d_{Y}(\rho(i),P)\leq d_{Y}(\partial{W},P)\leq e.$$
Then by the set up of the interval $E_{Y}$ we have that $i\in E_{Y}$. Thus $J_{W}\subseteq E_{Y}$. Moreover, by (\ref{eq : EXJX}), $J_{W}\cap E_{W}\neq\emptyset$. Thus $E_{W}\cap E_{Y}\neq\emptyset$.
\medskip

\noindent {\bf Case 2.2:}
  \begin{equation}\label{eq : dYbdWP}d_{Y}(\partial{W},P) > e.\end{equation}

The pants decomposition $P\in\Sigma_{\epsilon}(\mu^{-},\mu^{+})$ and the subsurface $Y$ is a component domain of $\rho$. Let $l\in J_{Y}\cap E_{Y}$ be form Claim \ref{claim : EXJX}. Let $x'_{Y}=\rho(l)\cap g_{Y}$. As we saw in the proof of the claim,
 \begin{equation}\label{eq : dYx'YP}d_{Y}(x'_{Y},P)\leq 2M+\epsilon+4.\end{equation} 
 Let $h$ be a geodesic in $\mathcal{C}(Y)$ connecting $\pi_{Y}(P)$ to $x'_{Y}$. 
 \begin{claim}\label{claim : hintnbhdbdW}
 The geodesic segment $h$ does not intersect the $1-$neighborhood of $\pi_{Y}(\partial{W})$. 
 \end{claim}
 Otherwise, there is a vertex $z\in h$ with $d_{Y}(z,\partial{W})\leq 1$ (see Figure \ref{fig : case2.2ftr}). Then we have 
\begin{eqnarray}\label{eq : dYPbdW}
d_{Y}(P,\partial{W})&\leq& d_{Y}(P,z)+d_{Y}(z,\partial{W})\\
&\leq& d_{Y}(P,x'_{Y})+1\leq (2M+\epsilon+4)+1.\nonumber
\end{eqnarray} 
The first inequality is the triangle inequality and the second inequality follows from the inequalities $d_{Y}(P,z)\leq d_{Y}(P,x'_{Y})$ (see Figure \ref{fig : case2.2ftr}) and $d_{Y}(z,\partial{W})\leq 1$. The third inequality holds by (\ref{eq : dYx'YP}). 

But then the upper bound (\ref{eq : dYPbdW}) contradicts the lower bound (\ref{eq : dYbdWP}) given as the assumption of Case 2.2. The proof of the claim is complete.
\begin{figure}
\centering
\scalebox{0.2}{\includegraphics{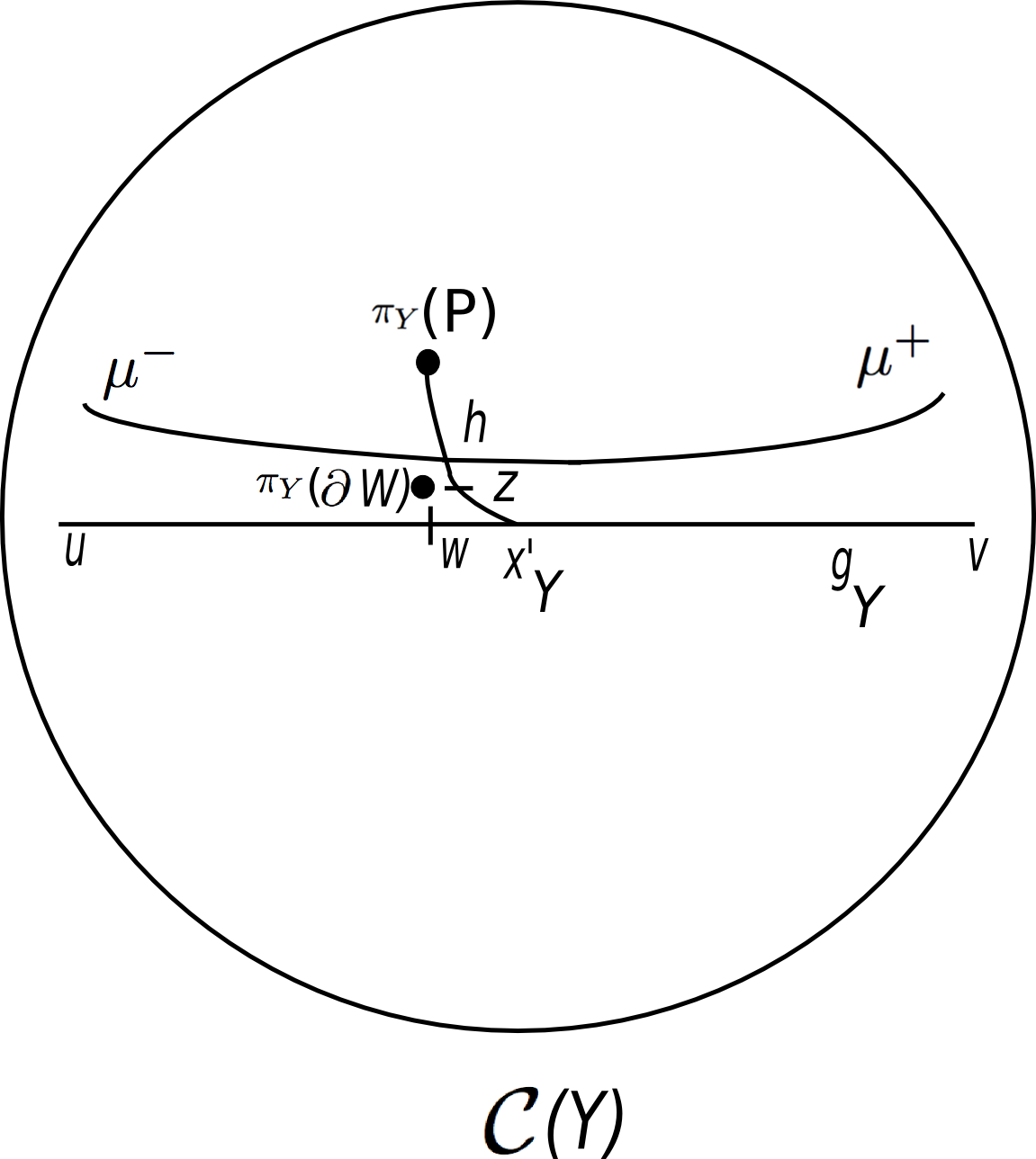}}
\caption{  {\bf Case 2.2:} If $h$ intersects the $1-$neighborhood of $\pi_{Y}(\partial{W})$, then the distance between $\pi_{Y}(\partial{W})$ and $\pi_{Y}(P)$ would be less than the lower bound in the assumption of Case 2.2.}
\label{fig : case2.2ftr}
\end{figure}
\medskip

By Claim (\ref{claim : hintnbhdbdW}), $\partial{W}$ intersects every vertex of $h$, so Theorem \ref{thm : bddgeod}(Bounded Geodesic Image) implies that 
\begin{equation}\label{eq : dWpx'Y}d_{W}(P,x'_{Y})\leq G.\end{equation}

We have $W\subsetneq Y$ and both $W$ and $Y$ are component domains of $\rho$. Let $\phi_{g_{Y}}(W)$ be the footprint of $W$ on $g_{Y}$ which consists of the vertices of $g_{Y}$ that do not intersect $W$. For more detail about footprints see $\S 4$ of \cite{mm2}, in particular Definition 4.9. By Lemma 4.10 of \cite{mm2}, $\phi_{g_{Y}}(W)$ is a sequence of $1,2$ or $3$ consecutive vertices of $g_{Y}$.
 
 If $x'_{Y}\in \phi_{g_{Y}}(W)$, then $x'_{Y}$ does not intersect $\partial{W}$, so $d_{Y}(\partial{W}, x'_{Y})\leq 1$. Moreover by (\ref{eq : dYx'YP}), $d_{Y}(x'_{Y},P)\leq 2M+\epsilon+4$. These two inequalities combined by the triangle inequality imply that 
 $$d_{Y}(\partial{W},P)\leq \epsilon+2M+5.$$
  But the above upper bound contradicts the lower bound (\ref{eq : dYbdWP}), because $\epsilon+2M+5<e$. So we conclude that $x'_{Y}\not\in \phi_{g_{Y}}(W)$.

Therefore, either  $x'_{Y}<\min\phi_{g_{Y}}(W)$ or $x_{Y}'>\max\phi_{g_{Y}}(W)$ as vertices on the geodesic $g_{Y}\subset\mathcal{C}(S)$. We proceed to discuss these two cases. Let $u$ be the initial and $v$ be the final vertex of the geodesic $g_{Y}$.
\medskip

\noindent {\bf Case 2.2.1:} $x'_{Y}< \min \phi_{g_{Y}}(W)$. 

Then $\partial{W}$ intersects every vertex of $g_{Y}$ between $u$ and $x'_{Y}$. Hence by Theorem \ref{thm : bddgeod} we have
 \begin{equation}\label{eq : dWx'Yu}d_{W}(x'_{Y},u)\leq G.\end{equation}
 Let $j_{W}^{-}$ be the initial parameter of the interval $J_{W}$. By condition $S3$ of slices of hierarchies at the beginning of \cite[\S 5]{mm2} there is a vertex $w\in \phi_{g_{Y}}(W)$ such that $w\subset \rho(j_{W}^{-})$. Let $i\in J_{Y}$ be such that $u\subset\rho(i)$. Then since $u\leq w$ on $g_{Y}$, Theorem \ref{thm : hrpath}(\ref{h : monoton}) (Monotonicity) implies that $i\leq j_{W}^{-}$. Then Theorem \ref{thm : hrpath}(\ref{h : lrbddproj}) implies that 
 \begin{equation}\label{eq : dWum-}d_{W}(u,\mu^{-})\leq M.\end{equation}
   Combining inequalities (\ref{eq : dWx'Yu}), (\ref{eq : dWum-}) and (\ref{eq : dWpx'Y}) by the triangle inequality we get
\begin{equation}\label{eq : dWPmu-}d_{W}(P,\mu^{-})\leq M+2G \leq 2M.\end{equation}
 The second inequality above follows from the fact that $M\geq 2G$. To see this fact note that $M_{1},M_{2}\geq G$ (see Lemma 6.1 (Sigma projection) in \cite{mm2}) and $M=M_{1}+M_{2}+M_{3}+M_{4}+B_{0}+4$.

Let $i\in [m,i_{W}^{-}]$. Then by the first part of Claim \ref{claim : dXimu}, $d_{W}(\rho(i),\mu^{-})\leq 2M+5$. This inequality and the inequality (\ref{eq : dWPmu-}) combined by the triangle inequality give us 
$$d_{W}(\rho(i),P)\leq 4M+5+2.$$
 Now since $e>4M+7$, by the set up of the interval $E_{W}$ we have $i\in E_{W}$, so
\begin{equation}\label{eq : EWsubseteq''}[m,i^{-}_{W}]\subseteq E_{W}.\end{equation}

By (\ref{eq : IXJX}), $i_{W}^{-}\in J_{W}$. So by condition $S3$ of slices of hierarchies in \cite[\S 5]{mm2}, as before, there is a vertex $w\in \phi_{g_{Y}}(W)$ such that $w\subset \rho(i_{W}^{-})$. Let $j\in J_{Y}$ be so that $x_{Y}'\subset\rho(j)$. Then $x_{Y}'\leq w$ as vertices along $g_{Y}$. Hence by Theorem \ref{thm : hrpath}(\ref{h : monoton}) (Monotonicity) we have that $j \leq i_{W}^{-}$. So by (\ref{eq : EWsubseteq''}), $j\in E_{W}$. Furthermore, by (\ref{eq : dYx'YP}) we have
$$d_{Y}(\rho(j),P)\leq d_{Y}(x_{Y}',P)\leq \epsilon+M$$
 Then since $e>\epsilon+M$, by the set up of the interval $E_{Y}$ any $j$ as above is in the interval $E_{Y}$. Therefore $E_{W} \cap E_{Y} \neq \emptyset$.
\medskip

\noindent {\bf Case 2.2.2:}  $x'_{Y}>\max \phi_{g_{Y}}(W)$.

In this case similar to Case 2.2.1 we can first show that $d_{W}(P,\mu^{+})\leq 2M$, which again using a similar argument implies that $[i_{W}^{+},n]\subseteq E_{W}$. Then by a similar proof we can conclude that $E_{W}\cap E_{Y}\neq \emptyset$.
\medskip

Finally Case 3 can be treated similar to Case (2) interchanging the role of subsurfaces $Y$ and $W$.
\medskip

In summary in all of the above cases we verified that (\ref{eq : EYEW}) holds. Thus as we explained earlier the lemma follows from Helly's Theorem in dimension one.
\end{proof}
We can now complete the proof of Theorem \ref{thm : narrowhull}. Lemma \ref{lem : nhullsubsurfbd} provides $d>0$ depending on $\epsilon$ and $A$, so that for any $P\in \Sigma_{\epsilon}(\mu^{-},\mu^{+})$ there is a $j_{P}\in [m,n]$ such that 
$$d_{X}(\rho(j_{P}),P)\leq d$$ 
 for every essential non-annular subsurface $X\subseteq S$. Let the threshold constant in the distance formula (\ref{eq : dsf}) be $\max\{M_{1},d\}$. Let $\Delta$ be the additive constant in the distance formula corresponding to this threshold constant. Then by the above inequality the distance formula gives us
$$d(P,\rho(j_{P})) \leq \Delta.$$ 
Therefore $\Sigma_{\epsilon}(\mu^{-},\mu^{+})$ is contained in the $\Delta$ neighborhood of $|\rho|$. We earlier proved that $|\rho|\subset \Sigma_{\epsilon}(\mu^{-},\mu^{+})$. These two facts together imply that the Hausdorff distance of $|\rho|$ and $\Sigma_{\epsilon}(\mu^{-},\mu^{+})$ is bounded by $\Delta$. Note that $d$ depends only on $A$ and $\epsilon$, so $\Delta$ depends only on $A$ and $\epsilon$. 

\end{proof}

\subsection{Stability}\label{subsec : stab}

First we recall the notion of a contracting subset of a metric space. Let $\mathcal{X}$ be a metric space.

\begin{definition}\textnormal{ (Contraction property)} \label{def : contraction}
 Given $R,B\geq 0$ and $0<\eta\leq 1$ a subset $\mathcal{Y}\subset\mathcal{X}$ is $(R,B,\eta)-$contracting if there is a map $\Pi : \mathcal{X} \to \mathcal{Y}$ with the following property. Suppose that $x,y\in\mathcal{X}$ and $d(x,\Pi x) > R$. Then $d(x,y)\leq \eta d(x,\Pi y)\Longrightarrow d(\Pi x,\Pi y)\leq B.$
  \end{definition}
If the map $\Pi:\mathcal{X}\to \mathcal{Y}$ satisfies the contraction property, the coarse Lipschitz property, and coarsely preserves $\mathcal{Y}$, then for any $K\geq 1$ and $C\geq 0$ the standard Morse lemma argument as is in the proof of \cite[Lemma 7.1]{mm1}, gives a $d>0$, such that a $(K,C)-$quasi-geodesic with end points in $\mathcal{Y}$ stays in the $d-$neighborhood of $\mathcal{Y}$. In this way we get a quantifier function $d:\mathbb{R}^{\geq 1}\times \mathbb{R}^{\geq 0}\to \mathbb{R}^{\geq 0}$, depending only on $R,B$ and $\eta$ such that $\mathcal{Y}$ is $d-$stable in $\mathcal{X}$ (see Definition \ref{def : stable}). 
\medskip

The following theorem is the main result of this subsection

\begin{thm} \textnormal{(Stable hierarchy resolution path)}\label{thm : stability}
Given $A>0$ there is a quantifier function $d_{A}:\mathbb{R}^{\geq0}\times \mathbb{R}^{\geq1}\to \mathbb{R}^{\geq0}$ such that any hierarchy path with $A-$narrow end points $\mu^{-}$ and $\mu^{+}$ is $d_{A}-$stable in the pants graph. 
 \end{thm}
 
\begin{proof} 
First note that by Theorem \ref{thm : narrowhull} the Hausdorff distance of $\Sigma_{\epsilon}(\mu^{-},\mu^{+})$ ($\epsilon>M$) and a hierarchy path $\rho$ between $\mu^{-}$ and $\mu^{+}$ is bounded by the constant $\Delta$ depending only on $A$ and $\epsilon$. Therefore, if $\Sigma_{\epsilon}(\mu^{-},\mu^{+})$ is $d-$stable, then the hierarchy path $\rho$ is stable with quantifier function $d+\Delta$. Thus it suffices to prove that $\Sigma_{\epsilon}(\mu^{-},\mu^{+})$ is stable.

The stability of the $\Sigma_{\epsilon}-$hull follows from the fact that for any $\epsilon$ sufficiently large the projection map $\Pi$ onto the $\Sigma_{\epsilon}-$hull defined in Theorem \ref{thm : projsigmahull} has the contraction property. We will prove the last fact in Theorem \ref{thm : contraction}.
\end{proof}

We borrow the following properties of $\delta-$hyperbolic spaces which are not necessarily locally compact (e.g. the curve complex) from \cite{mhlpcc}. These properties will be used in the proof of Lemma \ref{lem : contsubsurfs} which is the main part of the proof of Theorem \ref{thm : contraction}(Narrow hulls are contracting). Let $\mathcal{X}$ be a $\delta-$hyperbolic space which is not necessarily locally compact.

\begin{prop}\label{prop : treelike}
 Let $\zeta$ be a geodesic in $\mathcal{X}$. Given points $x,y$ in $\mathcal{X}$ or the Gromov boundary of $\mathcal{X}$. Let $x',y'$ be nearest points to $x$ and $y$ on $\zeta$, respectively.  When $x$ is at the boundary, let $x_{i}\in \mathcal{X}$ be a sequence of points with $x_{i}\to x$ as $i\to \infty$ and let $x'$ be the limit of nearest points to the points $x_{i}$ on $\zeta$. The same for $y$ and $y'$. Then
\begin{enumerate}[(i)]
\item A geodesic $[x,w]$ connecting $x$ to a point $w$ on $\zeta$ intersects the $3\delta$ neighborhood of $x'$.
\item We have $d(x',y')\leq d(x,y)+12\delta$.
\item (Tree like) Let $K=14\delta$ and $\delta'=24\delta$. Suppose that $d(x',y')> K$ then we have 
\begin{equation} \label {eq : tree}d(x,x')+d(x',y')+d(y',y) \leq d(x,y)+\delta' .\end{equation}
\noindent There is a function $b(a,\delta)$ increasing in both $a$ and $\delta$ with the following property. Let $[x,x']$ and $[y,y']$ be geodesics connecting $x$ to $x'$ and $y$ to $y'$, respectively. Given $a>0$ and $\mathcal{Z}\subset \mathcal{X}$, denote the $a-$neighborhood of $\mathcal{Z}$ by $\mathcal{N}_{a}(\mathcal{Z})$.
\item If $\mathcal{N}_{a}([x,x'])\cap \mathcal{N}_{a}([y,y'])\neq \emptyset$, then $d(x',y')\leq b$.
\end{enumerate}
 \end{prop}
  \begin{proof}
First suppose that $x,y\in\mathcal{X}$. Parts (i) and (iii) are respectively Propositions 3.2 and 3.4 of \cite{mhlpcc}. Part (iv) can be proved by a slight modification of the proof given for Proposition 3.4 in \cite{mhlpcc}. 
 
{\it Proof of part (ii):} By part (i) there is a point $z$ on $[x,y']$ with 
\begin{equation}\label{eq : X1}d(z,x')\leq 3\delta\end{equation}
 and there is a point $z'$ on $[y,x']$ with 
 \begin{equation}\label{eq : X2}d(z',y')\leq 3\delta.\end{equation} 
 Without loss of generality suppose that 
 \begin{equation}\label{eq : X3}d(x,z)\geq d(y,z')\end{equation}
  We claim that 
 \begin{equation}\label{eq : dzz'dxy}d(z,z')\leq d(x,y)+6\delta.\end{equation}
  Otherwise, 
  \begin{equation}\label{eq : X4}d(z,z')>d(x,y)+6\delta.\end{equation}
   Now we have 
 \begin{eqnarray*}
  d(x,y')&\leq& d(x,y)+d(y,z')+d(z',y')<d(z,z')-6\delta+d(x,z)+3\delta\\
  &\leq& d(z,z')+d(x,z)-d(z',y')\leq d(y',z)+d(x,z)=d(x,y')
  \end{eqnarray*}
The first inequality is the triangle inequality.  The second inequality follows from inequalities (\ref{eq : X2}), (\ref{eq : X3}) and (\ref{eq : X4}). The third inequality follows from (\ref{eq : X2}). The fourth inequality is the triangle inequality. The last equality holds because $z$ is on the geodesic segment $[x,y']$. But then we have $d(x,y')< d(x,y')$, which is contradictory and we obtain the claimed inequality. 

Finally, we have that
$$d(x',y')\leq d(x',z)+d(z,z')+d(z',y')\leq d(x,y)+12\delta,$$
 where the first inequality is the triangle inequality and the second one follows from inequalities (\ref{eq : X1}), (\ref{eq : dzz'dxy}) and (\ref{eq : X2}).

Now suppose that $x$ ($y$) is at infinity of $\mathcal{X}$. Let the points $x_{i}$ in $\mathcal{X}$ be so that $x_{i}\to x$ as $i\to\infty$ (similarly the points $y_{i}\to y$). For part (i) consider the geodesic segments $[x_{i},w]$, as we saw above $[x_{i},w]$ passes through $3\delta$ neighborhood of $x'_{i}$ (a nearest point to $x_{i}$ on $\zeta$). Then since $x_{i}'$ converges to a nearest point to $x$ on $\zeta$ as $i\to\infty$, we conclude that part (i) holds for $x$.

Similarly the conclusion of each part (ii)-(iv) follows from applying that to each $x_{i}$ ($y_{i}$) proved above and then taking limit.
 \end{proof}
  The constants $K$ and $\delta'$ in Proposition \ref{prop : treelike} depend only on $\delta$ the hyperbolicity constant of the metric space $\mathcal{X}$. For any essential non-annular subsurface $Y\subseteq S$ denote by $K_{Y}$ and $\delta'_{Y}$ the corresponding constants of $\mathcal{C}(Y)$, respectively. These constants depend only on $\delta_{Y}$ and thus the topological type of $Y$. 
 
 We will also need the following elementary lemmas.
 \begin{lem}\label{lem : orth}
 Given a point $z$ and a geodesic $\zeta$ in $\mathcal{X}$. Let $z'$ be a nearest point to $z$ on $\zeta$. Let $x$ be a point on $\zeta$ and $x'$ be a nearest point to $x$ on an infinite geodesic containing $[z,z']$. Then $d(z',x')\leq 6\delta$.
 \end{lem}
 \begin{proof} By Proposition \ref{prop : treelike}(i) any geodesic segment $[x,z']$ connecting $x$ and $z'$ intersects the $3\delta$ neighborhood of $x'$ at a point $u$. We claim that $d_{Z}(u,x')\leq 3\delta$. For otherwise the path $[x,u]\cup [u,z']$ would have length less than the length of $[x,x']$ which contradicts the fact that $[x,x']$ minimizes the distance between $x$ and $x'$. Then by the triangle inequality 
 $$d_{Z}(x',z')\leq d(x',u)+d(u,z')\leq 3\delta+3\delta=6\delta.$$
 \end{proof}

 \begin{lem}\label{lem : symm}
 Given a point $z$ and a geodesic $\zeta$ in $\mathcal{X}$. Let $z'$ be a nearest point to $z$ on $\zeta$. Then for any $p$ on $\zeta$, $d(p,z)+6\delta\geq d(p,z')$.
 \end{lem}
\begin{proof}
Extend $\zeta$ to an infinite geodesic $\bar{\zeta}$. Let $\mathcal{Z}$ be the set of nearest points to $z$ on $\bar{\zeta}$. By Proposition \ref{prop : treelike}(ii) the diameter of $\mathcal{Z}$ is at most $12\delta$. Let $p'$ be the point on $\bar{\zeta}$ so that the distance of $p'$ and $\mathcal{Z}$ is the same as the distance of $p$ and $\mathcal{Z}$.  By symmetry $d(p,z)=d(p',z)$ and by the triangle inequality $d(p,z)+d(z,p')\geq d(p,p')$. Thus $2d(p,z)\geq d(p,p')$. Moreover $d(p,p')+\diam(\mathcal{Z})\geq 2d(p,z')$. So we get 
$$2d(p,z)+12\delta\geq 2d(p,z').$$ 
Dividing both sides of the above inequality by $2$ we get the desired inequality.
 \end{proof}
 
  \begin{lem}\label{lem : ftrnearp}
  Let geodesics $f$ and $f'$, $D$ fellow travel in $\mathcal{X}$. Let $u$ and $\hat{u}$ be nearest points to a point $p$ on $f$ and $f'$ respectively. Then $d(\hat{u},u)\leq 3D+6\delta$.
 \end{lem}
 \begin{proof}Let $u'$ be a nearest point to $u$ on $f'$. Since $f$ and $f'$, $D$ fellow travel, $d(p,u')\leq d(p,u)+d(u,u')\leq d(p,u)+D$. Moreover, $\hat{u}$ is a nearest point to $p$ on $f'$ so $d(p,\hat{u})\leq d(u,p)+D$. Similarly $d(p,u)\leq d(\hat{u},p)+D$. The first and third inequalities imply that
$$d(u',p)\leq d(\hat{u},p)+2D.$$
 Let $u'$ be a nearest point to $u$ on $f'$, then $d(u',u)\leq D$. Then 
 $$d(p,u')\leq d(p,u)+d(u,u')\leq d(p,u)+D.$$
  Now we prove that $d(u,u')\leq 3\delta+2D$. By Proposition \ref{prop : treelike}(i) there is a point $z$ on $[p,u']$ such that  
  $$d(z,\hat{u})\leq 3\delta.$$
   Then by the triangle inequality 
  $$d(z,p)\geq d(p,\hat{u})-d(\hat{u},z)\geq d(p,\hat{u})-3\delta.$$
   Thus 
 \begin{eqnarray*}
 d(z,u')&=&d(p,u')-d(p,z)\leq d(p,u')-(d(p,\hat{u})-3\delta)\\
 &=&d(p,u')-d(p,\hat{u})+3\delta \leq 2D+3\delta.
 \end{eqnarray*}
  Then by the triangle inequality 
  $$d(u',\hat{u})\leq d(u',z)+d(z,\hat{u})\leq 6\delta+2D.$$
   Finally 
   $$d(u,\hat{u})\leq d(\hat{u},u')+d(u',u)\leq 6\delta+2D+D= 6\delta+3D.$$
\end{proof}
Let $F$ be the constant from Theorem \ref{thm : projsigmahull}.

\begin{thm}\textnormal{ (Narrow hulls are contracting)}\label{thm : contraction}
 Given $A>0$ and $\epsilon>F$, there are $R,B>0$ and $0<\eta \leq 1$ so that the following holds. Let $(\mu^{-},\mu^{+})$ be an $A-$narrow pair, then $\Sigma_{\epsilon}(\mu^{-},\mu^{+})$ has the contraction property with constants $R,B$ and $\eta$.
\end{thm}

\begin{proof}
We prove that the projection map $\Pi:P(S)\to\Sigma_{\epsilon}(\mu^{-},\mu^{+})$ defined in Theorem \ref{thm : projsigmahull} has the contraction property. Note that to have a projection onto the hull the theorem requires that $\epsilon>F$. Indeed we prove the contrapositive of the contraction property which asserts the following: Suppose that $d(P,\Pi P)>R$. Then $d(\Pi P,\Pi Q)>B \Longrightarrow d(P,\Pi P)> \eta d(P,Q).$  
 
\begin{lem} \label{lem : contsubsurfs}
 Given $A>0$ there are constants $B>0$ and $q>0$ with the following properties. Suppose that $\epsilon>F$, let $(\mu^{-},\mu^{+})$ be an $A-$narrow pair and $\Pi: P(S)\to \Sigma_{\epsilon}(\mu^{-},\mu^{+})$ be the projection map onto the $\Sigma_{\epsilon}-$hull. Let $P,Q\in P(S)$ be such that $d(\Pi P,\Pi Q)>B$. Then for every essential non-annular subsurface $Z\subseteq S$ we have 
\begin{equation} \label{eq : consubs}d_{Z}(P,Q) \geq d_{Z}(P,\Pi P)-q \end{equation}
\end{lem}
The above lemma implies the contraction property. To see this, suppose that (\ref{eq : consubs}) holds for every non-annular subsurface $Z \subseteq S$. Then by the distance formula (\ref {eq : dsf}) we have 
\begin{eqnarray*}
d(P, Q) &\asymp_{K,C}&\sum_{\substack{ Z\subseteq S\\ \nonannular}}\{d_{Z}(P, Q)\}_{A} \\
&\geq& \sum_{\substack{ Z\subseteq S\\ \nonannular}}\{d_{Z}(P, \Pi P)-q\}_{A}.
\end{eqnarray*}
Recalling the definition of the cut-off function $\{.\}_{A}$, for any term in the last sum above we have
$$ \{d_{Z}(P, \Pi P)-q\}_{A} \geq \frac{A}{A+q} \{ d_{Z}(P, \Pi P) \}_{A+q}.$$
Moreover, for the threshold constant $A_{1}=A+q$ there are constants $K_{1},C_{1}$ such that the distance formula (\ref{eq : dsf}) is written as
$$ \sum_{\substack{ Z\subseteq S\\ \nonannular}}\{d_{Z}(P, \Pi P)\}_{A+q} \asymp_{K_{1},C_{1}}d(P, \Pi P).$$
Therefore, we obtain 
$$d(P,Q)>\eta'd(P, \Pi P)-c$$
 for $\eta'=\frac{A}{K(A+q)K_{1}}$ and $c=\frac{AC}{K_{1}(A+q)K}+C_{1}$.

Now let $R$ be large enough such that  $\eta=\eta'-\frac{c}{R} >0$. Then for any $P \in P(S)$ such that $d(P, \Pi P)>R$, and any $Q\in P(S)$, we have that
\begin{eqnarray*}
d(P,Q)&>&\eta'd(P, \Pi P)-c 
= [(\eta' -\frac{c}{R}) d(P, \Pi P) ]-c+( \frac{c}{R}) d(P, \Pi P) \\
&>& \eta d(P, \Pi P).
\end{eqnarray*}
This shows that for $R$ and $\eta$ as above the projection map $\Pi:P(S) \to \Sigma_{\epsilon}(\mu^{-},\mu^{+})$ provided that $\epsilon>F$ and $(\mu^{-},\mu^{+})$ is $A-$narrow satisfies the contrapositive of the contraction property. So $\Pi$ has the contraction property.
\end{proof}

\begin{proof}[Proof of Lemma \ref {lem : contsubsurfs}]
 We set 
\begin{equation}\label{eq : a'}A'=A+12\delta+\delta' +\bar{K}+2F+2M+2(F_{1}+4)+G+2(3D+6\delta)+2(D_{1}+1),\end{equation}
as the threshold constant of the distance formula (\ref{eq : dsf}) in this proof. Note that $A'>M_{1}$.

The constants on the right hand side of (\ref{eq : a'}) are the following: $\delta =\max_{Y\subseteq S}\delta_{Y}$, $\bar{K}=\max_{Y\subseteq S}K_{Y}$ and $\delta'=\max_{Y\subseteq S}\delta'_{Y}$, where $\delta_{Y}$ is the hypebolicity constant of $\mathcal{C}(Y)$. Moreover $K_{Y}$ and $\delta'_{Y}$ are the constants from Proposition \ref{prop : treelike}(iv) for $\mathcal{C}(Y)$. Note that $\delta_{Y}, K_{Y}$ and $\delta'_{Y}$ depend only on the topological type of $Y$, so the above maxima exist. The constant $G$ is the bound from Theorem \ref{thm : bddgeod}(Bounded Geodesic Image). $M=M_{1}+M_{2}+M_{3}+M_{4}+B_{0}+4$, where $M_{i}, i=1,2,3,4$, are the constants from Theorem \ref{thm : hrpath} and $B_{0}$ is the constant from Theorem \ref{thm : behineq}(Behrstock inequality). As we mentioned after Theorem \ref{thm : projsigmahull} there are constants $F_{1},F_{2}\geq 1$ so that for any $P\in P(S)$, the tuple $(x_{Y})_{Y\subseteq S}$, where $x_{Y}$ is a nearest point to $\pi_{Y}(P)$ on $\hull_{Y}(\mu^{-},\mu^{+})$ satisfies the consistency conditions of Theorem \ref{thm : consistency}. Moreover, we let $F_{1}$ be greater than the constant we set up in Example \ref{ex : order}. Then by Theorem \ref{thm : consistency} there is a corresponding constant $F$ for which Theorem \ref{thm : projsigmahull} holds. 

The constant $D=\max_{Y\subseteq S}D_{Y}$, where $D_{Y}$ is the fellow traveling distance of two geodesics in $\mathcal{C}(Y)$ with end points within distance $2$ of each other. In particular, since the diameter of the projection of a (partial) marking or a lamination to a subsurface is at most $2$ (Lemma \ref{lem : diamproj}), for any subsurface $Y\subseteq S$, $D$ is the fellow traveling distance of any two geodesics in the convex hull of a pair of partial markings or laminations in $\mathcal{C}(Y)$. The constant $D_{1}$ is the maximum over all subsurfaces $Y$ of the fellow traveling distance of two geodesics with end points within distance $F+(3D+6\delta)$. For the fellow traveling property of geodesics in $\delta-$hyperbolic spaces see part III.H of \cite{bhnpc}. 
\medskip

We start by establishing some inequalities. Let $P\in P(S)$ and $X\subseteq S$ be a non-annular essential subsurface.

Let the vertex $p_{X}\in\pi_{X}(P)$ and the vertex $\hat{v}$ on $\hull_{X}(\mu^{-},\mu^{+})$ realize the distance between $\pi_{X}(P)$ and $\hull_{X}(\mu^{-},\mu^{+})$. Then by Theorem \ref{thm : projsigmahull} and the discussion after the theroem about its proof, we have that
 \begin{equation} \label{eq : dPPv^} d_{X}(\Pi P,\hat{v})\leq F.\end{equation}
    Let $f$ be a geodesic in $\hull_{X}(\mu^{-},\mu^{+})$ and let $v$ be a nearest point to $p_{X}$ on $f$. The point $\hat{v}$ lies on a geodesic $\hat{f}$ in the hull. Then Lemma \ref{lem : ftrnearp} applied to the geodesics $f$ and $\hat{f}$ and the point $p_{X}$ implies that
    \begin{equation}\label{eq : dvv^}d_{X}(v,\hat{v})\leq 3D+6\delta.\end{equation} 
    The inequalities (\ref{eq : dPPv^}) and (\ref{eq : dvv^}) combined by the triangle inequality give us
    \begin{equation}\label{eq : dvPiP}d_{X}(v,\Pi P)\leq 3D+6\delta+F.\end{equation}
Let $Q\in P(S)$. Let vertices $q_{X}\in\pi_{X}(Q)$ and $\hat{z}$ on $\hull_{X}(\mu^{-},\mu^{+})$ realize the distance between $\pi_{X}(Q)$ and $\hull_{X}(\mu^{-},\mu^{+})$. Let $f$ be a geodesic in $\hull_{X}(\mu^{-},\mu^{+})$. Let $v$ and $z$ be nearest points to $p_{X}$ and $q_{X}$ on $f$, respectively. Then by the triangle inequality and the inequalities (\ref{eq : dvv^}) and (\ref{eq : dPPv^}) we have that 
\begin{eqnarray}\label{eq : dvz>PPPQ}
 d_{X}(v,z)&\geq&d_{X}(\Pi P,\Pi Q)-d_{X}(\Pi P,\hat{v})-d_{X}(\Pi Q,\hat{z})-d_{X}(\hat{z},z)-d_{X}(\hat{v},v)\nonumber\\
&\geq& d_{X}(\Pi P,\Pi Q)-2F-2(3D+6\delta).
\end{eqnarray}

\noindent First note that if 
$$d_{Z}(P, \Pi P) \leq A'+5(F_{1}+4)$$ 
then for $q=A'+5(F_{1}+4)$, (\ref{eq : consubs}) holds. Thus in the rest of the proof we will assume that $Z$ is a subsurface with
\begin{equation}\label{eq : 1}d_{Z}(P,\Pi P)>A'+5(F_{1}+4).\end{equation} 
We proceed to discuss the following two cases depending on the range of the value of $d_{Z}(\Pi P,\Pi Q)$.
\medskip

\noindent{\bf Caes 1:} 
 \begin{equation}\label {eq : dZPiPPiQ>A'}d_{Z}(\Pi P,\Pi Q) > A'.\end{equation}

 Let $p_{Z}$ and $\hat{u}$ realize the distance between $\pi_{Z}(P)$ and $\hull_{Z}(\mu^{-},\mu^{+})$, and $q_{Z}$ and $\hat{w}$ realize the distance between $\pi_{Z}(Q)$ and the hull. Let $f$ be a geodesic in $\hull_{Z}(\mu^{-},\mu^{+})$. Let $u$ and $w$ be nearest points to $p_{Z}$ and $q_{Z}$ on $f$, respectively. Then by (\ref{eq : dvz>PPPQ}) and (\ref{eq : dZPiPPiQ>A'}) we have that
$$ d_{Z}(u,w)\geq A'-2F-2(3D+6\delta).$$
Then by the choice of $A'$ in (\ref{eq : a'}), $A'-2F-2(3D+6\delta)>K>K_{Z}$. Thus by the tree like property (\ref{eq : tree}) we have that 
$$d_{Z}(p_{Z},q_{Z})\geq d_{Z}(p_{Z},u)-\delta'_{Z}.$$
Then since $p_{Z}\in\pi_{Z}(P)$ and $q_{Z}\in\pi_{Z}(Q)$ we have
 \begin{equation}\label{eq : dZPQdZPu}d_{Z}(P,Q) \geq d_{Z}(P,u)-\delta_{Z}'-\diam_{Z}(P)-\diam_{Z}(Q)\geq d_{Z}(P,u)-\delta'-4.\end{equation}
Moreover, by the triangle inequality and the inequality (\ref{eq : dvPiP}) we have
\begin{eqnarray*}
d_{Z}(P,u)&\geq& d_{Z}(P,\Pi P)-d_{Z}(\Pi P,u)\\
&\geq& d_{Z}(P,\Pi P)-F-(3D+6\delta).
\end{eqnarray*} 
The above inequality and the inequality (\ref{eq : dZPQdZPu}) imply that
$$d_{Z}(P,Q) \geq d_{Z}(P,\Pi P)-\delta'-F-(3D+6\delta)-4.$$
Thus (\ref{eq : consubs}) holds for $q=\delta'+F+(3D+6\delta)+4$.
\medskip

 \noindent{\bf Case 2:} 
  \begin{equation}\label{eq : 2}d_{Z}(\Pi P,\Pi Q) \leq A'.\end{equation}

 Since the threshold constant of the distance formula is $A'$ the subsurface coefficient $d_{Z}(\Pi P,\Pi Q)$ has no contribution to $d(\Pi P,\Pi Q)$. Let $K'$ and $C'$ be the constants in the distance formula corresponding to the threshold constant $A'$. Set the constant
 \begin{equation}\label{eq : B}B=K'+C',\end{equation}
 in the statement of the lemma, also for the projection bound in the statement of Theorem \ref{thm : contraction}. By the assumption of the lemma we have that
  $$d(\Pi P,\Pi Q) > K'+C'.$$ 
Then by the distance formula (\ref{eq : dsf}) there is a subsurface $W\neq Z$ so that 
 \begin{equation}\label{eq : 3}d_{W}(\Pi P,\Pi Q) > A'.\end{equation}
 Let $p_{W}$ and $\hat{v}$ realize the distance between $\pi_{W}(P)$ and $\hull_{W}(\mu^{-},\mu^{+})$, and $q_{W}$ and $\hat{z}$ realize the distance between $\pi_{W}(Q)$ and the hull. Let $f$ be a geodesic in $\hull_{W}(\mu^{-},\mu^{+})$ between $\pi_{W}(\mu^{-})$ or $\mu^{-}|_{W}$ and $\pi_{W}(\mu^{+})$ or $\mu^{+}|_{W}$. Let $v$ and $z$ be nearest points to $p_{W}$ and $q_{W}$ on $f$, respectively. The inequalities (\ref{eq : dvz>PPPQ}) and (\ref{eq : 3}) imply that 
\begin{equation}\label{eq : dWvz}d_{W}(v,z)\geq A'-2F-2(3D+6\delta).\end{equation} 
Moreover $v,z$ are both on the geodesic $f$ which connects $\pi_{W}(\mu^{-})$ or $\mu^{-}|_{W}$ and $\pi_{W}(\mu^{+})$ or $\mu^{+}|_{W}$. Thus $d_{W}(\mu^{-},\mu^{+})\geq d_{W}(v,z)$. Then by (\ref{eq : dWvz}) we get
  \begin{equation}\label{eq : dWmu+-}d_{W}(\mu^{-},\mu^{+})\geq  A'-2F-2(3D+6\delta).\end{equation}
 By the choice of $A'$ in (\ref{eq : a'}), $A'-2F-2(3D+6\delta)>A$, so by (\ref{eq : dWmu+-}) we have
$$d_{W}(\mu^{-},\mu^{+})>A.$$
By the assumption of the lemma the pair $(\mu^{-},\mu^{+})$ is $A-$narrow, so the above inequality implies that the subsurface $W$ is a large subsurface. This excludes the possibility that $W$ and $Z$ be disjoint subsurfaces. Thus we need to discuss the following three subcases: 
\begin{enumerate}[(2.1)]
\item$W \pitchfork Z$, 
\item$W \subsetneq Z$ and 
\item$Z \subsetneq W$.
\end{enumerate}

\noindent {\bf Case 2.1:}  $W \pitchfork Z$.

Let $k:=k(P)= \lfloor \frac{d_{Z}(P, \Pi P)-2}{F_{1}+4} \rfloor$ and $k'= \lfloor \frac{A'}{F_{1}+4} \rfloor$. Here  $\lfloor x\rfloor$ is the floor function, which assigns to any $x\in \mathbb{R}$ the largest integer less than or equal to $x$.  
 
 Dividing both sides of the inequality (\ref{eq : 1}) by $F_{1}+4$ and subtracting $\frac{2}{F_{1}+2}$ from both sides we get 
 $$\frac{d_{Z}(P,\Pi P)-2}{F_{1}+4}\geq \frac{A'}{F_{1}+4}+5-\frac{2}{F_{1}+4}.$$
  Now since $0<\frac{2}{F_{1}+4}\leq 1$, taking the floor of both sides of the above inequality we have that
 \begin{equation}\label{eq : kk'}k\geq k'+4.\end{equation}
 We claim that 
\begin{equation}\label{eq : dZpipbdryW}d_{Z}(\Pi P, \partial{W}) \leq (k'+2)(F_{1}+4).\end{equation}
 Otherwise,  
  \begin{equation}\label{eq : dZpipbdryW>}d_{Z}(\Pi P, \partial{W}) > (k'+2)(F_{1}+4).\end{equation} 
 Let $(x_{Y})_{Y\subseteq S}$ be the tuple where for any non-annular subsurface $Y$, $x_{Y}$ is a vertex in $\pi_{Y}(\Pi P)$. Since $F_{1}$ is greater than the constant we set up in Example \ref{ex : order} and $F_{2}\geq 1$, the tuple $(x_{Y})_{Y\subseteq S}$ satisfies the consistency conditions of Theorem \ref{thm : consistency} for $F_{1}$ and $F_{2}$. So we may use the constants $F_{1},F_{2}$ and the tuple $(x_{Y})_{Y\subseteq S}$ to define a partial order as in Definition \ref{def : partialorder}. Then the inequality (\ref{eq : dZpipbdryW>}) can be written as
 \begin{equation} \label{eq : zw} Z\ll_{k'+2}W.  \end{equation} 
 Moreover, by (\ref{eq : 3}) and since $A'>2(F_{1}+4)$ we have
 \begin{equation} \label{eq : wpiq} W \ll_{2} \Pi Q. \end{equation} 
 Having (\ref{eq : zw}) and (\ref{eq : wpiq}), by the transitivity property of $\ll$ (Theorem \ref{thm : partialorder}(2)) we deduce that 
 $$Z\ll_{k'+1} \Pi Q,$$
 which means that
 $$d_{Z}(\Pi Q, \Pi P) \geq (k'+1)(F_{1}+4) > A'.$$
 But this lower bound contradicts (\ref{eq : 2}) and our claim follows. Therefore, in the rest of Case 2.1 we may assume that (\ref{eq : dZpipbdryW}) holds. 
\medskip

By the choice of $k$ we have
$$d_{Z}(\Pi P, P)\geq  k(F_{1}+4)+2.$$
The above inequality and (\ref{eq : dZpipbdryW}) combined by the triangle inequality imply that
\begin{equation} \label{eq : pbw} d_{Z}(P,\partial{W})\geq (k-k'-2)(F_{1}+4)+2-\diam_{Z}(\partial{W})\geq (k-k'-2)(F_{1}+4).\end{equation} 
Now use the constants $F_{1},F_{2}$ and the tuple $(x_{Y})_{Y\subseteq S}$, where for any non-annular subsurface $Y$, $x_{Y}$ is a vertex in $\pi_{Y}(P)$, to define a partial order. Then the inequality (\ref{eq : pbw}) can be written as
  \begin{equation} \label{eq : zkw} Z \ll_{k-k'-2} W \end{equation} 
note that by (\ref{eq : kk'}), $k-k'-2\geq 1$. 

Let $f$ be a geodesic in $\hull_{W}(\mu^{-},\mu^{+})$. Let $v$ and $z$, as before, be nearest points to $p_{W}$ and $q_{W}$ on $f$, respectively. Then by (\ref{eq : dWvz}) and the choice of $A'$, 
$$d_{W}(v,z)\geq A'-2F-2(3D+6\delta)>\bar{K}>K_{W}.$$
 Therefore, the tree like property (\ref{eq : tree}) implies that  
$$d_{W}(p_{W},q_{W})\geq d_{W}(v,z)-\delta'\geq A' -2F-\delta'-2(3D+6\delta).$$
Then since $p_{W}\in\pi_{W}(P)$ and $q_{W}\in\pi_{W}(Q)$ we have
\begin{eqnarray*}
d_{W}(P,Q)&\geq& d_{W}(p_{W},q_{W})-\diam_{W}(P)-\diam_{W}(Q)\\
&\geq& A' -2F-\delta'-2(3D+6\delta)-4.
\end{eqnarray*}
Thus we have
 \begin{equation} \label{eq : wq} W \ll_{m}Q \end{equation}
for $m=\lfloor \frac{A'-2F-\delta'-2(3D+6\delta)-4}{F_{1}+4} \rfloor$. Note that by the choice of $A'$ in (\ref{eq : a'}), $m\geq2$.
\medskip

Having (\ref{eq : zkw}) and (\ref{eq : wq}), by the transitivity property of $\ll$ (Theorem \ref{thm : partialorder}(2)) we deduce that
  $$Z\ll_{k-k'-3}Q,$$
where by (\ref{eq : kk'}), $k-k'-3\geq 1$. Therefore, 
\begin{eqnarray*}
d_{Z}(P,Q) &\geq& (\lfloor \frac{d_{Z}(P, \Pi P)-2}{F_{1}+4}\rfloor -k'-3)(F_{1}+4)\\
& \geq& d_{Z}(P, \Pi P)-(k'+4)(F_{1}+4)-2.
 \end{eqnarray*}
 So the inequality (\ref{eq : consubs}) holds for $q=(k'+4)(F_{1}+4)+2$. 
\medskip

\noindent{\bf Case 2.2:}   $W \subsetneq Z$.

 Let $\rho'$ be a hierarchy path between $P$ and $\Pi P$. By (\ref{eq : 1}) we have that $d_{Z}(P, \Pi P)>A'>M_{1}$, so by Theorem \ref{thm : hrpath}(\ref{h : largelink}) $Z$ is a component domain of $\rho'$. Let the vertices $q'\in\pi_{Z}(Q)$ and $x$ on $\hull_{Z}(P,\Pi P)$ realize the distance between $\pi_{Z}(Q)$ and $\hull_{Z}(P,\Pi P)$. By Theorem \ref{thm : hrpath}(\ref{h : hausd}) there is $T$ a slice of $\rho'$ such that 
 \begin{equation}\label{eq : dZTx}d_{Z}(T,x)\leq M_{4}.\end{equation} 
 Let $h$ be a geodesic in $\mathcal{C}(Z)$ connecting $\pi_{Z}(Q)$ and $\pi_{Z}(T)$ (see the left diagram of Figure \ref{fig : case2.2}).
  \begin{figure}
\centering
\scalebox{0.2}{\includegraphics{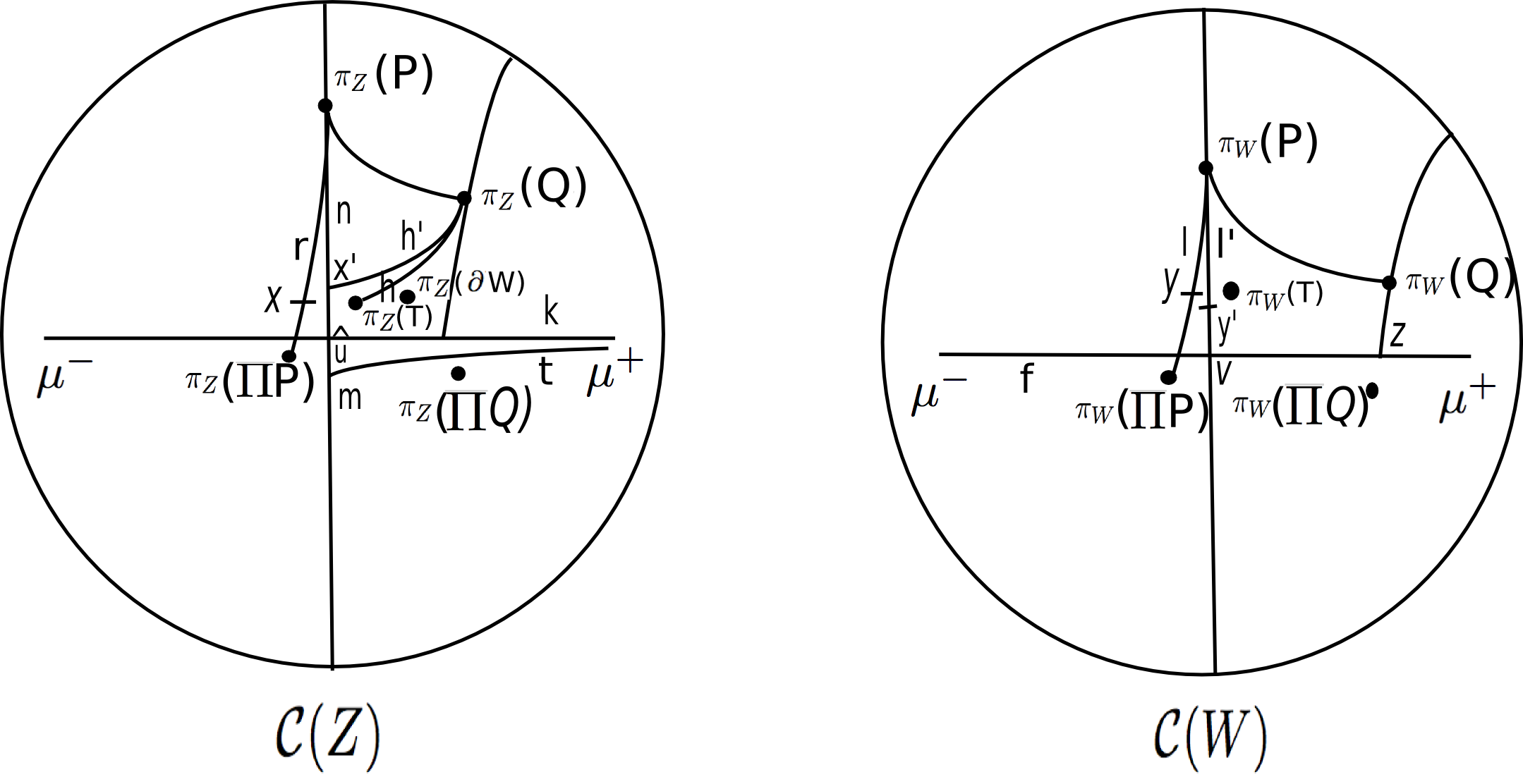}}
\caption{ {\bf Case 2.2:} Left diagram: The vertices $x$ and $q'$ realize the distance between $\pi_{Z}(Q)$ and $\hull_{Z}(P,\Pi P)$. The pants decomposition $T$ is a slice of $\rho'$ (a hierarchy path between $P$ and $\Pi P$) with $d_{Z}(T,x)\leq M_{4}$. $\pi_{Z}(\partial{W})$ is in the $1-$neighborhood of the geodesic $h$ connecting $\pi_{Z}(Q)$ to $\pi_{Z}(T)$ and any geodesic $k$ in $\hull_{Z}(\mu^{-},\mu^{+})$. Right diagram: The point $y$ on $\hull_{W}(P,\Pi P)$ is such that $d_{W}(y,T)\leq M_{4}$ and $y'$ is a nearest point to $y$ on the geodesic $l'$ connecting $\pi_{W}(P)$ to $v$ a nearest point to $\pi_{W}(P)$ on the geodesic $f$.}
\label{fig : case2.2}
\end{figure}
 
Since $T\in |\rho'|$, by Theorem \ref{thm : hrpath}(\ref{h : hausd}) there is a point $y \in \hull_{W}(P, \Pi P)$ such that  
\begin{equation}\label{eq : dWty}d_{W}(T ,y) \leq M_{4}\leq M\end{equation}
(see the right diagram of Figure \ref{fig : case2.2}). 

The point $y$ is on a geodesic $l$ connecting $\pi_{W}(P)$ to $\pi_{W}(\Pi P)$. The bound (\ref{eq : dvPiP}) guarantees that the end points of $l$ and any geodesic $l'$ connecting $p_{W}$ and $v$ are within $(3D+6\delta)+F$ distance of each other. Where $p_{W}\in\pi_{W}(P)$ is a nearest point to $\hull_{W}(\mu^{-},\mu^{+})$ and $v$ is a nearest point to $p_{W}$ on  given geodesic $f$ in the hull. Therefore $l$ and $l'$, $D_{1}$ fellow travel in $\mathcal{C}(W)$. Let $y'$ be a nearest point to $y$ on $l'$, then by the fellow traveling of $l$ and $l'$ we have that
$$d_{W}(y,y')\leq D_{1}$$ 
(see the right diagram of Figure \ref{fig : case2.2}). Moreover, $v$ and $z$ are nearest points to the points $y'$ and $q_{W}$ on the geodesic $f$, respectively. So by Proposition \ref{prop : treelike}(ii) we have that
 $$d_{W}(y',q_{W})\geq d_{W}(v,z)-12\delta_{W}\geq d_{W}(v,z)-12\delta.$$
 Combing the above two inequalities with the triangle inequality we get
  $$d_{W}(q_{W},y) \geq d_{W}(v,z)-12\delta-D_{1}$$
(see the right diagram of Figure \ref{fig : case2.2}). By (\ref{eq : dWvz}),
$$d_{W}(v,z)\geq A'-2F-2(3D+6\delta),$$
The above two inequalities give us 
$$d_{W}(q_{W},y)\geq A'-2F-2(3D+6\delta)-12\delta -D_{1}.$$
Then using the fact that $q_{W}\in\pi_{W}(Q)$ and $\diam_{W}(Q)\leq 2$ we have 
 \begin{equation}\label{eq : dWqy}d_{W}(Q,y)\geq A'-2F-2(3D+6\delta)-12\delta -D_{1}-2.\end{equation} 
 Now combining (\ref{eq : dWty}) and (\ref{eq : dWqy}) by the triangle inequality we get  
  $$d_{W}(Q ,T)\geq A'-2F-2(3D+6\delta)-D_{1}-12\delta-2-M.$$
   Then by the choice of $A'$ in (\ref{eq : a'}) we have that
 \begin{equation}\label{eq : dWqt}d_{W}(Q ,T)>G. \end{equation}
 Note that since $W\subsetneq Z$ we have $\partial{W}\pitchfork Z$. We claim that 
\begin{claim}\label{claim : bwh}
$\pi_{Z}(\partial{W})$ is in the $1-$neighborhood of $h$ in $\mathcal{C}(Z)$.
\end{claim}

 Otherwise, $\partial{W}$ would intersect every vertex of the geodesic $h$, which connects $\pi_{Z}(T)$ and $\pi_{Z}(Q)$. Then by Theorem \ref{thm : bddgeod}(Bounded Geodesic Image) $\diam_{W}(h)\leq G$ and therefore 
 $$d_{W}(Q,T)\leq G.$$
  This upper bound contradicts the lower bound (\ref{eq : dWqt}) and the claim follows.

 Let  $k \subset \hull_{Z}(\mu^{-},\mu^{+})$ be any geodesic connecting $\pi_{Z}(\mu^{-})$ or $\mu^{-}|_{Z}$ and $\pi_{Z}(\mu^{+})$ or $\mu^{+}|_{Z}$. We claim that 
 \begin{claim}\label{claim : bwk}
 $\pi_{Z}(\partial{W})$ is in the $1-$neighborhood of $k$.
 \end{claim} 
  Otherwise, $\partial{W}$ would intersect every vertex of $k$, so Theorem \ref{thm : bddgeod} implies that 
 $$d_{W}(\mu^{-},\mu^{+})\leq G.$$
 On the other hand,  by (\ref{eq : dWmu+-}) and the choice of $A'$ in (\ref{eq : a'}) we have 
 $$d_{W}(\mu^{-},\mu^{+})\geq A'-2F-2(3D+6\delta)>G.$$
  But this lower bound contradicts the upper bound above and the claim follows. 
  \medskip
 
As before let $p_{Z}\in\pi_{Z}(P)$ and $\hat{u}$ realize the distance between $\pi_{Z}(P)$ and $\hull_{Z}(\mu^{-},\mu^{+})$. Suppose that $\hat{u}$ is on the geodesic $k$ in the hull. Let $n$ be an infinite geodesic containing a geodesic segment $[p_{Z},\hat{u}]$. Let $m$ be a nearest point to $\pi_{Z}(\mu^{+})$ on the geodesic $n$. When $\mu^{+}|_{Z}$ is in $\mathcal{EL}(Z)$, let $x_{i}$ be a sequence of points on $k$ so that $x_{i}\to \mu^{+}|_{Z}$ as $i\to \infty$ and let $m$ be the limit of the nearest points to the points $x_{i}$ on the geodesic $n$. See the left diagram of Figure \ref{fig : case2.2}. Then Lemma \ref{lem : orth} applied to the point $p_{Z}$, the geodesic $k$ and the point $\pi_{Z}(\mu^{+})$ on $k$ (points $x_{i}$ on $k$) implies that
\begin{equation}\label{eq : dZmu^}d_{Z}(\hat{u},m)\leq  6\delta_{Z}\leq 6\delta.\end{equation}
 The above inequality guarantees that $k$ and any geodesic $t$ between $m$ and $\pi_{Z}(\mu^{+})$ or $\mu^{+}|_{Z}$, $D_{1}$ fellow travel (see the left diagram of Figure \ref{fig : case2.2}). By Claim \ref{claim : bwk} there is a point on $k$ within distance $1$ of $\pi_{Z}(\partial{W})$. Then the $D_{1}$ fellow traveling of the geodesics $k$ and $t$ implies that $\pi_{Z}(\partial{W})$ is within distance $D_{1}+1$ of $t$. 
 \medskip

Recall the point $q'\in\pi_{Z}(Q)$ and the point $x$ on $\hull_{Z}(P,\Pi P)$ that realize the distance between $\pi_{Z}(Q)$ and $\hull_{Z}(P,\Pi P)$. Let the vertex $x'$ be a nearest point to $q'$ on $n$. Let $h'$ be a geodesic connecting $q'$ to $x'$ (see the left diagram of Figure \ref{fig : case2.2}). Let $r$ be the geodesic in $\hull_{Z}(P,\Pi P)$ between $\pi_{Z}(P)$ and $\pi_{Z}(\Pi P)$ on which $x$ lies. We have that $d_{Z}(\Pi P,\hat{u})\leq F$, so the end points of $n$ and $r$ are within distance $F+2$ of each other. Therefore, the geodesics $n$ and $r$, $D_{1}$ fellow travel. Moreover $x$ is a nearest point to $q'$ on the geodesic $r$ and $x'$ is a nearest point to $q'$ on the geodesic $n$. Then Lemma \ref{lem : ftrnearp} applied to the point $q'$ and the geodesics $n$ and $r$ implies that 
\begin{equation}\label{eq : dZxx'}d_{Z}(x,x')\leq 3D_{1}+6\delta.\end{equation}
 The inequalities (\ref{eq : dZxx'}) and (\ref{eq : dZTx}) combined by the triangle inequality imply that
$$d_{Z}(T,x')\leq d_{Z}(T,x)+d_{Z}(x,x')\leq M+3D_{1}+6\delta.$$
This bound implies that the end points of $h$ and $h'$ are within distance $M+3D_{1}+6\delta+2$ of each other (because $\diam_{Z}(T)\leq 2$). Therefore, $h$ and $h'$, $D_{2}$ fellow travel. Here $D_{2}$ is the maximum over all subsurfaces $Y\subseteq S$ of the fellow traveling distance of two geodesics in $\mathcal{C}(Y)$ with end points within distance $M+3D_{1}+6\delta+2$ of each other. By Claim \ref{claim : bwh} there is a point on $h$ within distance $1$ of $\pi_{Z}(\partial{W})$. Then the fellow traveling of $h$ and $h'$ implies that $\pi_{Z}(\partial{W})$ is within distance $D_{2}+1$ of $h'$.
\medskip

 Let $a=\max\{D_{1}+3,D_{2}+3\}$. By the conclusions of the above two paragraphs any point in $\pi_{Z}(\partial{W})$ is in the $a$ neighborhood of the geodesics $t$ and $h'$ (see the left diagram of Figure \ref{fig : case2.2}). Therfeore Proposition \ref{prop : treelike}(iv) applied to the point $q'$, a vertex in $\pi_{Z}(\mu^{+})$ realizing the distance between $\pi_{Z}(\mu^{+})$ and $n$ or each point point $x_{i}$ (when $\mu^{+}|_{Z}$ is in $\mathcal{EL}(Z)$ the points $x_{i}$ on $k$ converge to $\mu^{+}|_{Z}$ in the Gromov boundary of $\mathcal{C}(Z)$) and the geodesic $n$ implies that 
 \begin{equation}\label{eq : dZmx'}d_{Z}(m,x')\leq b(\delta_{Z},a)\leq b(\delta ,a).\end{equation} 
 
 The inequalities (\ref{eq : dZmu^}), (\ref{eq : dZxx'}), (\ref{eq : dZmx'}) and (\ref{eq : dPPv^}) combined by the triangle inequality give us
  \begin{equation} \label{eq : xpip}d_{Z}(x, \Pi P) \leq (3D_{1}+6\delta)+F+b+6\delta.\end{equation}
 The vertex $x$ is a nearest point to $\pi_{Z}(Q)$ on the geodesic $r\subset\hull_{Z}(P,\Pi P)$. Then Lemma \ref{lem : symm} applied to the geodesic $r$ and the point $q'$ and any point $j\in\pi_{Z}(P)$ which is on $r$ implies that
 $$d_{Z}(q',j)+6\delta\geq d_{Z}(j,x).$$
 Then since $q'\in\pi_{Z}(Q)$ and $j\in\pi_{Z}(P)$ we have
\begin{equation}\label{eq : dZPQ>dZPx}d_{Z}(P,Q)+6\delta\geq d_{Z}(P,x)-\diam_{Z}(P)-\diam_{Z}(Q)\geq d_{Z}(P,x)-4.\end{equation}
Now we have 
\begin{eqnarray*}
d_{Z}(P,Q)&\geq& d_{Z}(P,x)-6\delta-4 \geq d_{Z}(P,\Pi P)-d_{Z}(\Pi P,x)-6\delta-4\\
&\geq&d_{Z}(P,\Pi P)-(3D_{1}+6\delta)-F-b-6\delta-6\delta-6 
\end{eqnarray*}
The first inequality is (\ref{eq : dZPQ>dZPx}). The second one is the triangle inequality. The third one follows from inequality (\ref{eq : xpip}). 

By the above inequality (\ref{eq : consubs}) holds for $q=F+(3D_{1}+6\delta)+12\delta+b+6$.
\medskip

\noindent {\bf Case 2.3:}   $Z \subsetneq W$.   

Note that $\partial{Z} \pitchfork W$. We claim that 
\begin{claim}\label{claim : bzppip}
$\pi_{W}(\partial{Z})$ is in the $1-$neighborhood of any geodesic $l \subset \hull_{W}(P, \Pi P)$ connecting $\pi_{W}(P)$ to $\pi_{W}(\Pi P)$.
\end{claim}
 
Otherwise, $\partial{Z}$ would intersect every vertex of $l$. Then Theorem \ref{thm : bddgeod} implies that $d_{Z}(P, \Pi P) \leq G<A'$, but this contradicts (\ref{eq : 1}).
 
 Let $f$ be a geodesic in $\hull_{W}(\mu^{-},\mu^{+})$. Let $v$ and $z$ be nearest points to $p_{W}$ and $q_{W}$ on $f$, respectively. Here $p_{W}\in\pi_{W}(P)$ and $q_{W}\in\pi_{W}(Q)$ are nearest point to $\hull_{W}(\mu^{-},\mu^{+})$. Let  $k$ be a geodesic connecting $\pi_{W}(Q)$ and $\pi_{W}(\Pi Q)$, and $k'$ be a geodesic connecting $q_{W}$ and $z$. Let $l$ be a geodesic connecting $p_{W}$ and $\pi_{W}(\Pi P)$, and $l'$ be a geodesic connecting $p_{W}$ and $v$ (see the left diagram of Figure \ref{fig : Case2.3}). By (\ref{eq : dvPiP}) the end points of $k$ and $k'$ are within distance $F+(3D+6\delta)$ of each other. So $k$ and $k'$, $D_{1}$ fellow travel. Similarly $l$ and $l'$, $D_{1}$ fellow travel. By Claim \ref{claim : bzppip} and the $D_{1}$ fellow traveling of $l$ and $l'$ there is a point $e$ on $l'$ so that 
 \begin{equation}\label{eq : dWbdZp^}d_{W}(\partial{Z},e)\leq 1+D_{1}.\end{equation} 
 
 By (\ref{eq : dWvz}) and the choice of $A'$ in (\ref{eq : a'}) we have
 $$d_{W}(v,z)\geq A'-2F-2(3D+6\delta)-4> K_{W}.$$
  Then the tree like property (\ref{eq : tree}) implies that for any point $i$ on $k'$,
 $$d_{W}(e,i)\geq d_{W}(v,z)-\delta'_{W}.$$
  Again by (\ref{eq : dWvz}) and the choice of $A'$ we have
  $$d_{W}(v,z)-\delta'_{W}\geq A'-2F-2(3D+6\delta)-4-\delta'>2D_{1}+3.$$
From the above two inequalities we obtain
  $$d_{W}(e,i)>2D_{1}+3.$$
  Since the point $i$ on $k'$ was arbitrary we may conclude that 
  $$d_{W}(e,k')>2D_{1}+3.$$
   The above inequality and (\ref{eq : dWbdZp^}) imply that 
   $$d_{W}(\partial{Z},k')>D_{1}+2.$$
    Finally, the above inequality and the $D_{1}$ fellow traveling of $k$ and $k'$ imply that 
   $$d_{W}(\partial{Z},k)>2.$$
    Therefore, $\partial{Z}$ intersects every vertex of $k$. So Theorem \ref{thm : bddgeod} implies that
 \begin{equation}\label{eq : dZQPiQ}d_{Z}(Q, \Pi Q) \leq G.\end{equation} 
Now by the triangle inequality 
$$d_{Z}(P,Q)\geq d_{Z}(P,\Pi Q)-d_{Z}(\Pi Q, Q)-2.$$
The above inequality and the inequality (\ref{eq : dZQPiQ}) imply that  
\begin{equation}\label{eq : dZPQ}d_{Z}(P,Q)\geq d_{Z}(P,\Pi Q)-G-2.\end{equation}
 Let $f$ be a geodesics in $\hull_{Z}(\mu^{-},\mu^{+})$. Let $u$ and $w$ be nearest points to respectively $p_{Z}$ and $q_{Z}$ on $f$. Since $u$ is a nearest point to $p_{Z}$ on $f$ and $w$ is a point on $f$ we have that
 $$d_{Z}(p_{Z},w)\geq d_{Z}(p_{Z},u).$$
 Then since $\diam_{Z}(P)\leq 2$ (Lemma \ref{lem : diamproj}) we get
 \begin{equation}\label{eq : dZPw>dZPu}d_{Z}(P,w)\geq d_{Z}(P,u)-2\end{equation}
(see the right digram of Figure \ref{fig : Case2.3}).

 Now by the triangle inequality and the inequalities (\ref{eq : dZPw>dZPu}) and (\ref{eq : dvPiP}) we have
  \begin{eqnarray}\label{eq : dZPPQ}
  d_{Z}(P,\Pi Q)&\geq& d_{Z}(P,w)-d_{Z}(w,\Pi Q)-\diam_{Z}(\Pi Q)\\
  &\geq& d_{Z}(P,u)-F-(3D+6\delta)-4.\nonumber
  \end{eqnarray}
Moreover, by the triangle inequality and the bound (\ref{eq : dvPiP}) we have
  \begin{eqnarray}\label{eq : dZPu}
  d_{Z}(P,u)&\geq& d_{Z}(P,\Pi P)-d_{Z}(\Pi P,u)\\
  &\geq& d_{Z}(P,\Pi P)-F-(3D+6\delta)-4.\nonumber
  \end{eqnarray}
  By inequalities (\ref{eq : dZPu}) and (\ref{eq : dZPPQ}) we have 
 \begin{equation}\label{eq : dZPPQ>dZPPP} d_{Z}(P, \Pi Q) \geq d_{Z}(P,\Pi P)-2F-2(3D+6\delta)-4. \end{equation}
  Then the inequalities (\ref{eq : dZPPQ>dZPPP}) and (\ref{eq : dZPQ}) imply that
$$d_{Z}(P, Q)\geq d_{Z}(P, \Pi P)-G-2F-2(3D+6\delta)-6.$$
 Thus (\ref{eq : consubs}) holds for $q=G+2F+2(3D+6\delta)+6$.

\begin{figure}
\centering
\scalebox{0.2}{\includegraphics{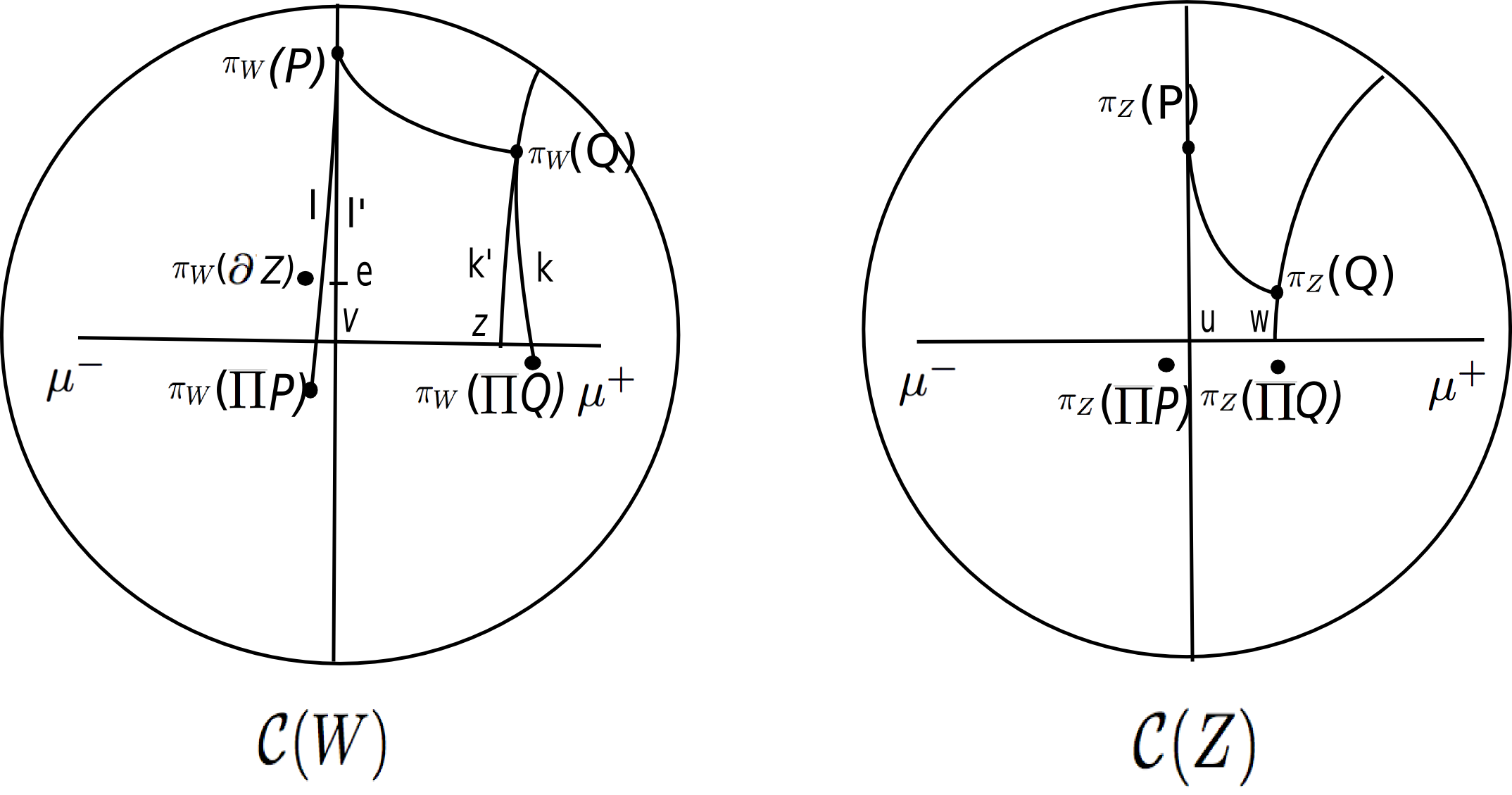}}
\caption{ {\bf Case 2.3:} Left diagram: $\pi_{W}(\partial{Z})$ is in the $1-$neighborhood of $l$. The geodesics $l$ and $l'$, and $k$ and $k'$, respectively, $D_{1}-$fellow travel, and $d_{W}(v,z)> 2D_{1}+3+G+\delta_{W}'$. It follows that $d_{W}(\partial{Z},k)>2$, and thus $\partial{Z}$ intersects every vertex of the geodesic $k$ connecting $\pi_{W}(Q)$ to $\pi_{W}(\Pi Q)$. Then Bounded Geodesic Image Theorem implies that $d_{Z}(Q,\Pi Q)\leq G$.}
\label{fig : Case2.3}
\end{figure}
\medskip

Establishing the inequality (\ref{eq : consubs}) in Cases (1) and (2) we may conclude that the inequality holds for $q$ the maximum of the constants $q$ we obtained in these two cases. This finishes the proof of the lemma.
\end{proof}
\begin{remark}
For the parameters $R,\eta$ and $B$ in the above theorem we have that $R \to \infty$, $\eta \to 0$ and $B \to \infty$, as $A \to \infty$. So applying the Morse lemma argument we get $\Sigma-$hulls with worse and worse stability property. More precisely, there are $K\geq 1$ and $C\geq 0$ such that $d_{A}(K,C)\to \infty$ as $A\to \infty$. 
\end{remark}


\subsection{Fellow traveling}\label{subsec : fellowtr} We start by the definition of fellow traveling of parametrized quasi-geodesics in a metric space.
   
  \begin{definition} \label{def : fellowtr}\textnormal{ (Fellow traveling)} 
  Given $D\geq 0$. Let $h_{1}:I_{1} \to \mathcal{X}$ and $h_{2}:I_{2} \to  \mathcal{X}$ be two parametrized quasi-geodesics. We say that $h_{1}$ and $h_{2}$, $D-$fellow travel if the following holds: For each $i \in I_{1}$ there is an $i' \in I_{2}$ such that $d(h_{1}(i), h_{2}(i')) \leq  D$ and vice versa. In other words, the Hausdorff distance of $h_{1}(I_{1})$ and $h_{2}(I_{2})$ is bounded by $D$.
 
Given a subinterval $I'_{1} \subset I_{1}$, we say that $h_{1}$, $D-$fellow travels $h_{2}$ over $I'_{1}$ if there is a subinterval $I'_{2} \subset I_{2}$ such that $h_{1}|_{I'_{1}}$ and $h_{2}|_{I'_{2}}$, $D-$fellow travel as above.
 \end{definition}
  
Let $h_{1}$ be a $(K_{1},C_{1})-$quasi-geodesic and $h_{2}$ be a $(K_{2},C_{2})-$quasi-geodesic. Let 
$$N_{h_{1},h_{2}}:I_{1}\to I_{2},$$
 be the coarse map which assigns to each $i \in I_{1}$ any $h_{2}(i')$ where $d(h_{1}(i),h_{2}(i')) \leq D$. Then $N_{h_{1},h_{2}}$ is a reparametrization of $h_{2}$ and for any $i,j\in I_{1}$, 
$$|N_{h_{1},h_{2}}(i),N_{h_{1},h_{2}}(j)|\asymp_{K,C}|i-j|$$
 The constants $K=K_{1}K_{2}$ and $C=\max\{K_{2}C_{1}+C_{2}+2K_{2}D,K_{1}C_{2}+C_{1}+2K_{1}D\}$. The above quasi-equality in particular implies that the diameter of $N_{h_{1},h_{2}}(i)$ is bounded above by $C$. Also the same holds exchanging $h_{1}$ and $h_{2}$. 
  
  \begin{thm} \label{thm : fellowtr}
  Given $A>0$ there is a constant $D$ with the following property. Let $g:[a,b]\to\Teich(S)$ be a WP geodesic segment with $A-$narrow end invariant $(\nu^{-},\nu^{+})$, and let $\rho$ be a hierarchy path  between $\nu^{-}$ and $\nu^{+}$. Then $\rho$ and $Q(g)$, $D-$fellow travel.
  \end{thm}
   \begin{proof}
The path $Q(g)$ is a $(K_{\WP}, C_{\WP})-$quasi geodesic in $P(S)$ where $K_{\WP}$ and $C_{\WP}$ depend only on the topological type of $S$ (Theorem \ref{thm : brockqisom}). The hierarchy path $\rho$ is a $(k,c)-$quasi-geodesic in $P(S)$ where $k$ and $c$ depend only on the topological type of $S$ (see $\S$\ref{sec : ccplx}).  By Theorem \ref{thm : stability}, $\rho$ is $d_{A}-$stable. So the Hausdorff distance of $Q(g)$ and $|\rho|$ is bounded by $D=d_{A}(K_{WP},C_{WP})$ in $P(S)$. Then by Definition \ref{def : fellowtr}, $\rho$ and $Q(g)$, $D$ fellow travel.  
\end{proof}

Let $g:[a,b]\to\Teich(S)$ be a WP geodesic segment and $\rho:[m,n]\to P(S)$ be a hierarchy path. Suppose that $Q(g)$ and $\rho$, $D-$fellow travel. Given $i \in [m,n]$ let $I_{i}\subseteq[a,b]$ be the smallest interval that contains every $t \in [a,b]$ such that $d(Q(g(t)),\rho(i))\leq D$. Then define the coarse map 
$$N_{\rho, g}:[m,n]\to [a,b],$$ 
so that $N_{\rho, g}(i)$ is any $t\in I_{i}$. The following proposition is a straightforward consequence of the definition of $N_{\rho,g}$.
 
  \begin{prop}\label{prop : parmap} The coarse map $N_{\rho,g}:[m,n]\to [a,b]$ has the following properties:
  \begin{itemize}
  \item Let $i\in[m,n]$. Let $r,s\in N_{\rho,g}(i)$, then $|r-s|\leq K_{\WP}(2D+C_{\WP})$.
   \item Let $i\in[m,n]$ and $r\in N_{\rho,g}(i)$, then $d(\rho(i),Q(g(r)))\leq D_{1}$ where $D_{1}=K_{\WP}(2DK_{\WP}+K_{\WP}C_{\WP})+C_{\WP}+D$.
  \item $\bigcup_{i\in [m,n]}N_{\rho,g}(i)$ covers $[a,b]$.
\item There are constants $K\geq 1$ and $C\geq 0$ depending on $D$ such that for any $i,j\in[m,n]$ we have
 $$|N_{\rho,g}(i)-N_{\rho,g}(j)| \asymp_{K,C} |i-j|.$$ 
 \end{itemize}
 \end{prop}
 By Theorem \ref{thm : fellowtr} this proposition in particular applies to a WP geodesic segment with narrow end invariant and a hierarchy path with the same end points.


\section{Itinerary of Weil-Petersson geodesic segments}\label{sec : itinerarywpgeod}

The itinerary of a WP geodesic $g$ in the Teichm\"{u}ller space refers to the list of short curves (curves with length less than a sufficiently small $\epsilon>0$), the time intervals along $g$ that each curve is short and the order that these intervals appear along $g$.

In this section we present our results on the control of length-functions and twist parameters along WP geodesics with narrow end invariant. 

Our main result Theorem \ref{thm : bdryzsh} asserts the following: Suppose that over an interval all of the subsurface coefficients are bounded, except possibly those of some annular subsurfaces whose core curves together consist the boundary of a large subsurface. Then the length of these curves are arbitrary short over a suitably shrunk subinterval of the interval. We prove this theorem at the end of $\S$\ref{subsec : lfcontrol}. This theorem describes a partial itinerary for WP geodesic segments with narrow end invariant. 

\subsection{Isolated annular subsurfaces}\label{subsec : comblem}

In this subsection we prove some combinatorial lemmas which together with the fellow traveling property of hierarchy paths between narrow pairs (Theorem \ref{thm : fellowtr}) provide us with a combinatorial frame work in which we will be able to control length-functions along WP geodesics. We also introduce the notion of an {\it isolated annular subsurface}, Definition \ref{def : isolatedann}.
\medskip

\noindent{\bf Bounded combinatorics:}
Given $R, R'>0$ and a subsurface $Z$, we say that $Z$ has $(R,R')-$bounded combinatorics between a pair of partial markings or laminations $\mu_{1}$ and $\mu_{2}$ if for any proper essential non-annular subsurface $Y \subsetneq Z$,
$$d_{Y}(\mu_{1}, \mu_{2}) \leq R,$$
and for any annular subsurface with core curve $\gamma \in \mathcal{C}_{0}(Z)$ which intersects both $\mu_{1}$ and $\mu_{2}$,
$$d_{\gamma}(\mu_{1}, \mu_{2}) \leq R'.$$
If only the first bound holds we say that the $Z$ has non-annular $R-$bounded combinatorics and if only the second bound holds we say that $Z$ has annular $R'-$bounded combinatorics.

\begin{lem}\textnormal{ (No backtracking of hierarchy paths)}
Let $\rho:[m,n]\to P(S)$ be a hierarchy path. Let $[i_{1},i_{2}]\subseteq [m,n]$ and $i, j\in [i_{1},i_{2}]$ with $i<j$. Then for every subsurface $Y\subseteq S$ we have that 
\begin{equation}\label{eq : nbtrack}d_{Y}(\rho(i_{1}),\rho(i_{2}))\geq d_{Y}(\rho(i),\rho(j))-2M_{2} .\end{equation}
\end{lem}
\begin{proof}
We have that
 \begin{eqnarray}\label{eq : dYri1ri2}
 &&d_{Y}(\rho(i_{1}),\rho(i))+d_{Y}(\rho(i),\rho(j))+d_{Y}(\rho(j),\rho(i_{2}))\\
 &\leq&d_{Y}(\rho(i_{1}),\rho(j))+d_{Y}(\rho(j),\rho(i_{2}))+M_{2}\nonumber\\
 &\leq& d_{Y}(\rho(i_{1}),\rho(i_{2}))+2M_{2}.\nonumber
 \end{eqnarray}
 The first inequality follows from Theorem \ref{thm : hrpath}(\ref{h : nobacktrack}) (no backtracking) considering $i_{1}<i<j$. The second inequality follows from no backtracking considering $i<j<i_{2}$. All of the terms in the first sum of (\ref{eq : dYri1ri2}) are non negative, thus 
 $$d_{Y}(\rho(i_{1}),\rho(i_{2}))\geq d_{Y}(\rho(i),\rho(j))-2M_{2},$$
  as was desired.    
  \end{proof}
 Now let $Z$ be a component domain of $\rho$ and $[i_{1},i_{2}]\subset J_{Z}$. If $Z$ has $(R,R')-$bounded combinatorics between $\rho(i_{1})$ and $\rho(i_{2})$ then by (\ref{eq : nbtrack}) for any $i,j\in[i_{1},i_{2}]$ and any subsurface $Y\subseteq S$,
 $$d_{Y}(\rho(i),\rho(j))\leq R+2M_{2}$$ 
 and for any $\gamma\in \mathcal{C}_{0}(Z)$
 $$d_{\gamma}(\rho(i),\rho(j))\leq R'+2M_{2}.$$
  In this situation we say that $Z$ has $(R,R')-$bounded combinatorics over the interval $[i_{1},i_{2}]$.
 \medskip
 
 In the following lemma we show that over a portion of a hierarchy path where a large subsurface $Z$ has non-annular bounded combinatorics $\pi_{Z}\circ \rho$ ($\pi_{Z}$ is the $Z$ subsurface projection) is a parametrization of the geodesic $g_{Z}\subset \mathcal{C}(Z)$ (see Theorem \ref{thm : hrpath}(\ref{h : J})) as a quasi-geodesic with constants depending only on $R$.

 \begin{lem} \label{lem : bddcomb} 
 
 Given $R>0$, there are $K_{R}\geq 1$ and $C_{R} \geq 0$ with the following properties. Let $\rho :[m,n] \to P(S)$ be a hierarchy path and let $Z$ be a large, non-annular component domain with non-annular $R-$bounded combinatorics over $[i_{1},i_{2}] \subset J_{Z}\subset [m,n]$. Then for any  $i,j \in [i_{1}, i_{2}] $ we have
 $$d(\rho(i),\rho(j))\asymp_{K_{R},C_{R}} d_{Z}(\rho(i),\rho(j)).$$
 \end{lem}
 
 \begin{proof}
 Given a threshold constant $A\geq M_{1}$, we write the distance formula (\ref {eq : dsf})
 \begin{equation} \label{eq : distf} d(\rho(i),\rho(j))\asymp_{K,C}\sum_{\substack{ Y\subseteq S\\ \nonannular}}\{ d_{Y}(\rho(i),\rho(j))\}_{A}.\end{equation}
Note that since the connected components of $S\backslash Z$ are only annuli and three holed spheres, every subsurface $Y$ contributing to the above sum either is a subsurface of $Z$ or overlaps $Z$. Suppose that $Y \pitchfork Z$.  Since $i,j \in J_{Z}$, by Theorem \ref{thm : hrpath}(\ref{h : J}), $\partial{Z}\subset \rho(i)$ and $\partial{Z}\subset\rho(j)$. Thus $\pi_{Y}(\rho(i))\cap\pi_{Y}(\rho(j))\neq\emptyset$. Then
 \begin{equation}\label{eq : YintZbd}d_{Y}(\rho(i),\rho(j))\leq 2.\end{equation} 
Now suppose that $Y \subsetneq Z$. With out loss of generality we may assume that $i_{1}<i<j<i_{2}$. Then the non-annular $R-$bounded combinatorics of $Z$ over $[i_{1},i_{2}]$ and no backtracking of hierarchy paths (\ref{eq : nbtrack}) imply that 
\begin{equation}\label{eq : YinZbd}d_{Y}(\rho(i),\rho(j)) \leq d_{Y}(\rho(i_{1}), \rho(i_{2}))+2M_{2}\leq R+2M_{2}.\end{equation}
Having the bounds (\ref{eq : YintZbd}) and (\ref{eq : YinZbd}) on the subsurface coefficients contributing to the sum (\ref{eq : distf}) if we let the threshold constant be $A_{R}=\max \{M_{1},R+2M_{2},2\}$ we get
$$d(\rho(i),\rho(j))\asymp_{K_{R},C_{R}} d_{Z}(\rho(i),\rho(j)),$$
where $K_{R}$ and $C_{R}$ are the constants corresponding to the threshold constant $A_{R}$ in the distance formula.
 \end{proof}
 
 An elementary move on a pants decomposition $P$, replaces a curve $\alpha \in P$ with a curve $\alpha'$ with distance $1$ in the complexity $1$ subsurface which $\alpha$ and $\alpha'$ fill and does not change any curve in $P-\alpha$. Thus the distance in the curve complex of a non-annular subsurface between the projections of $P$ and a pants decomposition obtained from $P$ by an elementary move is at most $1$. Then for any $P,Q \in P(S)$ and any non-annular subsurface $Y\subseteq S$ we have that
  \begin{equation} \label{eq : ddsubs}d_{Y}(P,Q) \leq d(P,Q).\end{equation}
 Let $\rho$ and $Q(g)$, $D-$fellow travel and $N=N_{g,\rho}$ be the parameter map from Proposition \ref{prop : parmap}. Given $i,j \in [m,n]$ let $r \in N(i)$ and $s\in N(j)$. Then by the second bullet of the proposition there is $D_{1}$ depending on $D$ so that $d(\rho(i),Q(g(r)))\leq D_{1}$ and $d(\rho(j),Q(g(s)))\leq D_{1}$. Then by (\ref{eq : ddsubs}), $d_{Y}(\rho(i),Q(g(r))\leq D$ and  $d_{Y}(\rho(j),Q(g(s)))\leq D$. So by the triangle inequality we obtain the quasi-equality of subsurface coefficients 
 $$d_{Y}(\rho(i),\rho(j)) \asymp_{1,2D_{1}} d_{Y}(Q(g(r)),Q(g(s))).$$
 But such a quasi-equality a-priori does not hold for annular subsurfaces, so we consider isolated annular subsurfaces along hierarchy paths. In Lemma \ref{lem : anncoeffcomp} we prove a quasi-equality for the subsurface coefficient of an isolated annular subsurface along $\rho$ and a $D-$fellow traveling WP geodesic depending only on $D$.  

 \begin{definition}\textnormal{ (Isolated annular subsurface)}\label{def : isolatedann}
Given $w,r,R>0$. Let $\rho:[m,n] \to P(S)$ be a hierarchy path. We say that an annular subsurface $A(\gamma)$ with core curve $\gamma$ is $(w,r,R)-$isolated at $i \in [m,n]$ if the following holds: There is a pants decomposition $\widehat{Q}$ such that $\gamma\in\widehat{Q}$ and $d(\widehat{Q},\rho(i))\leq r$. There are large, non-annular component domains of $\rho$, $Z_{1}$ and $Z_{2}$ with $\gamma\notin\partial{Z_{1}}$ and $\gamma\notin\partial{Z_{2}}$ and intervals $I_{1}\subseteq J_{Z_{1}}$ with $|I_{1}|\geq w$, and $I_{2}\subseteq J_{Z_{2}}$ with $|I_{2}|\geq w$ such that $\max I_{1}< i < \min I_{2}$. Moreover $Z_{1}$ and $Z_{2}$ have non-annular $R-$bounded combinatorics over $I_{1}$ and $I_{2}$, respectively. For the illustration of isolation see Figure \ref{fig : anncoeffbd}.
\end{definition}

\begin{lem} \label{lem : anncoeffcomp}\textnormal{(Annular coefficient comparison)}
Given $D,r$ and $R$ positive, there are constants ${\bf w}={\bf w}(D,r,R)$ and $B=B(D)$ with the following properties. Let $\rho :[m,n] \to P(S)$ be a hierarchy path. Let $A(\gamma)$ be $({\bf w},r,R)-$isolated at $i$. Let the intervals $I_{1},I_{2}\subset [m,n]$ be as in the definition of isolated annular subsurface. Moreover let $i_{1},i_{2}\in[m,n]$ be so that $i_{1}< \min  I_{1}$ and $i_{2}> \max  I_{2}$. Let $g:[a,b] \to \Teich(S)$ be a WP geodesic parametrized by arc-length such that $Q(g)$ and $\rho$, $D-$fellow travel. Let $t_{1}\in N(i_{1})$ and $t_{2}\in N(i_{2})$, where $N=N_{\rho,g}$ is the parameter map from Proposition \ref{prop : parmap}. Then we have the quasi-equality for annular subsurface coefficients
\begin{equation}\label{eq : anncoeffcomp}d_{\gamma}(Q(g(t_{1})),Q(g(t_{2})) \asymp_{1,B} d_{\gamma}(\rho(i_{1}),\rho(i_{2})).\end{equation}
Furthermore, we have the following lower bound for the length of $\gamma$,
\begin{equation}\label{eq : lnbzlb}\min \{ \ell_{\gamma}(g(t_{1})),\ell_{\gamma}(g(t_{2}))\} \geq \omega(L_{S}).\end{equation}
Here $\omega(l)$ denotes the width of the collar of a simple closed geodesic with length $l$ on a complete hyperbolic surface.
\end{lem}

\begin{proof}
Let ${\bf w}(D,r,R)=kK_{R}\big(D_{1}+r+C_{R}+7+2M_{2}\big)+kc$, where $K_{R}$ and $C_{R}$ are the constants from Lemma \ref{lem : bddcomb} depending on $R$ and $k,c$ are the constants for the quasi-geodesic $\rho$ depending only on the topological type of $S$. Moreover $M_{2}$ is the constant in (\ref{eq : nbtrack}) and $D_{1}$ is the constant from the second bullet of Proposition \ref{prop : parmap}. Since $\gamma$ is $({\bf w},r,R)-$isolated at $i$, there is a pants decomposition $\widehat{Q}$ such that $\gamma \in \widehat{Q}$ and $d(\widehat{Q},\rho(i))\leq r$. Then by (\ref{eq : ddsubs}) for any non-annular subsurface $Y\subseteq S$ we have
\begin{equation}\label{eq : dYQri}d_{Y}(\widehat{Q},\rho(i))\leq r.\end{equation}
By the second bullet of Proposition \ref{prop : parmap} we have $d(\rho(i_{1}),P)\leq D_{1}$ for every $P$ on a geodesic connecting $\rho(i_{1})$ and $Q(g(t_{1}))$. Then by (\ref{eq : ddsubs}) for any non-annular subsurface $Y\subseteq S$ we have
\begin{equation}\label{eq : dYQgt1ri1}d_{Y}(P,\rho(i_{1}))\leq D_{1}.\end{equation}
 Let $P$ be a pants decomposition on a geodesic in $P(S)$ connecting $\rho(i_{1})$ to $Q(g(t_{1}))$ (see Figure \ref{fig : anncoeffbd}). Let $I_{1}=[j,j']$. The hierarchy path $\rho$ is a $(k,c)-$quasi-geodesic, so 
 $$d(\rho(j),\rho(j'))\geq \frac{1}{k}|I_{1}|-c.$$
  Then since the subsurface $Z_{1}$ has $R$ non-annular bounded combinatorics over $I_{1}$, by Lemma \ref{lem : bddcomb} we have 
$$d_{Z_{1}}(\rho(j), \rho(j'))\geq \frac{1}{K_{R}}d(\rho(j), \rho(j'))-C_{R}\geq\frac{1}{K_{R}}(\frac{1}{k}|I_{1}|-c)-C_{R}.$$
 Then by the no backtracking property of hierarchy paths (\ref{eq : nbtrack}) for the times $i_{1},j,j'$ and $i$ we have 
 \begin{equation}\label{eq : dZ1riri1}d_{Z_{1}}(\rho(i), \rho(i_{1}))\geq \frac{1}{K_{R}}(\frac{1}{k}|I_{1}|-c)-C_{R}-2M_{2}.\end{equation}
  Then by the triangle inequality we have
\begin{eqnarray*}
d_{Z_{1}}(P, \widehat{Q}) &\geq& d_{Z_{1}}(\rho(i_{1}), \rho(i)) - d_{Z_{1}}(\widehat{Q},\rho(i)) - d_{Z_{1}}(\rho(i_{1}),P)-4 \\
&\geq& \frac{1}{K_{R}}\big(\frac{1}{k}|I_{1}|-c\big) -2M_{2}- C_{R}- r-D_{1}-4\geq 3.
\end{eqnarray*}
The second inequality follows from the inequalities (\ref{eq : dYQri}), (\ref{eq : dYQgt1ri1}) and (\ref{eq : dZ1riri1}). Since $\gamma$ is $({\bf w},r,R)-$isolated at $i$ and $I_{1}$ is the isolating interval on the left of $i$, $|I_{1}|\geq {\bf w}$. Then the last inequality follows from the choice of ${\bf w}$.

Now since $\gamma\in\widehat{Q}$ and $\gamma\notin\partial{Z_{1}}$, the above inequality implies that for any curve $\alpha\in P$, $d_{Z_{1}}(\alpha,\gamma)\geq 3$. Thus as we saw in $\S$\ref{sec : ccplx} the curves $\alpha$ and $\gamma$ intersect each other. Then by (\ref{eq : ddsubs}) the projection of the geodesic connecting $\rho(i_{1})$ and $Q(g(t_{1}))$ to $A(\gamma)$ has diameter bounded above by $D_{1}$. So we have
$$d_{\gamma}(Q(g(t_{1})),\rho(i_{1}))\leq D_{1}.$$
 Replacing $i_{1}$ by $i_{2}$, $t_{1}$ by $t_{2}$, $Z_{1}$ by $Z_{2}$ and $I_{1}$ by $I_{2}$ an argument similar to the above argument gives us
 $$d_{\gamma}(Q(g(t_{2})),\rho(i_{2}))\leq D_{1}.$$ 
 The above two inequalities and the triangle inequality imply that the quasi-equality for annular coefficients (\ref{eq : anncoeffcomp}) holds for $B=2D_{1}$.

As we saw above $\gamma \pitchfork Q(g(t_{1}))$ and $\gamma \pitchfork Q(g(t_{2}))$. Now since every curve in the pants decompositions $Q(g(t_{1}))$ and $Q(g(t_{2}))$ has length at most the Bers constant $L_{S}$, the lower bound for the length of $\gamma$ in (\ref{eq : lnbzlb}) follows from the Collar Lemma (see \cite[\S 4.1]{buser}).
\begin{figure}
\centering
\scalebox{0.2}{\includegraphics{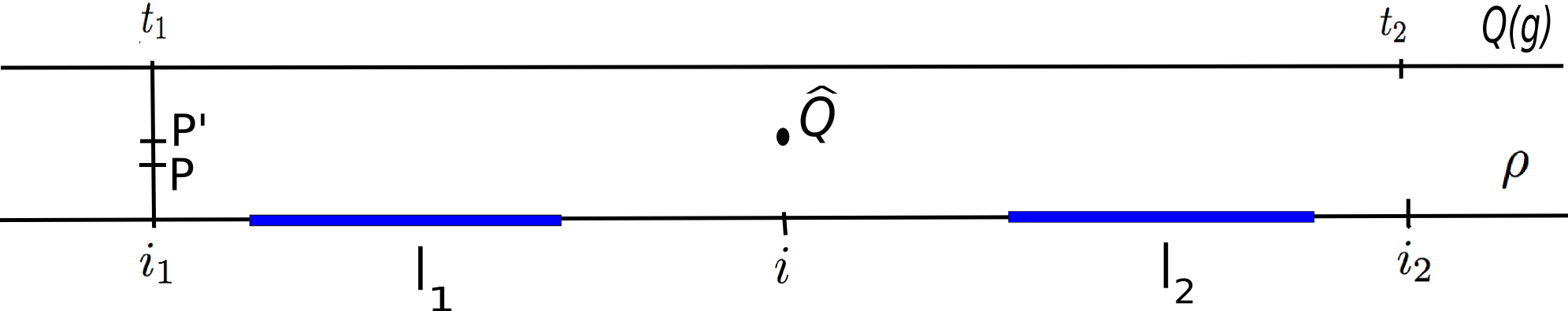}}
\caption{Isolated annular subsurface: The curve $\gamma\in\widehat{Q}$ and $d(\widehat{Q},\rho(i))\leq r$. Non-annular, large component domains $Z_{1}$ and $Z_{2}$ have non-annular $R-$bounded combinatorics on the left and the right blue subintervals, respectively, also $\gamma\notin \partial{Z}_{1}$ and $\gamma\notin\partial{Z}_{2}$. Moreover, each blue interval has length at least $w$. Then $A(\gamma)$ is $(w,r,R)-$isolated at $i$.} 
\label{fig : anncoeffbd}
\end{figure}

\end{proof}


\subsection{Length-function control}\label{subsec : lfcontrol}
 Corollaries \ref{cor : twsh} and \ref{cor : shtw} in $\S$\ref{sec : bddwpgeodseg} provide a local control on buildup of Dehn twists about a curve versus the change of its length along WP geodesic segments. The parameters of the control depend on the length of the geodesic segment and the supremum of the length-function along the geodesic segment. Lemma \ref{lem : anncoeffcomp} provides a quasi-equality for the subsurface coefficient of an isolated annular subsurface between two points of a hierarchy path and the corresponding points of a fellow traveling WP geodesic. The constant of the quasi-equality depends only on the fellow traveling distance. Using the quasi-equality we may pull back the annular coefficient along the geodesic to the hierarchy path and use the combinatorial properties of the hierarchy to control the length-functions along the geodesic. As we will see in $\S$\ref{sec : examples} the control of length-functions leads to the control of the global behavior of WP geodesics.
 
 \begin{lem} \label{lem : roughub}\textnormal{(Rough bounds)}

Given $D,R,R'>0$, there are constants $w=w(D,R)$, $\bar{\epsilon}=\bar{\epsilon}(D,R,R')$ and $l=l(D,R)$ with the following properties. Let $\rho:[m,n] \to P(S)$ be a hierarchy path. Suppose that a large component domain $Z$ has $(R,R')-$bounded combinatorics over $[m',n']\subset J_{Z}$. Let $g:[a,b] \to \overline{\Teich(S)}$ be a WP geodesic parametrized by arc-length such that $\rho$ and $Q(g)$, $D-$fellow travel. If $m'-n'>2w$ then
 \begin{enumerate}
\item $\ell_{\alpha}(g(t)) \leq l$ for every $\alpha \in \partial{Z}$ and
 \item $\ell_{\gamma}(g(t)) > \bar{\epsilon}$ for every $\gamma \notin \partial{Z}$
 \end{enumerate}
  for every $t \in [a', b']$, where $a' \in N(m'+w)$ and $b' \in N(n'-w)$. Here $N=N_{\rho,g}$ is the parameter map from Proposition \ref{prop : parmap}. 
 \end{lem}

\begin{proof}
We begin by establishing the lower bound (2). Let the constants $K$ and $C$ of the parameter map $N$ be as in Proposition \ref{prop : parmap}. Let $D_{1}$ be the constant from the second bullet of the proposition. Let $d_{1}=K_{\WP}D_{1}+K_{\WP}C_{\WP}$ and $D'=K_{\WP}d_{1}+C_{\WP}$. Let $w={\bf w}(D,D_{1},R)$, where ${\bf w}$ is the constant from Lemma \ref{lem : anncoeffcomp}. Let $T=Kw+C$ and $s=\frac{1}{K}w-C$. Suppose that $n'-m'>2w$. Let $i_{0}\in [m'+w,n'-w]$ and $t_{0}\in N(i_{0})$. Let $t_{0}^{+}\in N(i_{0}+w)$ and $t_{0}^{-}\in N(i_{0}-w)$. Then  $t_{0}^{+}-t_{0}\in[s,T]$ and $t_{0}-t_{0}^{-}\in[s,T]$. See Figure \ref{fig : fellowtraveling}. We denote by $\mu(g(t_{0}))$ a (partial) Bers marking of the surface $g(t_{0})$. Similarly, we denote (partial) Bers markings of $g(t_{0}^{-})$ and $g(t_{0}^{+})$ by $\mu(g(t_{0}^{-}))$ and $\mu(g(t_{0}^{+}))$, respectively. 
 
Let $\gamma\notin\partial{Z}$. If $\ell_{\gamma}(g(t_{0}))\geq L_{S}$ then we already have the lower bound (2) for the length of $\gamma$ at $g(t_{0})$. Otherwise, there is $Q_{0}$ a Bers pants decomposition of $g(t_{0})$ such that $\gamma\in Q_{0}$. Moreover by the second bullet of Proposition \ref{prop : parmap} we have $d(Q_{0},\rho(i_{0}))\leq D_{1}<D'$. Thus according to Definition \ref{def : isolatedann}, by the assumptions of the lemma, $A(\gamma)$ is $(w,D',R)$-isolated at $i_{0}$ where the non-annular bounded combinatorics component domain on both sides is $Z$. Then by Lemma \ref{lem : anncoeffcomp}(\ref{eq : lnbzlb}) we have the lower bound
\begin{equation}\label{eq : lgt-t+}\min \{ \ell_{\gamma}(g(t_{0}^{-})),\ell_{\gamma}(g(t_{0}^{+})) \} \geq \omega(L_{S}).\end{equation} 

 By the choice of $t_{0}^{-}$ and $t_{0}^{+}$ we have $|t_{0}^{+}-t_{0}^{-}|\in[2s,2T]$. By the bound (\ref{eq : lgt-t+}), Corollary \ref{cor : gradlfestimate} implies that there is an $s_{0}>0$ with $s_{0}\leq s$ so that for any $r\in[t_{0}^{-},t_{0}^{-}+s_{0}]$ and any $r\in[t_{0}^{+}-s_{0},t_{0}^{+}]$ we have 
 $$\ell_{\gamma}(u(r))\geq\frac{\omega(L_{S})}{2}.$$
  Then by Corollary \ref{cor : shtw} there is an $\bar{\epsilon}\leq L_{S}$ depending on $\omega(L_{S}),T,s_{0}$ and the lower bound in (\ref{eq : dgammaglb}) such that: If 
$$\inf_{t\in [t_{0}^{-},t_{0}^{+}]} \ell_{\gamma}(g(t)) \leq \bar{\epsilon},$$
 then 
\begin{equation} \label{eq : dgammaglb}d_{\gamma}(\mu(g(t_{0}^{-})),\mu(g(t_{0}^{+})))> R'+2M_{2}+B,\end{equation}
Here $B=B(D)$ is the constant from Lemma \ref{lem : anncoeffcomp}(\ref{eq : anncoeffcomp}) and $M_{2}$ is the constant in (\ref{eq : nbtrack}).

The annular subsurface $A(\gamma)$ is $(w,D',R)$-isolated at $i_{0}$, and $\rho$ and $Q(g)$, $D-$fellow travel. So by Lemma \ref{lem : anncoeffcomp}(\ref{eq : anncoeffcomp}) there is $B$ depending on $D$ so that
  $$ d_{\gamma}(\rho(i_{0}-w),\rho(i_{0}+w))\geq  d_{\gamma}(\mu(g(t_{0}^{-})),\mu(g(t_{0}^{+})))-B.$$
  Since $m'\leq i_{0}-w\leq i_{0}+w\leq n'$ by the no backtracking property of hierarchy paths (\ref{eq : nbtrack}) we have 
$$d_{\gamma}(\rho(m'),\rho(n'))\geq d_{\gamma}(\rho(i_{0}-w),\rho(i_{0}+w))-2M_{2}.$$
The above two inequalities imply that
$$d_{\gamma}(\rho(m'),\rho(n'))\geq d_{\gamma}(\mu(g(t_{0}^{-})),\mu(g(t_{0}^{+})))-B-2M_{2}.$$
The above inequality and (\ref{eq :  dgammaglb}) imply that $d_{\gamma}(\rho(m'),\rho(n'))>R'$. But this contradicts the assumption of the lemma that the $\gamma$ annular subsurface coefficient is bounded above by $R'$. The lower bound $\bar{\epsilon}$ for the length of $\gamma$ at $g(t_{0})$ follows from this contradiction. 

Now by Proposition \ref{prop : parmap}, $\bigcup_{i\in[m'+w,n'-w]}N(i)$ covers $[a',b']$. Thus $\ell_{\gamma}(g(t)) > \bar{\epsilon}$ for any $t \in [a',b']$. This is the desired lower bound (2). Note that $\bar{\epsilon}$ depends only on $R',R,D$. In particular $\bar{\epsilon}$ is uniform over $[a',b']$ i.e. does not depend on the value of the parameter $t_{0}$.
\medskip

We proceed to establish the rough upper bound (1). Fix $\alpha\in \rho(i_{0})-\partial{Z}$. Let $U_{L_{S}}(\rho(i_{0}))$ be the Bers region of $\rho(i_{0})$, consisting of all marked surfaces at which the length of every curve in $\rho(i_{0})$ is at most $L_{S}$. We choose a hyperbolic surface $x\in U_{L_{S}}(\rho(i_{0}))$ so that
\begin{itemize} 
\item $\ell_{\alpha'}(x)=L_{S}$ for every $\alpha' \in \rho(i_{0})-\alpha$, and 
\item $\ell_{\alpha}(x)=\frac{L_{S}}{2}$ if $\ell_{\alpha}(g(t_{0}))\geq \frac{3L_{S}}{4}$, and $\ell_{\alpha}(x)=L_{S}$ if $\ell_{\alpha}(g(t_{0}))<\frac{3L_{S}}{4}$.
\end{itemize}
 We observe that 
\begin{enumerate}[(a)]
\item $\rho(i_{0})$ is a Bers pants decomposition of $x$. 

\noindent Let ${\bf d}$ be the function in Corollary \ref{cor : gradlfestimate}, let $d_{0}=\min\{{\bf d}(L_{S},\frac{L_{S}}{4}),{\bf d}(\frac{3L_{S}}{4},\frac{L_{S}}{4})\}$. Then we have 
\item $d_{\WP}(g(t_{0}),x) \geq d_{0}$. 

\noindent Also observe that by changing the twist parameters about the curves in $\rho(i_{0})$ we can obtain $x$ such that
\item  a Bers marking of $x$, which we denote by $\mu_{x}$, is equal to a marking slice of any hierarchy $H(\mu^{-},\mu^{+})$ and $\base(\mu_{x})=\rho(i_{0})$. Here $\mu^{-}$ and $\mu^{+}$ are the end points of $\rho$. There may be several such surfaces and we choose one.
\item It follows from (a) that $x$ has injectivity radius $\inj(x)>0$ depending only on $L_{S}$.
\end{enumerate}

Let as before $D_{1}$ be the constant from the second bullet of Proposition \ref{prop : parmap} and $d_{1}=K_{\WP}D_{1}+K_{\WP}C_{\WP}$. Let $[x,g(t_{0})]$ be the WP geodesic connecting $x$ to $g(t_{0})$. Let $u:[0,d]\to\Teich(S)$ be the parametrization of $[x,g(t_{0})]$ by arc-length such that $u(0)=x$ and $u(d)=g(t_{0})$. Then by Theorem \ref{thm : brockqisom} the length of $u$ is bounded above by $d_{1}$. Thus $d\leq d_{1}$. On the other hand, by (b) we have $d\geq d_{0}$.

To get a rough upper bound for the length of the boundary components of $Z$ at $g(t_{0})$ first we establish a lower bound for the length of every $\gamma \notin \partial{Z}$ along $u$. Then the rough upper bound will follow from a compactness argument.
\medskip

 Let $\bar{\epsilon}=\bar{\epsilon}(D,R,R')$ be the lower bound (2) we established earlier for $\ell_{\gamma}(g(t))$ for each $\gamma\notin\partial{Z}$ and every $t\in[a'+T,b'-T]$. Here we make the following two choices:

\begin{figure}
\centering
\scalebox{0.2}{\includegraphics{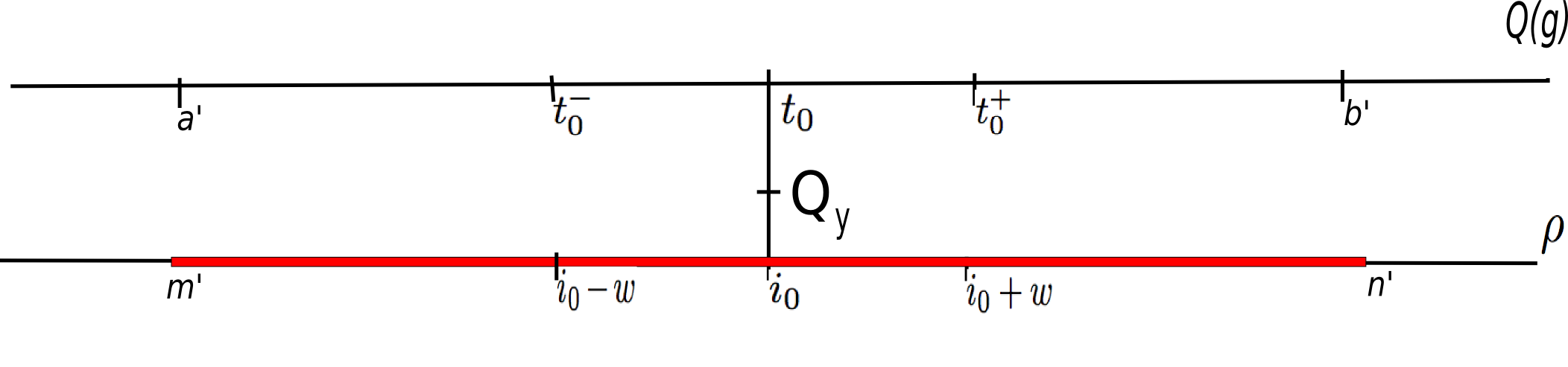}}
\caption{ The component domain $Z$ has $(R,R')-$bounded combinatorics over the red interval. At the point $y$ on $[x,g(t)]$ the infimum of the length of $\gamma$ is realized. Thus there is $Q_{y}$ a Bers pants decomposition of $y$ so that $\gamma\in Q_{y}$. Then $A(\gamma)$ is $(w,D',R)-$isolated at $i_{0}$.}
\label{fig : fellowtraveling}
\end{figure}

 \begin{enumerate}[(i)]
\item \label{ch : N0}Let $I^{-}= [t_{0}^{-},t_{0}]$ and $I^{+}=[t_{0},t_{0}^{+}]$. By (\ref{eq : lgt-t+}), $\ell_{\gamma}(g(t_{0}^{-}))\geq \omega(L_{S})$. Thus $\sup_{t\in I^{-}}\ell_{\gamma}(g(t))\geq \omega(L_{S})$. Moreover, the length of the interval $I^{-}$, $|I^{-}|\in[s,T]$. Then Corollary \ref{cor : twsh} applied to the geodesic segment $g|_{I^{-}}$ implies that there is an $N_{0}\in\mathbb{N}$ depending only on $T,\omega(L_{S})$ and $\bar{\epsilon}$ such that: If 
$$d_{\gamma}(\mu(g(t_{0}^{-})),\mu(g(t_{0})))>N_{0},$$ 
then 
$$\inf_{t \in I^{-}} \ell_{\gamma}(g(t))\leq\bar{\epsilon}.$$ 
Similarly, $\sup_{t\in I^{+}}\ell_{\gamma}(g(t))\geq \omega(L_{S})$ and $|I^{+}|\in[s,T]$. Then Corollary \ref{cor : twsh} applied to the geodesic segment $g|_{I^{+}}$ implies that: If 
$$d_{\gamma}(\mu(g(t_{0}^{+})),\mu(g(t_{0})))>N_{0},$$
 then 
$$\inf_{t \in I^{+}} \ell_{\gamma}(g(t)) \leq \bar{\epsilon}.$$
 
 \item\label{ch : ep2} Let $\gamma\notin \partial{Z}$. By the lower bound (2), $\ell_{\gamma}(g(t_{0}))\geq\bar{\epsilon}$ and by (d), $\ell_{\gamma}(x)\geq 2\inj(x)$. Moreover by (b) the length of $u$ is at least $d_{0}$. Let $e=\min\{\inj (x), \frac{\bar{\epsilon}}{2}\}$. Then by Corollary \ref{cor : gradlfestimate} there is an $s_{0}$ with $2s_{0}<d_{0}<d$ so that for any $r\in[0,s_{0}]$ and any $r\in [d-s_{0},d]$ we have
 $$\ell_{\gamma}(u(r))\geq e.$$  
 By what we just said $\sup_{r \in [0,d]} \ell_{\gamma}(u(r)) \geq e$. Then Corollary \ref{cor : shtw} applied to $u$ implies that there is an $\epsilon_{2}<\min\{e,L_{S}\}$ depending on $d,s_{0}$ and $e$ and the lower bound in (\ref{eq : dgammalargeu}) such that: If 
 \begin{equation} \label{eq : infgamma}\inf_{r \in [s_{0},d-s_{0}]}\ell_{\gamma}(u(r))\leq\epsilon_{2},\end{equation}
 then
 \begin{equation} \label{eq : dgammalargeu}d_{\gamma}( \mu_{x}, \mu(g(t_{0})))>N_{0}+R'+2M_{2}+B+4.\end{equation}
\end{enumerate}

In what follows we prove that the lower bound
 $$\inf_{t\in [0,d]}\ell_{\gamma}(u(t))>\epsilon_{2},$$
 for every $\gamma\notin\partial{Z}$, where $\epsilon_{2}$ is the constant we chose in (\ref{ch : ep2}) holds. Indeed, we rule out the possibility that $\inf_{r \in [0,d]}\ell_{\gamma}(u(t)) \leq \epsilon_{2}$ for some $\gamma \notin \partial{Z}$. 

The proof is by contradiction. First note that $\mu_{x}$ is a Bers marking of $u(0)=x$ and $\mu(g(t_{0}))$ is a Bers marking of $u(d)=g(t_{0})$. Moreover, by (\ref{eq : infgamma}) there is a point $y \in u([s_{0},d-s_{0}])$ such that $\gamma$ is in $Q_{y}$ a Bers pants decomposition of $y$. 
  
Let $D'=K_{\WP}d_{1}+C_{\WP}$ and $w={\bf w}(D,D',R)$ be as before. Since the length of $u$ is less than $d_{1}$ by Theorem \ref{thm : brockqisom} we have $d(Q_{y},\rho(i_{0}))\leq D'$. Then by Definition \ref{def : isolatedann}, $\gamma$ is $(w,D',R)-$isolated at $i_{0}$.

 Moreover $\rho$ and $Q(g)$, $D-$fellow travel. Then as we saw in the proof of Lemma \ref{lem : anncoeffcomp} there is a constant $B=B(D)$ such that
\begin{equation}\label{eq : dgammabdghr}d_{\gamma}(\rho(i_{0}-w),\mu(g(t_{0}^{-}))\leq B. \end{equation} 
By (c), $\base(\mu_{x})=\rho(i_{0})$. Then (\ref{eq : nbtrack}) and the assumption of the lemma that the annular coefficients are bonded over $[m',n']$ imply that
\begin{equation}\label{eq : dgammabdhr}d_{\gamma}(\mu_{x},\rho(i_{0}-w)) \leq R'+2M_{2}.\end{equation}
The inequalities (\ref{eq : dgammalargeu}), (\ref{eq : dgammabdghr}) and (\ref{eq : dgammabdhr}) combined by the triangle inequality give us
   $$d_{\gamma}(\mu(g(t_{0})), \mu(g(t_{0}^{-})))> N_{0}.$$
So the choice of $N_{0}$ in (\ref{ch : N0}) implies that $\inf_{t \in [t_{0}^{-},t_{0}]} \ell_{\gamma}(g(t)) \leq \bar{\epsilon}$. Then since $[t_{0}^{-},t_{0}]\subset[t_{0}^{-},t_{0}^{+}]$, $\inf_{t \in [t_{0}^{-},t_{0}^{+}]}\ell_{\gamma}(g(t)) \leq \bar{\epsilon}$. Then by (\ref{eq : dgammaglb}) we have
 $$d_{\gamma}(\mu(g(t_{0}^{-})),\mu(g(t_{0}^{+})))>R'+B+2M_{2}.$$
Then again since $\gamma$ is $(w,D',R)-$isolated at $i_{0}$ by Lemma \ref{lem : anncoeffcomp},
$$d_{\gamma}(\rho(i_{0}-w),\rho(i_{0}+w))>R'+2M_{2}.$$
 Then since $m'<i_{0}-w<i_{0}+w<n'$ the no backtracking property of hierarchy paths (\ref{eq : nbtrack}) and the above inequality imply that 
$$d_{\gamma}(\rho(m'),\rho(n'))>R'.$$
 This lower bound contradicts the assumption of the lemma that the $\gamma$ annular coefficient over $[m',n']$ is bounded above by $R'$. 
 
 \begin{figure}
\centering
\scalebox{0.2}{\includegraphics{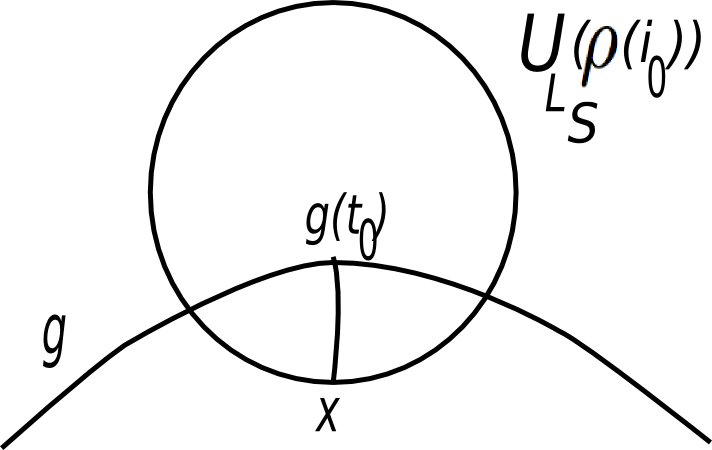}}
\caption{The closed set $U_{L_{S}}(\rho(i_{0}))$ is the Bers region of $\rho(i_{0})$ in $\Teich(S)$. The geodesic $g$ intersects this region. By (b) there is a point $x\in U_{L_{S}}(\rho(i_{0}))$ so that $d_{\WP}(x,g(t_{0}))>d_{0}$. Thus the length of the WP geodesic segment connecting $x$ and $g(t_{0})$ is at least $d_{0}$. Moreover we have the lower bounds for the length of any $\gamma\notin\partial{Z}$ at $g(t_{0})$ from part (2) of lemma and at $x$ from (d) independent of $t_{0}$. Thus we may apply Corollary \ref{cor : gradlfestimate} to the geodesic segment $u$ to choose $s_{0}$ with $2s_{0}<d_{0}$ in (ii) independent of $t_{0}$ (uniformly along $g$).}
\label{fig : choicex}
\end{figure}

 Note that $\epsilon_{2}$ depends only on $R,R'$ and $D$. In particular it is uniform along $g$.
 \medskip

We will finish establishing of the upper bound (2) by a compactness argument. 

Let $u_{n}:[0,d_{n}]\to \Teich(S)$ be a sequence of WP geodesic segments parametrized by arc-length with $d_{0}\leq d\leq d_{1}$.  Suppose that subsurfaces $Z_{n}$ are such that the length of any curve in $\partial{Z}_{n}$ at $u_{n}(0)$ is $L_{S}$ and $\inf_{t\in [0,d_{n}]}\ell_{\gamma}(u_{n}(t))\geq\epsilon_{2}$ for every $\gamma\notin\partial{Z}_{n}$. Moreover suppose that there are times $t^{*}_{n}\in [0,d_{n}]$ and curves $\hat{\alpha}_{n}\in\partial{Z}_{n}$ such that $\ell_{\hat{\alpha}_{n}}(u_{n}(t^{*}_{n}))\to \infty$ as $n\to \infty$. After possibly passing to a subsequence and remarking (applying elements of the mapping class group) we may assume that there is a subsurface $Z$ such that  the length of any curve in $\partial{Z}$ at $u_{n}(0)$ is $L_{S}$ and $\inf_{t\in [0,d_{n}]}\ell_{\gamma}(u_{n}(t))\geq\epsilon_{2}$  for every $\gamma\notin\partial{Z}$. Moreover, there is a curve $\hat{\alpha}\in\partial{Z}$ such that $\ell_{\hat{\alpha}}(u_{n}(t^{*}_{n}))\to\infty$ as $n\to \infty$.

 We proceed to get a contradiction. Let $\hat{u}:[0,d]\to \overline{\Teich(S)}$ be the geodesic limit of geodesics $u_{n}$ as in Theorem \ref{thm : geodlimit}(Geodesic Limit Theorem). Let the partition $0=t_{0}<t_{1}<...<t_{k+1}=T$, multi-curves $\sigma_{0},...,\sigma_{k+1}$, the multi-curve $\hat{\tau}$ and the elements of $\Mod(S)$, $\psi_{n}$, $\mathcal{T}_{i,n}$ and $\varphi_{i,n}=\mathcal{T}_{i,n}\circ...\circ\mathcal{T}_{i,n}\circ \psi_{n}$ be as in the theorem. We claim that
\begin{itemize}
\item $\ell_{\alpha}(\psi_{n}(x))=\ell_{\alpha}(x)$ for every $\alpha \in \partial{Z}$,
\item $\sigma_{i} \subseteq \partial{Z}$ for $i=1,...,k+1$.
 \end{itemize}
 In the proof of the Geodesic Limit Theorem $\psi_{n}$ is applied to keep $u_{n}(0)$ in a compact subset of the Teichm\"{u}ller space. Now since the length of any curve in $\partial{Z}$ at $u_{n}(0)$ is $L_{S}$ we may choose each $\psi_{n}$ to be supported on $S\backslash\partial{Z}$ i.e. $\psi_{n}$ is the identity map in a regular neighborhood of every curve in $\partial{Z}$. We have the first bullet. We proceed to prove the second bullet by an induction. In the Geodesic Limit Theorem for each $i=1,...,k+1$, $\sigma_{i}$ is the multi-curve determining the stratum that the limit of $\varphi_{i-1,n}(u_{n}|_{[t_{i-1},d]})$ after possibly passing to a subsequence intersects. Then the lower bound on the length of each $\gamma \notin \partial{Z}$ along $u_{n}([0,d])$ and the fact that $\psi_{n}$ does not change the isotopy class of any curve in $\partial{Z}$ ($\psi_{n}$ is supported on $S\backslash\partial{Z}$) the limit of $\psi_{n}(u_{n}|_{[0,d]})$ intersects the stratum of a multi-curve $\sigma_{1}\subseteq \partial{Z}$. Now let $i> 1$. Suppose that $\sigma_{j}\subseteq \partial{Z}$ for every $0\leq j\leq i$. Then since $\mathcal{T}_{j,n}\in \tw(\sigma_{j}-\hat{\tau})$ for each $j=1,...,i$, the composition of $\mathcal{T}_{j,n}$'s does not change the isotopy class of any curve in $\partial{Z}$. Thus $\varphi_{i,n}$ does not change the isotopy class of any curve in $\partial{Z}$. Moreover the length of each $\gamma\notin\partial{Z}$ along $u_{n}([t_{i},d])$ is bounded below. Thus we may conclude that the limit of $\varphi_{i,n}(u_{n}|_{[t_{i},d]})$ as $n\to \infty$ intersects the stratum of a multi-curve $\sigma_{i+1}\subseteq\partial{Z}$. 

The above claim and the Geodesic Limit Theorem guarantee that any curve whose length is $0$ at some time along $\hat{u}$ is a boundary curve of $Z$. The curve $\hat{\alpha}$ does not intersect $\partial{Z}$. So the length of $\hat{\alpha}$ along $\hat{u}$ is bounded above  by some $l_{0}>0$. By the second bullet above, each $\varphi_{i,n}$ is the composition of $\psi_{n}$ and powers of Dehn twists about some of the curves in $\partial{Z}$. So the isotopy class of each curve in $\partial{Z}$ is preserved by $\varphi_{i,n}$. After possibly passing to a subsequence we may assume that $t^{*}_{n}\in [t_{i},t_{i+1}]$ for all $n$. Then by the third part of the Geodesic Limit Theorem, $\varphi_{i,n}(u_{n}(t^{*}_{n}))\to \hat{u}(t^{*})$ as $n\to \infty$ for some $t^{*}\in[t_{i},t_{i+1}]$. So for all $n$ sufficiently large $\ell_{\hat{\alpha}}(\varphi_{i,n}(u_{n}(t^{*}_{n})))\leq 2l_{0}$. But $\ell_{\hat{\alpha}}(u_{n}(t^{*}_{n}))=\ell_{\hat{\alpha}}(\varphi_{i,n}(u_{n}(t^{*}_{n})))$ ($\hat{\alpha}\in \partial{Z}$ and $\varphi_{i,n}$ preserves the isotopy class of each curve in $\partial{Z}$). Thus $\ell_{\hat{\alpha}}(u_{n}(t^{*}_{n}))\leq 2l_{0}$. This contradicts the assumption that $\ell_{\hat{\alpha}}(u_{n}(t^{*}_{n}))\to\infty$ as $n\to \infty$. Therefore the upper bound stated in part (2) of the lemma exists. Moreover the upper depends only on $\epsilon_{2}$ and therefore $R,R'$ and $D$. 

\end{proof}
\begin{prop}\label{prop : rec} 
Given $A,R,R'>0$, there are constants $w=w(A,R)$ and $\bar{\epsilon}=\bar{\epsilon}(A,R,R')$ with the following properties. Let $g:[a,b] \to \overline{\Teich(S)}$ be a WP geodesic segment parametrized by arc-length with $A-$narrow end invariant $(\nu^{-},\nu^{+})$. Let $\rho:[m,n]\to P(S)$ be a hierarchy path between $\nu^{-}$ and $\nu^{+}$. Suppose that $S$ has $(R,R')-$bounded combinatorics over $[m',n']\subseteq  [m,n]$ and $m'-n'>2w$, then
$$ \inj(g(t)) \geq\frac{\bar{\epsilon}}{2} $$
for every $t \in [a',b']$, where $a'\in N(m'+w)$ and $b'\in N(n'-w)$.
\end{prop}

 \begin{proof}
 Let $D=D(A)$ be the fellow traveling distance of $g$ and the hierarchy path between $\nu^{-}$ and $\nu^{+}$ from Theorem \ref{thm : fellowtr}. By the assumption that $S$ is the subsurface with $(R,R')-$bounded combinatorics Lemma \ref{lem : roughub}(2) implies that for every $\gamma \in \mathcal{C}(S)$ we have $\ell_{\gamma}(g(t))\geq \bar{\epsilon}$ for every $t \in [a',b']$. So $\inj(g(t)) \geq \frac{\bar{\epsilon}}{2}$ on this interval.
\end{proof}

\begin{remark}
 Compare the above corollary with the main result of \cite{bmm2} which asserts that given $R>0$ there is an $\epsilon>0$ such that if the end invariant of a WP geodesic $g$ is $(R,R)-$bounded combinatorics then $g$ stays in the $\epsilon-$think part of the Teichm\"{u}ller space.
 \end{remark}

We are ready to prove the main result of this section. 
\begin{thm}\textnormal{ (Short Curve)} \label{thm : bdryzsh}
Given $A,R,R'>0$ and a sufficiently small $\epsilon>0$, there is a constant $\bar{w}=\bar{w}(A, R, R',\epsilon)$ with the following property. Let $g:[a,b] \to \Teich(S)$ be a WP geodesic segment with $A-$narrow end invariant $(\nu^{-},\nu^{+})$. Let $\rho:[m,n] \to P(S)$ be a hierarchy path between $\nu^{-}$ and $\nu^{+}$. Suppose that $Z$ a large component domain of $\rho$ has $(R,R')-$bounded combinatorics over $[m',n'] \subset J_{Z}$. 

If $m'-n' \geq 2\bar{w}$, then for every $\alpha \in \partial{Z}$ we have
$$\ell_{\alpha}(g(t)) \leq  \epsilon,$$
for every $t \in [a',b']$, where $a'\in N(m'+\bar{w})$ and $b' \in N(n'-\bar{w})$. Here $N=N_{\rho,g}$ is the parameter map from Proposition \ref{prop : parmap}.
\end{thm}
 \begin{proof}To get arbitrary short boundary curves we sharpen the rough upper bound obtained in Lemma \ref{lem : roughub}. This is done in the following lemma. 
  \begin{lem}\label{lem : bdryzsh}\textnormal{ (Sharpening the rough upper bound)}
 Given $l,\bar{\epsilon}>0$ and $0<\epsilon\leq \bar{\epsilon}$, there is a constant $\bar{s}=\bar{s}(l,\bar{\epsilon},\epsilon)>0$ with the following property. Let $g:[a',b']\to \overline{\Teich(S)}$ be a WP geodesic such that the length-function bounds
\begin{enumerate}
\item $\ell_{\alpha}(g(t))\leq l$ for every $\alpha \in \partial{Z}$, and
\item $\ell_{\gamma}(g(t))\geq\bar{\epsilon}$ for every $\gamma \notin \partial{Z}$
\end{enumerate}
hold for every $t\in [a',b']$. Furthermore, suppose that $b'-a'>2\bar{s}$. Then for every $\alpha \in \partial{Z}$ we have that
\begin{equation}\label{eq : bdryzsh}\ell_{\alpha}(g(t))\leq \epsilon\end{equation} 
for every $t \in [a'+\bar{s},b'-\bar{s}]$. 
\end{lem}
Having Lemma \ref{lem : bdryzsh} the proof of the theorem is completed as follows. Let $D=D(A)$ be the fellow traveling distance from Theorem \ref{thm : fellowtr}. Let $e\in N(m')$ and $f\in N(n')$. Then $Q(g|_{[e,f]})$ and $\rho|_{[m',n']}$, $D-$fellow travel.  Moreover, the subsurface $Z$ has $(R,R')-$bounded combinatorics over $[m',n']$. Then by Lemma \ref{lem : roughub} there are constants $l,\bar{\epsilon}$ and $w$ depending on $R,R'$ and $D$ with the property that if $m'-n'\geq 2w$, then the length-function bounds (1) and (2) of Lemma \ref{lem : bdryzsh} hold for every $t\in[a',b']$, where $a'\in N(m'+w)$ and $b'\in N(n'-w)$. Let the constants $l,\bar{\epsilon}$ be from Lemma \ref{lem : roughub} and $\epsilon$ be as in the statement of the theorem. Then let $\bar{s}$ be the corresponding constant from Lemma \ref{lem : bdryzsh}. Let $K,C$ be the constants for the quasi-isometry $N$ from Proposition \ref{prop : parmap}. Now let $\bar{w}=K\bar{s}+KC+w$. Then $m'-n'>2\bar{w}$ guarantees that $b'-a'>2\bar{s}$. Then Lemma \ref{lem : bdryzsh} gives us the asserted length-function bound of the theorem. 
\end{proof}
\begin{proof}[Proof of Lemma \ref{lem : bdryzsh}]
The proof is by induction on $|\partial{Z}|$ the number of boundary components of $Z$. Suppose that for a subsurface the bounds (1) and (2) on length-functions hold. In what follows we use convexity of length-functions along WP geodesics to get an arbitrary short curve in the boundary of the subsurface $Z$ over an arbitrary long interval. Then using the fact that over this interval the geodesic is close to the stratum of the short curve, we project this portion of the geodesic to the stratum. Then we are in the same set up in the Teichm\"{u}ller space of a surface with lower complexity where a subsurface with fewer number of boundary curves is considered.  

To begin note that: If $|\partial{Z}|=0$ ($\partial{Z}=\emptyset$), then the lemma holds vacuously and provides us with the base of the induction. 

\begin{claim}\label{claim : ashc} There is a constant $T=T(l,\epsilon)>0$ such that if $[c,d]\subseteq [a',b']$ is a subinterval with $d-c>2T$ then there is a curve $\alpha \in \partial{Z}$ such that 
$$\ell_{\alpha}(g(t))\leq \epsilon$$ 
for some $t \in [c+T,d-T]$.
\end{claim}
First we show that if $g|_{[c,d]}$ stays in the $\epsilon-$thick part of the Teichm\"{u}ller space then there is an upper bound for the length of the interval $[c,d]$. Suppose that the geodesic segment $g|_{[c,d]}$ is in the $\epsilon-$thick part. Fix $\hat{\alpha} \in \partial{Z}$. Then by Theorem \ref{thm : convlf}(\ref{eq : convlf}) there is $c(\epsilon)>0$ for which the differential inequality 
$$\ddot{\ell}_{\hat{\alpha}}(g(t)) \geq \epsilon c(\epsilon)$$ 
holds on the interval $[c,d]$. 

Suppose that $d-c>\frac{l}{\epsilon \sqrt{2c(\epsilon)}}$. By the upper bound $(1)$ in the lemma, $\ell_{\hat{\alpha}}(g(t))\leq l$ for every $t \in [c,d]$. Moreover, $\ell_{\hat{\alpha}}(g(t))\geq 0$ for any $t\in[c,d]$. Thus by the Mean-value Theorem there is $t^{*}\in [c,d]$ such that $|\dot{\ell}_{\hat{\alpha}}(g(t^{*}))| <\epsilon\sqrt{2c(\epsilon)}$. Let $t\in[c,d]$. Integrating the above differential inequality from $t^{*}$ to $t$ we get
\begin{equation}\label{eq : lquadf} \ell_{\hat{\alpha}}(g(t))\geq \ell_{\hat{\alpha}}(g(t^{*}))+\dot{\ell}_{\hat{\alpha}}(g(t^{*}))(t-t^{*})+\frac{1}{2}\epsilon c(\epsilon)(t-t^{*})^{2}.\end{equation}
We have $\ell_{\hat{\alpha}}(g(t^{*}))>\epsilon$ and $|\dot{\ell}_{\hat{\alpha}}(g(t^{*}))|<\epsilon\sqrt{2c(\epsilon)}$. Then computing $\Delta$ the discriminant of the quadratic function on the right hand side of (\ref{eq : lquadf}) we have 
$$\Delta=(\dot{\ell}_{\hat{\alpha}}(g(t^{*})))^{2}-4(\frac{\epsilon c(\epsilon)}{2})\ell_{\hat{\alpha}}(g(t^{*}))\leq (\epsilon\sqrt{2c(\epsilon)})^{2}-4(\frac{1}{2}\epsilon^{2}c(\epsilon))= 0.$$ 
This guarantees that the quadratic function is positive on $\mathbb{R}$. 

As before $\ell_{\hat{\alpha}}(g(t))<l$ for every $t\in[c,d]$. Then by completing the square we get $l\geq (\sqrt{\frac{\epsilon c(\epsilon)}{2}}(t-t^{*})+\frac{\dot{\ell}_{\hat{\alpha}}(g(t^{*})}{\sqrt{2\epsilon c(\epsilon)}})^{2}-\frac{\Delta}{2\epsilon c(\epsilon)}$. Then using the inequalities $|\dot{\ell}_{\hat{\alpha}}(g(t^{*}))|<\epsilon\sqrt{2c(\epsilon)}$ and $\Delta\leq 0$ we may conclude that 
$$|t-t^{*}|\leq \frac{\sqrt{2}}{\sqrt{c(\epsilon)}}+\frac{\sqrt{2l}}{\sqrt{\epsilon c(\epsilon)}}.$$
Note that for $\epsilon$ sufficiently small the right-hand side is positive. Then since $t\in[c,d]$ was arbitrary we have that $d-c\leq 2(\frac{\sqrt{2}}{\sqrt{c(\epsilon)}}+\frac{\sqrt{2l}}{\sqrt{\epsilon c(\epsilon)}})$. Now taking into account the assumption that $d-c>\frac{l}{\epsilon \sqrt{2c(\epsilon)}}$ we made above we conclude that for $T=\max\{\frac{l}{2\epsilon \sqrt{2c(\epsilon)}},\frac{\sqrt{2}}{\sqrt{c(\epsilon)}}+\frac{\sqrt{2l}}{\sqrt{\epsilon c(\epsilon)}}\}$ the upper bound $d-c\leq2T$ holds.

The contrapositive of what we just proved is that if $d-c>2T$, then there are curves which get shorter than $\epsilon$ at some time along $g|_{[c,d]}$. Moreover, $[c,d]\subseteq [a,b]$ so by the bound (2) component curves of $\partial{Z}$ are the only curves which can get shorter than $\epsilon<\bar{\epsilon}$ along $g|_{[c,d]}$. Thus we conclude that if $d-c>2T$, then there is a time $t \in [c,d]$ and a curve $\alpha \in \partial{Z}$ such that $\ell_{\alpha}(g(t)) \leq \epsilon$ as was desired. 
\begin{remark} 
For $\epsilon$ sufficiently small, $T=\frac{l}{2\epsilon \sqrt{2c(\epsilon)}}$.
\end{remark}
 
\begin{claim} \label{claim : ashcl}
Given $L>0$. If $b'-a'>(2|\partial{Z}|+1)L+2T$ then there is a curve $\alpha\in\partial{Z}$ such that $\ell_{\alpha}(g(t))\leq\epsilon$ on a subinterval of $[a',b']$ of length at least $L$.
\end{claim}
\begin{figure}
\centering
\scalebox{0.2}{\includegraphics{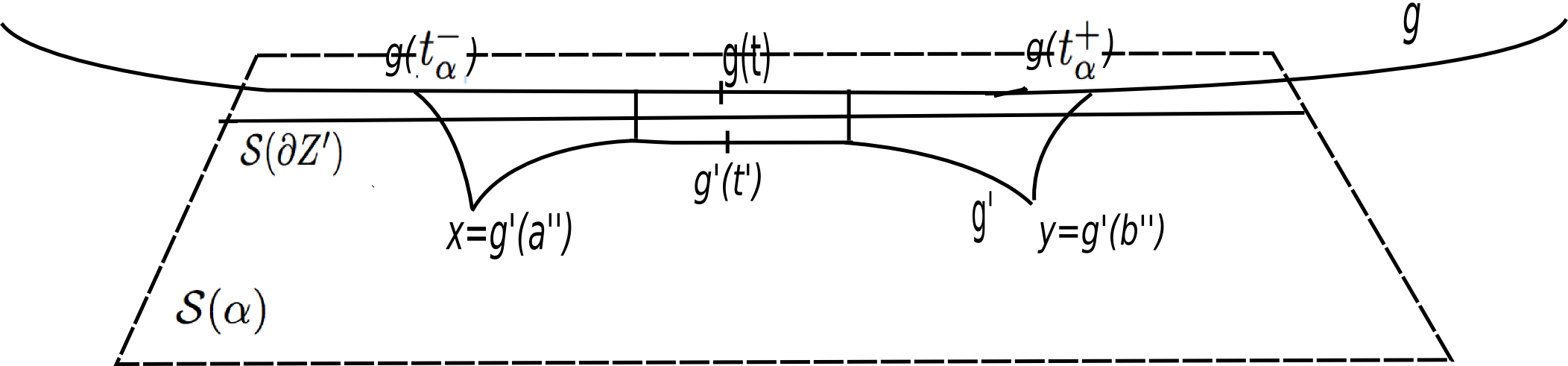}}
\caption{ The curve $\alpha \in \partial{Z}$ is shorter than $\epsilon'$ along $g|_{[t_{\alpha}^{-},t_{\alpha}^{+}]}$. The geodesic segment $g':[a'',b'']\to \mathcal{S}(\alpha)$ connects $x$ and $y$, the nearest points to $g(t_{\alpha}^{-})$ and $g(t_{\alpha}^{-})$ on $\mathcal{S}(\alpha)$, respectively. The length-function bounds ($1'$) and ($2'$) hold for the subsurface $Z'=Z\cup A(\alpha)$ along $g'$. Then by the assumption of the induction the length of every $\alpha' \in \partial{Z'}$ is shorter than $\frac{\epsilon}{2}$ along $g'$ over a suitably shrunk subinterval of $[a'',b'']$.}  
\label{fig : projstratum}
\end{figure}

By the assumption that $b'-a'>(2|\partial{Z}|+1)L+2T$ there is a partition of the interval $[a'+T, b'-T]$ into $2|\partial{Z}|+1$ subintervals $I_{1},...,I_{2|\partial{Z}|+1}$ where $|I_{i}|\geq L$ for $i=1,...,2|\partial{Z}|+1$. For each $i$ let $I_{i}=[r_{i},s_{i}]$. Claim \ref{claim : ashc} applied to each interval $[r_{i}-T,s_{i}+T]$ implies that there is a time $t_{i}\in I_{i}$ at which a component curve of $\partial{Z}$ is shorter than $\epsilon$. Now the pigeon-hole principle implies that there is a curve $\alpha \in \partial{Z}$ and indices $i_{1}, i_{2}$ and $i_{3}$ with $i_{1}< i_{2}< i_{3}$ such that $\ell_{\alpha}(g(t_{i_{1}})),\ell_{\alpha}(g(t_{i_{2}}))$ and $\ell_{\alpha}(g(t_{i_{3}}))$ are less than $\epsilon$. Then by the convexity of the $\alpha-$length-function along $g$, $\ell_{\alpha}(g(t))\leq \epsilon$ on the interval $[t_{i_{1}},t_{i_{3}}]$. Moreover $t_{i_{3}}-t_{i_{1}}\geq |I_{i_{2}}|\geq L$. So $[t_{i_{1}},t_{i_{3}}]$ is the claimed subinterval.
\medskip

Let $\epsilon'<\min\{\bar{\epsilon}, \epsilon\}$, which will be determined. Let $L=\bar{s}'+2\sqrt{2\pi\epsilon'}$, where $\bar{s}'$ will be determined.
 
Let $T=T(\epsilon',l)$ be the constant from Claim \ref{claim : ashc}. Then for $L$ as above by Claim \ref{claim : ashcl} if $b'-a'>(2|\partial{Z}|+1)L+2T$, then there is a curve $\alpha\in \partial{Z}$ such that 
\begin{equation}\label{eq : ag}\ell_{\alpha}(g(t))\leq \epsilon'\end{equation}
on an interval of length at least $L$. Denote this interval by $[t_{\alpha}^{-},t_{\alpha}^{+}]$. Let  $x$ and $y$ be the nearest points to $g(t_{\alpha}^{-})$ and $g(t_{\alpha}^{+})$ on the $\alpha-$stratum, respectively. Since the $\alpha-$stratum is geodesically convex, there is a WP geodesic segment $g':[a'',b''] \to \mathcal{S}(\alpha)$ parametrized by arc-length connecting $x$ to $y$ (see Figure \ref{fig : projstratum}).

Since $\ell_{\alpha}(g(t_{\alpha}^{-}))\leq \epsilon'$ and $\ell_{\alpha}(g(t_{\alpha}^{+})) \leq \epsilon'$, by Proposition \ref{prop : diststr} we get the upper bounds
$d_{\WP}(g(t_{\alpha}^{-}),x) \leq \sqrt{2 \pi \epsilon'}$ and $d_{\WP}(g(t_{\alpha}^{+}),y) \leq \sqrt{2 \pi \epsilon'}$, respectively. Moreover, $\overline{\Teich(S)}$ equipped with the WP metric is a $\CAT(0)$ space. Therefore by convexity of the distance function between two geodesics in a $\CAT(0)$ space (see e.g. Proposition 2.2 in Chapter II of \cite{bhnpc}) the distance between any point on $g([t_{\alpha}^{-}, t_{\alpha}^{+}])$ and its nearest point on $g'$ is less than $\sqrt{2 \pi \epsilon'}$. Here we choose $\epsilon'$ so that 
$$\sqrt{2\pi \epsilon'} \leq \min\{{\bf d}(\bar{\epsilon},\frac{\bar{\epsilon}}{2}),{\bf d}(2l,l),{\bf d}(\epsilon,\frac{\epsilon}{2})\},$$
 where ${\bf d}$ is the function from Corollary \ref{cor : gradlfestimate}.

 By the choice of $\epsilon'$ we have the following length-function bounds 
\begin{enumerate}[($1'$)]
\item $\ell_{\alpha}(g'(t)) \leq 2l$ for every $\alpha' \in \partial{Z'}$, and
\item $\ell_{\gamma}(g'(t))\geq\frac{\bar{\epsilon}}{2}$ for every  $\gamma \not\in \partial{Z'}$
\end{enumerate}
for every $t\in[a'',b'']$. Here $Z'=Z\cup A(\alpha)$.

Now $Z'$ is a large subsurface and $|\partial{Z}'|=|\partial{Z}|-1$. So by the above two length-function bounds, the assumption of the induction for the geodesic $g':[a'',b'']\to \overline{\Teich(S)}$ implies that there is $\bar{s}'>0$ such that if $b''-a''>2\bar{s}'$ then for every $\alpha' \in \partial{Z'}=\partial{Z}-\{\alpha\}$ we have
\begin{equation}\label{eq : a'g'}\ell_{\alpha'}(g'(t)) \leq \frac{\epsilon}{2}\end{equation} 
for every $t \in [a''+\bar{s}',b''-\bar{s}']$.  

Let $\bar{s}=\bar{s}'+2\sqrt{2\pi\epsilon}$. Given $t\in [t_{\alpha}^{-}+\bar{s},t_{\alpha}^{+}-\bar{s}]$, let $t'\in [a'',b'']$ be the nearest point to $g(t)$ on $g'|_{[a'',b'']}$. Since $d_{\WP}(g(t_{\alpha}^{-}),g'(a''))\leq \sqrt{2\pi\epsilon}$ and $d_{\WP}(g(t_{\alpha}^{+}),g'(b''))\leq \sqrt{2\pi\epsilon}$ the convexity of the distance function between two geodesics in a $\CAT(0)$ space implies that 
$$d_{\WP}(g(t),g'(t'))\leq \sqrt{2\pi\epsilon}.$$
Now by the triangle inequality 
$$d_{\WP}(g'(t'),g'(a''))\geq d_{\WP}(g(t),g(t_{\alpha}^{-}))-d_{\WP}(g(t),g'(t'))-d_{\WP}(g(t_{\alpha}^{-}),g'(a'')).$$
 Then we have that
 $$t'-a''\geq (t-t_{\alpha}^{-})-2\sqrt{2\pi\epsilon}\geq \bar{s}'.$$
  Similarly, we have 
  $$b''-t'\geq (t_{\alpha}^{+}-t)-2\sqrt{2\pi\epsilon}\geq \bar{s}'.$$
 Thus $t'\in [a''+\bar{s}',b''-\bar{s}']$. Then by the bound (\ref{eq : a'g'}) and the choice of $\epsilon'$, Corollary \ref{cor : gradlfestimate} implies that for every curve $\alpha' \in \partial{Z'}$, 
 $$\ell_{\alpha'}(g(t)) \leq \epsilon$$
  for every $t\in [a'+\bar{s},b'-\bar{s}]$. Moreover  $[a'+\bar{s},b'-\bar{s}] \subset [t_{\alpha}^{-},t_{\alpha}^{+}]$, so by (\ref{eq : ag}),
  $$\ell_{\alpha}(g(t))\leq \epsilon$$
 for every $t\in[a'+\bar{s},b'-\bar{s}]$. 
   
   We established the bound for the length of all of the component curves of $\partial{Z}$ on the interval $[a'+\bar{s},b'-\bar{s}]$. This finishes the step of the induction.
 
\end{proof}

\section{Laminations with prescribed subsurface coefficients}\label{sec : preendinv}

Our purpose in this section is to construct pairs of partial markings or laminations on a surface $S$ with a given list of subsurface coefficients. More precisely, given a sequence of integers $\{ e_{i} \}_{i}$ we will construct a pair of laminations or markings $(\mu_{I},\mu_{T})$ such that  there is a list of large subsurfaces $\{Z_{i}\}_{i}$ with 
$$d_{Z_{i}}(\mu_{I},\mu_{T})\asymp_{K_{1},C_{1}}|e_{i}|,$$
 where the constants $K_{1}$ and $C_{1}$ depend on certain initial choices. Moreover there are constants ${\bf m}$ and ${\bf m}'$ depending on the initial choices such that the subsurface coefficient of any proper non-annular subsurface of $S$ which is not in the list of subsurfaces $Z_{i}$ is bounded by ${\bf m}$, and all annular subsurface coefficients are bounded by ${\bf m}'$. This is a kind of symbolic coding for laminations using subsurface coefficients which can be thought of as continued fraction expansions with a specific pattern of coefficients. We will use these constructions in $\S$\ref{sec : examples} to construct examples of WP geodesics with certain behavior in the moduli space.
\medskip

The construction uses compositions of powers of (partial) pseudo-Anosov maps. A {\it partial pseudo-Anosov map} $f$ on a surface $S$ is a reducible element of $\Mod(S)$ which preserves the isotopy class of curves consisting a multi-curve $\sigma$ on $S$ and does not rearrange the connected components of $S\backslash\sigma$. Moreover, the restriction of $f$ to each of the connected components of $S\backslash\sigma$ is a pseudo-Anosov maps. We say that the partial pseudo-Anosov map $f$ is supported on $S\backslash \sigma$. 

We start with some background about the action of (partial) pseudo-Anosov maps on the curve complex of a surface and its subsurfaces also the space of projective measured laminations. 

The following proposition is a straightforward consequence of \cite[Proposition 4.6]{mm1}.

\begin{prop}\label{prop : palb}
Let $f$ be a (partial) pseudo-Anosov map supported on a subsurface $Y\subseteq S$. There is a constant $\tau_{f}>0$ such that for every $\alpha \pitchfork Y$ and every integer $e$ we have
$$d_{Y}(\alpha, f^{e}(\alpha)) \geq \tau_{f} |e|.$$
\end{prop}
 
 \begin{lem}\label{lem : limsup}
 Let $f$ be a (partial) pseudo-Anosov map supported on a subsurface $X$. There is a constant $\bar{\tau}_{f}>0$ such that for every $\alpha \in \mathcal{C}_{0}(X)$ we have
 $$\limsup_{n\to \infty} \frac{d_{X}(\alpha,f^{n}(\alpha))}{n}=\bar{\tau}_{f}.$$ 
 \end{lem}
 \begin{proof}
 Let $\alpha \in \mathcal{C}_{0}(X)$. Using the triangle inequality and the fact that $f$ is an isometry of $\mathcal{C}(X)$, for any positive integer $n$ we have that  
 $$d_{X}(\alpha,f^{n}(\alpha))\leq \sum_{i=0}^{n-1}d_{X}(f^{i}(\alpha),f^{i+1}(\alpha))\leq n d_{X}(\alpha,f(\alpha)).$$
  Therefore $\frac{d_{X}(\alpha,f^{n}(\alpha))}{n}\leq d_{X}(\alpha,f(\alpha))$. Thus $\limsup_{n\to \infty} \frac{d_{X}(\alpha,f^{n}(\alpha))}{n}$ is a finite number. 
  
Now let $\beta\in\mathcal{C}_{0}(X)$ with $\alpha \neq \beta$. By the triangle inequality and the fact that $f$ and therefore any $f^{n}$ ($n\in\mathbb{N}$) is an isometry of $\mathcal{C}(X)$ we have that
  \begin{eqnarray*}
  d_{X}(\beta,f^{n}(\beta))&\leq&d_{X}(\beta,\alpha)+d_{X}(\alpha,f^{n}(\alpha))+d_{X}(f^{n}(\alpha),f^{n}(\beta))\\
  &\leq&d_{X}(\alpha,f^{n}(\alpha))+2d_{X}(\beta,\alpha).
   \end{eqnarray*}
  Thus
  $$d_{X}(\beta,f^{n}(\beta))\leq d_{X}(\alpha,f^{n}(\alpha))+2d_{X}(\beta,\alpha).$$
   Now dividing both sides of the above inequality by $n$ and taking limsup we see that 
   $$\limsup_{n\to \infty} \frac{d_{X}(\beta,f^{n}(\beta))}{n} \leq \limsup_{n\to \infty} \frac{d_{X}(\alpha,f^{n}(\alpha))}{n}.$$
   Similarly we may show that $\limsup_{n\to \infty} \frac{d_{X}(\alpha,f^{n}(\alpha))}{n} \leq \limsup_{n\to \infty} \frac{d_{X}(\beta,f^{n}(\beta))}{n}$.
  Therefore, the limsup does not depend on the choice of $\alpha$. Finally, by Proposition \ref{prop : palb} we have that $\bar{\tau}_{f}\geq \tau_{f}>0$.
 \end{proof}
 
 We continue by reviewing some facts about the action of (partial) pseudo-Anosov maps on the space of projective measured laminations $\mathcal{PML}(S)$. We essentially follow expos\'{e} 11 of \cite{FLP} and $\S 3$ and appendix A of the book by Ivanov \cite{ivanovsubgpmcg}. Here we replace measured geodesic laminations with measured geodesic foliations used in these two references. The correspondence of measured foliations and measured geodesic lamination is explained in \cite{fol=lam}. Given a reducible element of the mapping class group by Theorem 11.7 of \cite{FLP} there is a multi-curve $\sigma=\{\delta_{j}\}_{j=1}^{m}$ and subsurfaces $\{X_{a}\}_{a=1}^{n}$ such that $S\backslash\sigma=\bigsqcup_{a} X_{a}^{o}$, where $X_{a}^{o}$ denotes the interior of the subsurface $X_{i}$, and the restriction of $f$ to each $X_{a}$ is either pseudo-Anosov or periodic. For a partial pseudo-Anosov map we assume that the restriction of the map to each $X_{a}$ is a pseudo-Anosov map. Then as in expos\'{e} 11 of \cite{FLP} for each $a=1,...,n$, there are measured laminations on $S$, $\mathcal{L}_{a}^{\pm}=(\lambda_{a}^{\pm},m_{a}^{\pm})$ the attracting and repelling measured laminations of $f|_{X_{a}}$ and real numbers $s_{a}>1$ such that for each $a$,
 \begin{itemize}
 \item $f(\lambda_{a}^{+})= \lambda_{a}^{+}$ and $f_{*}(m_{a}^{+})\geq s_{a}m_{a}^{+}$,
 \item $f(\lambda_{a}^{-})= \lambda_{a}^{-}$ and $f_{*}(m_{a}^{-})\leq s_{a}^{-1}m_{a}^{-}$. 
 \end{itemize}
 Moreover, both $\lambda_{a}^{\pm}$ are uniquely ergodic laminations on $X_{a}$. In particular the support of $\lambda_{a}^{\pm}$ is minimal filling on $X_{a}$. Also note that $\lambda_{a}^{\pm}$ contain the curves in $\partial{X}_{a}$.
 
  Let $i:\mathcal{ML}(S)\times \mathcal{ML}(S)\to \mathbb{R}$ be the intersection number defined for any pair of measured laminations (see $\S 2.7$ of \cite{ivanovsubgpmcg}). Given a complete hyperbolic metric on $S$, let $\ell:\mathcal{ML}(S)\to \mathbb{R}^{\geq 0}$ be the length-function (for the definition see \cite{bonlam}). Note that both $i$ and $\ell$ are homogeneous function of degree one in each of their variables. For example, $i(s\mathcal{L},\mathcal{L}')=si(\mathcal{L},\mathcal{L}')$.

As in Appendix $A$ of  \cite{ivanovsubgpmcg}(see also $\S 3$ of the book) let  $\Delta_{f}^{+}$ be the set of projective classes of measured geodesic laminations 
$$\{\sum_{a=1}^{n}t_{a}\mathcal{L}^{+}_{a} : t_{a}\geq 0\;\text{for}\; a=1,...,n,\;\text{and}\; \sum_{a=1}^{n}t_{a}>0\}.$$ 
Also let $\Psi_{f}^{+}$ be the set of projective classes of measured geodesic laminations 
$$\{\mathcal{L}\neq 0:i(\mathcal{L},\mathcal{L}_{a}^{+})=0\;\text{for}\;a=1,...,n\}.$$
 Note that $\Delta_{f}^{+}\subseteq \Psi_{f}^{+}$. Similarly, define the sets $\Delta_{f}^{-}$ and $\Psi_{f}^{+}$ and note that $\Delta_{f}^{-}\subseteq \Psi_{f}^{-}$.
 
 Define the functions 
 $$L^{\pm}:\mathcal{PML}(S)\to \mathbb{R}^{\geq 0}$$
  by $L^{\pm}([\mathcal{L}])=\frac{1}{\ell(\mathcal{L})}(\sum_{a=1}^{n}i(\mathcal{L},\mathcal{L}_{a}^{\pm})+\sum_{j=1}^{m}i(\mathcal{L},\delta_{j}))$. Note that $(L^{+})^{-1}(0)=\Delta^{-}$ and  $(L^{-})^{-1}(0)=\Delta^{+}$.  
 
 In \cite[\S3]{ivanovsubgpmcg} is shown that given a compact subset $K\subset \mathcal{PML}(S)\backslash \Psi^{+}$, there are constants $c_{1},d_{1},d_{1}'$ and $c_{2},d_{2},d_{2}'$, depending only on $f$ and $K$ such that for any measured lamination $\mathcal{L}$ with $[\mathcal{L}]\in K$, $\delta_{j}\in\sigma$ and $n\in\mathbb{N}$, $\frac{i(f^{n}(\mathcal{L}),\delta_{j})}{\ell(f^{n}(\mathcal{L}))}\leq \frac{d_{1}'}{c_{1}n-d_{1}}$ and $\frac{i(f^{n}(\mathcal{L}),\mathcal{L}_{a}^{+})}{\ell(f^{n}(\mathcal{L}))}\leq \frac{s_{a}^{n}d_{2}'}{c_{2}n-d_{2}}$.  Similarly, given a compact subset $K\subset \mathcal{PML}(S)\backslash \Psi^{-}$ there are constants $c_{1},d_{1},d_{1}'$ and $c_{2},d_{2},d_{2}'$, depending only on $f$ and $K$ such that for any measured lamination $\mathcal{L}$ with $[\mathcal{L}]\in K$, $\delta_{j}\in\sigma$ and $n\in \mathbb{N}$, $\frac{i(f^{n}(\mathcal{L}),\delta_{j})}{\ell(f^{n}(\mathcal{L}))}\leq \frac{d_{1}'}{c_{1}n-d_{1}}$ and $\frac{i(f^{n}(\mathcal{L}),\mathcal{L}_{a}^{-})}{\ell(f^{n}(\mathcal{L}))}\leq \frac{s_{a}^{n}d_{2}'}{c_{2}n-d_{2}}$. 
  
Therefore, for any $n\in\mathbb{N}$, $L^{+}(f^{n}([\mathcal{L}]))\leq \frac{s_{a}^{-n}d'}{cn-d}+\frac{d'}{cn-d}$. Similarly, for any $n\in\mathbb{N}$, $L^{-}(f^{-n}([\mathcal{L}]))\leq \frac{s_{a}^{n}d'}{cn-d}+\frac{d'}{cn-d}$. 
 
Using the above bounds for the functions $L^{\pm}$ and the fact that $L^{\pm}$ are continuous on $\mathcal{PML}(S)$ one may easily verify that the action of $f$ on $\mathcal{PML}(S)\backslash \Psi_{f}^{-}\cup \Psi_{f}^{+}$ has a compact fundamental domain, denoted by $K_{f}$. Furthermore, Ivanov in Theorem A.2 of the Appendix of \cite{ivanovsubgpmcg}(see also Theorem 3.5 in $\S 3$ of the book) proves that
 
 \begin{thm}\label{thm : unifconv}
 Let $U$ be an open subset and $K$ be a compact subset of $\mathcal{PLM}(S)$. If $\Delta_{f}^{+}\subset U$ and $K\subset \mathcal{PML}(S)\backslash \Psi_{f}^{-}$, then there is $N\in\mathbb{N}$ such that $f^{n}(K)\subset U$ for any integer $n\geq N$. If $\Delta_{f}^{-}\subset U$ and $K\subset \mathcal{PML}(S)\backslash \Psi_{f}^{+}$, then there is $N\in\mathbb{N}$ such that $f^{-n}(K)\subset U$ for any integer $n\geq N$. 
 \end{thm}
 We use the following lemma to obtain certain upper bounds for subsurface coefficients in Subsections \ref{subsec : pre1}, \ref{subsec : pre2}.

\begin{lem}\label{lem : paub}
Let $f$ be a (partial) pseudo-Anosov map supported on a large surface $X$. Given $\zeta,\eta\in \mathcal{C}_{0}(X)$, there is a constant $h$ depending on $f$ and $\zeta,\eta$, so that for any subsurface $W$ which is neither $X$ nor an annular subsurface with core curve a boundary curve of $X$ we have
\begin{equation}\label{eq : dWfg}d_{W}(f^{e_{1}}(\zeta),f^{e_{2}}(\eta))\leq h,\end{equation}
for any $e_{1},e_{2}\in \mathbb{Z}$ so that $f^{e_{1}}(\zeta)\pitchfork W$ and $f^{e_{2}}(\eta)\pitchfork W$. 
\end{lem}
\begin{proof} 

Suppose that $W$ is a non-annular subsurface. Since $X$ is a large subsurface either $\partial{W}\pitchfork X$ or $X\subsetneq W$. 
 
 First we establish the bound (\ref{eq : dWfg}) for non-annular subsurface $W$ so that $\partial{W}\pitchfork X$. 
  
Let $K:=K_{f}$ be the fundamental domain for the action of $f$ on $\mathcal{PML}(S)\backslash \Psi_{f}^{-}\cup \Psi_{f}^{+}$. Applying an appropriate power of $f$, let us say $e$, to $\partial{W}$ we may assume that the projective class of all of the component curves of $\partial{W}$ which overlap $X$ are in $K$. Applying $f^{-e}$ to the subsurface coefficient in (\ref{eq : dWfg}) we get the subsurface coefficient $d_{f^{-e}(W)}(f^{e_{1}-e}(\zeta),f^{e_{2}-e}(\eta))$. Note that since $\partial{W}\pitchfork X$ and $f$ is supported on $X$ we have that $f^{-e}(\partial{W})\pitchfork X$ and thus $\partial{f^{-e}(W)}\pitchfork X$. Moreover, since $f^{e_{1}}(\zeta)\pitchfork W$, $f^{e_{1}-e}(\zeta)\pitchfork f^{-e}(W)$, and since $f^{e_{2}}(\eta)\pitchfork W$, $f^{e_{2}-e}(\eta)\pitchfork f^{-e}(W)$. Thus it suffices to bound the subsurface coefficient in (\ref{eq : dWfg}) assuming that a component of $\partial{W}$ is in $K$.  

Fix a complete hyperbolic metric on $S$. Realize all curves and laminations geodesically in this metric. We claim that there is a constant $l_{1}>0$ and $N_{1}\in\mathbb{N}$ depending only on $\zeta$ and $f$ such that if for an integer $n\geq N_{1}$, $f^{n}(\zeta)\pitchfork W$ then the length of $f^{n}(\zeta)\cap W$ is bounded above by $l_{1}$. We essentially follow the compactness argument given by Minksy in \cite{mklcc}(see also \cite[Theorem 3.9]{convexccgp}). Suppose that the claim does not hold. Then there is a sequence $\{n_{r}\}_{r=1}^{\infty}$, subsurfaces $W_{r}$ so that $f^{n_{r}}(\zeta)\pitchfork W_{r}$, and arcs $a_{r}$ in $f^{n_{r}}(\zeta)\cap W_{r}$ for which the length of $a_{r}$ goes to $\infty$ as $r\to\infty$. Theorem \ref{thm : unifconv} applied to the pseudo-Anosov map $f$, the curve $\zeta$ and arbitrary open subsets $U$ so that $U\subset\mathcal{PML}(S)\backslash\Psi_{f}^{-}$ and $\Delta_{f}^{+}\subset U$ implies that the projective measured laminations $[f^{n_{r}}(\zeta)]$ converge into the subset $\Delta^{+}_{f}$. Moreover $X$ is a large subsurface, so $\Delta^{+}_{f}=[\mathcal{L}^{+}]$, where $\mathcal{L}^{+}$ is the attracting measured lamination of $f|_{X}$. Denote the support of $[\mathcal{L}^{+}]$ by $\lambda^{+}$. Then since $a_{r}\subset f^{n_{r}}(\zeta)$ and the length of $a_{r}$ goes to $\infty$, the arcs $a_{r}$ converge to a sublamination of $\lambda^{+}$. But $\lambda^{+}$ is minimal so the arcs $a_{r}$ converge to the lamination $\lambda^{+}$ itself. Let $\beta_{r}$ be a component curve of $\partial{W}_{r}$ which overlaps $X$ and consider the projective measured laminations $[\beta_{r}]\in K$. Note that $K$ is a compact subset of $\mathcal{PML}(S)\backslash \Psi_{f}^{-}\cup \Psi_{f}^{+}$, so after possibly passing to a subsequence $[\beta_{r}]$ converge to a projective measured lamination $[\mathcal{E}]$ in $K$. Denote the support of $\mathcal{E}$ by $\xi$. Furthermore, note that for each $r$, $a_{r}$ is disjoint from $\beta_{r}$. Thus $\xi$ is disjoint from $\lambda^{+}$. But then $i(\mathcal{E},\mathcal{L}^{+})=0$, and therefore $[\mathcal{E}]\in\Psi_{f}^{+}$. This contradicts the fact that $K$ is disjoint from $\Psi_{f}^{+}$. This contradiction implies that the claim holds.

A similar argument shows that there exist a constant $l_{2}>0$ and $N_{2}\in\mathbb{N}$, depending only on $\zeta$ and $f$, such that if  for an integer $n\geq N_{2}$, $f^{-n}(\zeta)\pitchfork W$, then the lenght of $f^{-n}(\zeta)\cap W$ is bounded above by $l_{2}$. 

There are finitely many curves $f^{e}(\zeta)$ so that $-N_{1}\leq e\leq N_{2}$ and $f^{e}(\zeta)\pitchfork W$, so the length of these curves is bounded above by some $l_{3}>0$. Then in particular the length of $f^{e}(\zeta)\cap W$ is bounded by $l_{3}$. 

By the bounds we established above, if for an $e\in\mathbb{Z}$, $f^{e}(\zeta)\pitchfork W$ then the length of $f^{e}(\zeta)\cap W$ is bounded above by $l=\max\{l_{i}:i=1,2,3\}$.

Similarly we can show that there is an $l'>0$ depending on $f$ and $\eta$, so that if for an $e\in\mathbb{Z}$, $f^{e}(\eta)\pitchfork W$ then the length of $f^{e}(\eta)\cap W$ is bounded above by $l'$.

Let $e_{1},e_{2}\in\mathbb{Z}$ be so that $f^{e_{1}}(\zeta)\pitchfork W$ and $f^{e_{2}}(\eta)\pitchfork W$, as we saw above $f^{e_{1}}(\zeta)\cap W$ and $f^{e_{2}}(\eta)\cap W$ have length bounded by $l$ and $l'$ respectively. This implies that the intersection number of $f^{e_{1}}(\zeta)\cap W$ and $f^{e_{2}}(\eta)\cap W$ is uniformly bounded above. Then Lemma \ref{lem : di} gives the upper bound for (\ref{eq : dWfg}).
\medskip

Now we establish the bound (\ref{eq : dWfg}) for non-annular subsurface $W$ with $X\subsetneq W$. In this situation $\partial{X}$ has at least one component curve $\beta\in\mathcal{C}_{0}(W)$. Then by the triangle inequality we have 
$$d_{W}(f^{e_{1}}(\zeta), f^{e_{2}}(\eta))\leq d_{W}(f^{e_{1}}(\zeta),\beta)+d_{W}(\beta,f^{e_{2}}(\eta)).$$
 Applying $f^{-e_{1}}$, $d_{W}(f^{e_{1}}(\zeta),\beta)=d_{f^{-e_{1}}W}(\zeta,\beta)$. Here we use the fact that since $\beta$ is a component curve of $\partial{X}$ and $f$ is supported on $X$ we have $f^{-e_{1}}(\beta)=\beta$. By Lemma \ref{lem : di}, we have $d_{f^{-e_{1}}(W)}(\zeta,\beta)\leq 2i(\zeta,\beta)+1$. So the first subsurface coefficient above is bounded by $2i(\zeta,\beta)+1$. Similarly the second subsurface coefficient above, using Lemma \ref{lem : di}, is bounded by $2i(\beta,\eta)+1$. There are finitely many $\beta\in \partial{X}$, and $\eta,\zeta$ are fixed. Thus the intersection numbers above are uniformly bounded. Thus we obtain the desired bound for (\ref{eq : dWfg}).
 \medskip
 
 Finally suppose that $W$ is an annular subsurface whose core curve $\beta$ is not a boundary curve of $X$. Applying an appropriate power of $f$, let us say $e$, to $\beta$ we may assume that the projective class of $\beta$ which overlaps $X$ is in $K$. Applying this power of $f$ to the subsurface coefficient in (\ref{eq : dWfg}) we obtain $d_{f^{-e}(\beta)}(f^{e_{1}-e}(\zeta),f^{e_{2}-e}(\eta))$, where $f^{-e}(\beta)$ is not a boundary curve of $X$ and the curves $f^{e_{1}-e}(\zeta)$ and $f^{e_{2}-e}(\eta)$ overlap $f^{-e}(\beta)$. Thus it suffices to bound the subsurface coefficient in (\ref{eq : dWfg}) assuming that $\beta\in K$. 

Let $\lambda,\lambda'$ be two curves or laminations that are realized geodesically in the fixed metric on $S$. Let $p\in\lambda\cap\lambda'$. Denote the smaller angle between the leaves $l\subseteq\lambda$ and $l'\subseteq\lambda$ passing through $p$ by $\vartheta_{p}$. Define the angle between $\lambda$ and $\lambda'$ to be the infimum of $\vartheta_{p}$ for all $p\in\lambda\cap\lambda'$. 

 We claim that there exist $N_{1}\in\mathbb{N}$ and a constant $\theta_{1}>0$ such that if for any integer $n\geq N_{1}$, $f^{n}(\zeta)\pitchfork W$ then the angle between the curves $f^{n}(\zeta)$ and $\beta$ is bounded from below by $\theta_{1}$. Suppose that the claim does not hold. Then there is a sequence $\{n_{r}\}_{r=1}^{\infty}$ and annular subsurfaces $W_{r}$ with core curves $\beta_{r}$ so that $f^{n_{r}}(\zeta)\pitchfork\beta_{r}$ and the angle between $\beta_{r}$ and $f^{n_{r}}(\zeta)$ goes to $0$ as $r\to\infty$. As we saw the discussion for non-annular subsurfaces Theorem \ref{thm : unifconv} implies that $[f^{n_{r}}(\zeta)]\to [\mathcal{L}^{+}]$ as $r\to\infty$ ($\mathcal{L}^{+}$ is the attracting lamination of $f|_{X}$). Furthermore, since $K$ is compact after possibly passing to a subsequence, $[\beta_{r}]\to [\mathcal{E}]$ for some $[\mathcal{E}]$ in $K$. Denote the support of $\mathcal{E}$ by $\xi$. Then since the angle between $\beta_{r}$ and $f^{n_{r}}(\zeta)$ goes to $0$, $\xi$ is a sub-lamination of $\lambda^{+}\cup \partial{X}$ ($\lambda^{+}$ is the support of $\mathcal{L}^{+}$). Therefore, $[\mathcal{E}]$ is in $\Psi_{f}^{+}$. But this contradicts the fact that $K$ is disjoint from $\Psi_{f}^{+}$. 
 
 Similarly we can show that there exist $N_{2}\in\mathbb{N}$ and $\theta_{2}>0$ so that if for an integer $n\geq N_{2}$, $f^{-n}(\zeta)\pitchfork\beta$ then the angle between $\beta$ and $f^{-n}(\zeta)$ is bounded below by $\theta_{2}$. 
 
 There are finitely many curves $f^{e}(\zeta)$ with $N_{1}\leq e\leq N_{2}$ and $f^{e}(\zeta)\pitchfork\beta$. Moreover,  the angle between any of them and $\beta$ is positive. Thus the angle between any of them and $\beta$ is bounded below by some $\theta_{3}>0$.
 
 Thus for an $e\in\mathbb{Z}$ if $f^{e}(\zeta)\pitchfork\beta$ then the angle between $f^{e}(\zeta)$ and $\beta$ is bounded below by $\theta:=\min\{\theta_{1},\theta_{2},\theta_{3}\}$.
 
 Similarly we can show that there is $\theta'>0$ depending on $f$ and $\eta$ such that for an $e\in\mathbb{Z}$ if $f^{e}(\eta)\pitchfork\beta$ then the angle between $f^{e}(\eta)$ and $\beta$ is bounded below by $\theta'$.
 
 Having the lower bounds $\theta$ and $\theta'$ for the angles, Lemma 2.6 of \cite{convexccgp} applied to the curves $f^{e_{1}}(\zeta)$ and $f^{e_{2}}(\eta)$ and the annular subsurface $W$ with core curve $\beta$ gives us an upper bound for $d_{\beta}(f^{e_{1}}(\zeta),f^{e_{2}}(\eta))$, depending only on $\theta,\theta'$ and the lower bound for the length of $\beta$. The length of $\beta$ is bounded below by twice of the injectivity radius of the hyperbolic metric which was fixed on the surface. This finishes the proof of the uniform upper bound in (\ref{eq : dWfg}) for annular subsurfaces whose core curve is not in the boundary of $X$. 
\end{proof}
\begin{prop}\label{prop : pAlam}
Let $f$ be a partial pseudo-Anosov map supported on a large subsurface $X$. Let $\mathcal{L}^{\pm}$ be the attracting/repelling measured laminations of $f$. Denote the support of $\mathcal{L}^{\pm}$ by $\lambda^{\pm}$.  Suppose that $Y\subsetneq X$ is an essential subsurface. We have the following:
\begin{itemize}
\item There are finitely many laminations containing $\lambda^{\pm}$.
\item Any leaf of $\lambda^{\pm}$ intersects $Y$ essentially. 
\item  $\pi_{Y}(\lambda^{\pm})$ consists of finitely many curves.
\end{itemize}
\end{prop}
\begin{proof}
We give the proofs of the three bullets for $\lambda^{+}$, for $\lambda^{-}$ the proofs are identical. 

Equip $S$ with a complete hyperbolic metric and realize curves and laminations geodeiscally in this metric. The complement $S\backslash\lambda^{+}$ consists of finite sided ideal polygons and crowns of the closed geodesics in the boundary of $X$ (see \cite[\S 4]{autsurf}). Therefore any lamination that contains $\lambda^{+}$ is obtained from $\lambda^{+}$ by adding some of the diagonal geodesics of the polygons and crowns which do not intersect each other. There are finitely many choices for these geodesics. Thus there are finitely many laminations that contain $\lambda^{+}$. The proof of the first bullet is complete.

Since $\lambda^{+}$ is filling $X$, there is at least one leaf $l$ of $\lambda^{+}$ which intersects $\partial{Y}\subset X$ essentially. Then in particular $l$ intersects $Y$ essentially. Let $p,q\in l$ be two consecutive intersection points of $l$ and $\partial{Y}$. Any other leaf $l'$ of $\lambda^{+}$ is dense in $\lambda^{+}$. Thus the points $p$ and $q$ are accumulation points of $l'$. Then considering the subarc of $l'$ between two points $p'$ and $q'$ on $l'$ which are arbitrary close to $p$ and $q$ respectively, one can see that $l'$ also intersects $\partial{Y}$ essentially. Thus in particular $l'$ intersects $Y$ essentially. The proof of the second bullet is complete.

Let $\mathcal{F}$ be the foliation corresponding to $\lambda^{+}$ (see \cite{fol=lam}). Note that $\mathcal{F}$ has finitely many $p-$prong singularities ($p\geq 1$) on $S$. The arcs in $\mathcal{F}\cap Y$ with end points on boundary curves of $Y$ can be divided into parallel arcs which consist finitely many rectangles inside $Y$ with two opposite sides on one or two boundary curves of $Y$. The number of rectangles is bounded above by the sum of the number of prongs at singularities of $\mathcal{F}$. Any leaf of $\lambda^{+}$ fellow travels a leaf of $\mathcal{F}$ or a concatenation of singular leaves of $\mathcal{F}$. Let $a$ be an arc in a rectangle as above with end points on boundary curves $\alpha$ and $\beta$ (a single boundary curve $\alpha$). Consider the set of curves in the boundary of a regular neighborhood of $a\cup\alpha\cup\beta$ ($a\cup\alpha$). Let $C$ be the set of all curves we obtain in this way for arcs $a$ (up to homotopy) in each one of the rectangles. There are finitely many rectangles and therefore finitely many arcs $a$ up to homotopy. Moreover in the boundary of any regular neighborhood of the union of $a$ and the boundary curves on which the end points of $a$ lie there are finitely many curves. Therefore, $C$ is a finite set. It follows from the definition that any curve in $\pi_{Y}(\lambda^{+})$ is in $C$. This finishes the proof of the third bullet.
\end{proof}
Denote the set of laminations containing $\lambda^{\pm}$ by $\Lambda^{\pm}$. By the first bullet of the Proposition \ref{prop : pAlam} $\Lambda^{\pm}$ is a finite set.
\begin{lem} \label{lem : Hnbhdl}
Given $\epsilon>0$, there is $V^{\pm}$ a neighborhood of $[\mathcal{L}^{\pm}]$ in $\mathcal{PML}(S)$, so that the support of every projective measured lamination in $V^{\pm}$ is within the $\epsilon$ Hausdorff distance of a lamination in $\Lambda^{\pm}$. 
 \end{lem}
 \begin{proof}
 Suppose that the lemma does not hold for $[\mathcal{L}^{+}]$ (The proof for $[\mathcal{L}^{-}]$ is similar). Then there is a sequence of neighborhood $V_{n}\subset \mathcal{PML}(S)$ containing $[\mathcal{L}^{+}]$, so that $\bigcap_{n=1}^{\infty} V_{n}=[\mathcal{L}^{+}]$ and for each $n\in\mathbb{N}$, there is a projective measured lamination $[\mathcal{L}_{n}]\in V_{n}$ so that the Hausdorff distance of $\lambda_{n}$ the support of $\mathcal{L}_{n}$ and all of the laminations in $\Lambda^{+}$ is at least $\epsilon$. Then $\{[\mathcal{L}_{n}]\}_{n=1}^{\infty}$ is a sequence of projective measured laminations that converges to $[\mathcal{L}^{+}]$. But any convergent subsequence of the supports $\lambda_{n}$ does not converge to a lamination containing $\lambda^{+}$. This is a contradiction to Proposition \ref{prop : hlimwk*lim}. 
 \end{proof}

 \subsection{Scheme I} \label{subsec : pre1} 
 
 The construction of this subsection will be used in $\S\ref{subsec : div}$ to construct examples of divergent WP geodesic rays and in $\S$\ref{subsec : closed} to construct examples of closed WP geodesics in the thin part of the moduli space.  
 
 Let $\alpha$ and $\beta$ be two disjoint curves on $S$ such that the subsurfaces $S\backslash \alpha$, $S\backslash \beta$ and $S\backslash \{\alpha,\beta\}$ are large subsurfaces. Consider indexed large subsurfaces $X_{0}=S\backslash\{\alpha,\beta\}$, $X_{1}=S\backslash\alpha$, $X_{2}=S\backslash\{\alpha,\beta\}$ and $X_{3}=S\backslash\beta$. Note that $X_{0}$ and $X_{2}$ are both the same subsurface $S\backslash\{\alpha,\beta\}$ with different indices. Let $f_{0},f_{1}, f_{2}$ and $f_{3}$ be partial pseudo-Anosov maps supported on $X_{0},X_{1},X_{2}$ and $X_{3}$, respectively, where $f_{0}$ and $f_{2}$ are the same partial pseudo-Anosov maps with different indices. Then in particular, $f_{a}$, $a=0,1,2,3$, preserves the homotopy class of each component of $\partial{X}_{a}$. Moreover suppose that the restriction of $f_{a}$ to a regular neighborhood of each curve in $\partial{X}_{a}$ is the identity map.

Let $q_{0}:\mathbb{N}\to \{0,1,2,3\}$ be the function $q_{0}(i)\equiv i$ (mod 4). Let $q_{1}(i)=q_{0}(i+1)$, $q_{2}(i)=q_{0}(i+2)$ and $q_{3}(i)=q_{0}(i+3)$. Let $q$ denote any of the functions $q_{0},q_{1},q_{2}$ and $q_{3}$ or the restriction of any of them to the set $\{1,...,k\}$, for some $k\in\mathbb{N}$. 

When the domain of $q$ is $\mathbb{N}$ let $\{e_{i}\}_{i=1}^{\infty}$ be an infinite sequence of integers and when the domain of $q$ is $\{1,...,k\}$ let $\{e_{i}\}_{i=1}^{k}$ be a sequence of integers with $k$ elements. For simplicity of notation some times we denote the sequence $\{e_{i}\}_{i}$ by $e$. 

\begin{notation}
Given $r,s\in \mathbb{N}$ with $s>r$, we denote the composition of powers of (partial) pseudo-Anosov maps $f_{q(j)}^{e_{j}}$, $j=r,...,s$, $f_{q(r)}^{e_{r}}\circ ...\circ f_{q(s)}^{e_{s}}$ by 
$$f_{q(r)}^{e_{r}}...f_{q(s)}^{e_{s}}.$$
\end{notation}
 For any $i$ in the domain of $q$ set the subsurface 
 \begin{equation}\label{eq : Zi}Z_{i}(q,e)=f_{q(1)}^{e_{1}}...f_{q(i-1)}^{e_{i-1}}(X_{q(i)}).\end{equation}
Let $\mu_{I}(q,e)$ be a marking so that $\base(\mu_{I})$ contains $\{\partial{X}_{a}\}_{a=0,1,2,3}=\{\alpha,\beta\}$. Throughout the following lemmas and propositions we assume that the domain of $q$ is $\{1,..,k\}$ for some $k\in\mathbb{N}$. We let $\mu_{T}(q,e)=f_{q(1)}^{e_{1}}...f_{q(k)}^{e_{k}}\mu_{I}(q,e)$ and establish several bounds for the subsurface coefficients of $\mu_{I}(q,e)$ and $\mu_{T}(q,e)$.  

When there is no ambiguity we drop the reference to $q$ and $e$. For example we denote $Z_{i}(q,e)$ by $Z_{i}$.

\begin{remark}
The construction of this subsection and the estimates for subsurface coefficients can be carried out in a more general setting. Here we restrict ourself to be able to provide detailed step by step estimates and complete arguments.
\end{remark}
 
\begin{lem} \label{lem : ldomain} There are constants $K'_{1}>0, C'_{1}\geq 0$ and $E_{1}>0$, depending only on the partial pseudo-Anosov maps $f_{0},f_{1},f_{2}$ and $f_{3}$, and $\mu_{I}$ with the following properties. Given $q$ and the sequence of integers $\{e_{i}\}_{i=1}^{k}$  such that $|e_{i}|> E_{1}$ for any $i\in\{1,...,k\}$, we have 
\begin{enumerate}[(i)]
\item\label{ldomain : ld} For any $i\in\{1,...,k\}$,
\begin{equation}\label{eq : ldomain1}d_{Z_{i}(q,e)}(\mu_{I}(q,e),\mu_{T}(q,e)) \geq K'_{1} |e_{i}|-C'_{1}\end{equation}
\item \label{ldomain : ord}Let $k\geq 3$. Let $i,j\in\{1,...,k\}$ and $j\geq i+2$. Then $Z_{i}(q,e)<Z_{j}(q,e)$ between $\mu_{I}(q,e)$ and $\mu_{T}(q,e)$.
\end{enumerate}
\end{lem}  

\begin{proof} 
Our proof modifies the proof of Theorem 5.2 in \cite{clm}. There the authors assume that any two of the subsurfaces which support partial pseudo-Anosov maps either overlap or are disjoint. But here $X_{0}(=X_{2}$) is a subsurface of $X_{1}$ and $X_{3}$. As a result their argument does not go through completely to prove the lemma and needs some modification. Furthermore, our set up is different.

{\it Proof of part part (\ref{ldomain : ld}).} The proof is by induction on $k$.  Denote $\mu_{I}$ by $\mu$. Set the constant
 $$K'_{1}=\min\{\tau_{a} : a=0,1,2,3\}.$$
 Here $\tau_{a}=\tau_{f_{a}}$, $a=0,1,2,3$, is the constant from Proposition \ref{prop : palb} for the partial pseudo-Anosov map $f_{a}$ supported on $X_{a}$. 
 
 Let 
 $$\eta=\max\{d_{X_{a}}(\mu,f_{b}^{e}\mu) : a,b\in\{0,1,2,3\},\; X_{a}\neq X_{b} \;\text{and}\;  e\in \mathbb{Z}\}.$$
  Note that since $\mu$ is fixed and $f_{b}$ is not supported on $X_{a}$, by Lemma \ref{lem : paub} the above maximum exists and is finite. Set the constant
 \begin{center}$C'_{1}=2(B_{0}+\eta+2)$.\end{center}
Here $B_{0}$ is the constant in Theorem \ref{thm : behineq}(Behrstock Inequality).  

Let
$$\omega=\max\{d_{W}(\mu,\partial{X}_{a}) : W\subseteq S \;\text{and}\; a=0,1,2,3\}.$$
 Note that since the marking $\mu$ and subsurfaces $\{X_{a}\}_{a=0,1,2,3}$ are fixed the intersection numbers $i(\mu,\partial{X}_{a})$ are uniformly bounded. Then by Lemma \ref{lem : di} subsurface coefficients $d_{W}(\mu,\partial{X}_{a})$ ($W\subseteq S$) are uniformly bounded and therefore the above maximum exists and is finite. Set the constant
$$E_{1}=\frac{B_{0}+\omega+4M+4+C_{1}'}{K'_{1}}.$$ 
By Proposition \ref{prop : palb}, 
$$d_{X_{a}}(f_{a}^{e}\mu,\mu)\geq \tau_{a}|e|\geq K'_{1}|e|,$$
for $a=0,1,2,3$. So we have the base of the induction for $k=1$. 
 
 Suppose that for any function $q:\{1,...,k'\}\to \{0,1,2,3\}$ with $k'<k$ as the beginning of this section and any sequence of integers $\{e_{i}\}_{i=1}^{k'}$ with $|e_{i}|>E_{1}$ for $i\in\{1,...,k'\}$, part (\ref{ldomain : ld}) holds. Fix $i\in\{1,...,k\}$ and let $g=f_{q(1)}^{e_{1}}...f_{q(i-1)}^{e_{i-1}}$. Applying $g^{-1}$ to $d_{Z_{i}}(\mu_{I},\mu_{T})$ we get
 $$d_{Z_{i}}(\mu_{I},\mu_{T})=d_{X_{q(i)}}(g^{-1}\mu,f_{q(i)}^{e_{i}}h\mu)$$
 where $h=f_{q(i+1)}^{e_{i+1}}...f_{q(k)}^{e_{k}}$. By the triangle inequality the right hand side is bounded below by
 \begin{equation}\label{eq : dXqilb}d_{X_{q(i)}}(h\mu,f_{q(i)}^{e_{i}}h\mu)- d_{X_{q(i)}}(g^{-1}\mu,h\mu)-2.\end{equation}
 By Proposition \ref{prop : palb} the first term of (\ref{eq : dXqilb}) is bounded below by $K'_{1}|e_{i}|$. This gives us the multiplicative constant in (\ref{eq : ldomain1}). To get the additive constant in (\ref{eq : ldomain1}) we proceed to bound the second term of (\ref{eq : dXqilb}). By the triangle inequality it is bounded above by
 \begin{equation}\label{eq : dXqig-}d_{X_{q(i)}}(g^{-1}\mu,\mu)+d_{X_{q(i)}}(\mu,h\mu)+2.\end{equation}
 First we show that $d_{X_{q(i)}}(\mu,h\mu)\leq \frac{C'_{1}}{2}+1$. Let $q'(j)=q(j+i)$ for $j=1,...,k-i$ and $e'_{j}=e_{j+i}$ for $j=1,...,k-i$. Then by the set up of the subsurfaces for $q',e'$ (see (\ref{eq : Zi})), $Z_{2}(q',e')=f_{q(i+1)}^{e_{i+1}}(X_{q(i+2)})$ so
$$ d_{f_{q(i+1)}^{e_{i+1}}(X_{q(i+2)})}(\mu,h\mu)=d_{Z_{2}(q',e')}(\mu,h\mu).$$
The assumption of the induction applied to $q'$ implies that the right hand side subsurface coefficient is greater than or equal to $K'_{1}|e_{i+1}|-C'_{1}$. So we have
 \begin{equation}\label{eq : dfqeXlb} d_{f_{q(i+1)}^{e_{i+1}}(X_{q(i+2)})}(\mu,h\mu)\geq K'_{1}|e_{i+1}|-C'_{1}. \end{equation} 
 
  We claim that 
  \begin{claim} \label{claim : bdXqi-1i+1} 
 For any $i\in\{1,...,k\}$ we have that $X_{q(i)}\pitchfork f_{q(i+1)}^{e_{i+1}}(X_{q(i+2)})$. 
  \end{claim}
  
  To see this, first suppose that $q(i)=1$ then $X_{q(i)}=S\backslash \alpha$. Furthermore, since $q(i+2)=3$, $X_{q(i+2)}=S\backslash\beta$ and since $q(i+1)=2$, $X_{q(i+1)}=S\backslash\{\alpha, \beta\}$. The map $f_{q(i+1)}$ preserves each component of $\partial{X}_{q(i+1)}=\{\alpha, \beta\}$, so $f_{q(i+1)}^{e_{i+1}}(X_{q(i+2)})=S\backslash \beta$. Then since $\alpha\neq\beta$, $X_{q(i)}\pitchfork f_{q(i)}^{e_{i}}(X_{q(i+2)})$. 
  
  If $q(i)=3$, then $X_{q(i)}=S\backslash \beta$ and the claim follows from a similar argument exchanging $\alpha$ and $\beta$. 
  
  Now suppose that $q(i)=0 (2)$, then $X_{q(i)}=S\backslash\{\alpha,\beta\}$. Furthermore, since $q(i+2)=2(0)$, $X_{q(i+2)}=S\backslash \{\alpha,\beta\}$. So $f_{q(i+1)}^{e_{i+1}}(\alpha),f_{q(i+1)}^{e_{i+1}}(\beta)\in f_{q(i+1)}^{e_{i+1}}(\partial{X}_{q(i+2)})$. Moreover, $f_{q(i+1)}$ is supported on $X_{q(i+1)}=S\backslash \alpha$ or $S\backslash \beta$.
  \begin{itemize}
  \item If $f_{q(i+1)}$ is supported on $S\backslash \beta$, then by Proposition \ref{prop : palb}, 
  $$d_{X_{q(i+1)}}(f_{q(i+1)}^{e_{i+1}}(\partial{X}_{q(i)}),\partial{X}_{q(i)})>K'_{1}|e_{i+1}|>K'_{1}E_{1}>4.$$ 
  The above inequality implies that $f_{q(i+1)}^{e_{i+1}}(\alpha)\pitchfork\alpha$ (see $\S$\ref{sec : ccplx}). Then since $f_{q(i+1)}^{e_{i+1}}(\alpha)\in f_{q(i+1)}^{e_{i+1}}(\partial{X}_{q(i+2)})$ we have $X_{q(i)}\pitchfork f_{q(i+1)}^{e_{i+1}}(X_{q(i+2)})$. 
  \item If $f_{q(i+1)}$ is supported on $S\backslash \alpha$, then similar to the above bullet we have that $f_{q(i+1)}^{e_{i+1}}(\beta)\pitchfork\beta$. Then since $f_{q(i+1)}^{e_{i+1}}(\beta)\in f_{q(i+1)}^{e_{i+1}}(\partial{X}_{q(i+2)})$ we have $X_{q(i)}\pitchfork f_{q(i+1)}^{e_{i+1}}(X_{q(i+2)})$.
  \end{itemize}
   The proof of the claim is complete.
  
  By Claim \ref{claim : bdXqi-1i+1} in hand we may write
 \begin{equation}\label{eq : di+1Xi+2o}d_{f_{q(i+1)}^{e_{i+1}}(X_{q(i+2)})}(\partial{X}_{q(i)},h\mu).\end{equation}
This subsurface coefficient by the triangle inequality is bounded below by
 \begin{equation}\label{eq : di+1Xi+2}d_{f_{q(i+1)}^{e_{i+1}}(X_{q(i+2)})}(h\mu,\mu)- d_{f_{q(i+1)}^{e_{i+1}}(X_{q(i+2)})}(\mu,\partial{X}_{q(i)})-2.\end{equation}
By (\ref{eq : dfqeXlb}) the first term of (\ref{eq : di+1Xi+2}) is greater than or equal to $K'_{1}|e_{i+1}|$. Moreover,  the second term of (\ref{eq : di+1Xi+2}) is bounded above by $\omega$. So if $|e_{i+1}|> E_{1}$ then (\ref{eq : di+1Xi+2}) is greater than $ B_{0}$, and therefore (\ref{eq : di+1Xi+2o}) is greater than $B_{0}$. By Claim \ref{claim : bdXqi-1i+1}, $X_{q(i)}\pitchfork f_{q(i+1)}^{e_{i+1}}(X_{q(i+2)})$, so Theorem \ref{thm : behineq}(Behrstock Inequality) implies that
 \begin{equation}\label{eq : dXqifqi+1bdXqi+2hm}d_{X_{q(i)}}(f_{q(i+1)}^{e_{i+1}}(\partial{X}_{q(i+2)}), h\mu)\leq B_{0}.\end{equation} 
 Now by the triangle inequality
\begin{eqnarray}\label{eq : dXimuhmu}
d_{X_{q(i)}}(\mu,h\mu)&\leq& d_{X_{q(i)}}(\mu,f_{q(i+1)}^{e_{i+1}}(\partial{X}_{q(i+2)}))+d_{X_{q(i)}}(f_{q(i+1)}^{e_{i+1}}(\partial{X}_{q(i+2)}),h\mu)+2\nonumber\\
&\leq& \eta+B_{0}+2= \frac{C'_{1}}{2}.
\end{eqnarray}
The second inequality follows from the choice of $\eta$ and the inequality (\ref{eq : dXqifqi+1bdXqi+2hm}). The third one follows from the choice of $C'_{1}$.

Let $q'(j)=q(i+1-j)$ for $j=1,...,i$ and $e'_{j}=-e_{i+1-j}$ for $j=1,...,i$. Then the same argument we gave above using $q'$ and $e'$ instead of $q$ and $e$ implies that $d_{X_{q(i)}}(g^{-1}\mu,\mu)\leq \frac{C'_{1}}{2}$. This inequality and (\ref{eq : dXimuhmu}) give us the bound $C'_{1}+2$ for (\ref{eq : dXqig-}). Using this bound for the second term of (\ref{eq : dXqilb}) we obtain the bound $C'_{1}$ for the additive constant of (\ref{eq : ldomain1}). 
\medskip

{\it Proof of part (\ref{ldomain : ord}).} The proof is by induction on $k$. The base of the induction for $k=3$ is obtained as follows. 

We would like to show that $Z_{1}<Z_{3}$. Then we need to verify that the conditions of Definition \ref{def : ordersubsurf} hold for markings $\mu_{I},\mu_{T}$ and subsurfaces $Z_{1}$ and $Z_{3}$. First note that by part (\ref{ldomain : ld}), 
$$d_{Z_{1}}(\mu_{I},\mu_{T})>K'_{1}|e_{1}|-C_{1}'>K'_{1}E_{1}-C_{1}'>4M$$
 and similarly $d_{Z_{3}}(\mu_{I},\mu_{T})>4M$. 

We proceed to show that $Z_{1}\pitchfork Z_{3}$. If $q(1)=1$, then by the set up of subsurfaces in (\ref{eq : Zi}),  $Z_{1}(q,e)=S\backslash\alpha$, $Z_{2}(q,e)=f_{1}^{e_{1}}(S\backslash\{\alpha,\beta\})$ and $Z_{3}(q,e)=f_{1}^{e_{1}}f_{2}^{e_{2}}(S\backslash\beta)$. Moreover, $f_{1}$ preserves $\alpha$ and $f_{2}$ preserves both $\alpha$ and $\beta$. So $Z_{2}(q,e)=S\backslash\{\alpha,f_{1}^{e_{1}}(\beta)\}$ and $Z_{3}(q,e)=S\backslash f_{1}^{e_{1}}(\beta)$. Then applying $f_{1}^{-e_{1}}$ to $\partial{Z_{1}}=\alpha$ and $\partial{Z_{3}}=f_{1}^{e_{1}}(\beta)$, we get $\alpha$ and $\beta$. Moreover, $S\backslash \alpha \pitchfork S\backslash \beta$, so $Z_{1}\pitchfork Z_{3}$. 

If $q(1)=3$ a similar argument implies that $Z_{1}<Z_{3}$. If $q(1)=0 (2)$, then $Z_{1}(q,e)=S\backslash \{\alpha,\beta\}$, $Z_{2}(q,e)=f_{0}^{e_{1}}(S\backslash\alpha)$ ($f_{2}^{e_{1}}(S\backslash \beta)$) and $Z_{3}(q,e)=f_{0}^{e_{1}}f_{1}^{e_{2}}(S\backslash \{\alpha,\beta\})$ $(f_{2}^{e_{1}}f_{3}^{e_{2}}(S\backslash \{\alpha,\beta\}))$. Moreover $f_{q(1)}=f_{0}(f_{2})$ preserves $\alpha$ and $\beta$, and $f_{q(2)}=f_{1}(f_{3})$ preserves $\alpha (\beta)$. So  $Z_{2}(q,e)=S\backslash \alpha$ ($S\backslash \beta$) and $Z_{3}(q,e)=S\backslash\{\alpha,f_{0}^{e_{1}}f_{1}^{e_{2}}(\beta)\}$ $(S\backslash\{f_{2}^{e_{1}}f_{3}^{e_{3}}(\alpha),\beta\})$. Apply $f_{1}^{-e_{2}}f_{0}^{-e_{1}}(f_{3}^{-e_{2}}f_{2}^{-e_{1}})$ to $\partial{Z}_{1}$ and $\partial{Z}_{3}$. The resulting multi-curves contain the curves $\beta$ and $f_{1}^{-e_{2}}(\beta)$ ($\alpha$ and $f_{3}^{-e_{3}}(\alpha)$), respectively. Moreover, by Proposition \ref{prop : palb}, 
$$d_{S\backslash\alpha}(\beta,f_{1}^{e_{1}}(\beta))>\tau_{1}|e_{1}|>\tau_{1}E_{1}>4\;(d_{S\backslash\beta}(\alpha,f_{3}^{-e_{3}}(\alpha))>\tau_{3}E_{1}>4),$$
  so $\beta$ and $f_{1}^{e_{1}}(\beta)$ overlap ($\alpha$ and $f_{3}^{-e_{3}}(\alpha)$ overlap) (see $\S$\ref{sec : ccplx}). This implies that $Z_{1}\pitchfork Z_{3}$. 

Now we may write $d_{Z_{1}}(\mu_{I},\partial{Z}_{3})$. Let $q'(i)=q(i)$ for $i=1,2$ and $e'_{i}=e_{i}$ for $i=1,2$. Then since $Z_{1}(q,e)=Z_{1}(q',e')$, $\mu_{I}(q,e)=\mu_{I}(q',e')$ and $\partial{Z}_{3}\subset\mu_{T}(q',e')$, (\ref{eq : ldomain1}) for $q',e'$ gives us 
\begin{eqnarray*}
d_{Z_{1}}(\mu_{I},\partial{Z}_{3})&\geq& d_{Z_{1}}(\mu_{I},\mu_{T}(q',e'))\\
&\geq& K'_{1}|e_{1}|-C'_{1}>K'_{1}E_{1}-C'_{1}>2M.
\end{eqnarray*}
We showed that the conditions of Definition \ref{def : ordersubsurf} are satisfied and thus $Z_{1}<Z_{3}$. This finishes establishing of the base of the induction.
\medskip

Suppose that the assertion of part (\ref{ldomain : ord}) holds for every $q:\{1,...,k'\}\to\{0,1,2,3\}$, with $k'<k$. Let $q_{init}(l)=q(l)$ for $l=1,...,k-1$ and $e_{init,l}=e_{l}$ for $l=1,...,k-1$. Then $\mu_{I}(q_{init},e_{init}) = \mu_{I}(q,e)$ and $Z_{l}(q_{init},e_{init})=Z_{l}(q,e)$ for $l=1,...,k-1$.

Let $q_{term}(l)=q(l+1)$ for $l=1,...,k-1$ and $e_{term,l}=e_{l+1}$ for $l=1,...,k-1$. 

Suppose that $i,j\in \{1,...,k\}$ and $j\geq i+2$. By (\ref{eq : ldomain1}) (part (\ref{ldomain : ld}) of the lemma), assuming that $|e_{i}|>E_{1}$ (see the choice of $E_{1}$) we have 
$$d_{Z_{i}}(\mu_{I},\mu_{T})>4M\; \text{and}\; d_{Z_{j}}(\mu_{I},\mu_{T})>4M.$$
 If $j< k$, then the assumption of the induction applied to $q_{init}$ and $e_{init}$ implies that $Z_{i}(q_{init},e_{init})<Z_{j}(q_{init},e_{init})$ between $\mu_{I}(q_{init},e_{init})$ and $\mu_{T}(q_{init},e_{init})$. Thus by Definition \ref{def : ordersubsurf} in particular we have that $Z_{i}(q_{init},e_{init})\pitchfork Z_{j}(q_{init},e_{init})$ and 
$$d_{Z_{j}(q_{init},e_{init})}(\mu_{I}(q_{init},e_{init}),\partial{Z}_{i}(q_{init},e_{init}))>2M.$$
 Then since $\mu_{I}(q_{init},e_{init}) = \mu_{I}(q,e)$, $Z_{j}(q_{init},e_{init})=Z_{j}(q,e)$ and $Z_{i}(q_{init},e_{init})=Z_{i}(q,e)$, we have that $Z_{i}(q,e)<Z_{j}(q,e)$. 

If $i>1$, then by the assumption of the induction $Z_{i-1}(q_{term},e_{term})<Z_{j-1}(q_{term},e_{term})$ between $\mu_{T}(q_{term},e_{term})$ and $\mu_{T}(q_{term},e_{term})$. Thus in particular, $Z_{i-1}(q_{term},e_{term})\pitchfork Z_{j-1}(q_{term},e_{term})$ and 
$$d_{Z_{j-1}(q_{term},e_{term})}(\mu_{T}(q_{term},e_{term}),\partial{Z_{i-1}}(q_{term},e_{term}))>2M.$$
Now applying $f_{q(1)}^{e_{1}}$ to subsurfaces $Z_{i-1}(q_{term},e_{term})$ and $Z_{j-1}(q_{term},e_{term})$ and the marking $\mu_{T}(q_{term},e_{term})$ we get the subsurfaces $Z_{i}(q,e)$ and $Z_{j}(q,e)$ and the marking $\mu_{T}(q,e)$, respectively. Thus $Z_{i}(q,e)\pitchfork Z_{j}(q,e)$ and
$$d_{Z_{j}(q,e)}(\mu_{T}(q,e),\partial{Z_{i}}(q,e))>2M.$$
Then by Definition \ref{def : ordersubsurf},  $Z_{i}(q,e)< Z_{j}(q,e)$.

Now suppose that $i=1$ and $j=k$. When $k=4$, we may proceed as the base of the induction and directly prove that $Z_{1}\pitchfork Z_{4}$ and $d_{Z_{1}}(\mu_{I}(q,e),\partial{Z}_{4})>2M$, which implies that $Z_{1}<Z_{4}$. The details are similar so we skip them. When $k\geq 5$, let $l\in \{3,...,k-2\}$. As we saw above $Z_{1}<Z_{l}$ and $Z_{l}<Z_{k}$. Then by the transitivity of the relation $<$ on subsurfaces (Proposition \ref{prop : ordersubsurf}) we conclude that $Z_{1}<Z_{k}$.

\end{proof}

The second part of the above lemma does not say anything about the order of two consecutive subsurfaces $Z_{i}$ and $Z_{i+1}$. The following lemma gives an order for times in the two consecutive intervals $J_{Z_{i}}$ and $J_{Z_{i+1}}$. See Theorem \ref{thm : hrpath}(\ref{h : J}) for the definition of $J$ intervals corresponding to component domains of a hierarchy. 

\begin{lem}\label{lem : Jnext}
Given $q$ and $\{e\}_{i=1}^{k}$. Let $i\in\{1,...,k\}$ be such that $q(i)=1$ or $3$ ($Z_{i}(q,e)$ has one boundary curve). Let $i>1$. If $j\in J_{Z_{i}(q,e)}$ then $j\geq \min J_{Z_{i-1}(q,e)}$. If $j\in J_{Z_{i-1}(q,e)}$ then $j\leq \max  J_{Z_{i+1}(q,e)}$. 
\end{lem}

\begin{proof}
By the set up of subsurfaces $Z_{i}(q,e)$ in (\ref{eq : Zi}) and the function $q$ we have $\partial{Z_{i}}=\partial{Z_{i-1}}\cap\partial{Z_{i+1}}$. 

By Lemma \ref{lem : ldomain}(\ref{ldomain : ord}), $Z_{i-1}<Z_{i+1}$, so by Definition \ref{def : ordersubsurf} in particular, $d_{Z_{i+1}}(\mu_{T},\partial{Z}_{i-1})> 2M$. Then Theorem \ref{thm : behineq} implies that $d_{Z_{i-1}}(\mu_{T},\partial{Z}_{i+1})\leq M$.  Since $j\in J_{Z_{i}}$, $\partial{Z}_{i}\subset\rho(j)$, then since $\partial{Z}_{i+1}\subset \partial{Z}_{i}$, we have that $\partial{Z}_{i+1}\subset\rho(j)$. Thus we obtain 
$$d_{Z_{i-1}}(\mu_{T},\rho(j))\leq M$$
This inequality and $d_{Z_{i-1}}(\mu_{I},\mu_{T})>4M$, combined by the triangle inequality imply that $d_{Z_{i-1}}(\mu_{I},\rho(j))> 3M-2$. Now assume that $j\leq \min J_{Z_{i-1}}$, then by Theorem \ref{thm : hrpath}(\ref{h : lrbddproj}), $d_{Z_{i-1}}(\mu_{I},\rho(j))\leq M$, which contradicts the lower bound we just proved. Thus $j\geq \min J_{Z_{i-1}}$. 

The proof of that $j\leq \max  J_{Z_{i+1}}$ is similar.
\end{proof}

\begin{lem}\label{lem : Wub1}
There is a constant $E_{2}>E_{1}$, depending on $f_{0},f_{1},f_{2}$ and $f_{3}$ and $\mu_{I}(q,e)$ with the following properties. Given $q$ and $\{e_{i}\}_{i=1}^{k}$  such that $|e_{i}|>E_{2}$ for any $i\in\{1,..., k\}$, we have 
\begin{enumerate}[(i)]
\item \label{Wub1 : bd} Suppose that $W$ is a non-annular subsurface which is neither $Z_{i}(q,e)$ for some $i\in\{1,...,k\}$ nor $S$. If $Z_{r}(q,e)<W<Z_{s}(q,e)$ for some $r,s\in\{1,...,k\}$ with $r<s$, then there is a constant $m>0$ depending on $s-r$ such that $d_{W}(\partial{Z}_{r}(q,e),\partial{Z}_{s}(q,e)) \leq m$. If $W<Z_{s}(q,e)$ for some $s\in\{1,...,k\}$, then there is a constant $m>0$ depending on $s$, so that $d_{W}(\mu_{I}(q,e),\partial{Z}_{s}(q,e))\leq m$. Also if $Z_{r}(q,e)<W$ for some $r\in\{1,...,k\}$, then there is a constant $m>0$ depending on $k-r$, so that $d_{W}(\partial{Z}_{r}(q,e),\mu_{T}(q,e))\leq m$.
\item \label{Wub1 : abd} Suppose that  $A(\gamma)$ is an annular subsurface. If $Z_{r}(q,e)<\gamma<Z_{s}(q,e)$ for some $r,s\in\{1,...,k\}$ with $r<s$, then there is a constant $m'>0$ depending on $s-r$, such that $d_{\gamma}(\partial{Z}_{r}(q,e),\partial{Z}_{s}(q,e))\leq m'$. If $\gamma<Z_{s}(q,e)$ for some $s\in\{1,...,k\}$, then there is a constant $m'>0$ depending on $s$, so that $d_{\gamma}(\mu_{I}(q,e),\partial{Z}_{s}(q,e))\leq m'$. Also if $Z_{r}(q,e)<W$ for some $r\in\{1,...,k\}$, then there is a constant $m'>0$ depending on $k-r$, so that $d_{\gamma}(\partial{Z}_{r}(q,e),\mu_{T}(q,e))\leq m'$.

\end{enumerate}
\end{lem}
\begin{proof}
{\it Proof of part (\ref{Wub1 : bd}):} 

Since $Z_{r}<W<Z_{s}$, the subsurface $W$ overlaps $\partial{Z}_{r}$ and $\partial{Z}_{s}$ (see Definition \ref{def : ordersubsurf}). First suppose that $W$ overlaps all of the boundary curves $\partial{Z}_{j}$, where $r< j< s$ and $q(j)=0$ or $2$ ($Z_{j}$ is a subsurface with two boundary curves). Let $j$ be so that $q(j)=0$ or $2$. Let $g=f_{q(1)}^{e_{1}}...f_{q(j-1)}^{e_{j-1}}$. Applying $g^{-1}$ to $d_{W}(\partial{Z}_{j},\partial{Z}_{j+2})$ we obtain
$$d_{g^{-1}(W)}(\partial{X}_{q(j)},f_{q(j)}^{e_{j}}f_{q(j+1)}^{e_{j+1}}(\partial{X}_{q(j+2)})).$$
Since $q(j)=0$ or $2$, $\partial{X}_{q(j)}=\partial{X}_{q(j+2)}=\{\alpha,\beta\}$, and $f_{g(j)}$ preserves each of $\alpha$ and $\beta$. Then applying $f_{q(j)}^{-e_{j}}$ to the above subsurface coefficient we get
$$d_{f_{q(j)}^{-e_{j}}\circ g^{-1}(W)}(\{\alpha,\beta\},f_{q(j+1)}^{e_{j+1}}(\{\alpha,\beta\})).$$
Moreover, $f_{q(j)}^{-e_{j}}\circ g^{-1}(W)\neq X_{q(j)}$, for otherwise $W=g(X_{q(j)})=Z_{j}$, which contradicts the assumption that $W$ is not in the list of subsurfaces $Z_{i}$. Then Lemma \ref{lem : paub} provides an upper bound $h$ depending only on $\alpha,\beta$ and partial pseudo-Anosov maps $f_{1}$ and $f_{3}$ (because $f_{q(j+1)}=f_{1}$ or $f_{3}$) for the above subsurface coefficient. Thus we have the upper bound 
\begin{equation}\label{eq : dWbdZjbdZj+2}d_{W}(\partial{Z}_{j},\partial{Z}_{j+2})\leq h.\end{equation}

Let $r\leq j_{1}\leq s$ be the smallest index with $q( j_{1})=0$ or $2$. Moreover let $r\leq j_{2}\leq s$ be the largest index so that $q(j_{2})=0$ or $2$.
\begin{claim}\label{claim : dWbdZrbdZj1}
We have  $d_{W}(\partial{Z}_{r},\partial{Z}_{j_{1}})\leq 2$ and $d_{W}(\partial{Z}_{j_{2}},\partial{Z}_{s})\leq 2$.
 \end{claim}
For the first inequality, note that $j_{1}=r$ or $r+1$. If $j_{1}=r$ then the bound follows immediately from Lemma \ref{lem : diamproj}. Suppose that $j_{1}=r+1$. Then since $\partial{Z}_{r}\subset\partial{Z}_{r+1}$ the bound again follows from Lemma \ref{lem : diamproj}. The proof of the second inequality is similar. 
 
 Using the two bounds from Claim \ref{claim : dWbdZrbdZj1} and the bound (\ref{eq : dWbdZjbdZj+2}) for all $r\leq j\leq s$ with $q(j)=0$ or $2$, by the triangle inequality we have that
\begin{eqnarray*}
d_{W}(\partial{Z}_{r},\partial{Z}_{s})&\leq& \sum_{\substack{ j:q(j)=0\;\text{or} 2\;\text{and}\\ j_{1}\leq j< j_{2}}} d_{W}(\partial{Z}_{j},\partial{Z}_{j+2})+d_{W}(\partial{Z}_{r},\partial{Z}_{j_{1}})+d_{W}(\partial{Z}_{j_{2}},\partial{Z}_{s})\\
&+&\sum_{\substack{ j:q(j)=0\;\text{or} 2\;\text{and}\\ j_{1}< j< j_{2}}}\diam_{W}(\partial{Z}_{j})\\
&\leq& h\lfloor\frac{s-r}{2}\rfloor+2\lfloor\frac{s-r}{2}\rfloor+4.
\end{eqnarray*}
 Which is the desired bound.
\medskip

Now suppose that $W$ does not overlap a boundary curve $\partial{Z}_{j}$, where $r< j< s$ and $q(j)=0$ or $2$. Then since $Z_{j}$ is a large subsurface, $W\subseteq Z_{j}$, and since $W$ is not in the list of subsurfaces $Z_{i}$, $W\subsetneq Z_{j}$.

Let $g=f_{q(1)}^{e_{1}}....f_{q(j-1)}^{e_{j-1}}$. Applying $g^{-1}$ to the subsurface coefficient $d_{W}(\partial{Z}_{r},\partial{Z}_{s})$ we get 
$$d_{g^{-1}(W)}(f_{q(j-1)}^{-e_{j-1}}...f_{q(r)}^{-e_{r}}(\partial{X_{q(r)}}),f_{q(j)}^{e_{j}}...f_{q(s-1)}^{e_{s-1}}(\partial{X_{q(s)}})).$$
Denote the multi-curves $\partial^{-}=f_{q(j-1)}^{-e_{j-1}}...f_{q(r)}^{-e_{r}}(\partial{X_{q(r)}})$ and $\partial^{+}=f_{q(j+1)}^{e_{j+1}}...f_{q(s-1)}^{e_{s-1}}(\partial{X_{q(s)}})$ and the subsurface $Y=g^{-1}(W)$. Then the above subsurface coefficient may be written as
\begin{equation}\label{eq : dg-W}d_{Y}(\partial^{-},f_{q(j)}^{e_{j}}(\partial^{+})).\end{equation}
The subsurface coefficient $d_{W}(\partial{Z}_{r},\partial{Z}_{s})$ is equal to the subsurface coefficient (\ref{eq : dg-W}), so it suffices to bound the subsurface coefficient (\ref{eq : dg-W}). 

Since $Y\subset X_{q(j)}$, and $\partial^{-},f^{e_{j}}_{q(j)}(\partial^{+})$ both intersect $X_{q(j)}$ and $Y$ essentially, the second part of Lemma \ref{lem : diamproj} guarantees that $\diam_{Y}(\pi_{Y}(\pi_{X_{q(j)}}(\partial^{-}))\cup\pi_{Y}(\partial^{-}))$ and $\diam_{Y}(\pi_{Y}(\pi_{X_{q(j)}}(f_{q(j)}^{e_{j}}(\partial^{+})))\cup\pi_{Y}(f_{q(j)}^{e_{j}}(\partial^{+})))$ are uniformly bounded. Therefore,
\begin{equation}\label{eq : dYd-fd+=dYXd-Xfd+}d_{Y}(\partial^{-},f^{e_{j}}_{q(j)}(\partial^{+}))\asymp d_{Y}(\pi_{X_{q(j)}}(\partial^{-}),\pi_{X_{q(j)}}(f_{q(j)}^{e_{j}}(\partial^{+}))).\end{equation}

 For partial pseudo-Anosov maps $f_{a}$, $a=0,1,2,3$, let $\Psi^{\pm}_{a}$ and $\Delta^{\pm}_{a}$ be the subsets of $\mathcal{PML}(S)$ defined at the beginning of $\S$\ref{sec : preendinv}. Since the support of each $f_{a}$ is a large subsurface we have that $\Delta^{\pm}_{a}=[\mathcal{L}^{\pm}_{a}]$, where $\mathcal{L}^{\pm}_{a}$ is the attracting/repelling measured lamination of $f_{a}$. Denote the support of $\mathcal{L}_{a}^{\pm}$ by $\lambda_{a}^{\pm}$. By Proposition \ref{prop : pAlam} there are finitely many laminations that contain $\lambda_{a}^{\pm}$. Denote the set of laminations containing $\lambda_{a}^{\pm}$ by $\Lambda_{a}^{\pm}$. 
 
We have that $q(j)=0$ or $2$, so $X_{q(j)}\subsetneq X_{q(j+1)}$. Also $Y\subsetneq X_{q(j)}$. Thus any leaf of $\lambda_{q(j+1)}^{+}$ intersects both $X_{q(j)}$ and $Y$ essentially (Proposition \ref{prop : pAlam}). Any lamination $\lambda\in\Lambda_{q(j+1)}^{+}$ contains $\lambda^{+}_{q(j+1)}$ so there are leaves of $\lambda$ that intersect both $X_{q(j)}$ and $Y$ essentially. Similarly, any lamination $\lambda'\in\Lambda_{q(j-1)}^{-}$ contains leaves which intersect both $X_{q(j)}$ and $Y$ essentially. Thus the following subsurface projections are well-defined and are non-empty: $\pi_{X_{q(j)}}(\lambda)$, $\pi_{X_{q(j)}}(\lambda')$, $\pi_{Y}(\pi_{X_{q(j)}}(\lambda))$, $\pi_{Y}(\pi_{X_{q(j)}}(\lambda'))$, $\pi_{Y}(\lambda)$ and $\pi_{Y}(\lambda')$.
 
There are finitely many laminations in $\bigcup_{a=0,1,2,3}\Lambda_{a}^{\pm}$. Applying Lemma \ref{lem : Hdistdprojbd} to each lamination $\lambda\in\bigcup_{a=0,1,2,3}\Lambda_{a}^{\pm}$ and subsurface $X_{b}$, $b=0,1,2,3$ and taking minimum of the constants we obtain there is an $\epsilon>0$ so that: For any $a,b\in\{0,1,2,3\}$ if a lamination $\lambda\in\Lambda_{a}^{\pm}$ intersects the subsurface $X_{b}$ essentially and a curve $\gamma$ is within the $\epsilon$ Hausdorff distance of  $\lambda$, then $\diam_{X_{b}}(\pi_{X_{b}}(\gamma)\cup\pi_{X_{b}}(\lambda))\leq 4$.
  
 For each $a\in\{0,1,2,3\}$, let $U^{\pm}_{a}$ be a neighborhood of $[\mathcal{L}_{a}^{\pm}]$ in $\mathcal{PML}(S)$, so that 
 \begin{itemize}
 \item$\overline{U_{a}^{-}}$ (the closure of $U_{a}^{-}$) and $\overline{U_{a}^{+}}$ (the closure of $U_{a}^{+}$) are disjoint from $\Psi_{b}^{-}\cup \Psi_{b}^{+}$ for any $b\in \{0,1,2,3\}$ with $X_{a}\neq X_{b}$. Note that since $\mathcal{PML}(S)$ is compact each subset $\overline{U^{\pm}_{a}}$ is compact. 
 \item Every lamination in $U^{+}_{a}$ is in the $\epsilon$ Hausdorff distance of a lamination $\lambda\in \Lambda^{+}_{a}$ and every lamination in $U^{-}_{a}$ is in the $\epsilon$ Hausdorff distance of a lamination $\lambda'\in \Lambda^{-}_{a}$.
 \end{itemize}
 Lemma \ref{lem : Hnbhdl} guarantees that for $\epsilon$ as above we may choose the neighborhoods $U^{\pm}_{a}$ so that the second bullet holds. 

 Applying Theorem \ref{thm : unifconv} to the pseudo-Anosov map $f_{a}$ and compact sets $\partial{X}_{b}$, $\overline{U_{b}^{-}}$ and $\overline{U_{b}^{+}}$ for any $a,b\in\{0,1,2,3\}$ with $X_{a}\neq X_{b}$ and taking the maximum of the constants we obtain, there exists a constant  $E_{2}>E_{1}$ ($E_{1}$ is the constant from Lemma \ref{lem : ldomain}) such that 
\begin{itemize}
\item if $\delta\in \partial{X}_{b}$ and $\delta\pitchfork X_{a}$, then $f_{a}^{n}(\delta) \subset U_{a}^{+}$ for all $n\geq E_{2}$ and $f_{a}^{-n}(\delta)\subset U_{a}^{-}$ for all $n\geq E_{2}$, and
\item $f_{a}^{n}(\overline{U_{b}^{\pm}}) \subset U_{a}^{+}$ for all $n\geq E_{2}$ and $f_{a}^{-n}(\overline{U_{b}^{\pm}})\subset U_{a}^{-}$ for all $n\geq E_{2}$.
\end{itemize}
 Then assuming that $|e_{i}|>E_{2}$ for all $i\in\{1,...,k\}$, we have that $\partial^{+}\subset U^{+}_{q(j+1)}\cup\{\alpha,\beta\}$ and $\partial^{-}\subset U_{q(j-1)}^{-}\cup\{\alpha,\beta\}$. 
 
The multi-curve $\partial^{\pm}$ consists of one or two curves. Let $b^{-}=\{\zeta\in\partial^{-}:\zeta\pitchfork Y\}$ and $b^{+}=\{\eta\in\partial^{+}:f_{q(j)}^{e_{j}}(\eta)\pitchfork Y\}$. 
 
 Suppose that $b^{+}\cap \{\alpha,\beta\}\neq\emptyset$ and $b^{-}\cap \{\alpha,\beta\}\neq\emptyset$ ($b^{+}\cap \{\alpha,\beta\}\neq\emptyset$ occurs when $j-r\leq 2$, and $b^{-}\cap \{\alpha,\beta\}\neq\emptyset$ occurs when $s-j\leq 2$). Then (\ref{eq : dg-W}) is bounded above by
 $$\max\{d_{Y}(\zeta,f_{q(j)}^{e_{j}}(\eta)):\zeta,\eta\in\{\alpha,\beta\}\}.$$
Lemma \ref{lem : paub} provides an upper bound depending only on $\alpha,\beta$ and $f_{q(j)}=f_{0}$ for all of the subsurface coefficients above. Thus the above maximum is a finite number and the bound for (\ref{eq : dg-W}) follows.
 
So in the rest of the proof we assume that either $b^{+}\subset U^{+}_{q(j+1)}$ or $b^{-}\subset U_{q(j-1)}$. 

First suppose that $b^{+}\subset U_{q(j+1)}^{+}$ and $b^{-}\subset U_{q(j-1)}^{-}$. Then there is a lamination $\lambda'\in \Lambda_{q(j-1)}^{-}$ so that $b^{-}$ is within the $\epsilon$ Hausdorff distance of $\lambda'$. Then by the choice of $\epsilon$ 
 $$\diam_{X_{q(j)}}(\pi_{X_{q(j)}}(b^{-})\cup\pi_{X_{q(j)}}(\lambda'))\leq 4.$$
 Similarly there is a lamination $\lambda\in \Lambda_{q(j)}^{+}$, so that $b^{+}$ is within the $\epsilon$ Hausdorff distance of $\lambda$. Then by the choice of $\epsilon$ we have that $\diam_{X_{q(j)}}(\pi_{X_{q(j)}}(b^{+})\cup \pi_{X_{q(j)}}(\lambda))\leq 4$. Then since $f_{q(j)}$ acts as isometry on $\mathcal{C}(X_{q(j)})$ we have 
$$\diam_{X_{q(j)}}(f_{q(j)}^{e_{j}}(\pi_{X_{q(j)}}(b^{+}))\cup f_{q(j)}^{e_{j}}(\pi_{X_{q(j)}}(\lambda)))\leq 4.$$
By the above two bounds we have 
\begin{equation}\label{eq : dYXd-fXd+=dYXl'fXl}d_{Y}(\pi_{X_{q(j)}}(\partial^{-}),f_{q(j)}^{e_{j}}(\pi_{X_{q(j)}}(\partial^{+})))\asymp d_{Y}(\pi_{X_{q(j)}}(\lambda'),f_{q(j)}^{e_{j}}(\pi_{X_{q(j)}}(\lambda))).\end{equation}
 Then since $f_{q(j)}$ is supported on $X_{q(j)}$, $f_{q(j)}^{e_{j}}(\pi_{X_{q(j)}}(\partial^{+}))=\pi_{X_{q(j)}}(f_{q(j)}^{e_{j}}(\partial^{+}))$. Thus from (\ref{eq : dYXd-fXd+=dYXl'fXl}) we get
 \begin{equation}\label{eq : dYXd-Xfd+=dYXl'fXl}d_{Y}(\pi_{X_{q(j)}}(\partial^{-}),\pi_{X_{q(j)}}(f_{q(j)}^{e_{j}}(\partial^{+})))\asymp d_{Y}(\pi_{X_{q(j)}}(\lambda'),f_{q(j)}^{e_{j}}(\pi_{X_{q(j)}}(\lambda))).\end{equation}
 Putting the quasi-equalities (\ref{eq : dYd-fd+=dYXd-Xfd+}) and (\ref{eq : dYXd-Xfd+=dYXl'fXl}) together we have
\begin{equation}\label{eq : dYd-fd+=dYXl'fXl}d_{Y}(\partial^{-},f_{q(j)}^{e_{j}}(\partial^{+}))\asymp d_{Y}(\pi_{X_{q(j)}}(\lambda'),f_{q(j)}^{e_{j}}(\pi_{X_{q(j)}}(\lambda))).\end{equation}
Lemma \ref{lem : paub}, gives a uniform bound for 
$$d_{Y}(\zeta,f_{q(j)}^{e_{j}}(\eta))$$
 for any $\zeta\in \pi_{X_{q(j-1)}}(\lambda')$ and $\eta\in\pi_{X_{q(j+1)}}(\lambda)$, depending on $f_{q(j)}$ and $\zeta,\eta$. Moreover  there are finitely many laminations in $\bigcup_{a=0,1,2,3}\Lambda_{a}^{\pm}$, and finitely many curves in the projection of each one of these laminations onto $X_{q(j)}$ (Proposiiton \ref{prop : pAlam}). Thus the right hand side of (\ref{eq : dYd-fd+=dYXl'fXl}) is uniformly bounded. The upper bound for (\ref{eq : dg-W}) follows.
 \medskip
 
 Assuming that $b^{+}\subset U_{q(j+1)}^{+}$ and $b^{-}\cap\{\alpha,\beta\}\neq\emptyset$ similar to above we can get 
 $$d_{Y}(\partial^{-},f_{q(j)}^{e_{j}}(\partial^{+}))\asymp d_{Y}(\delta,f_{q(j)}^{e_{j}}(\pi_{X_{q(j)}}(\lambda)))$$
 where $\lambda\in\Lambda_{q(j+1)}^{+}$ and $\delta\in\{\alpha,\beta\}$. As above we can show that the right hand side of the above quasi-equality is uniformly bounded. The upper bound for (\ref{eq : dg-W}) follows.
 
 The only case left is that $b^{-}\subset U_{q(j-1)}^{-}$ and $b^{+}\cap\{\alpha,\beta\}\neq\emptyset$, which can be treated similarly.
 
 Now let's bound $d_{W}(\mu_{I},\partial{Z}_{s})$. Note that $W$ overlaps $\base(\mu_{I})$ and $\partial{Z}_{s}$. Suppose that $W$ overlaps all of the boundary curves $\partial{Z}_{j}$ where $1\leq j< s$ and $q(j)=0$ or $2$. Let $1\leq j_{1}\leq s$ be the smallest index with $q(j_{1})=0$ or $2$. Note that $j_{1}=1$ or $2$, then $\partial{Z}_{j_{1}}\subset \mu_{I}$. Then by Lemma \ref{lem : diamproj}, $d_{W}(\mu_{I},\partial{Z}_{j_{1}})\leq 2$.
 
  Let $1\leq j_{2}\leq s$ be the largest index so that $q(j_{2})=0$ or $2$. Then similar to Claim \ref{claim : dWbdZrbdZj1} we have that $d_{W}(\partial{Z}_{j_{2}},\partial{Z}_{s})\leq 2$.
  
   The above two bounds and the bound (\ref{eq : dWbdZjbdZj+2}) for all $1\leq j\leq s$ with $q(j)=0$ or $2$ give us
 \begin{eqnarray*}
 d_{W}(\mu_{I},\partial{Z}_{s})&\leq&\sum_{\substack{ j:q(j)=0\;\text{or} 2\;\text{and}\\ j_{1}\leq j<s}} d_{W}(\partial{Z}_{j},\partial{Z}_{j+2})+d_{W}(\mu_{I},\partial{Z}_{j_{1}})+d_{W}(\partial{Z}_{j_{2}},\partial{Z}_{s})\\
 &+&\sum_{\substack{ j:q(j)=0\;\text{or} 2\;\text{and}\\ j_{1}< j< s}}\diam_{W}(\partial{Z}_{j})\\
 &\leq& h\lfloor\frac{s-1}{2}\rfloor+2\lfloor\frac{s-1}{2}\rfloor+2.
  \end{eqnarray*}
  Which is the desired bound.
  
  If $W$ does not overlap one of the boundary curves $\partial{Z}_{j}$, $1\leq j< s$. Then the bound for $d_{W}(\mu_{I},\partial{Z}_{s})$ follows from an argument similar to the one we gave above to bound $d_{W}(\partial{Z}_{r},\partial{Z}_{s})$. 
  
  Define the function $q'(i)=q(k-i+1)$ for $i=1,...,k$ and the sequence $e'_{i}=e_{k-i+1}$ for $i=1,...,k$. The bound for $d_{W}(\partial{Z}_{r},\mu_{T})$ can be obtained similar to the bound for $d_{W}(\mu_{I},\partial{Z}_{s})$ considering $q'$ and $e'$ instead of $q$ and $e$.
 \medskip

{\it Proof of part (\ref{Wub1 : abd}).} Suppose that $\gamma$ overlaps all boundary curves $\partial{Z}_{j}$, where $r<  j< s$ and $q(j)=0$ or $2$. Then as in the beginning of the proof of part (\ref{Wub1 : bd}) we have that 
$$d_{\gamma}(\partial{Z}_{r},\partial{Z}_{s})\leq h\lfloor\frac{r-s}{2}\rfloor+2\lfloor\frac{r-s}{2}\rfloor+4.$$

Now suppose that $\gamma$ does not overlap one of the boundary curves $\partial{Z}_{j}$ where $r< j< s$ and $q(j)=0$ or $2$.

Let $g=f_{q(1)}^{e_{1}}...f_{q(j-1)}^{e_{j-1}}$, applying $g^{-1}$ to $d_{\gamma}(\partial{Z}_{r},\partial{Z}_{s})$ we obtain
$$d_{g^{-1}(\gamma)}(f_{q(j-1)}^{-e_{j-1}}...f_{q(r)}^{-e_{r}}(\partial{X_{q(r)}}),f_{q(j)}^{e_{j}}...f_{q(s-1)}^{e_{s-1}}(\partial{X_{q(s)}})).$$
Denote multi-curves $\partial^{-}=f_{q(j-1)}^{-e_{j-1}}...f_{q(r)}^{-e_{r}}(\partial{X_{q(r)}})$ and $\partial^{+}=f_{q(j+1)}^{e_{j+1}}...f_{q(s-1)}^{e_{s-1}}(\partial{X_{q(s)}})$ and the subsurface $\delta=g^{-1}(\gamma)$. The above subsurface coefficient may written as
\begin{equation}\label{eq : dg-gamma}d_{\delta}(\partial^{-},f_{q(j)}^{e_{j}}(\partial^{+}).\end{equation}
Since $\delta\subset X_{q(j)}$, and $\partial^{-},f_{q(j)}^{e_{j}}(\partial^{+})$ both intersect $X_{q(j)}$ and $\delta$ essentially, the second part of Lemma \ref{lem : diamproj} guarantees that $\diam_{\delta}(\pi_{\delta}(\pi_{X_{q(j)}}(\partial^{-}))\cup\pi_{\delta}(\partial^{-}))$ and $\diam_{\delta}(\pi_{\delta}(\pi_{X_{q(j)}}(f_{q(j)}^{e_{j}}(\partial^{+})))\cup\pi_{\delta}(f_{q(j)}^{e_{j}}(\partial^{+})))$ are uniformly bounded. Therefore,
\begin{equation}\label{eq : ddeld-fd+=ddelXd-fXd+} d_{\delta}(\partial^{-},f_{q(j)}^{e_{j}}(\partial^{+}))\asymp d_{\delta}(\pi_{X_{q(j)}}(\partial^{-}),\pi_{X_{q(j)}}(f_{q(j)}^{e_{j}}(\partial^{+})))\end{equation}

First suppose that $\gamma\not\in\partial{Z}_{j}$. Then $\gamma\in\mathcal{C}_{0}(Z_{j})$ and $\delta\in\mathcal{C}_{0}(X_{q(j)})$. 

Let the finite sets of laminations $\Lambda_{a}^{\pm}$, $a=0,1,2,3$, be as in the proof of part (\ref{Wub1 : bd}).
 
We have that $q(j)=0$ or $2$, so $X_{q(j)}\subsetneq X_{q(j+1)}$. Also $\delta\subsetneq X_{q(j)}$. Thus any leaf of $\lambda_{q(j+1)}^{+}$ intersects both $X_{q(j)}$ and $\delta$ essentially (Proposition \ref{prop : pAlam}).  Any lamination $\lambda\in\Lambda_{q(j+1)}^{+}$ contains $\lambda^{+}_{q(j+1)}$, so there are leaves of $\lambda$ that intersect both $X_{q(j)}$ and $\delta$ essentially. Similarly,  there are leaves of any lamination $\lambda'\in\Lambda_{q(j-1)}^{-}$ which intersect $X_{q(j)}$ and $\delta$ essentially. Thus the following subsurface projections are well-defined and are non-empty: $\pi_{X_{q(j)}}(\lambda)$, $\pi_{X_{q(j)}}(\lambda')$, $\pi_{\delta}(\pi_{X_{q(j)}}(\lambda))$, $\pi_{\delta}(\pi_{X_{q(j)}}(\lambda'))$, $\pi_{\delta}(\lambda)$ and $\pi_{\delta}(\lambda')$.
 
 Let $\epsilon>0$, the neighborhoods $U_{a}^{\pm}\subset\mathcal{PML}(S)$ of $[\mathcal{L}_{a}^{\pm}]$ ($a=0,1,2,3$) and the constant $E_{2}$ be as in the proof of part (\ref{Wub1 : bd}).

 Then by the assumption that $|e_{i}|>E_{2}$ for any $i\in\{1,...,k\}$, we have that $\partial^{+}\subset U_{q(j+1)}^{+}\cup\{\alpha,\beta\}$ and $\partial^{-}\subset U_{q(j-1)}^{-}\cup\{\alpha,\beta\}$. 
  
The multi-curve $\partial^{\pm}$ consists of one or two curves. Let $b^{-}=\{\zeta\in\partial^{-}:\zeta\pitchfork Y\}$ and $b^{+}=\{\eta\in\partial^{+}:f_{q(j)}^{e_{j}}(\eta)\pitchfork Y\}$.
  
   Suppose that $b^{+}\cap \{\alpha,\beta\}\neq\emptyset$ and $b^{-}\cap \{\alpha,\beta\}\neq\emptyset$. Then (\ref{eq : dg-gamma}) is bounded above by
 $$\max\{d_{\delta}(\zeta,f_{q(j)}^{e_{j}}(\eta)):\zeta,\eta\in\{\alpha,\beta\}\}.$$
 Lemma \ref{lem : paub} provides an upper bound depending only on $\alpha,\beta$ and $f_{q(j)}$ for all of the subsurface coefficients above. Thus the above maximum is a finite number and the bound for (\ref{eq : dg-gamma}) follows.
 
So in the rest of the proof we assume that either $b^{+}\subset U^{+}_{q(j+1)}$ or $b^{-}\subset U_{q(j-1)}$. 

First suppose that $b^{+}\subset U_{q(j+1)}^{+}$ and $b^{-}\subset U_{q(j-1)}^{-}$. Then there is a lamination $\lambda'\in \Lambda_{q(j-1)}^{-}$ so that $b^{-}$ is within the $\epsilon$ Hausdorff distance of $\lambda'$. Then by the choice of $\epsilon$ we have 
 $$\diam_{X_{q(j)}}(\pi_{X_{q(j)}}(b^{-})\cup\pi_{X_{q(j)}}(\lambda'))\leq 4.$$ 
 Similarly there is a lamination $\lambda\in \Lambda_{q(j+1)}^{+}$ so that $b^{+}$ is within the $\epsilon$ Hausdorff distance of $\lambda$. Then by the choice of $\epsilon$ we have $\diam_{X_{q(j)}}(\pi_{X_{q(j)}}(b^{+})\cup \pi_{X_{q(j)}}(\lambda))\leq 4$. Then by the fact that $f_{q(j)}$ acts as an isometry on $\mathcal{C}(X_{q(j)})$ we have 
 $$\diam_{X_{q(j)}}(f_{q(j)}^{e_{j}}(\pi_{X_{q(j)}}(b^{+}))\cup f_{q(j)}^{e_{j}}(\pi_{X_{q(j)}}(\lambda)))\leq 4.$$ 
 By the above bounds on the diameters we have that
\begin{equation}\label{eq : ddelXd-fd+=dYXl'fl}d_{\delta}(\pi_{X_{q(j)}}(\partial^{-}),f_{q(j)}^{e_{j}}(\pi_{X_{q(j)}}(\partial^{+})))\asymp d_{\delta}(\pi_{X_{q(j)}}(\lambda'),f_{q(j)}^{e_{j}}(\pi_{X_{q(j)}}(\lambda))).\end{equation}
Moreover since $f_{q(j)}$ is supported on $X_{q(j)}$, $f_{q(j)}^{e_{j}}(\pi_{X_{q(j)}}(\partial^{+}))=\pi_{X_{q(j)}}(f_{q(j)}^{e_{j}}(\partial^{+}))$. Thus
\begin{equation}\label{eq : ddelXd-fd+=dYXl'fl}d_{\delta}(\pi_{X_{q(j)}}(\partial^{-}),\pi_{X_{q(j)}}(f_{q(j)}^{e_{j}}(\partial^{+})))\asymp d_{\delta}(\pi_{X_{q(j)}}(\lambda'),f_{q(j)}^{e_{j}}(\pi_{X_{q(j)}}(\lambda))).\end{equation}
Putting the quasi-equalities (\ref{eq : ddelXd-fd+=dYXl'fl}) and (\ref{eq : ddeld-fd+=ddelXd-fXd+}) together we have that
\begin{equation}\label{eq : ddeld-fd+=ddelXl'fl}d_{\delta}(\partial^{-},f_{q(j)}^{e_{j}}(\partial^{+}))\asymp d_{\delta}(\pi_{X_{q(j)}}(\lambda'),\pi_{X_{q(j)}}(f_{q(j)}^{e_{j}}(\lambda))).\end{equation}
Lemma \ref{lem : paub}, gives a uniform bound for $d_{\delta}(\zeta,f_{q(j)}^{e_{j}}(\eta))$ for any $\zeta\in \pi_{X_{q(j)}}(\lambda')$ and $\eta\in\pi_{X_{q(j)}}(\lambda)$, depending on $\zeta,\eta$ and $f_{q(j)}$. Moreover there are finitely many laminations in $\bigcup_{a=0,1,2,3}\Lambda_{a}^{\pm}$, and finitely many curves in the projection of each one of these laminations onto $X_{q(j)}$ (Proposition \ref{prop : pAlam}). Thus the right hand side of (\ref{eq : ddeld-fd+=ddelXl'fl}) is uniformly bounded. The upper bound for (\ref{eq : dg-gamma}) follows.

Assuming that 
\begin{itemize}
\item $b^{+}\subset U_{q(j+1)}^{+}$ and $b^{-}\cap\{\alpha,\beta\}\neq\emptyset$, or
\item $b^{-}\subset U_{q(j-1)}^{-}$ and $b^{-}\cap\{\alpha,\beta\}\neq\emptyset$
\end{itemize}
 a similar argument provides us with the upper bound for (\ref{eq : dg-gamma}).
 \medskip

Now suppose that $\gamma\in \partial{Z}_{j}$, then $\delta\in\partial{X}_{q(j)}$. Then by the assumption that the restriction of $f_{q(j)}$ to a neighborhood of $\delta\in\partial{X}_{a}$ is the identity map we have
$$d_{\delta}(\lambda',f_{q(j)}^{e_{j}}(\lambda))\asymp d_{\delta}(\lambda',\lambda).$$
 The subsurface coefficient $d_{\delta}(\lambda',\lambda)$ is bounded by 
$$\max\{d_{\delta}(\lambda',\lambda):\lambda'\in \Lambda_{q(j-1)}^{-},\lambda\in\Lambda_{q(j+1)}^{+},\delta\in\partial{X}_{q(j)}\}.$$
Since there are finitely many laminations in $\bigcup_{a=0,1,2,3}\Lambda_{a}^{\pm}$, the above maximum exists and is finite. So we obtain the upper bound for (\ref{eq : dg-gamma}).
 \medskip

 The bound for $d_{\gamma}(\mu_{I},\partial{Z}_{s})$ and $d_{\gamma}(\partial{Z}_{r},\mu_{T})$ may be obtained similarly.


\end{proof}

\begin{prop}\label{prop : subsurfub}
There are constants ${\bf m},{\bf m}'>4M$, depending on $f_{0},f_{1},f_{2}$ and $f_{3}$ and $\mu_{I}$ with the following properties. Given $q$ and $\{e_{i}\}_{i=1}^{k}$ such that $|e_{i}|>E_{2}$ for any $i\in\{1,..., k\}$, we have
\begin{enumerate}[(i)]
\item \label{subsurfub : bd}For any non-annular subsurface $W$ which is neither $Z_{i}(q,e)$ for some $i\in\{1,...,k\}$ nor $S$ we have $d_{W}(\mu_{I}(q,e),\mu_{T}(q,e)) \leq {\bf m}$.
\item \label{subsurfub : abd} Given $\gamma\in \mathcal{C}_{0}(S)$ we have $d_{\gamma}(\mu_{I}(q,e),\mu_{T}(q,e))\leq {\bf m}'$.
\end{enumerate}
\end{prop}

\begin{proof} 
{\it Proof of part (\ref{subsurfub : bd}).} If $d_{W}(\mu_{I},\mu_{T})\leq 4M$, then we already have the upper bound. So we may assume that $d_{W}(\mu_{I},\mu_{T})> 4M$. 

Let $l\in\{1,...,k\}$ with $q(l)=1$ or $3$. Then $Z_{l}$ is a subsurface with one boundary curve. We claim that $\partial{Z}_{l-1}$ and $\partial{Z}_{l+1}$ fill $Z_{l}$. To see this, suppose that $q(l)=1$, then $X_{q(l)}=S\backslash \alpha$ ( Assuming that $q(l)=3$ the proof is similar). Let $g=f_{q(1)}^{e_{1}}...f_{q(l-1)}^{e_{l-1}}$. We have $g^{-1}(Z_{l})=X_{q(l)}=S\backslash \alpha$. Since $q(l-1)=0$ or $2$, $X_{q(l-1)}=S\backslash\{\alpha,\beta\}$. Then $g^{-1}(\partial{Z}_{l-1})=f_{q(l-1)}^{-e_{l-1}}(\partial{X}_{q(l-1)})=\{f_{q(l-1)}^{-e_{l-1}}(\alpha),f_{q(l-1)}^{-e_{l-1}}(\beta)\}$. Moreover, $f_{q(l-1)}$ is supported on $X_{q(l-1)}$, so preserves each component of $\partial{X}_{q(l-1)}= \{\alpha,\beta\}$ i.e. $f_{q(l-1)}(\alpha)=\alpha$ and $f_{q(l-1)}(\beta)=\beta$. Thus $g^{-1}(\partial{Z}_{l-1})=\{\alpha,\beta\}$. Since $q(l+1)=0$ or $2$, $X_{q(l+1)}=S\backslash\{\alpha,\beta\}$, so $g^{-1}(\partial{Z}_{l+1})=\{\alpha, f_{q(l)}^{e_{l}}(\beta)\}$. Now $\beta \pitchfork S\backslash \alpha$, so by Proposition \ref{prop : palb}, for $|e_{l}|>E_{1}$ ($E_{1}$ is the constant from Lemma \ref{lem : ldomain}),
$$d_{S\backslash \alpha}(\beta,f_{q(l)}^{e_{l}}(\beta))\geq \tau_{1}|e_{l}|>4.$$
 So the curves $\beta$ and $f_{q(l)}^{e_{l}}(\beta)$ fill $S\backslash\alpha$. Then $g^{-1}(\partial{Z}_{l-1})$ and $g^{-1}(\partial{Z}_{l+1})$ fill $g^{-1}(Z_{l})$, and therefore $\partial{Z}_{l-1}$ and $\partial{Z}_{l+1}$ fill $Z_{l}$. 

We make the following observations.
\begin{claim}\label{claim : WZi}
Given $i\in\{1,...,k\}$, suppose that $W\pitchfork Z_{i}$. Then $W$ and $Z_{i}$ are ordered. 
\end{claim}
By Lemma \ref{lem : ldomain}(\ref{ldomain : ld}) and since $|e_{i}|>E_{1}$, $d_{Z_{i}}(\mu_{I},\mu_{T})> 4M$. Moreover by our assumption $d_{W}(\mu_{I},\mu_{T})>4M$. Then by Definition \ref{def : ordersubsurf}, $W$ and $Z_{i}$ are ordered.

\begin{claim}\label{claim : WZl+-} 
Let $l\in\{1,...,k\}$ with $q(l)=1$ or $3$. Then either $W$ and $Z_{l-1}$ or $W$ and $Z_{l+1}$ are ordered. 
\end{claim}
Since $Z_{l}$ is a large subsurface and $W$ is a non-annular subsurface, $W\subseteq Z_{l}$. $W=Z_{l}$ is excluded by the assumption that $W$ is not a subsurface in the list of subsurfaces $Z_{i}$. So we may assume that $W\subsetneq Z_{l}$. As we saw above $\partial{Z}_{l-1}$ and $\partial{Z}_{l+1}$ fill $Z_{l}$, so $W$ overlaps either $Z_{l-1}$ or $Z_{l+1}$. If $W\pitchfork Z_{l-1}$ then by Claim \ref{claim : WZi} $W$ and $Z_{l-1}$ are ordered. Similarly if $W\pitchfork Z_{l+1}$ then $W$ and $Z_{l+1}$ are ordered.

\begin{claim} \label{claim : trans}
 Suppose that $i,j\in\{1,...,k\}$ with $i+2\leq j$. If $Z_{j}<W$ then $Z_{i}<W$.  If $W<Z_{i}$ then $W<Z_{j}$. 
  \end{claim}
  By Lemma \ref{lem : ldomain}(\ref{ldomain : ord}) $Z_{i}<Z_{j}$. If $Z_{j}<W$  then by transitivity of $<$ (Proposition \ref{prop : ordersubsurf}), $Z_{i}<W$. Similarly if $W<Z_{i}$ then by transitivity of $<$, $W<Z_{j}$. 
  \medskip

 Let $l\in\{1,...,k\}$ with $q(l)=1$ or $3$. We proceed to place the subsurface $W$ in the list of subsurfaces $Z_{i}$.

 If $W\pitchfork Z_{l}$ then by Claim \ref{claim : WZi}, either $W<Z_{l}$ or $Z_{l}<W$. Suppose that $W$ does not overlap $Z_{l}$. Then by Claim \ref{claim : WZl+-}, either $W$ and $Z_{l-1}$, or $W$ and $Z_{l+1}$ are ordered. Suppose that $W$ and $Z_{l+1}$ are oredered. Then either $Z_{l+1}<W$ or $W<Z_{l+1}$.
 \medskip

{\bf Suppose that $Z_{l+1}<W$.} If $W\pitchfork Z_{l+2}$, then by Claim \ref{claim : WZi} either $W<Z_{l+2}$ or $Z_{l+2}<W$. If $W<Z_{l+2}$, then $Z_{l+1}<W<Z_{l+2}$ and we are in Case (1) below. If $Z_{l+2}<W$ then by Claim \ref{claim : trans}, $Z_{l}<W$. Suppose that $W$ does not overlap $Z_{l+2}$. By Claim \ref{claim : WZl+-}, either $W$ and $Z_{l+1}$, or $W$ and $Z_{l+3}$ are ordered. If $W$ and $Z_{l+3}$ are ordered, then either $W<Z_{l+3}\Longrightarrow Z_{l+1}<W<Z_{l+3}$ and we are in Case (1) below, or $Z_{l+3}<W\Longrightarrow Z_{l}<W$ (Claim \ref{claim : trans}). 

Suppose that $W$ and $Z_{l+3}$ are not ordered. Observe that either $W\pitchfork Z_{l+4}$ or not. If $W\pitchfork Z_{l+4}$ then by Claim \ref{claim : WZi} either $W<Z_{l+4}$ or $Z_{l+4}<W$. If $W<Z_{l+4}$ then $Z_{l+1}<W<Z_{l+4}$ and we are in Case (1) below. If $Z_{l+4}<W$ then by Claim \ref{claim : trans}, $Z_{l}<W$. Suppose that $W$ does not overlap $Z_{l+4}$. By Claim \ref{claim : WZi}, either $W$ and $Z_{l+3}$ or $W$ and $Z_{l+5}$ are ordered. Since we assumed that $W$ and $Z_{l+3}$ are not ordered, $W$ and $Z_{l+5}$ are ordered. Then either $W<Z_{l+5} \Longrightarrow Z_{l+1}<W<Z_{l+5}$ and we are in Case (1) below, or $Z_{l+5}<W \Longrightarrow Z_{l}<W$ (Claim \ref{claim : trans}).
\medskip

{\bf Suppose that $W<Z_{l+1}$.} If $W\pitchfork Z_{l-2}$, then by Claim \ref{claim : WZi} either $W<Z_{l-2}\Longrightarrow W<Z_{l}$ (Claim \ref{claim : trans}), or $Z_{l-2}<W\Longrightarrow Z_{l-2}<W<Z_{l+1}$ and we are in Case (1) below. Suppose that $W$ does not overlap $Z_{l-2}$. By Claim \ref{claim : WZl+-}, either $W$ and $Z_{l-3}$ or $W$ and $Z_{l-1}$ are ordered. If $W$ and $Z_{l-1}$ are ordered, then by Claim \ref{claim : WZi}, either $W<Z_{l-1}$ or $Z_{l-1}<W$. If $Z_{l-1}<W$ then $Z_{l-1}<W<Z_{l-1}$ and we are in Case (1). Suppose that $W<Z_{l-1}$. We assumed that $W$ does not overlap $Z_{l-2}$. By Claim \ref{claim : WZl+-}, either $W$ and $Z_{l-1}$ or $W$ and $Z_{l-3}$ are ordered. If $W$ and $Z_{l-3}$ are ordered, then either $W<Z_{l-3}$ and thus by Claim \ref{claim : trans}, $W<Z_{l}$, or $Z_{l-3}<W \Longrightarrow Z_{l-3}<W<Z_{l-1}$ and we are in Case (1) below. Suppose that $W$ and $Z_{l-3}$ are not ordered. If $W\pitchfork Z_{l-4}$, then by Claim \ref{claim : WZi}, either $W<Z_{l-4}\Longrightarrow W<Z_{l}$ (Claim \ref{claim : trans}), or $Z_{l-4}<W\Longrightarrow Z_{l-4}<W<Z_{l-1}$ and we are in Case (1) below. Suppose that $W$ does not overlap $Z_{l-4}$. By Claim \ref{claim : WZl+-}, either $W$ and $Z_{l-5}$ or $W$ and $Z_{l-3}$ are ordered. We supposed that $W$ and $Z_{l-3}$ are not ordered, so $W$ and $Z_{l-5}$ are ordered. Thus either $W<Z_{l-5}$ or $Z_{l-5}<W$. If $Z_{l-5}<W$ then $Z_{l-5}<W<Z_{l-1}$ and we are in Case (1). If $W<Z_{l-5}$ then by Claim \ref{claim : trans}, $W<Z_{l}$.
 
We showed that either $W<Z_{l}$, $Z_{l}<W$ or $W$ is ordered in the list of subsurfaces $Z_{i}$ as Case (1) below. Applying this result to every $l\in\{1,...,k\}$ with $q(l)=1$ or $3$ we conclude that $W$ is ordered in the list of subsurfaces $Z_{i}$ as one of the following cases:  
\begin{enumerate}
\item $Z_{r}<W<Z_{s}$, where $s-r\leq 4$,
\item $W<Z_{s}$, where $s\leq 4$,
\item $Z_{r}<W$, where $k-r\leq 4$. 
\end{enumerate}
We proceed to establish the upper bound for $d_{W}(\mu_{I},\mu_{T})$ in each of these cases.
\medskip

\noindent{\bf Case (1).} Since $Z_{r}<W$, by Definition \ref{def : ordersubsurf}, $d_{W}(\partial{Z_{r}},\mu_{I})\leq M$. Similarly, since $W<Z_{s}$, $d_{W}(\partial{Z_{s}},\mu_{T})\leq M$. These two inequalities and the triangle inequality give us 
 \begin{eqnarray*}
 d_{W}(\mu_{I},\mu_{T}) &\leq& d_{W}(\mu_{I},\partial{Z}_{r})+d_{W}(\partial{Z_{r}},\partial{Z_{s}})+d_{W}(\partial{Z}_{s},\mu_{T})+4\\
 &\leq& d_{W}(\partial{Z_{r}},\partial{Z_{s}})+2M+4.
 \end{eqnarray*}
 The subsurface coefficient $d_{W}(\partial{Z_{r}},\partial{Z_{s}})$ is bounded above by Lemma \ref{lem : Wub1}(\ref{Wub1 : bd}). Thus the desired bound follows.
  \medskip
 

\noindent{\bf Case (2).} By Definition \ref{def : ordersubsurf}, $d_{W}(\partial{Z}_{s},\mu_{T})\leq M$. Then the triangle inequality and this bound give us
$$d_{W}(\mu_{I},\mu_{T}) \leq d_{W}(\mu_{I},\partial{Z_{s}})+M+2.$$
The subsurface coefficient $d_{W}(\mu_{I},\partial{Z_{s}})$ is bounded by Lemma \ref{lem : Wub1}(\ref{Wub1 : bd}). Thus the desired bound follows.
 \medskip
 
\noindent{\bf Case (3).} We have
$$d_{W}(\mu_{I},\mu_{T}) \leq d_{W}(\partial{Z_{r}},\mu_{T})+M+2.$$
The subsurface coefficient $d_{W}(\partial{Z_{r}},\mu_{T})$ is bounded by Lemma \ref{lem : Wub1}(\ref{Wub1 : bd}). Thus the desired bound follows.
 \medskip
 
 The upper bounds in Cases (1), (2) and (3) depend on the upper bound for $s-r$, $s$, $k-r$ respectively, the marking $\mu_{I}$ and pseudo-Anosov maps $f_{0},f_{1},f_{2},f_{3}$. Having the bounds there is a constant ${\bf m}$ so that
  $$d_{W}(\mu_{I},\mu_{T})\leq{\bf m}$$
  for any proper, non-annular subsurfaces $W$ which is not in the list of subsurfaces $Z_{i}$. The proof of part (\ref{subsurfub : bd}) is complete.
\medskip

{\it Proof of part (\ref{subsurfub : abd})}. If $d_{\gamma}(\mu_{I},\mu_{T})\leq 4M$ we already have the bound. So we may assume that $d_{\gamma}(\mu_{I},\mu_{T})> 4M$.

We claim that if $q(l)=1$ or $3$ ($\partial{Z}_{l}$ consists of one curve), then $\partial{Z}_{l-3}$ and $\partial{Z}_{l-7}$ fill $S$. To see this, suppose that $q(l)=1$, then $X_{q(l)}=S\backslash \alpha$ (Assuming that $q(l)=3$ the proof is similar).  Let $g=f_{q(1)}^{e_{1}}...f_{q(l-4)}^{e_{l-4}}$. Since $q(l-3)=0$ or $2$, $X_{q(l-3)}=S\backslash\{\alpha,\beta\}$. Then $g^{-1}(\partial{Z}_{l-3})=\{\alpha,\beta\}$. 

We have $q(l-4)=1$ so $f_{q(l-4)}$ preserves $\partial{X}_{q(l-4)}=\alpha$, $q(l-5)=0$ so $f_{q(l-5)}$ preserves each component of $\partial{X}_{q(l-5)}=\{\alpha,\beta\}$, $q(l-6)=3$ so $f_{q(l-6)}$ preserves $\partial{X}_{q(l-6)}=\beta$, and $q(l-7)=2$ so $f_{q(l-7)}$ preserves each component of $\partial{X}_{q(l-7)}=\{\alpha,\beta\}$. Then we have that
$$g^{-1}(\partial{Z}_{l-7})=\{f_{q(l-4)}^{-e_{l-4}}f_{q(l-5)}^{-e_{l-5}}f_{q(l-6)}^{-e_{l-6}}(\alpha),f_{q(l-4)}^{-e_{l-4}}(\beta)\}.$$
First note that  $\beta \pitchfork S\backslash\alpha$. Moreover $f_{q(l-4)}$ is supported on $S\backslash \alpha$. So by Lemma \ref{lem : ldomain} and since $|e_{l}|>E_{2}$ (note that $E_{2}>E_{1}$),
 $$d_{S\backslash \alpha}(\beta,f_{q(l-4)}^{-e_{l-4}}(\beta))\geq K_{1}'|e_{l-4}|-C_{1}'\geq K_{1}'E_{2}-C_{1}'>4.$$ 
 The above inequality implies that $\beta$ and $f_{q(l-4)}^{-e_{l-4}}(\beta)$ fill $S\backslash \alpha$ (see $\S$\ref{sec : ccplx}). 
 
 Second define $q'(i)=q(l-4-i+1)$ for $i=1,2,3$ and $e'_{i}=-e_{l-4-i+1}$ for $i=1,2,3$. Then by Lemma \ref{lem : ldomain} we have that 
$$d_{Z_{3}(q',e')}(f_{q(l-4)}^{-e_{l-4}}f_{q(l-5)}^{-e_{l-5}}f_{q(l-6)}^{-e_{l-6}}(\alpha),\alpha)\geq K'|e_{l-6}|-C'\geq K'E_{2}-C'>4.$$
 Thus $f_{q(l-4)}^{-e_{l-4}}f_{q(l-5)}^{-e_{l-5}}f_{q(l-6)}^{-e_{l-6}}(\alpha)\pitchfork\alpha$. 
 
So we may conclude that $g^{-1}(\partial{Z_{l-3}})$ and $g^{-1}(\partial{Z_{l-7}})$ fill $S$. Therefore $\partial{Z_{l-3}}$ and $\partial{Z_{l-5}}$ fill $S$. Then given $\gamma\in\mathcal{C}_{0}(S)$ we have
\begin{itemize}
\item $\gamma\pitchfork \partial{Z_{l-3}}$ or $\gamma\pitchfork \partial{Z_{l-7}}$.
\end{itemize} 
 Similarly we can prove that $\partial{Z}_{l+3}$ and $\partial{Z}_{l+7}$ fill $Z_{l}$. Thus given $\gamma\in\mathcal{C}_{0}(S)$ we have that
\begin{itemize}
\item $\gamma\pitchfork \partial{Z_{l+3}}$ or $\gamma\pitchfork \partial{Z_{l+7}}$.
\end{itemize}

\begin{claim} \label{claim : gZi}
If $\gamma\pitchfork Z_{i}$ then $\gamma$ and $Z_{i}$ are ordered.
\end{claim}
By Lemma \ref{lem : ldomain}, $d_{Z_{i}}(\mu_{I},\mu_{T})\geq K'_{1}|e_{i}|-C'_{1}$ for any $i\in\{1,...,k\}$. Then since $|e_{i}|>E_{2}$,  
$$d_{Z_{i}}(\mu_{I},\mu_{T})>4M.$$
 Moreover, by our assumption, $d_{\gamma}(\mu_{I},\mu_{T})>4M$. These two inequalities imply $\gamma$ and $Z_{i}$ are ordered (see Definition \ref{def : ordersubsurf}).
 \medskip

Claim \ref{claim : gZi} and the above two bullets imply that given $l\in\{1,...,k\}$ with $q(l)=1$ or $3$, $\gamma$ and subsurfaces $Z_{r}$ and $Z_{s}$ are ordered, where $r=l-3$ or $l-7$, and $s=l+3$ or $l+7$. If $Z_{r}<\gamma<Z_{s}$ we are in Case (1) below. Otherwise, either $Z_{s}<\gamma$ (by transitivity of $<$, $Z_{r}<\gamma$), or $\gamma<Z_{r}$ (by transitivity of $<$, $\gamma<Z_{r}$). Repeating the comparison for all $l\in\{1,...,k\}$ with $q(l)=1$ or $3$ we end up in one of the following cases  
\begin{enumerate}
\item  $Z_{r}<\gamma<Z_{s}$, where $s-r\leq 14$,
\item $\mu_{I}<\gamma<Z_{s}$, where $s\leq 7$,
\item $Z_{r}<\gamma<\mu_{T}$, where $k-r\leq 7$.
\end{enumerate}
We proceed to establish the bound for $d_{\gamma}(\mu_{I},\mu_{T})$ in each of these cases.
\medskip

\noindent{\bf Case (1).} Since $Z_{r}<\gamma$, by Definition \ref{def : ordersubsurf}, $d_{\gamma}(\mu_{I},\partial{Z_{r}})\leq M$. Similarly, since $\gamma<Z_{s}$, $d_{\gamma}(\mu_{T},\partial{Z}_{s})\leq M$. Having these bounds, by the triangle inequality we get
\begin{eqnarray*}
d_{\gamma}(\mu_{I},\mu_{T})&\leq&d_{\gamma}(\partial{Z}_{r},\partial{Z}_{s})+2M+\diam_{\gamma}(\partial{Z}_{r})+\diam_{\gamma}(\partial{Z}_{s})\\
&\leq&d_{\gamma}(\partial{Z}_{r},\partial{Z}_{s})+2M+4
\end{eqnarray*}
The subsurface coefficient $d_{\gamma}(\partial{Z}_{r},\partial{Z}_{s})$ is bounded by Lemma \ref{lem : Wub1}(\ref{Wub1 : abd}). The desired bound follows.
\medskip

\noindent{\bf Case (2).} Since $\gamma<Z_{s}$, by Definition \ref{def : ordersubsurf}, $d_{\gamma}(\mu_{T},\partial{Z}_{s})\leq M$. Then by the triangle inequality 
 $$d_{\gamma}(\mu_{I},\mu_{T})\leq d_{\gamma}(\mu_{I},\partial{Z}_{s})+M+2.$$
 The subsurface coefficient $d_{\gamma}(\mu_{I},\partial{Z}_{s})$ is bounded by Lemma \ref{lem : Wub1}(\ref{Wub1 : abd}). The desired bound follows.
\medskip

\noindent{\bf Case (3).} Since $Z_{r}<\gamma$, by Definition \ref{def : ordersubsurf}, $d_{\gamma}(\mu_{I},\partial{Z}_{r})\leq M$. Then by the triangle inequality 
 $$d_{\gamma}(\mu_{I},\mu_{T})\leq d_{\gamma}(\partial{Z}_{r},\mu_{T})+M+2.$$
 The subsurface coefficient $d_{\gamma}(\partial{Z}_{r},\mu_{T})$ is bounded by Lemma \ref{lem : Wub1}(\ref{Wub1 : abd}). The desired bound follows.
\medskip

The bounds in Cases (1), (2) and (3) depend on the upper bound for $s-r$,$s$, $k-r$, respectively, the marking $\mu_{I}$ and pseudo-Anosov maps $f_{0},f_{1},f_{2},f_{3}$ 
Establishing the bounds there is a constant ${\bf m}'$ so that 
$$d_{\gamma}(\mu_{I},\mu_{T})\leq {\bf m}'$$
 for any $\gamma\in\mathcal{C}_{0}(S)$.
 \end{proof}
 
 \begin{prop} \label{prop : ldomain} There are constants $K_{1}\geq 1$, $C_{1}\geq 0$ and $E>E_{2}$, depending on the partial pseudo-Anosov maps $f_{0},f_{1},f_{2}$ and $f_{3}$ and $\mu_{I}$ with the following properties. Given $q$ and $\{e_{i}\}_{i=1}^{k}$  such that $|e_{i}|>E$ for any $i\in\{1,...,k\}$, we have 
\begin{equation}\label{eq : ldomain}d_{Z_{i}(q,e)}(\mu_{I}(q,e),\mu_{T}(q,e)) \asymp_{K_{1},C_{1}}|e_{i}|.\end{equation}
\end{prop}  
 \begin{proof}
Let the constants $K'_{1},C'_{1}$ and $E_{1}$ be from Lemma \ref{lem : ldomain}. Then by the lemma: If $|e_{i}|>E_{1}$ for all $i\in\{1,...,k\}$, then the lower bound in (\ref{eq : ldomain}) holds for $K'_{1}$ and $C'_{1}$. We proceed to obtain the upper bound in (\ref{eq : ldomain}). Denote $\mu=\mu_{I}(q,e)$. Let $i\in\{1,...,k\}$ and $g=f_{q(1)}^{e_{1}}...f_{q(i-1)}^{e_{i-1}}$. Applying $g^{-1}$ to the subsurface coefficient $d_{Z_{i}}(\mu,\mu_{T})$ we obtain 
$$d_{X_{q(i)}}(g^{-1}\mu,f_{q(i)}^{e_{i}}h\mu),$$
 where $h=f_{q(i+1)}^{e_{i+1}}...f_{q(k)}^{e_{k}}$. By the triangle inequality
\begin{equation}\label{eq : dXqi}d_{X_{q(i)}}(g^{-1}\mu,f_{q(i)}^{e_{i}}h\mu)\leq d_{X_{q(i)}}(f_{q(i)}^{e_{i}}h\mu,h\mu)+d_{X_{q(i)}}(h\mu, g^{-1}\mu)+2.\end{equation}
In Lemma \ref{lem : ldomain} we proved the second term on the right hand side of (\ref{eq : dXqi}) is bounded by $C'_{1}$. Thus we only need to show that for some $K_{1}\geq 1$, 
$$d_{X_{q(i)}}(f_{q(i)}^{e_{i}}h\mu,h\mu)\leq K_{1}|e_{i}|.$$
Define $q'(j)=q(j+i)$ for $j=1,...,k-i$ and $e'_{j}=e_{j+i}$ for $j=1,...,k-i$. The subsurface $X_{q(i)}$ is not in the list of subsurfaces $Z_{j}(q',e')$. By assumption of the proposition there $|e'_{j}|=|e_{j+i}|>E_{2}$ for $j=1,...,k-i$. So by Proposition \ref{prop : subsurfub}(\ref{subsurfub : bd}) for $q',e'$, there is a constant ${\bf m}$ such that:  
\begin{equation}\label{eq : dXqimhm}d_{X_{q(i)}}(\mu,h\mu)\leq {\bf m}.\end{equation}
Let $\bar{\tau}_{a}:=\bar{\tau}_{f_{a}}$, $a=0,1,2,3$, be the constant from Lemma \ref{lem : limsup}. Then
$$\limsup_{n\to\infty}\frac{d_{X_{a}}(\delta,f_{a}^{n}(\delta))}{n}=\bar{\tau}_{a}$$
 for every $\delta \in \mathcal{C}_{0}(X_{a})$. Therefore for any $\delta\in\mathcal{C}_{0}(X_{a})$, there is a positive integer $N(\delta)$ such that for any integer $n\geq N$ we have 
\begin{equation}\label{eq : dXa}d_{X_{a}}(\delta,f_{a}^{n}(\delta))\leq (\bar{\tau}_{a}+1)n.\end{equation}
 Moreover, it follows from the proof of Lemma \ref{lem : limsup} that for $\delta'\neq \delta$, $|N(\delta)-N(\delta')|$ is bounded by a constant depending only on the distance of $\delta$ and $\delta'$ in $\mathcal{C}(X_{a})$. Thus we may choose $E>E_{2}$ such that for any integer $n\geq E$ and any curve $\delta$ in the ${\bf m}-$neighborhood of $\pi_{X_{a}}(\mu)$ in $\mathcal{C}(X_{a})$, (\ref{eq : dXa}) holds. By (\ref{eq : dXqimhm}), $\pi_{X_{q(i)}}(h\mu)$ is in the ${\bf m}-$neighborhood of $\pi_{X_{q(i)}}(\mu)$. Thus if $|e_{i}|> E$, then we have 
 $$d_{X_{q(i)}}(h\mu,f_{q(i)}^{e_{i}}h\mu)\leq (\bar{\tau}_{a}+1)|e_{i}|.$$ 
Let $K_{1}=\max\{\bar{\tau}_{a}+1,\frac{1}{\tau_{a}} : a=0,1,2,3\}$, then 
$$d_{X_{q(i)}}(\mu,h\mu)\leq K_{1}|e_{i}|.$$
The lower bound in (\ref{eq : ldomain}) is established. 

Recall from Lemma \ref{lem : ldomain} that $K'_{1}=\min\{\tau_{a} : a=0,1,2,3\}$, so $K'_{1}>\frac{1}{K_{1}}$. Thus we may conclude that (\ref{eq : ldomain}) holds for constants $K_{1},C_{1}=C'_{1}$ and $E$. 
 \end{proof}

Given $q$ and $e$ so that the domain of  $q$ is a finite set  $\{1,...,k\}$ ($k\in\mathbb{N}$) and $e$ is a sequence of integers $\{e_{i}\}_{i=1}^{k}$. Let the markings $\mu_{I}(q,e)$ and $\mu_{T}(q,e)$ be as before. Lemma \ref{lem : ldomain} and Propositions \ref{prop : subsurfub} and \ref{prop : ldomain} together prescribe the list of all subsurface coefficients of the pair of markings $\mu_{I}$ and $\mu_{T}$. 

We proceed to generalize this construction to the case that the domain of $q$ is $\mathbb{N}$. Let the curves $\alpha$ and $\beta$, the indexed subsurfaces $X_{0},X_{1},X_{2}$ and $X_{3}$, and indexed partial pseudo-Anosov maps $f_{0},f_{1}, f_{2}$ and $f_{3}$ supported on indexed subsurface $X_{0},X_{1},X_{2}$ and $X_{3}$ respectively be as the beginning of this subsection. Let $q:\mathbb{N}\to \{0,1,2,3\}$ be as the beginning of this subsection and let $\{e_{i}\}_{i=1}^{\infty}$ be an infinite sequence of integers. For each $i\in\mathbb{N}$ let the subsurface $Z_{i}(q,e)$ be as in (\ref{eq : Zi}).

 For any $k\in\mathbb{N}$, define $q^{k} :\{1,...,k\}\to\{0,1,2,3\}$ by $q^{k}(i)=q(i)$ for $i=1,...,k$, and define the sequence $e^{k}_{i}=e_{i}$ for $i=1,...,k$. Let $\mu_{I}\equiv \mu_{I}(q^{k} ,e^{k})$ be a marking so that $\base(\mu_{I})$ contains $\{\partial{X}_{a}\}_{a=0,1,2,3}=\{\alpha,\beta\}$ and let $\mu_{k}=\mu_{T}(q^{k} ,e^{k})$ be the marking defined as before. Note that $Z_{i}(q^{k},e^{k})=Z_{i}(q,e)$. Let $\delta_{k}$ be the multi-curve $f_{q(1)}^{e_{1}}...f_{q(k)}^{e_{k}}(\{\alpha,\beta\})$. Observe that $\delta_{k}\subset \base(\mu_{k})$ and $\partial{Z}_{k+1}(q,e)\subseteq \delta_{k}$. Let $\delta_{k_{m}}$ be a convergent subsequence of $\delta_{k}$ in the Hausdorff topology of $M_{\infty}(S)$. Denote the limit lamination by $\lambda$. Remove isolated leaves of $\lambda$ (if any) and denote the resulting lamination by $\mu'$. Adding simple closed curves disjoint from $\mu'$ so that each curve has bounded intersection number with the isolated leaves of $\lambda$ we obtain a lamination $\mu_{T}(q,e)$ with the property that for any subsurface $Y$ intersecting $\lambda$ and $\mu_{T}(q,e)$,
 \begin{equation}\label{eq : dZmTl} |d_{Y}(\mu_{T}(q,e),\lambda)|\leq d, \end{equation}
 where $d$ depends on the upper bound for the intersection numbers.

  \begin{prop}\label{prop : pre1} There are constants $K_{1}\geq 1$, $C_{1}\geq 0$, $E>0$ and ${\bf m},{\bf m}'>4M$, depending only on $f_{0},f_{1},f_{2}$ and $f_{3}$, and $\mu_{I}$ with the following properties. Given $q$ and $\{e_{i}\}_{i}$  such that $|e_{i}|\geq E$ for any $i$ in the domain of $q$, we have
  \begin{enumerate}[(i)]
  \item\label{pre1 : ld} For any integer $i$ in the domain of $q$, 
  $$d_{Z_{i}(q,e)}(\mu_{I}(q,e),\mu_{T}(q,e)) \asymp_{K_{1},C_{1}} |e_{i}|.$$
  \item\label{pre1 : bd} For any non-annular subsurface $W$ which is neither $Z_{i}(q,e)$ for some $i$ nor $S$ we have $d_{W}(\mu_{I}(q,e),\mu_{T}(q,e)) \leq {\bf m}$.
\item\label{pre1 : bda} For any $\gamma\in \mathcal{C}_{0}(S)$ we have $d_{\gamma}(\mu_{I}(q,e),\mu_{T}(q,e))\leq {\bf m}'$.
\item \label{pre1 : ord} Given $i,j$ in the domain of $q$, if $j\geq i+2$, then $Z_{i}(q,e)<Z_{j}(q,e)$.
 \item \label{pre1 : ordnext} Let $i$ in the domain of $q$ be such that $q(i)=1$ or $3$. If $j\in J_{Z_{i}(q,e)}$ then $j\geq \min J_{Z_{i-1}(q,e)}$ and $j\leq \max  J_{Z_{i+1}(q,e)}$.
\item\label{pre1 : minfill} If the domain of $q$ is $\mathbb{N}$, then $\mu_{T}(q,e)$ is a minimal filling lamination.
  \end{enumerate}
 \end{prop}
 \begin{proof}
Suppose that the domain of $q$ is $\{1,...,k\}$ for some $k\in\mathbb{N}$ and $\{e_{i}\}_{i=1}^{k}$ is a finite sequence of integers. Let $E$ be the constant form Proposition \ref{prop : ldomain}. Suppose that $|e_{i}|>E$. By Proposition \ref{prop : ldomain} there are constants $K_{1},C_{1}$ so that the subsurface coefficient quasi-equality in part (\ref{pre1 : ld}) holds. Moreover, by Proposition \ref{prop : subsurfub}(i) and (ii) there are constants ${\bf m}$ and ${\bf m}'$ so that the subsurface coefficient bounds in parts (\ref{pre1 : bd}) and (\ref{pre1 : bda}) hold, respectively. The order of domains in part (\ref{pre1 : ord}) is proved in Lemma \ref{lem : ldomain}(ii). Moreover part (\ref{pre1 : ordnext}) is proved in Lemma \ref{lem : Jnext}.
 
We proceed to establish the bounds on subsurface coefficients when the domain of $q$ is $\mathbb{N}$. Moreover determine the order of subsurfaces when the domain of $q$ is $\mathbb{N}$. The proofs use the bounds on subsurface coefficients and the order of subsurfaces when the domain of $q$ is finite and taking various limits of curves. In the proof for any $r\in\mathbb{N}$, we denote $Z_{r}(q,e)$ by $Z_{r}$, $\mu_{I}(q,e)$ by $\mu_{I}$ and $\mu_{T}(q,e)$ by $\mu_{T}$. Also for any $k\in\mathbb{N}$, $\mu_{k}=\mu(q^{k},e^{k})$ as was defined before.
\medskip

 {\it Proof of part (\ref{pre1 : ld}):} Given $i\in\mathbb{N}$ by Proposition \ref{prop : ldomain} there exist $K_{1}\geq 1$ and $C_{1}\geq 0$ so that for every $k> i$ we have
$$d_{Z_{i}(q^{k},e^{k})}(\mu_{I},\mu_{k})\asymp_{K_{1},C_{1}} |e_{i}|.$$
Moreover by Lemma \ref{lem : ldomain}(ii): If $l> k$, then $Z_{i}(q^{l},e^{l})<Z_{k+1}(q^{l},e^{l})$ between $\mu_{I}$ and $\mu_{l}$. Thus in particular, $Z_{k+1}(q^{l},e^{l})\pitchfork Z_{i}(q^{l},e^{l})$. Note that for any integer $r<l$, $Z_{r}=Z_{r}(q^{l},e^{l})$, so $Z_{k+1}\pitchfork Z_{i}$. Then since $\partial{Z_{k+1}}\subseteq \delta_{k}$ we have that
\begin{equation}\label{eq : dkintZi}\delta_{k}\pitchfork Z_{i}.\end{equation}
 Moreover since $\delta_{k}\subset \base(\mu_{k})$ and $\diam_{Z_{i}}(\mu_{k})\leq 2$ (Lemma \ref{lem : diamproj}) from the above quasi-equality we get
 $$d_{Z_{i}}(\mu_{I},\delta_{k})\asymp_{K_{1},C_{1}+2} |e_{i}|.$$ 
 Let the sequence $k_{m}$ and the lamination $\lambda$ be as before. Since  $\delta_{k_{m}}\to\lambda$ in the Hausdorff topology by Proposition \ref{prop : Hlimsubsurfcoeff}, $d_{Z_{i}}(\mu_{I},\delta_{k_{m}}) \asymp_{1,4} d_{Z_{i}}(\mu_{I},\lambda)$  for all $m$ sufficiently large. Then (\ref{eq : dZmTl}) implies that
 $$d_{Z_{i}}(\mu_{I},\delta_{k_{m}}) \asymp_{1,4+d} d_{Z_{i}}(\mu_{I},\mu_{T}).$$
Set $C_{1}$ to be $C_{1}+d+6$. Then by the above quasi-equalities we get
$$d_{Z_{i}}(\mu_{I},\mu_{T})\asymp_{K_{1},C_{1}}|e_{i}|,$$
as was desired.
\medskip

{\it Proof of part (\ref{pre1 : bd}):} Let $W$ be an essential subsurface which is neither $Z_{i}$ for some $i\in\mathbb{N}$ nor $S$. By Proposition \ref{prop : subsurfub}(\ref{subsurfub : bd}), there exists ${\bf m}$, so that for any $k\in\mathbb{N}$ we have that $d_{W}(\mu_{I},\mu_{k}) \leq {\bf m}$. We have that $\mu_{T}\pitchfork W$. Then by the set up of $\mu_{T}$, $\lambda$ intersects $W$. Then the convergence of $\delta_{k_{m}}$ to $\lambda$ in the Hausdorff topology guarantees that $\delta_{k_{m}}\pitchfork W$ for all $m$ sufficiently large. Then since $\delta_{k_{m}}\subset\base(\mu_{k_{m}})$ and $\diam_{W}(\mu_{k_{m}})\leq 2$ we have 
$$d_{W}(\mu_{I},\delta_{k_{m}})\leq {\bf m}+2.$$ 
Moreover by Proposition \ref{prop : Hlimsubsurfcoeff} the convergence of $\delta_{k_{m}}$ to $\lambda$ in the Hausdorff topology implies that $d_{W}(\mu_{I},\lambda) \asymp_{1,4} d_{W}(\mu_{I},\delta_{k_{m}})$ for all $m$ sufficiently large.. Then by (\ref{eq : dZmTl}) we have
$$d_{W}(\mu_{I},\mu_{T}) \asymp_{1,4+d} d_{W}(\mu_{I},\delta_{k_{m}})$$
 Putting the above bound and quasi-equality together we get the bound $d_{W}(\mu_{I},\mu_{T})\leq {\bf m}+d+8$. We set ${\bf m}$ to be ${\bf m}+d+6$.
\medskip

{\it Proof of part (\ref{pre1 : bda}):}  Similar to the proof of part (\ref{pre1 : bd}) using the Hausdorff convergence of $\delta_{k_{m}}$ to $\lambda$ we get
$$d_{\gamma}(\mu_{I},\mu_{T})\leq {\bf m}'+d+6$$
 for any $\gamma\in\mathcal{C}_{0}(S)$. We set ${\bf m}'$ to be ${\bf m}'+d+6$.
 \medskip
 
 {\it Proof of part (\ref{pre1 : ord}):} Let $E$ be the constant from Proposition \ref{prop : ldomain} and set $E$ to be larger than $\frac{4M+C_{1}}{K_{1}}$. Suppose that $|e_{i}|>E$. 
 
Let $k\in\mathbb{N}$ and $ i,j\in\{1,...,k\}$ with $j\geq i+2$. By Lemma \ref{lem : ldomain}(\ref{ldomain : ord}) we have that $Z_{i}(q^{k},e^{k})< Z_{j}(q^{k},e^{k})$ between $\mu_{I}(q^{k},e^{k})$ and $\mu_{T}(q^{k},e^{k})$. So in particular $Z_{i}(q^{k},e^{k}) \pitchfork Z_{j}(q^{k},e^{k})$ and $d_{Z_{i}(q^{k},e^{k})}(\mu_{I},\partial{Z}_{j}(q^{k},e^{k}))>2M$. Then since for any $r<k$, $Z_{r}(q^{k},e^{k})=Z_{r}$ we have
\begin{equation}\label{eq : dZimbdZj}d_{Z_{i}}(\mu_{I},\partial{Z}_{j})>2M.\end{equation}
  Moreover by part (i) the subsurface coefficient inequalities 
$$d_{Z_{i}}(\mu_{I},\mu_{T})\geq K_{1}|e_{i}|-C_{1}\;\text{and}\; d_{Z_{j}}(\mu_{I},\mu_{T})\geq K_{1}|e_{i}|-C_{1}$$ 
hold. Then since $|e_{i}|>E$ we have that
\begin{equation}\label{eq : dZijmmT}d_{Z_{i}}(\mu_{I},\mu_{T})> 4M\;\text{and}\; d_{Z_{j}}(\mu_{I},\mu_{T})> 4M.\end{equation}
 The bounds (\ref{eq : dZimbdZj}) and (\ref{eq : dZijmmT}) and the fact that $Z_{i}\pitchfork Z_{j}$ according to Definition \ref{def : ordersubsurf} imply that $Z_{i}<Z_{j}$ between the markings $\mu_{I}$ and $\mu_{T}$. 
\medskip

{\it Proof of part (\ref{pre1 : ordnext}):}  The proof of Lemma \ref{lem : Jnext} uses the fact that $d_{Z_{i}}(\mu_{I},\mu_{T})\geq 2M$ for any $i\in\{1,...,k\}$ (in the domain of $q$), and the pattern that the subsurfaces $Z_{i-1}$ and $Z_{i}$, and subsurfaces $Z_{i}$ and $Z_{i+1}$ overlap. By part (\ref{pre1 : ld}) and since $|e_{i}|>E$ , $d_{Z_{i}}(\mu_{I},\mu_{T})\geq 2M$. Moreover, the pattern that $Z_{i-1}$ and $Z_{i}$, and $Z_{i}$ and $Z_{i+1}$ overlap is the same as Lemma \ref{lem : Jnext}. Then the proof of the lemma goes through and gives us part (\ref{pre1 : ordnext}) when the domain of $q$ is $\mathbb{N}$.
\medskip

{\it Proof of part (\ref{pre1 : minfill}):} For each $k\in\mathbb{N}$, let $\rho_{k}$ be a hierarchy path between markings $\mu_{I}$ and $\mu_{k}$. Let $\rho$ be a hierarchy path between $\mu_{I}$ and $\mu_{T}$. For any $i\in\mathbb{N}$ with $q(i)= 1$ or $3$, by part (i) and since $|e_{i}|>E$, $d_{Z_{i}}(\mu_{I},\mu_{T})>M$, so $Z_{i}$ is a component domain of $\rho$ with one boundary curve. Thus $\partial{Z}_{i}$ is a curve on the main geodesic of $\rho$. Moreover by part (\ref{pre1 : ord}), $Z_{i}<Z_{j}$ between $\mu_{I}$ and $\mu_{T}$ whenever $j\geq i+2$. This implies that $\partial{Z}_{i}$ is before $\partial{Z}_{j}$ along the main geodesic of $\rho$. To see this, note that since $\partial{Z}_{i}$ and $\partial{Z}_{j}$ are curves on the main geodesic of $\rho$, the tight geodesics $g_{Z_{i}}$ and $g_{Z_{j}}$ are time ordered as is defined in \cite[\S 4]{mm2}. For otherwise, $\partial{Z}_{i}$ is after $\partial{Z}_{j}$. Then there is an $l\in J_{Z_{i}}$ so that $l> \max J_{Z_{j}}$. But this contradicts the second part of Proposition \ref{prop : ordersubsurf}.  See \cite[\S 5]{mm2} for the detail about resolution of hierarchies and the parametrization of hierarchy paths specially Proposition 5.4 there. Thus the curves $\partial{Z}_{i}$ go to $\infty$ on the main geodesic of $\rho$ as $i\to\infty$. Therefore $\partial{Z}_{i}$ converge to a point $x$ in the Gromov boundary of $\mathcal{C}(S)$ as $i\to \infty$. This point by Theorem \ref{thm : ccbdry} determines a projective measured lamination $[\mathcal{E}]$ with minimal filling support $\xi$. Moreover, given $k\in\mathbb{N}$, for each $1\leq i\leq k$, $d_{Z_{i}}(\mu_{I},\mu_{k})>M$. Thus $\partial{Z}_{i}$ is on the main geodesic of the hierarchy path $\rho_{k}$ between $\mu_{I}$ and $\mu_{k}$. Note that $\delta_{k}\subset \base(\mu_{k})$. Then $\delta_{k}$ also converge to the point $x$ in the Gromov boundary of $\mathcal{C}(S)$ as $k\to \infty$. The sequence of curves $\delta_{k_{m}}$ converges to $\lambda$ in the Hausdorff topology. Take a convergent subsequence of $[\delta_{k_{m}}]$ in the topology of $\mathcal{PML}(S)$. By Theorem \ref{thm : gobdrycc} the support of the limit projective measured lamination is $\xi$. Moreover by Proposition \ref{prop : hlimwk*lim}, $\xi\subseteq\lambda$. Since $\xi$ is filling, $S\backslash\xi$ is the union of topological disks and annuli. Thus $\lambda$ is the union of $\xi$ and some isolated leaves in components of $S\backslash\xi$. Moreover $\mu_{T}(q,e)$ is obtained from $\lambda$ by removing isolated leaves and adding closed curves. But by the topology of the complement of $\xi$, no closed curve would be added. Thus $\mu_{T}=\xi$. Then in particular, $\mu_{T}$ is a minimal filling lamination on $S$.
\end{proof}

\subsection{Scheme II} \label{subsec : pre2}The construction of this subsection will be used in $\S$\ref{subsec : recurrence} to construct examples of recurrent WP geodesic rays.

Let $\alpha$ be a curve such that $S\backslash \alpha$ is a large subsurface. Consider the indexed subsurfaces $X_{0}=S$ and $X_{1}=S\backslash \alpha$. Let $f_{0}$ and $f_{1}$ be indexed (partial) pseudo-Anosov maps supported on $X_{0}$ and $X_{1}$, respectively. Moreover suppose that the restriction of $f_{1}$ to a regular neighborhood of $\partial{X}_{1}=\alpha$ is the identity map.

Define functions $q_{0}(i)\equiv i$ (mod 2) and $q_{1}(i)=q_{0}(i+1)$. Let $q$ be any of $q_{0}$ and $q_{1}$ or their restriction to $\{1,...,k\}$ for some $k\in\mathbb{N}$. 

When the domain of $q$ is $\mathbb{N}$ let $\{e_{i}\}_{i=1}^{\infty}$ be an infinite sequence of integers and when the domain of $q$ is $\{1,...,k\}$ let $\{e_{i}\}_{i=1}^{k}$ be a sequence of integers with $k$ elements.

 For any $i$ in the domain of $q$ set the subsurface 
 \begin{equation}\label{eq2 : Zi}Z_{i}(q,e)=f_{q(1)}^{e_{1}}...f_{q(i-1)}^{e_{i-1}}(X_{q(i)}).\end{equation}
Let $\mu_{I}(q,e)$ be a marking so that $\base(\mu_{I})$ contains $\{\partial{X}_{a}\}_{a=0,1}=\{\alpha\}$. When the domain of $q$ is $\{1,...,k\}$ for some $k\in\mathbb{N}$, let $\mu_{T}(q,e)=f_{q(1)}^{e_{1}}...f_{q(k)}^{e_{k}}\mu_{I}(q,e)$. When the domain of $q$ is $\mathbb{N}$, for each $k\in\mathbb{N}$ define the function $q^{k}(i)=q(i)$ where $i=1,..,k$ and the sequence $e^{k}_{i}=e_{i}$ where $i=1,...,k$. Let $\delta_{k}=f_{q(1)}^{e_{1}}...f_{q(k)}^{e_{k}}(\alpha)$. By the set up $\delta_{k}\subset \base(\mu_{T}(q^{k} ,e^{k}))$ and $\partial{Z}_{k+1}(q,e)\subseteq \delta_{k}$. Let $\delta_{k_{m}}$ be a subsequence of $\delta_{k}$ convergent to a lamination $\lambda$ in the Hausdorff topology of $M_{\infty}(S)$. As in $\S$\ref{subsec : pre1}(Scheme I) remove isolated leaves of $\lambda$ and denote the resulting lamination by $\mu'$. Then add simple closed curves disjoint from $\mu'$ with bounded intersection number with the isolated leaves of $\lambda$ to obtain a generalized marking $\mu_{T}(q,e)$ with the property that for any subsurface $Y$ intersecting $\lambda$ and $\mu_{T}(q,e)$,
$$|d_{Y}(\mu_{T}(q,e),\lambda)|\leq d,$$
 where $d$ depends on the upper bound for the intersection numbers.

\begin{remark}
The construction of this subsection and the estimates for subsurface coefficients can be carried out in a more general setting. Here we restrict ourself to be able to provide detailed step by step estimates and complete arguments.
\end{remark}

\begin{prop}\label{prop : pre2}
There are constants $K_{1}\geq 1, C_{1}\geq 0$, $E>0$ and ${\bf m},{\bf m}'>4M$, depending on $f_{0},f_{1}$ and $\mu_{I}$ with the following properties. Given $q$ and $\{e_{i}\}_{i}$  such that $|e_{i}|>E$ for any $i$ in the domain of $q$, we have
\begin{enumerate}[(i)]
\item\label{pre2 : ld} For any integer $i$ in the domain of $q$ with $q(i)=1$, 
$$d_{Z_{i}(q,e)}(\mu_{I}(q,e),\mu_{T}(q,e))\asymp_{K_{1},C_{1}} |e_{i}|.$$
\item\label{pre2 : ord} Given $i,j$ in the domain of $q$, if $i<j$ and $q(i)=q(j)=1$ then $Z_{i}(q,e)<Z_{j}(q,e)$. 
 \item\label{pre2 : bd} For any non-annular subsurface $W$ which is neither $Z_{i}(q,e)$ for some $i$ nor $S$ we have $d_{W}(\mu_{I}(q,e),\mu_{T}(q,e)) \leq {\bf m}$.
\item\label{pre2 : bda} For any $\gamma\in \mathcal{C}_{0}(S)$ we have $d_{\gamma}(\mu_{I}(q,e),\mu_{T}(q,e))\leq {\bf m}'$.
\item \label{pre2 : minfill} If the domain of $q$ is $\mathbb{N}$, then $\mu_{T}(q,e)$ is a minimal filling lamination.
\end{enumerate}
\end{prop}

\begin{proof}
First we prove the statements (\ref{pre2 : ld}) to (\ref{pre2 : bda}) when for some $k\in\mathbb{N}$, $q:\{1,...,k\}\to \{0,1\}$ and $\{e_{i}\}_{i=1}^{k}$ is a sequence of integers with $k$ elements. Most of the details are similar to the ones given in $\S$\ref{subsec : pre1}(Scheme I) so here we mainly sketch them and explain the necessary modifications. Denote $\mu_{I}$ by $\mu$. We set the constants: 
$$K'_{1}=\min\{\tau_{a}:a=0,1\},$$
 where the constant $\tau_{a}:=\tau_{f_{a}}$, $a=0,1$, is from Proposition \ref{prop : palb},
 $$C'_{1}=2(B_{0}+\eta+2),$$
 where $\eta=\max\{d_{X_{a}}(f_{b}^{e}\mu,\mu) : a,b\in \{0,1\}\;\text{and }\; a\neq b\}$, and 
 $$E_{1}=\frac{B_{0}+4+4M+\omega+C_{1}'}{K'_{1}},$$ 
 where  $\omega=\max\{d_{W}(\mu,\partial{X_{a}}) : W\subseteq S \;\text{and}\; a=0,1\}$. 

Let $i\in \{1,...,k\}$ be such that $q(i)=1$. Let $g=f_{q(1)}^{e_{1}}...f_{q(i-1)}^{e_{i-1}}f_{q(i)}^{e_{i}}$. The partial pseudo-Anosov map $f_{q(i)}$ is supported on $S\backslash\alpha$ and preserves $\alpha$. Moreover $q(i+2)=1$ so $X_{q(i+2)}=S\backslash\alpha$. Then applying $g^{-1}$ to $d_{S}(\partial{Z}_{i},\partial{Z}_{i+2})$ we get $d_{S}(\alpha,f_{q(i+1)}^{e_{i+1}}(\alpha))$. Furthermore, $f_{q(i+1)}=f_{0}$ is supported on $S$, so by Proposition \ref{prop : palb}, 
\begin{equation}\label{eq : dSaf(i+1)a}d_{S}(\alpha, f_{q(i+1)}^{e_{i+1}}(\alpha))\geq \tau_{0}|e_{i+1}|,\end{equation}
 thus we have
\begin{equation}\label{eq : dSbdZ-+}d_{S}(\partial{Z}_{i},\partial{Z}_{i+2})\geq \tau_{0}|e_{i+1}|.\end{equation}

Let $i\in \{1,...,k\}$ be such that $q(i)=1$. Then similar to the proof of Lemma \ref{lem : ldomain}(\ref{ldomain : ld}) we may obtain
\begin{equation}\label{eq : ldomain'}d_{Z_{i}(q,e)}(\mu_{I}(q,e),\mu_{T}(q,e))\geq K'_{1}|e_{i}|-C'_{1}.\end{equation}
  The only difference is that instead of Claim \ref{claim : bdXqi-1i+1} we have the following claim: For any $i\in\{1,...,k\}$ with $q(i)=1$ we have that $X_{q(i)}\pitchfork f_{q(i+1)}^{e_{i+1}}(X_{q(i+2)})$. The proof is as follows: $\partial{X}_{q(i)}=\alpha$ and since $q(i+2)=1$, $\partial{X}_{q(i+2)}=\alpha$. Moreover by (\ref{eq : dSaf(i+1)a}), 
  $$d_{S}(\alpha,f_{q(i+1)}^{e_{i+1}}(\alpha))\geq\tau_{0}E_{1}>4.$$
 This implies that $\alpha\pitchfork f_{q(i+1)}^{e_{i+1}}(\alpha)$ and therefore $X_{q(i)}\pitchfork f_{q(i+1)}^{e_{i+1}}(X_{q(i+2)})$.  
   
 Then as part (\ref{ldomain : ord}) of Lemma \ref{lem : ldomain} by induction on $k$ we may show that the subsurfaces with $q(i)=1$ are ordered. Here we establish the base of the induction for $k=3$ as follows:  By (\ref{eq : dSbdZ-+}) and since $|e_{2}|>E_{1}$ we have that $d_{S}(\partial{Z}_{1},\partial{Z}_{3})\geq 3$. Thus $\partial{Z}_{1}\pitchfork \partial{Z}_{3}$. Let $q'(i)=q(i)$ for $i=1,2$ and $e'_{i}=e_{i}$ for $i=1,2$. Then since $\mu_{I}(q,e)=\mu_{I}(q',e')$. $Z_{1}(q,e)=Z_{1}(q',e')$ and $\partial{Z}_{3}(q,e)\subset\mu_{T}(q',e')$, (\ref{eq : ldomain'}) for $q',e'$ gives us $d_{Z_{1}}(\mu_{I},\partial{Z}_{3})>2M$. Then by Definition \ref{def : ordersubsurf} we have $Z_{1}<Z_{3}$.
 
 We need the following lemma for the rest of the proof. 
\begin{lem}\label{lem : Wub2}
There is a constant $E_{2}>E_{1}$, depending on $f_{0},f_{1}$ and $\mu_{I}$ with the following property. Given $q$ and $\{e_{i}\}_{i=1}^{k}$ such that $|e_{i}|>E_{2}$ for any $i\in\{1,..., k\}$, we have 
\begin{enumerate}[(i)]
\item \label{Wub2 : bd} Suppose that $W$ is a non-annular subsurface which is neither $Z_{i}(q,e)$ for some $i\in\{1,...,k\}$ nor $S$. If  $Z_{r}(q,e)<W<Z_{s}(q,e)$ for some $r,s\in\{1,...,k\}$ with $r<s$, then there is a constant $m>0$ depending on $s-r$, such that $d_{W}(\mu_{I}(q,e),\mu_{T}(q,e)) \leq m$. If $W<Z_{s}(q,e)$ for some $s\in\{1,...,k\}$, then there is a constant $m>0$ depending on $s$, so that $d_{W}(\mu_{I}(q,e),\partial{Z}_{s}(q,e))\leq m$. Also if $Z_{r}(q,e)<W$ for some $r\in\{1,...,k\}$, then there is a constant $m>0$ depending on $k-r$, so that $d_{W}(\partial{Z}_{r}(q,e),\mu_{T}(q,e))\leq m$. 
\item \label{Wub2 : abd} Suppose that  $A(\gamma)$ is an annular subsurface. If $Z_{r}(q,e)<\gamma<Z_{s}(q,e)$ for some $r,s\in\{1,...,k\}$ with $r<s$, then there is a constant $m'>0$ depending on $s-r$, such that $d_{\gamma}(\mu_{I}(q,e),\mu_{T}(q,e))\leq m'$. If $\gamma<Z_{s}(q,e)$ for some $s\in\{1,...,k\}$, then there is a constant $m'>0$ depending on $s$, so that $d_{\gamma}(\mu_{I}(q,e),\partial{Z}_{s}(q,e))\leq m'$. Also if $Z_{r}(q,e)<W$ for some $r\in\{1,...,k\}$, then there is a constant $m'>0$, depending on $k-r$, so that $d_{\gamma}(\partial{Z}_{r}(q,e),\mu_{T}(q,e))\leq m'$.
\end{enumerate}
\end{lem}
\begin{proof}
{\it Proof of part (\ref{Wub2 : bd}):} 

First suppose that $W$ overlaps all of the boundary curves $\partial{Z}_{j}$, where $r<j<s$ and $q(j)=1$.  Let $g=f_{q(1)}^{e_{1}}...f_{q(j)}^{e_{j}}$. Note that $f_{q(j)}$ is supported on the subsurface $S\backslash\alpha$ and preserves the curve $\alpha$. Thus applying $g^{-1}$ to $d_{W}(\partial{Z}_{j},\partial{Z}_{j+2})$ we obtain 
$$d_{g^{-1}(W)}(\alpha,f_{q(j+1)}^{e_{j+1}}(\alpha)).$$
 By Lemma \ref{lem : paub} the above subsurface coefficient is uniformly bounded by some $h$ depending on $f_{q(j+1)}=f_{0}$ and $\alpha$. Then by the triangle inequality we have that
\begin{eqnarray*}
d_{W}(\partial{Z}_{r},\partial{Z}_{s})&\leq&\sum_{\substack{ j:q(j)=1\;\text{and}\\ r\leq j< s}} d_{W}(\partial{Z}_{j},\partial{Z}_{j+2})+\sum_{\substack{ j:q(j)=1\;\text{and}\\ r< j<s}}\diam_{W}(\partial{Z}_{j})\\
&\leq&h\lfloor \frac{s-r}{2}\rfloor+2\lfloor \frac{s-r}{2}\rfloor,
\end{eqnarray*}
which is the desired bound.

Now suppose that $W$ does not overlap a boundary curve $\partial{Z}_{j}$ where $r<j<s$ and $q(j)=1$. Then since $Z_{j}$ is a large subsurface, $W\subseteq Z_{j}$, and since $W$ is not in the list of subsurfaces $Z_{i}$, $W\subsetneq Z_{j}$.

Let $g=f_{q(1)}^{e_{1}}....f_{q(j-1)}^{e_{j-1}}$. Applying $g^{-1}$ to the subsurface coefficient $d_{W}(\partial{Z}_{r},\partial{Z}_{s})$ we get 
$$d_{g^{-1}(W)}(f_{q(j-1)}^{-e_{j-1}}...f_{q(r)}^{-e_{r}}(\partial{X_{q(r)}}),f_{q(j)}^{e_{j}}...f_{q(s-1)}^{e_{s-1}}(\partial{X_{q(s)}})).$$
Denote curves $\partial^{-}=f_{q(j-1)}^{-e_{j-1}}...f_{q(r)}^{-e_{r}}(\partial{X_{q(r)}})$ and $\partial^{+}=f_{q(j+1)}^{e_{j+1}}...f_{q(s-1)}^{e_{s-1}}(\partial{X_{q(s)}})$ and the subsurface $Y=g^{-1}(W)$. The above subsurface coefficient may be written as 
\begin{equation}\label{eq2 : dg-W}d_{Y}(\partial^{-},f_{q(j)}^{e_{j}}(\partial^{+})).\end{equation}

Since $Y\subset X_{q(j)}$, and $\partial^{-},f^{e_{j}}_{q(j)}(\partial^{+})$ both intersect $Y$ essentially the second part of Lemma \ref{lem : diamproj} guarantees that $\diam_{Y}(\pi_{Y}(\pi_{X_{q(j)}}(\partial^{-}))\cup\pi_{Y}(\partial^{-}))$ and $\diam_{Y}(\pi_{Y}(\pi_{X_{q(j)}}(f_{q(j)}^{e_{j}}(\partial^{+})))\cup\pi_{Y}(f_{q(j)}^{e_{j}}(\partial^{+})))$ are uniformly bounded. Therefore,
\begin{equation}\label{eq2 : dYd-fd+=dYXd-fd+}d_{Y}(\partial^{-},f^{e_{j}}_{q(j)}(\partial^{+}))\asymp d_{Y}(\pi_{X_{q(j)}}(\partial^{-}),\pi_{X_{q(j)}}(f_{q(j)}^{e_{j}}(\partial^{+}))).\end{equation}
 
 For partial pseudo-Anosov maps $f_{a}$, $a=0,1$, let $\Psi^{\pm}_{a}$ and $\Delta^{\pm}_{a}$ be the subsets of $\mathcal{PML}(S)$ defined at the beginning of $\S$\ref{sec : preendinv}. Since the support of each $f_{a}$ is a large subsurface we have that $\Delta^{\pm}_{a}=[\mathcal{L}^{\pm}_{a}]$, where $\mathcal{L}^{\pm}_{a}$ is the attracting/repelling measured lamination of $f_{a}$. Denote the support of $\mathcal{L}_{a}^{\pm}$ by $\lambda_{a}^{\pm}$. By Proposition \ref{prop : pAlam} there are finitely many laminations that contain $\lambda_{a}^{\pm}$. Denote the set of laminations containing $\lambda_{a}^{\pm}$ by $\Lambda_{a}^{\pm}$. 
  
 We have that $q(j)=1$, so $X_{q(j+1)}\subsetneq X_{q(j)}$. Also $Y\subsetneq X_{q(j)}$. Thus any leaf of $\lambda_{q(j+1)}^{+}$ intersects both $X_{q(j)}$ and $Y$ essentially (Proposition \ref{prop : pAlam}). Any lamination $\lambda\in\Lambda_{q(j+1)}^{+}$ contains $\lambda^{+}_{q(j+1)}$, so there are leaves of $\lambda$ that intersect both $X_{q(j)}$ and $Y$ essentially. Similarly, any lamination $\lambda'\in\Lambda_{q(j-1)}^{-}$ contains leaves which intersect both $X_{q(j)}$ and $Y$ essentially.  So there are infinite leaves of $\lambda$ that intersect $X_{q(j)}$ and $Y$ essentially. Similarly,  there are leaves of any lamination $\lambda'\in\Lambda_{q(j-1)}^{-}$ which intersect $X_{q(j)}$ and $Y$ essentially. Thus the following subsurface projections are well-defined an non-empty: $\pi_{X_{q(j)}}(\lambda)$, $\pi_{X_{q(j)}}(\lambda')$, $\pi_{Y}(\pi_{X_{q(j)}}(\lambda))$, $\pi_{Y}(\pi_{X_{q(j)}}(\lambda'))$, $\pi_{Y}(\lambda)$ and $\pi_{Y}(\lambda')$. 
 
 There are finitely many laminations in $\bigcup_{a=0,1}\Lambda_{a}^{\pm}$. Applying Lemma \ref{lem : Hdistdprojbd} to each one of these laminations and subsurfaces $X_{a}$, $a=0,1$, and taking the minimum of the constants we obtain, there is an $\epsilon>0$ so that: If a lamination $\lambda\in\Lambda_{a}^{\pm}$ intersects a subsurface $X_{b}$ essentially, where $a,b\in\{0,1\}$, and a curve $\gamma$ is within the $\epsilon$ Hausdorff distance of  $\lambda$, then $\diam_{X_{b}}(\pi_{X_{b}}(\gamma)\cup\pi_{X_{b}}(\lambda))\leq 4$.

For each $a\in\{ 0,1\}$, let neighborhoods $U^{\pm}_{a}$ of $[\mathcal{L}_{a}^{\pm}]$ in $\mathcal{PML}(S)$ be so that
 \begin{itemize} 
 \item$\overline{U_{a}^{-}}$ (the closure of $U_{a}^{-}$) and $\overline{U_{a}^{+}}$ (the closure of $U_{a}^{+}$) are disjoint from $\Psi_{b}^{-}\cup \Psi_{b}^{+}$ for any $b\in \{0,1\}$ with $X_{a}\neq X_{b}$. 
  \item Every lamination in $U^{+}_{a}$ is in the $\epsilon$ Hausdorff distance of a lamination $\lambda\in \Lambda^{+}_{a}$ and every lamination in $U^{-}_{a}$ is in the $\epsilon$ Hausdorff distance of a lamination $\lambda'\in \Lambda^{-}_{a}$.
 \end{itemize}
 Lemma \ref{lem : Hnbhdl} guarantees that for $\epsilon$ as above we may choose the neighborhoods $U^{\pm}_{a}$ so that the second bullet holds. 

 Applying Theorem \ref{thm : unifconv} to the pseudo-Anosov map $f_{a}$ and the compact sets $\partial{X}_{b}$, $\overline{U_{b}^{-}}$ and $\overline{U_{b}^{+}}$, for any $a,b\in\{0,1\}$ with $X_{a}\neq X_{b}$ and taking maximum of the constants we obtain, there exists $E_{2}>E_{1}$ such that 
\begin{itemize}
\item $f_{0}^{n}(\partial{X}_{1}) \subset U_{0}^{+}$ for all $n\geq E_{2}$ and $f_{0}^{-n}(\partial{X}_{1})\subset U_{0}^{-}$ for all $n\geq E_{2}$ and
\item $f_{a}^{n}(\overline{U_{b}^{\pm}}) \subset U_{a}^{+}$ for all $n\geq E_{2}$ and $f_{a}^{-n}(\overline{U_{b}^{\pm}})\subset U_{a}^{-}$ for all $n\geq E_{2}$.
\end{itemize}
 Then assuming that $|e_{i}|>E_{2}$ for every $i\in\{1,...,k\}$, $\partial^{+}\in U_{q(j+1)}^{+}$ and $\partial^{-}\in U_{q(j-1)}^{-}$. 
 
Then there is a lamination $\lambda\in \Lambda_{q(j)}^{+}$ so that $\partial^{+}$ is within the $\epsilon$ Hausdorff distance of $\lambda$. Similarly there is a lamination $\lambda'\in \Lambda_{q(j-1)}^{-}$ so that $\partial^{-}$ is within the $\epsilon$ Hausdorff distance of $\lambda'$. Then by the choice of $\epsilon$ we have 
$$\diam_{X_{q(j)}}(\pi_{X_{q(j)}}(\partial^{-})\cup\pi_{X_{q(j)}}(\lambda'))\leq 4.$$
 Similarly by the choice of $\epsilon$ we have $\diam_{X_{q(j)}}(\pi_{X_{q(j)}}(\partial^{+})\cup \pi_{X_{q(j)}}(\lambda))\leq 4$. Then by the fact that $f_{q(j)}$ acts as isometry on $\mathcal{C}(S)$ we have 
 $$\diam_{X_{q(j)}}(f_{q(j)}^{e_{j}}(\pi_{X_{q(j)}}(\partial^{+}))\cup f_{q(j)}^{e_{j}}(\pi_{X_{q(j)}}(\lambda)))\leq 4.$$
  Then the rest of the proof is similar to the proof Lemma \ref{lem : Wub1}(\ref{Wub1 : bd}) and gives us the upper bound for (\ref{eq2 : dg-W}).
\medskip

 Now let us bound $d_{W}(\mu_{I},\partial{Z}_{s})$. First suppose that $W$ overlaps all of the boundary curves $\partial{Z}_{j}$, where $1<j<s$ and $q(j)=1$. Then similar to the proof of part (\ref{Wub2 : bd}) we may obtain the bound
 $$d_{W}(\mu_{I},\partial{Z}_{s})\leq h\lfloor\frac{s-1}{2}\rfloor+2\lfloor\frac{s-1}{2}\rfloor.$$
  Now suppose that $W$ does not overlap at least one boundary curve $\partial{Z}_{j}$ with $q(j)=1$ and $1\leq j\leq s$. Then the bound for $d_{W}(\mu_{I},\partial{Z}_{s})$ follows from exact the same argument we gave above to bound $d_{W}(\partial{Z}_{r},\partial{Z}_{s})$. 
  
  Define the function $q'(i)=q(k-i+1)$ for $i=1,...,k$, and the sequence $e'_{i}=e_{k-i+1}$ for $i=1,...,k$. Then the bound for $d_{W}(\partial{Z}_{r},\mu_{T})$ can be obtained similar to that of $d_{W}(\mu_{I},\partial{Z}_{s})$ considering $q'$ and $e'$.
 \medskip

{\it Proof of part (\ref{Wub2 : abd}).} Suppose that $\gamma$ overlaps all boundary curves $\partial{Z}_{j}$, where $r\leq j\leq s$ with $q(j)=1$. Then as the beginning of the proof of part (\ref{Wub2 : bd}) we may obtain the bound
$$d_{\gamma}(\partial{Z}_{r},\partial{Z}_{s})\leq h\lfloor\frac{r-s}{2}\rfloor+2\lfloor\frac{r-s}{2}\rfloor.$$
Now suppose that $\gamma$ does not overlap at least one of the boundary curves $\partial{Z}_{j}$ where $r<j<s$ and $q(j)=1$.

Let $g=f_{q(1)}^{e_{1}}...f_{q(j-1)}^{e_{j-1}}$, applying $g^{-1}$ to $d_{\gamma}(\partial{Z}_{r},\partial{Z}_{s})$ we obtain
$$d_{g^{-1}(\gamma)}(f_{q(j-1)}^{-e_{j-1}}...f_{q(r)}^{-e_{r}}(\partial{X_{q(r)}}),f_{q(j)}^{e_{j}}...f_{q(s-1)}^{e_{s-1}}(\partial{X_{q(s)}})).$$
Denote curves $\partial^{-}=f_{q(j-1)}^{-e_{j-1}}...f_{q(r)}^{-e_{r}}(\partial{X_{q(r)}})$ and $\partial^{+}=f_{q(j+1)}^{e_{j+1}}...f_{q(s-1)}^{e_{s-1}}(\partial{X_{q(s)}})$ and the subsurface $\delta=g^{-1}(\gamma)$. The above subsurface coefficient may be written as
\begin{equation}\label{eq2 : dg-gamma}d_{\delta}(\partial^{-},f_{q(j)}^{e_{j}}(\partial^{+})).\end{equation}
Since $\delta\subset X_{q(j)}$, and $\partial^{-},f^{e_{j}}_{q(j)}(\partial^{+})$ both intersect $\delta$ essentially the second part of Lemma \ref{lem : diamproj} guarantees that $\diam_{\delta}(\pi_{\delta}(\pi_{X_{q(j)}}(\partial^{-}))\cup\pi_{\delta}(\partial^{-}))$ and $\diam_{\delta}(\pi_{\delta}(\pi_{X_{q(j)}}(f_{q(j)}^{e_{j}}(\partial^{+})))\cup\pi_{\delta}(f_{q(j)}^{e_{j}}(\partial^{+})))$ are uniformly bounded. Therefore,
\begin{equation}\label{eq2 : ddeld-fd+=ddelXd-fd+}d_{\delta}(\partial^{-},f^{e_{j}}_{q(j)}(\partial^{+}))\asymp d_{\delta}(\pi_{X_{q(j)}}(\partial^{-}),\pi_{X_{q(j)}}(f_{q(j)}^{e_{j}}(\partial^{+}))).\end{equation}

Suppose that $\gamma\not\in\partial{Z}_{j}$. Then $\gamma\in\mathcal{C}_{0}(Z_{j})$ and $\delta\in\mathcal{C}_{0}(X_{q(j)})$. 

Let the finite set of laminations $\Lambda_{a}^{\pm}$, $a=0,1$ be as in the proof of part (\ref{Wub2 : bd}).

We have that $q(j)=1$ so $X_{q(j)}\subsetneq X_{q(j+1)}$. Also $\delta\subsetneq X_{q(j)}$. Thus any leaf of $\lambda_{q(j+1)}^{+}$ intersects both $X_{q(j)}$ and $Y$ essentially (Proposition \ref{prop : pAlam}). Any lamination $\lambda\in\Lambda_{q(j+1)}^{+}$ contains $\lambda^{+}_{q(j+1)}$, so there are leaves of $\lambda$ that intersect $X_{q(j)}$ and $\delta$ essentially. Similarly,  there are leaves of any lamination $\lambda'\in\Lambda_{q(j-1)}^{-}$ which intersect $X_{q(j)}$ and $\delta$ essentially. Thus the following subsurface projections are well-defined and non-empty: $\pi_{X_{q(j)}}(\lambda)$, $\pi_{X_{q(j)}}(\lambda')$, $\pi_{\delta}(\pi_{X_{q(j)}}(\lambda))$, $\pi_{\delta}(\pi_{X_{q(j)}}(\lambda'))$, $\pi_{\delta}(\lambda)$ and $\pi_{\delta}(\lambda')$.

Let $\epsilon>0$ and the neighborhoods $U_{a}^{\pm}$ of $[\mathcal{L}_{a}^{\pm}]$ ($a=0,1$) and the constant $E_{2}$ be as in the proof of part (\ref{Wub2 : bd}).

 Then by the assumption that $|e_{i}|>E_{2}$, for any $i\in\{1,...,k\}$, $\partial^{+}\subset U^{+}_{q(j+1)}$ and $\partial^{-}\subset U^{-}_{q(j-1)}$.
 
 Then there is a lamination $\lambda\in \Lambda_{q(j+1)}^{+}$ so that $\partial^{+}$ is within the $\epsilon$ Hausdorff distance of $\lambda$. Similarly there is a lamination $\lambda'\in \Lambda_{q(j-1)}^{-}$ so that $\partial^{-}$ is within the $\epsilon$ Hausdorff distance of $\lambda'$. Then by the choice of $\epsilon$ we have 
 $$\diam_{X_{q(j)}}(\pi_{X_{q(j)}}(\partial^{-})\cup\pi_{X_{q(j)}}(\lambda'))\leq 4.$$
  Similarly by the choice of $\epsilon$ we have $\diam_{X_{q(j)}}(\pi_{X_{q(j)}}(\partial^{+})\cup\pi_{X_{q(j)}}(\lambda)\leq 4$. Then by the fact that $f_{q(j)}$ acts as isometry on $\mathcal{C}(S)$ we have 
  $$\diam_{X_{q(j)}}(f_{q(j)}^{e_{j}}(\pi_{X_{q(j)}}(\partial^{+}))\cup f_{q(j)}^{e_{j}}(\pi_{X_{q(j)}}(\lambda))\leq 4.$$
   The rest of the proof is similar to the proof of Lemma \ref{lem : Wub1}(\ref{Wub1 : abd}) and gives us the upper bound for (\ref{eq2 : dg-gamma}).
 \medskip

Suppose that $\gamma\in \partial{Z}_{j}$, then $\delta\in\partial{X}_{q(j)}$. Then by the assumption that the restriction of $f_{q(j)}$ to a neighborhood of $\delta\in\partial{X}_{a}$ is the identity map we have that $d_{\delta}(\lambda',f_{q(j)}^{e_{j}}(\lambda))\asymp d_{\delta}(\lambda',\lambda)$. The subsurface coefficient $d_{\gamma}(\lambda',\lambda)$ is bounded by 
$$\max\{d_{\delta}(\lambda',\lambda):\lambda'\in \Lambda_{q(j-1)}^{-},\lambda\in\Lambda_{q(j+1)}^{+},\delta\in\partial{X}_{q(j)}\}.$$
 Since there are finitely many laminations in $\bigcup_{a=0,1}\Lambda_{a}^{\pm}$ the above maximum exists and is finite. So we obtain the bound for (\ref{eq2 : dg-gamma}).
 \medskip

 The bounds for $d_{\gamma}(\mu_{I},\partial{Z}_{s})$ and $d_{\gamma}(\partial{Z}_{r},\mu_{T})$ may be obtained similarly.


\end{proof}

{\it Proof of part (\ref{pre2 : bd}).} Suppose that $d_{W}(\mu_{I},\mu_{T})>4M$ (otherwise we already have the bound). 
\begin{claim}\label{claim2 : WZi}
 If $W\pitchfork Z_{i}$, then $W$ and $Z_{i}$ are ordered.
 \end{claim}
By Lemma \ref{lem : ldomain}, $d_{Z_{i}}(\mu_{I},\mu_{T})\geq K'_{1}|e_{i}|-C'_{1}$ for any $i\in\{1,...,k\}$. Then since $|e_{i}|>E_{2}$,  
$$d_{Z_{i}}(\mu_{I},\mu_{T})>4M.$$
 Moreover, by our assumption, $d_{W}(\mu_{I},\mu_{T})>4M$. Then the claim holds by Definition \ref{def : ordersubsurf}.
 \medskip

 Let $l\in\{1,...,k\}$ be such that $q(l)=0$. Then $q(l+1)=1$, so by (\ref{eq : dSbdZ-+}) and the choice of $E_{1}$, we have that 
 $$d_{S}(\partial{Z}_{l+1},\partial{Z}_{l+3})\geq\tau_{0}|e_{l+2}|\geq \tau_{0}E_{1}>4.$$
  The above inequality implies that $\partial{Z_{l+1}}$ and $\partial{Z_{l+3}}$ fill $S$ (see $\S$\ref{sec : ccplx}). Thus 
\begin{itemize}
\item $W\pitchfork Z_{l+1}$ or $W\pitchfork Z_{l+3}$. 
\end{itemize}
Similarly 
\begin{itemize}
\item $W\pitchfork Z_{l-1}$ or $W\pitchfork Z_{l-3}$.
\end{itemize}
Claim \ref{claim2 : WZi} and the above two bullets imply that given $l\in\{1,...,k\}$ with $q(l)=1$ or $3$, $W$ and subsurfaces $Z_{r}$ and $Z_{s}$ are ordered, where $r=l-1$ or $l-3$ and $s=l+1$ or $l+3$. If $Z_{r}<W<Z_{s}$ we are in Case (1) below. Otherwise, either $Z_{s}<W$ (by transitivity of $<$ also $Z_{r}<W$), or $W<Z_{r}$ (by transitivity also $W<Z_{s}$). Repeating this argument for all $l\in\{1,...,k\}$ with $q(l)=1$ or $3$ we will end up in one of the following cases:  
 \begin{enumerate}
\item $Z_{r} < W < Z_{s}$ where $s-r\leq 6$, 
\item $W<Z_{s}$, where $s\leq 3$,
\item $Z_{r}<W$, where $k-r\leq 3$. 
\end{enumerate}
We proceed to establish a bound for $d_{W}(\mu_{I},\mu_{T})$ in each of these cases.  
\medskip

\noindent{\bf Case (1).} By Definition \ref{def : ordersubsurf}, $d_{W}(\mu_{I},\partial{Z}_{r})\leq M$ and $d_{W}(\partial{Z}_{s},\mu_{T})\leq M$. Then by the triangle inequality and this bound we get
$$d_{W}(\mu_{I},\mu_{T})\leq d_{W}(\partial{Z_{r}},\partial{Z_{s}})+2M+4.$$
The subsurface coefficient $d_{W}(\partial{Z_{r}},\partial{Z_{s}}) $ is bounded above by Lemma \ref{lem : Wub2}. Thus we obtain the desired bound.
\medskip

\noindent{\bf Case (2).} By Definition \ref{def : ordersubsurf}, $d_{W}(\partial{Z}_{s},\mu_{T})\leq M$. Then by the triangle inequality and this bound we get
$$d_{W}(\mu_{I},\mu_{T})\leq d_{W}(\mu_{I},\partial{Z_{s}})+M+2,$$
The subsurface coefficient $d_{W}(\mu_{I},\partial{Z_{s}}) $ is bounded above by Lemma \ref{lem : Wub2}. Thus we obtain the desired bound.
 \medskip
 
\noindent{\bf Case (3).} This case can be treated similar to Case (2).
\medskip

The bounds in Cases (1), (2) and (3) depend on the upper bound for $s-r$,$s$, $k-r$ respectively, the marking $\mu_{I}$ and pseudo-Anosov maps $f_{0},f_{1}$. 
Establishing the bounds there is a constant ${\bf m}$ so that 
$$d_{W}(\mu_{I},\mu_{T})\leq {\bf m}$$
 for any proper, non-annular subsurface $W$ which is not in the list of subsurfaces $Z_{i}$. The proof of part (\ref{pre2 : bd}) is complete.
  \medskip

 {\it Proof of part (\ref{pre2 : bda}).} Let $l\in\{1,...,k\}$ with $q(l)=0$. Then $q(l+2)=0$, thus by (\ref{eq : dSbdZ-+}) we have that $d_{S}(\partial{Z}_{l+1},\partial{Z}_{l+3})\geq 3$. Thus $\partial{Z}_{l+1}$ and $\partial{Z}_{l+3}$ fill $S$. Similarly $d_{S}(\partial{Z}_{l-1},\partial{Z}_{l-3})\geq 3$, so $\partial{Z}_{l-1}$ and $\partial{Z}_{l-3}$ fill $S$. Thus 
\begin{itemize}
\item $\gamma\pitchfork \partial{Z}_{l+1}$ or $\gamma\pitchfork \partial{Z}_{l+3}$, and
\item $\gamma\pitchfork \partial{Z}_{l-1}$ or $\gamma\pitchfork \partial{Z}_{l-3}$.
\end{itemize}
Then similar to the proof of part (\ref{subsurfub : abd}) of Proposition \ref{prop : subsurfub} we can show that $\gamma$ is ordered in the list of subsurfaces $Z_{i}$ as one of the following cases:
\begin{enumerate}
\item $Z_{r} < \gamma < Z_{s}$, where $s-r\leq 6$,
\item $\mu_{I}<\gamma<Z_{s}$, where $s\leq 3$,
\item $Z_{r}<\gamma<\mu_{T}$, where $k-r\leq 3$.
\end{enumerate}
\medskip

In Case (1) we have that
$$d_{\gamma}(\mu_{I},\mu_{T})\leq d_{\gamma}(\partial{Z}_{r},\partial{Z}_{s})+2M+4.$$
The subsurface coefficient $d_{\gamma}(\partial{Z}_{r},\partial{Z}_{s})$ is bounded by Lemma \ref{lem : Wub2}(\ref{Wub2 : abd}). The desired bound for $d_{\gamma}(\mu_{I},\mu_{T})$ follows. Similarly we may obtain the bounds in Cases (2) and (3).  The bounds in Cases (1), (2) and (3) depend on the upper bound for $s-r,r,k-r$ respectively, the marking $\mu_{I}$ and the partial pseudo-Anosov maps $f_{0},f_{1}$. Then there is a constant ${\bf m}'$ so that 
$$d_{\gamma}(\mu_{I},\mu_{T})\leq {\bf m}'$$
 for any $\gamma\in\mathcal{C}_{0}(S)$. The proof of part (iv) is complete.
\medskip

The lower bound  (\ref{eq : ldomain'}) and part (\ref{pre2 : bd}) together as in the proof of Proposition \ref{prop : ldomain} give us part (\ref{pre2 : ld}). As in the proof of Proposition \ref{prop : ldomain} we set $E>E_{2}$ such that for any integer $e$ with $|e|>E$ we have 
$$d_{X_{a}}(\delta,f_{a}^{e}(\delta))\leq (\bar{\tau}_{a}+1)|e|,$$ 
for every $\delta$ in the ${\bf m}-$neighborhood of $\pi_{X_{a}}(\mu)$ in $\mathcal{C}(X_{a})$ and $a=0,1$ (${\bf m}$ is the constant from part (\ref{pre2 : bd})). Then for $K_{1}=\max\{\frac{1}{\tau_{a}},\bar{\tau}_{a}+1\}$, $C_{1}=C'_{1}$ and $E$, part (\ref{pre2 : ld}) holds.
\medskip

We established all of the bounds when the domain of $q$ is $\{1,...,k\}$ for some $k\in\mathbb{N}$. Then the bounds when the domain of $q$ is $\mathbb{N}$ and the fact that $\mu_{T}$ is a minimal filling lamination (part (\ref{pre2 : minfill})) follow from the limiting argument we gave in the proof of Proposition \ref{prop : pre1}. The order of subsurfaces (part (\ref{pre2 : ord})) follows from the order of subsurfaces when the domain of $q$ is finite which we established above and the limiting argument we gave for the proof of Proposition \ref{prop : pre1}(\ref{pre1 : ord}).

\end{proof}


\section{Behavior of geodesics}\label{sec : examples}

In this section we use the control on length-functions along WP geodesic segments from $\S$\ref{sec : itinerarywpgeod} and the pair of laminations or markings with prescribed list of subsurface coefficients from $\S$\ref{sec : preendinv} to construct examples of Weil-Petersson geodesics with certain behavior in the moduli space. For example divergent rays in the moduli space and closed geodesics in the thin part of the moduli space. We also provide a recurrence condition for WP geodesics in terms of ending laminations. Our results in this section can be considered as a kind of symbolic coding of WP geodesics.

In \cite{dynwppww} the authors construct WP geodesics which are dense in the moduli space when $\xi(S)=1$. Jeffery Brock \cite{brockcommunication} constructs divergent WP geodesic rays with minimal filling ending laminations. Both constructions start with a piecewise geodesic in the Weil-Petersson completion of the Teichm\"{u}ller space. Then applying high powers of Dehn twists about curves in the multi-curves which determine the strata that the piecewise geodesic intersects manage to replace the piecewise geodesic with a piecewise geodesic in the completion of the Teichm\"{u}ller space with arbitrary small exterior angles. Then using a variation of shadowing lemma perturb the piecewise geodesic to a single WP geodesic in the Teichm\"{u}ller space. In these constructions the relation between the itinerary of the ray and the end invariants and their associated subsurface coefficients is not explicit.

These constructions are analogue of the ones for Teichm\"{u}ller geodesics.  Cheung and Masur in \cite{cmdiv} give examples of divergent Teichm\"{u}ller geodesic rays with uniquely ergodic vertical lamination. Rafi constructs closed Teichm\"{u}ller geodesics staying in the thin part of the moduli space and divergent geodesic rays. Rafi uses the control on length-functions in terms of subsurface coefficients along Teichm\"{u}ller geodesics developed in \cite{rshteich} and \cite{rteichhyper}


\subsection{Weil-Petersson geodesic rays with prescribed itinerary}\label{subsec : rayprescribed}

In this subsection by extracting limits of WP geodesic segments with end invariants on a single infinite hierarchy path with narrow end points we construct WP geodesics whose behavior mimic the combinatorial properties of hierarchy paths.

Let $\nu^{+}$ be a measurable geodesic lamination. Suppose that there is a collection of pairwise disjoint subsurfaces $Z_{a}$, $a=1,...,m$, with $\xi(Z_{a})\geq 1$ so that: any simple closed curve in $S\backslash \cup_{a=1}^{m}Z_{a}$ is isotopic to a boundary curve of one of the subsurfaces $Z_{a}$. Moreover $\nu_{a}$ the restriction of $\nu^{+}$ to $Z_{a}$ is minimal and fills $Z_{a}$. By Theorem \ref{thm : ccbdry}, $\nu_{a}$ determines a point in the Gromov boundary of $\mathcal{C}(Z_{a})$. For each $a=1,...,m$ let $\mathcal{L}_{a}$ be a measured lamination supported on $\nu_{a}$. Let $\gamma^{a}_{n}\in\mathcal{C}(Z_{a})$ be a sequence of curves so that the projective classes $[\gamma^{a}_{n}]$ converge to $[\mathcal{L}_{a}]$ as $n\to \infty$ in $\mathcal{PML}(Z_{a})$. Here $\gamma^{a}_{n}$ is equipped with the transverse measure $i(\gamma^{a}_{n},.)$. For each $n\in\mathbb{N}$ let $Q_{n}$ be a pants decomposition that contains $\{\partial{Z}_{a},\gamma^{a}_{n}\}_{a=1}^{m}$. Let $c_{n}$ be a maximally nodal hyperbolic surface at $Q_{n}$. Let $[x,c_{n}]$ be the WP geodesic segment connecting a base point $x$ in the interior of Teichm\"{u}ller space to $c_{n}$. Denote the parametrization of $[x,c_{n}]$ by arc-length by $r_{n}$. The proof of the following two lemmas essentially follows the proof of surjectivity of weighted ending laminations of WP geodesic rays in the Teichm\"{u}ller space of a surface with complexity $3$ given in \cite[\S 4]{bm}.

\begin{lem}\label{lem : inftyray1}
After possibly passing to a subsequence the geodesic segments $r_{n}$ converge to an infinite ray $r$ in the Weil-Petersson visual sphere at $x$. Moreover, the length of each measured lamination $\mathcal{L}_{a}$, $a=1,...,m$ and any curve $\alpha\in\partial{Z}_{a}$, $a=1,...,m$, is bounded along $r$.
\end{lem}
\begin{proof}
Theorem \ref{thm : nonrefraction}(Non-refraction Theorem) guarantees that the interior of each one of the WP geodesic segments $r_{n}$ is inside the Teichm\"{u}ller space. Moreover, there is a lower bound for the distance of $x$ and all completion strata of $\overline{\Teich(S)}$. Thus the length of $r_{n}$'s is bounded below by a positive constant. Then by the local compactness of the WP metric at $x$ after possibly passing to a subsequence the initial parts of the geodesic segments $r_{n}$ converge to a geodesic segment $r_{\infty}$ starting at $x$. Let $r$ be the maximal geodesic ray in the Teichm\"{u}ller space with initial part $r_{\infty}$. We prove that $r$ is an infinite ray.

Let $s_{n}=\frac{1}{\ell_{Q_{n}}(x)}$, where $\ell_{Q_{n}}(x)=\max_{\gamma\in Q_{n}}\ell_{\gamma}(x)$. Then for each $n\in\mathbb{N}$, and $t$ in the domain of $r_{n}$  we have that $\ell_{s_{n}Q_{n}}(r_{n}(t))\leq 1$. To see this, note that $\ell_{s_{n}Q_{n}}(x)=1$ and $\ell_{s_{n}Q_{n}}(c_{n})=0$. Then the bound follows from the convexity of length-functions along WP geodesics. In what follows assuming that $r(t)$ has finite length $T$ we get a contradiction. 

After possibly passing to a subsequence the projective classes $[s_{n}Q_{n}]$ converge to a projective measured lamination $[\mathcal{L}]$ in $\mathcal{PML}(S)$. Let $1\leq a\leq m$. By our choice of the pants decompositions $Q_{n}$, $\gamma^{a}_{n}\in Q_{n}$. Moreover the projective classes $[s_{n}\gamma^{a}_{n}]=[\gamma^{a}_{n}]$ converge to the projective class $[\mathcal{L}_{a}]$. Thus the support of $[\mathcal{L}]$ contains the support of $[\mathcal{L}_{a}]$. The measured lamination $[\mathcal{L}_{a}]$ is supported on $\nu_{a}$. So the support of $[\mathcal{L}]$ contains $\nu_{a}$. Therefore, any curve $\gamma$ with $i(\gamma,\mathcal{L})=0$ is disjoint from the laminations $\nu_{a}$. Moreover $\nu_{a}$ fills $Z_{a}$, so $\gamma$ is disjoint from the subsurface $Z_{a}$. Then by our assumption about the subsurfaces $Z_{a}$, $\gamma\in\partial{Z}_{a}$ for some $1\leq a\leq m$.

The length-function $\ell_{\mathcal{L}}(x):\Teich(S) \times \mathcal{ML}(S)\to \mathbb{R}^{\geq 0}$ is continuous in both $x$ and $\mathcal{L}$ variables, which implies that $\ell_{\mathcal{L}}(r(t)) \leq 1$ for any $t <T$. So we have that 
\begin{equation}\label{eq : lmubd}\lim_{t \to T} \ell_{\mathcal{L}}(r(t)) \leq 1. \end{equation}
The geodesic ray $r$ is the maximal geodesic with initial part $r_{\infty}$ and has length $T$. Thus there is a maximal multi-curve $\sigma$ such that $r(T) \in \mathcal{S}(\sigma)$. Then for each curve $\gamma \in \sigma$ we have that $i(\gamma, \mathcal{L})=0$. For otherwise, since $\ell_{\gamma}(r(t)) \to 0$ as $t \to T$, we would have that $\ell_{\mathcal{L}}(r(t)) \to \infty$ as $t \to T$, which contradicts the bound (\ref{eq : lmubd}). Then as we saw above
\begin{equation}\label{eq : sigma<bdryZ}\sigma\subseteq\{\partial{Z}_{a}\}_{a=1}^{m}.\end{equation}

\begin{claim}\label{claim : dwprTc} 
$d_{\WP}(r(T),c_{n})\to \infty$ as $n\to \infty$, and $d_{\WP}(x,c_{n})\to \infty$ as $n\to \infty$. 
\end{claim}
Let $1\leq a\leq m$. The projective class of the curves $\gamma^{a}_{n}\in Q_{n}$ converge to $[\mathcal{L}_{a}]$ as $n\to \infty$. Moreover $[\mathcal{L}_{a}]$ determines a point in the Gromov boundary of $\mathcal{C}(Z_{a})$. Thus $d_{Z_{a}}(Q_{n},Q(r(T)))\to\infty$ as $n\to\infty$. But $d_{Z_{a}}(Q_{n},Q(r(T)))$ is a term on the right hand side of the distance formula (\ref{eq : dsf}) for $d(Q(c_{n}),Q(r(T)))$, therefore $d(Q(c_{n}),Q(r(T)))\to\infty$ as $n\to\infty$. Then by Theorem \ref{thm : brockqisom}(Quasi-Isometric Model) $d_{\WP}(r(T),c_{n})\to\infty$ as $n\to \infty$. The proof of the second part is similar.
\begin{claim}\label{claim : cnonstr}
For all $n\in\mathbb{N}$, $c_{n}\in\overline{\mathcal{S}(\sigma)}$. 
\end{claim}
By (\ref{eq : sigma<bdryZ}), $\sigma\subseteq\{\partial{Z}_{a}\}_{a=1}^{m}$. Then by the choice of pants decompositions $Q_{n}$, $\sigma\subset Q_{n}$ for all $n\in\mathbb{N}$. Then since $c_{n}$ is the maximally nodal surface at $Q_{n}$, $c_{n}\in\overline{\mathcal{S}(\sigma)}$.
\medskip

By the second part of Claim \ref{claim : dwprTc}, the length of the geodesic segments $r_{n}$ goes to $\infty$.
\begin{claim}\label{claim : rntstr}
For any $t'>T$ the points $r_{n}(t')$ converge to a point $y\in\overline{\mathcal{S}(\sigma)}$. 
\end{claim}
  Let $\eta_{n}$ be the parametrization of the geodesic segment $[r(T),c_{n}]$ by arc-length with $\eta_{n}(0)=r(T)$. We have that $r(T)\in \mathcal{S}(\sigma)$ and by Claim \ref{claim : cnonstr}, $c_{n}\in \overline{\mathcal{S}(\sigma)}$. Then by Theorem \ref{thm : nonrefraction}(Non-refraction Theorem) the interior of $\eta_{n}\subset\mathcal{S}(\sigma)$. By Claim \ref{claim : dwprTc}, $d_{\WP}(r(T),c_{n})\to\infty$ as $n\to \infty$.   
Moreover $\eta_{n}$ and $[r_{n}(T),c_{n}]$ have the same end point $c_{n}$, and $r_{n}(T)$ converges to $r(T)$ as $n \to \infty$. Then the $\CAT(0)$ comparison for the triangle with vertices $r(T),r_{n}(T)$ and $c_{n}$ implies that the Hausdorff distance between $\eta_{n}$ and $[r_{n}(T),c_{n}]$ tends to zero as $n\to \infty$. Therefore $r_{n}(t')$ converges to a point $y$ on $\eta_{n}$ in the $\sigma-$stratum. 
 \medskip

Let $t'>T$. By Claim \ref{claim : rntstr} the points $r_{n}(t')$ converge to a point $y\in\mathcal{S}(\sigma)$. By the Non-refraction Theorem the interior of $[x,y]$ is inside the maximal stratum containing $x$ and $y$ which is the Teichm\"{u}ller space. By the $\CAT(0)$ comparison the Hausdorff distance of $r_{n}|_{[0,t']}$ and $[x,y]$ goes to $0$ as $n\to \infty$. So $r_{n}(T)$ converges to a point in the interior of $[x,y]$ which is inside the Teichm\"{u}ller space. But this contradicts the assumption that $r_{n}(T)$ converge to a point in the $\sigma-$stratum. This contradiction finishes the proof of that the geodesic $r$ is an infinite geodesic ray. Then since the initial parts of the geodesic segments $r_{n}$ converge to the initial part of $r$ and the length of $r_{n}$'s go to $\infty$ (the second part of Claim \ref{claim : dwprTc}), we may conclude that $r_{n}$ converge to $r$ as $n\to\infty$. 
\medskip

To prove the last statement of the lemma fix $1\leq a\leq m$. Let $\alpha\in\partial{Z}_{a}$. For each $n\in\mathbb{N}$, $\gamma^{a}_{n}$ is a pinching curve of $r_{n}$. Moreover $[\gamma^{a}_{n}]\to[\mathcal{L}_{a}]$ as $n\to\infty$. In addition, the geodesic segments $r_{n}$ converge to $r$. Thus Lemma \ref{lem : limraylimmeas} guarantees that the length of $\mathcal{L}_{a}$ is decreasing along $r$. Also for each $n\in\mathbb{N}$, $\partial{Z}_{a}\subseteq Q_{n}$, so $\alpha$ is a pinching curve of $r_{n}$. Then since the geodesic segments $r_{n}$ converge to $r$, Lemma \ref{lem : limraylimmeas} guarantees that the length of $\alpha$ is decreasing along $r$.  
 \end{proof}
 
 Let $(\nu^{-},\nu^{+})$ be a narrow pair where $\nu^{-}$ is a marking. Then there is at most one large subsurface $Z$ so that the restriction of $\nu^{+}$ to $Z$ is minimal and fills $Z$. Suppose that such a subsurface exists. Since $Z$ is a large subsurface any simple closed curve in $S\backslash Z$ is isotopic to a curve in $\partial{Z}$. Let $\rho$ be a hierarchy path between $\nu^{-}$ and $\nu^{+}$. Then there is a unique infinite geodesic in $\mathcal{C}(Z)$ and an interval $J_{Z}=[N,\infty]$ ($N\in\mathbb{N}$) with the property explained in part (\ref{h : J}) of Theorem \ref{thm : hrpath}. Then for any $n\in\mathbb{N}$, 
$$\partial{Z}\subseteq \rho(n+N).$$
For more detail see \cite[\S 5]{mm2}. For any $n\in\mathbb{N}$ let $Q_{n}=\rho(n+N)$. Let $g_{Z}\subset\mathcal{C}(Z)$ be the tight geodesic of $\rho$ corresponding to the component domain $Z$ and let $\gamma_{n}=Q_{n}\cap g_{Z}$ (Theorem \ref{thm : hrpath}(\ref{h : J})). The curves $\gamma_{n}$ converge to a point in the Gromov boundary of $\mathcal{C}(Z)$ determined by the lamination $\nu$. Let $x\in\Teich(S)$ be a point  so that one of the Bers markings of $x$ is $\nu^{-}$. Let $c_{n}$ be the maximally nodal hyperbolic surface at $Q_{n}$. Let $r$ be the limit of the geodesic segments $[x,c_{n}]$ after possibly passing to a subsequence as in Lemma \ref{lem : inftyray1}.
 
 \begin{lem}\label{lem : inftyray}\textnormal{ (Infinite ray)}
  The ending lamination of $r$ contains $\nu$.
  \end{lem}
  \begin{proof}
   The pants decompositions $Q_{n}$ $(n\in\mathbb{N})$ consist an stable subset of  $P(S)$ (Theorem \ref{thm : stability}). Then by the Brock's quasi-isometry (Theorem \ref{thm : brockqisom}) the points $c_{n}$ consist a stable subset of $\overline{\Teich(S)}$. So there is a $d>0$ so that all of the geodesic segments $[x,c_{n}]$ are in the $d-$neighborhood of the collection of points $\{c_{n}\}_{n=1}^{\infty}$. Thus $r$ the limit of the geodesic segments $r_{n}$ in the visual sphere stays in the $d-$neighborhood of the collection of points $\{c_{n}\}_{n=1}^{\infty}$. Therefore there are times $t_{n}\in [0,\infty)$ such that $d_{\WP}(r(t_{n}),c_{n})\leq d$. 

Let $D=K_{\WP}d+C_{\WP}$. Then by Theorem \ref{thm : brockqisom}, $d(Q(r(t_{n})),Q_{n})\leq D$. Then by (\ref{eq : ddsubs}), $d_{Z}(Q(r(t_{n})),Q_{n})\leq D$. For each $n\in\mathbb{N}$ let $\alpha_{n}\in Q(r(t_{n}))$ be such that $\alpha_{n}\pitchfork Z$.  By Lemma \ref{lem : diamproj}, $\diam_{Z}(Q_{n})\leq 2$ and $\diam_{Z}(Q(r(t_{n}))\leq 2$, so we obtain 
$$d_{Z}(\alpha_{n},\gamma_{n})\leq D+4.$$
The curves $\gamma_{n}$ converge to a point in the Gromov boundary of $\mathcal{C}(Z)$ determined by $\nu$. The curves $\alpha_{n}$ converge to the same point. To see this, note that fixing a point $p$ in $\mathcal{C}(Z)$, the above bound on the distance of $\gamma_{n}$ and $\alpha_{n}$, and the fact that the distance of $\gamma_{n}$ and $p$ goes to $\infty$ imply that the Gromov product of $\gamma_{n}$ and $\alpha_{n}$ with respect to the point $p$ goes to $\infty$ as $n\to\infty$. This implies that their limits determine the same point in the Gromov boundary of $\mathcal{C}(Z)$ (see \cite[$\S$III.H.3]{bhnpc}). Then by Theorem \ref{thm : gobdrycc} after possibly passing to a subsequence the projective classes $[\alpha_{n}]$ converge to the projective class of a measured lamination $\mathcal{E}$ supported on $\nu$. Now $\alpha_{n}$ is a sequence of distinct Bers curves along $r$, so by Definition \ref{def : endmeaslam}, $\mathcal{E}$ is an ending measured lamination of $r$. Then by Definition \ref{def : endlamwp}, $\nu$ is contained in the ending lamination of the ray $r$. As was desired.
\end{proof}
 
 Suppose that the pair $(\nu^{-},\nu^{+})$ is a narrow pair. Let $\rho$ be any hierarchy path between $\nu^{-}$ and $\nu^{+}$. Then Lemma \ref{lem : inftyray} provides us with an infinite ray, denoted by $r_{\nu^{\pm}}$, whose forward ending lamination contains $\nu$ the minimal component of $\nu^{+}$ that fills a large subsurface. The following theorem shows that the behavior of $r_{\nu^{\pm}}$ mimics the combinatorial properties of hierarchy paths encoded in the end invariant listed in Theorem \ref{thm : hrpath}. 
 \begin{remark}
 By Theorem \ref{thm : stability} any two hierarchy paths between a narrow pair $\nu^{-}$ and $\nu^{+}$ fellow-travel. This implies that the  rays corresponding to them fellow travel in the Teichm\"{u}ller space.
 \end{remark}
 
 \begin{thm}\textnormal{ (Infinite ray with prescribed itinerary)} \label{thm : inftyray}
Given $A,R,R'>0$. Let $(\nu^{-},\nu^{+})$ be an $A-$narrow pair. Let $\rho$ be a hierarchy path between $\nu^{-}$ and $\nu^{+}$. Let $r_{\nu^{\pm}}:[0,\infty) \to \Teich(S)$ be a corresponding infinite WP geodesic ray. Let $D=D(A)$ be the fellow traveling distance from Theorem \ref{thm : fellowtr}.
 
Let $\bar{\epsilon}=\bar{\epsilon}(D,R,R')$ be the constant from Lemma \ref{lem : roughub} and for an $\epsilon\leq \bar{\epsilon}$ let $\bar{w}=\bar{w}(D,R,R',\epsilon)$ be the constant from Theorem \ref{thm : bdryzsh}. Suppose that $Z$ a large component domain of $\rho$ has $(R,R')-$bounded combinatorics over an interval $[m',n']\subset J_{Z}$ with $n'-m'>2\bar{w}$. Then 
   \begin{enumerate}
 \item $\ell_{\gamma}(r_{\nu^{\pm}}(t))\geq\bar{\epsilon}$ for every $\gamma \notin \partial{Z}$, and
  \item $\ell_{\alpha}(r_{\nu^{\pm}}(t))\leq \epsilon $ for every $\alpha \in \partial{Z}$
 \end{enumerate}
 for every $t \in [a',b']$, where $a'\in N(m'+\bar{w})$ and $b'\in N(n'-\bar{w})$. Here $N:=N_{\rho,g}$ is the parameter map from Proposition \ref{prop : parmap}.
 
Moreover, suppose that $Z_{1}$ and $Z_{2}$ are subsurfaces as above with bounded combinatorics over intervals $[m'_{1},n'_{1}]$ and $[m'_{2},n'_{2}]$, respectively. Then  $n'_{1}<m'_{2}$ implies that $b'_{1}<a'_{2}$.
 \end{thm}
 \begin{proof}
The assumption that the pair $(\nu^{-},\nu^{+})$ is $A-$narrow guarantees that for any subsurface $Y\subseteq S$ which is not large, 
$$d_{Y}(\nu^{-},\nu^{+})\leq A.$$
 So for any $n\in\mathbb{N}$, no backtracking property of hierarchy paths (\ref{eq : nbtrack}) implies that 
 $$d_{Y}(\rho(0),\rho(n))\leq A+2M_{2}.$$
  Thus the end invariant of the geodesic segment $[x,c_{n}]$ is $A+2M_{2}$ narrow. Moreover the subsurface $Z$ has $(R,R')-$bounded combinatorics over $[m,n]$, so by Lemma \ref{lem : roughub}, there are constants $\bar{\epsilon}$ and $w$ such that for any $\gamma \notin \partial{Z}$, 
$$\ell_{\gamma}(r_{n}(t))>\bar{\epsilon}$$
 for every $t\in N(j)$ where $j\in[m'+w,n'-w]$. Moreover, since $w \leq \bar{w}$ the above bound holds on the interval $[a',b']$ as well. Furthermore, by Theorem \ref{thm : bdryzsh}, for every $\alpha\in \partial{Z}$, $\ell_{\alpha}(r_{n}(t))\leq \epsilon$ for every $t\in[a',b']$. 
 
 Now note that the geodesic segments $r_{n}$ converge to $r_{\nu^{\pm}}$ point-wise in the Teichm\"{u}ller space. Thus by Theorem \ref{thm : continuitylf}(Continuity of length-functions) the same bounds on length-functions hold along $r_{\nu^{\pm}}$. Finally, the statement about order of intervals follows from the fellow traveling of $\rho$ and $r_{\nu^{\pm}}$. 
 \end{proof}

\subsection{Divergent geodesic rays} \label{subsec : div}
In this subsection we construct divergent Weil-Petersson geodesic rays in the moduli space. A ray is divergent if eventually leaves every compact subset of the moduli space. The existence of uncountably many divergent WP geodesic rays starting at a given point with minimal filling ending lamination is a consequence of our construction in this subsection. 
\begin{definition}
 A ray $r:[0,\infty) \to \mathcal{M}(S)$ is recurrent to a compact subset $\mathcal{K}\subset \mathcal{M}(S)$, if there is a sequence of times $t_{i}\to \infty$, such that $r(t_{i})\in \mathcal{K}$. A ray $r:[0,\infty) \to \mathcal{M}(S)$ is divergent if it is not recurrent to any compact subset of the moduli space. In other words, $r$ is divergent if for every compact set $\mathcal{K}\subset \mathcal{M}(S)$, there is a $T\geq 0$ such that $r([T,\infty))$ does not intersect $\mathcal{K}$.
\end{definition}

\begin{thmdiv}\textnormal{ (Divergent geodesics)}
Starting from any point $x\in\mathcal{M}(S)$ there are uncountably many divergent WP geodesic rays with minimal filling ending lamination.
\end{thmdiv}

\begin{proof} Let the curves $\alpha$ and $\beta$ be two disjoint curves in the set of Bers curves of the surface $x$ so that the subsurfaces $S\backslash\alpha$, $S\backslash\beta$ and $S\backslash\{\alpha,\beta\}$ are large subsurfaces. Then let the indexed subsurfaces $X_{0},X_{1},X_{2}$ and $X_{3}$, and the indexed partial pseudo-Anosov maps $f_{0},f_{1},f_{2}$ and $f_{3}$ supported on $X_{0},X_{1},X_{2}$ and $X_{3}$ respectively be as in $\S$\ref{subsec : pre1}. Recall that $X_{0}$ and $X_{2}$ are the same subsurface with different indices. Also $f_{0}$ and $f_{2}$ are the same partial pseudo-Anosov maps with different indices. Let the function $q(i)\equiv i$ (mod 4) and the sequence of integers $e_{i}$ be as in $\S$\ref{subsec : pre1}. For each $i\in\mathbb{N}$, let the subsurface $Z_{i}$ be as in (\ref{eq : Zi}). Fix a marking $\mu_{I}$ so that $\base(\mu_{I})$ contains $\{\partial{X}\}_{a=0,1,2,3}=\{\alpha,\beta\}$ and let $\mu_{T}$ be as in $\S$\ref{subsec : pre1}. 

Let the constants $K_{1},C_{1}$ and $E>0$ be as in Proposition \ref{prop : pre1}. Suppose that $|e_{i}|>E$ for all $i\in\mathbb{N}$. Then by Proposition \ref{prop : pre1}(\ref{pre1 : ld}) we have that
\begin{equation}\label{eq : dZiITe}d_{Z_{i}}(\mu_{I},\mu_{T})\asymp_{K_{1},C_{1}}|e_{i}|.\end{equation}
Then by the choice of $E$,
$$d_{Z_{i}}(\mu_{I},\mu_{T})> 4M.$$  
By Proposition \ref{prop : pre1}(\ref{pre1 : bd}), every subsurface $Z$ with 
$$d_{Z}(\mu_{I}(q,e),\mu_{T}(q,e))> {\bf m}$$
 is in the list of the subsurfaces $Z_{i}$ and therefore is a large subsurface. So the pair $(\mu_{I},\mu_{T})$ is ${\bf m}-$narrow. Let $\rho:[0,\infty]\to P(S)$ be a hierarchy path between $\mu_{I}$ and $\mu_{T}$. Let $r:[0,\infty)\to \Teich(S)$ be the corresponding WP geodesic ray with the end invariant $(\mu_{I},\mu_{T})$ and prescribed itinerary as in Theorem \ref{thm : inftyray}. By Proposition \ref{prop : parmap} we have the parameter map $N$ from the parameters of $\rho$ to the parameters of $r$.

Recall the $J$ intervals from Theorem \ref{thm : hrpath}(\ref{h : J}) corresponding to each component domain of a hierarchy. For each $i\in\mathbb{N}$ let $J_{Z_{i}}=[s_{i}^{-},s_{i}^{+}]$.

For each $i\in\mathbb{N}$ odd let 
\begin{center}$k_{i}=\max  J_{Z_{i}}\cap J_{Z_{i-1}}$ and  $l_{i}=\min  J_{Z_{i}}\cap J_{Z_{i+1}},$\end{center}
 and suppose that $l_{i}>k_{i}$. For each $i\in\mathbb{N}$ even let 
  \begin{center}$k_{i}=\min  J_{Z_{i}}$ and $l_{i}=\max  J_{Z_{i}}$.\end{center} 
\medskip

\noindent{\bf Subsurface coefficient bounds:} Here we collect some bounds on subsurface coefficients.
\medskip

\noindent{\bf $i$ is even.}  Since $i$ is even, $k_{i}=s_{i}^{-}$ and $l_{i}=s_{i}^{+}$. By Proposition \ref{prop : pre1}(\ref{pre1 : bd}) and no backtracking property of hierarchy paths (\ref{eq : nbtrack}), for every proper subsurface $W\subsetneq S$ which is not in the list of subsurfaces $Z_{i}$ we have 
\begin{eqnarray}\label{eq : dWJ2}
d_{W}(\rho(s_{i}^{-}),\rho(s_{i}^{+}))&\leq&{\bf m}+2M,\;\text{or equivalently}\\
d_{W}(\rho(k_{i}),\rho(l_{i}))&\leq& {\bf m}+2M.\nonumber
\end{eqnarray}
By Proposition \ref{prop : pre1}(\ref{pre1 : ord}), for all $j\in\mathbb{N}$ with $j-i \geq 2$, $Z_{i}<Z_{j}$ and for all  $j\in\mathbb{N}$ with $i-j\geq 2$, $Z_{j}<Z_{i}$. Suppose that $Z_{i}<Z_{j}$. Then by Definition \ref{def : ordersubsurf}, $d_{Z_{j}}(\mu_{I},\partial{Z}_{i})\leq M.$ Moreover $\partial{Z}_{i}\subset \rho(s_{i}^{-})$ and $\partial{Z}_{i}\subset\rho(s_{i}^{+})$. So 
$$d_{Z_{j}}(\mu_{I},\rho(s_{i}^{-}))\leq M$$
 and similarly $d_{Z_{j}}(\mu_{I},\rho(s_{i}^{+}))\leq M$. Then by the triangle inequality 
$$d_{Z_{j}}(\rho(s_{i}^{-}),\rho(s_{i}^{+}))\leq d_{Z_{j}}(\rho(s_{i}^{-}),\mu_{I})+d_{Z_{j}}(\mu_{I},\rho(s_{i}^{+}))+\diam_{Z_{i}}(\mu_{I})\leq 2M+2.$$
Assuming that $Z_{j}<Z_{i}$, by a similar argument, using the inequality $d_{Z_{j}}(\mu_{T},\partial{Z}_{i})\leq M$ from Definition \ref{def : ordersubsurf}, we can get the above bound. We recored these bounds
\begin{eqnarray}\label{eq : dZjJ2}
d_{Z_{j}}(\rho(s_{i}^{-}),\rho(s_{i}^{+}))&\leq& 2M+2,\;\text{or equivalently}\\
d_{Z_{j}}(\rho(k_{i}),\rho(l_{i}))&\leq&2M+2.\nonumber
\end{eqnarray}
By the set up of the subsurfaces $Z_{i}$ in (\ref{eq : Zi}),  $\partial{Z}_{i}\pitchfork Z_{i-1}$.  By Theorem \ref{thm : hrpath}(\ref{h : J}), $\partial{Z}_{i}\subset\rho(s_{i}^{-})$ and $\partial{Z}_{i}\subset\rho(s_{i}^{+})$.  Then we have
$$d_{Z_{i-1}}(\rho(s_{i}^{-}),\rho(s_{i}^{+}))\leq \diam_{Z_{i-1}}(\rho(s_{i}^{-}))+ \diam_{Z_{i-1}}(\rho(s_{i}^{+}))\leq 4.$$
We record the above inequality
\begin{eqnarray}\label{eq : dZi-1rs-+}
d_{Z_{i-1}}(\rho(s_{i}^{-}),\rho(s_{i}^{+}))&\leq&4,\;\text{or equivalently}\\
d_{Z_{i-1}}(\rho(k_{i}),\rho(l_{i}))&\leq& 4\nonumber
\end{eqnarray}
Similarly since $\partial{Z}_{i}\pitchfork Z_{i+1}$ and $\partial{Z}_{i}\subset\rho(s_{i}^{-})$ and $\partial{Z}_{i}\subset\rho(s_{i}^{+})$ we have that
\begin{eqnarray}\label{eq : dZi+1rs-+}
d_{Z_{i+1}}(\rho(s_{i}^{-}),\rho(s_{i}^{+}))&\leq&4,\;\text{or equivalently}\\
d_{Z_{i+1}}(\rho(k_{i}),\rho(l_{i}))&\leq& 4\nonumber
\end{eqnarray}
Also since $\partial{Z}_{i}\pitchfork S$ and $\partial{Z}_{i}\subset\rho(s_{i}^{-})$ and $\partial{Z}_{i}\subset\rho(s_{i}^{+})$ we have
\begin{eqnarray}\label{eq : dSrs-+}
d_{S}(\rho(s_{i}^{-}),\rho(s_{i}^{+}))&\leq& 4,\;\text{or equivalently}\\
d_{S}(\rho(k_{i}),\rho(l_{i}))&\leq& 4\nonumber
\end{eqnarray}

\noindent{\bf $i$ is odd.}  By Proposition \ref{prop : pre1}(\ref{pre1 : bd}) and no backtracking property of hierarchy paths (\ref{eq : nbtrack}) for every proper subsurface $W\subsetneq S$ which is not in the list of subsurfaces $Z_{i}$ we have 
\begin{equation}\label{eq : dWJ1}d_{W}(\rho(s_{i}^{-}),\rho(s_{i}^{+}))\leq {\bf m}+2M,\end{equation}
and
 \begin{equation}\label{eq : dWkl1}d_{W}(\rho(k_{i}),\rho(l_{i}))\leq {\bf m}+2M.\end{equation}

By Proposition \ref{prop : pre1}(\ref{pre1 : ord}), for all  $j\in\mathbb{N}$ with $j\geq i+2$, $Z_{i}<Z_{j}$ and for all  $j\in\mathbb{N}$ with $j\geq i+2$, $Z_{j}<Z_{i}$. Then by an argument similar to the proof of (\ref{eq : dZjJ2}) we can get 
 \begin{equation}\label{eq : dZjJ1}d_{Z_{j}}(\rho(s_{i}^{-}),\rho(s_{i}^{+}))\leq 2M+2,\end{equation}
 Moreover $k_{i},l_{i}\in J_{Z_{i}}$ so $\partial{Z}_{i}\subset \rho(k_{i})$ and $\partial{Z}_{i}\subset\rho(l_{i})$. Thus again by an argument similar to the proof of (\ref{eq : dZjJ2}) we can get  
  \begin{equation}\label{eq : dZjkl1}d_{Z_{j}}(\rho(k_{i}),\rho(l_{i}))\leq 2M+2.\end{equation}

By Theorem \ref{thm : hrpath}(\ref{h : nobacktrack}) (No backtracking) we have
$$d_{Z_{i-1}}(\rho(k_{i}),\rho(l_{i}))+d_{Z_{i-1}}(\rho(l_{i}),\mu_{T})\leq d_{Z_{i-1}}(\rho(k_{i}),\mu_{T})+M.$$
Now by Theorem \ref{thm : hrpath}(\ref{h : lrbddproj}), $d_{Z_{i-1}}(\mu_{T},\rho(k_{i}))\leq M$. Then by the above inequality we have
 \begin{equation}\label{eq : dZi-1kl1}d_{Z_{i-1}}(\rho(k_{i}),\rho(l_{i}))\leq 2M.\end{equation}
 Similarly, by the no backtracking we have 
 $$d_{Z_{i+1}}(\rho(k_{i}),\mu_{I})+d_{Z_{i+1}}(\rho(k_{i}),\rho(l_{i}))\leq d_{Z_{i+1}}(\rho(l_{i}),\mu_{I})+M,$$
 then since $d_{Z_{i+1}}(\mu_{I},\rho(l_{i}))\leq M$ we get
  \begin{equation}\label{eq : dZi+1kl1}d_{Z_{i+1}}(\rho(k_{i}),\rho(l_{i}))\leq 2M.\end{equation} 
  Also similar to (\ref{eq : dSrs-+}) we can get
    \begin{equation}\label{eq : dSrs-+1}d_{S}(\rho(s_{i}^{-}),\rho(s_{i}^{+}))\leq 4,\end{equation}
  and
  \begin{equation}\label{eq : dSkl1}d_{S}(\rho(k_{i}),\rho(l_{i}))\leq 4.\end{equation}
 \medskip

Finally, by Proposition \ref{prop : pre1}(\ref{pre1 : bda}) and no backtracking for any $i\in\mathbb{N}$ and any $\gamma \in \mathcal{C}_{0}(S)$ we have  
 \begin{equation}\label{eq : dakl}d_{\gamma}(\rho(k_{i}),\rho(l_{i}))\leq {\bf m}'+2M.\end{equation}
 \medskip

\noindent{\bf The length of $J$ intervals:}  Now we proceed to estimate the length of $J$ intervals using the above subsurface coefficient bounds. For each $i$ even by the inequalities (\ref{eq : dWJ2}), (\ref{eq : dZjJ2}), (\ref{eq : dZi-1rs-+}), (\ref{eq : dZi+1rs-+}) and (\ref{eq : dSrs-+}), all of the subsurface coefficients of $\rho(s_{i}^{-})$ and $\rho(s_{i}^{+})$ except that of $Z_{i}$ are bounded above by ${\bf m}+2M$ (note that ${\bf m}>2$). Let the threshold constant in the distance formula (\ref{eq : dsf}) be ${\bf m}+2M$ (note that ${\bf m}+2M>M_{1}$) and let $K_{2},C_{2}$ be the constants corresponding to this threshold constant. Then we have  
\begin{equation}\label{eq : dsi-+dZisi-+}d(\rho(s_{i}^{-}),\rho(s_{i}^{+}))\asymp_{K_{2},C_{2}}d_{Z_{i}}(\rho(s_{i}^{-}),\rho(s_{i}^{+})).\end{equation}
 Similarly for each $i$ odd by the inequalities (\ref{eq : dWJ1}), (\ref{eq : dZjJ1}) and (\ref{eq : dSrs-+1}) and the distance formula we have
 \begin{eqnarray}\label{eq : dsi-+dZi1-+}
 d(\rho(s_{i}^{-}),\rho(s_{i}^{+})))\asymp_{K_{2},C_{2}} d_{Z_{i-1}}(\rho(s_{i}^{-}),\rho(s_{i}^{+}))\\
 +d_{Z_{i}}(\rho(s_{i}^{-}),\rho(s_{i}^{+}))+d_{Z_{i+1}}(\rho(s_{i}^{-}),\rho(s_{i}^{+})).
 \end{eqnarray}
 By the no backtracking property of hierarchy paths (\ref{eq : nbtrack}) and Theorem \ref{thm : hrpath}(\ref{h : lrbddproj}) each subsurface coefficient on the right hand side of (\ref{eq : dsi-+dZisi-+}) and (\ref{eq : dsi-+dZi1-+}) is quasi-equal with constants $1$ and $2M$ to the corresponding subsurface coefficient of $\mu_{I}$ and $\mu_{T}$. For example 
 $$d_{Z_{i-1}}(\rho(s_{i}^{-}),\rho(s_{i}^{+}))\asymp_{1,2M}d_{Z_{i-1}}(\mu_{I},\mu_{T}).$$
  Moreover, the hierarchy path $\rho$ is a $(k,c)-$quasi-geodesic where $k\geq 1$ and $c\geq 0$ depend only on the topological type of $S$. Thus for each $i\in\mathbb{N}$ we have
  \begin{equation}\label{eq : JZi=dsi-+}|J_{Z_{i}}|\asymp_{k,c}d(\rho(s_{i}^{-}),\rho(s_{i}^{+})).\end{equation}
  Let $K_{3}=kK_{2}$ and $C_{3}=kC_{2}+2kK_{2}M+c$. Then by (\ref{eq : dsi-+dZisi-+}) and (\ref{eq : JZi=dsi-+}) for each $i$ even we have
\begin{equation}\label{eq : Jiebd}|J_{Z_{i}}|\asymp_{K_{3},C_{3}} d_{Z_{i}}(\mu_{I},\mu_{T}),\end{equation}
and by (\ref{eq : dsi-+dZi1-+}) and (\ref{eq : JZi=dsi-+}) for each $i$ odd we have
\begin{equation}\label{eq : Jiobd}|J_{Z_{i}}|\asymp_{K_{3},C_{3}} d_{Z_{i-1}}(\mu_{I},\mu_{T})+d_{Z_{i}}(\mu_{I},\mu_{T})+d_{Z_{i+1}}(\mu_{I},\mu_{T}).\end{equation}
\medskip

We proceed to set the sequence $\{e_{i}\}_{i=1}^{\infty}$ and using the estimates for the length of $J$ intervals study the intersection pattern of the $J$ intervals and estimate the length of intervals $[k_{i},l_{i}]$.

 Let $\epsilon_{n} \to 0$ be a decreasing sequence. Let $R={\bf m}+2M$ and $R'={\bf m}'+2M$. For each $n\in\mathbb{N}$ let 
\begin{equation}\label{eq : ei>div}y_{n}\geq 2\bar{w}({\bf m},R,R',\epsilon_{n})\end{equation}
where $\bar{w}$ is the constant from Theorem \ref{thm : inftyray}. 
\medskip

Let $K=K_{3}K_{1}$ and $C=C_{1}+C_{3}$. Define the sequence $\{e_{i}\}_{i=1}^{\infty}$ as follows: For $i$ even, $e_{i}=Ky_{\frac{i}{2}}+KC$, and for $i$ odd, $e_{i}=Ky_{\frac{i-1}{2}}+(K^{3}y_{\frac{i-1}{2}}+K^{3}C+KC)+(K^{3}y_{\frac{i+1}{2}}+K^{3}C+KC)$. 
\medskip

\noindent{\bf The length of intervals $[k_{i},l_{i}]$:} We investigate the pattern that the $J$ intervals overlap each other and estimate the length of each interval $[k_{i},l_{i}]$.  
\medskip

\noindent{\bf $i$ is even.} Recall that when $i$ is an even integer $[k_{i},l_{i}]=J_{Z_{i}}$. Then (\ref{eq : Jiebd}) implies that $l_{i}-k_{i}\geq \frac{1}{K_{3}}d_{Z_{i}}(\mu_{I},\mu_{T})-C_{3}$. Then by (\ref{eq : dZiITe}), $l_{i}-k_{i}\geq \frac{1}{K_{3}K_{1}}|e_{i}|-\frac{C_{1}}{K_{3}}-C_{3}$. Therefore, $l_{i}-k_{i}\geq y_{\frac{i}{2}}$. 
\medskip

\noindent{\bf $i$ is odd.} By (\ref{eq : Jiebd}) and (\ref{eq : dZiITe}), $|J_{Z_{i-1}}|\leq K|e_{i-1}|+C$. This implies that 
\begin{equation}\label{eq : JZiJZi-1}|J_{Z_{i}}\cap J_{Z_{i-1}}|\leq K|e_{i-1}|+C.\end{equation}
 Similarly,  $|J_{Z_{i+1}}|\leq K|e_{i+1}|+C$, which implies that 
 \begin{equation}\label{eq : JZiJZi+1}|J_{Z_{i}}\cap J_{Z_{i+1}}|\leq K|e_{i+1}|+C.\end{equation}
  Moreover by (\ref{eq : Jiobd}) and (\ref{eq : dZiITe}), 
  \begin{equation}\label{eq : JZi}|J_{Z_{i}}|\geq\frac{1}{K}(|e_{i-1}|+|e_{i}|+|e_{i+1}|)-C.\end{equation} 

By the construction of slices of hierarchies (see \cite[\S 5]{mm2}) since $Z_{i-1}\subset Z_{i}$ we have $J_{Z_{i}}\cap J_{Z_{i-1}}\neq\emptyset$. Similarly since $Z_{i+1}\subset Z_{i}$ we have $J_{Z_{i}}\cap J_{Z_{i+1}}\neq\emptyset$. By Proposition \ref{prop : pre1}(\ref{pre1 : ordnext}) for every $s\in J_{Z_{i}}$, $s\geq \min J_{Z_{i-1}}$. Similarly by Proposition \ref{prop : pre1}(\ref{pre1 : ordnext}) for every $s\in J_{Z_{i}}$, $s\leq \max J_{Z_{i+1}}$. Furthermore by the bounds (\ref{eq : JZi}), (\ref{eq : JZiJZi-1}) and (\ref{eq : JZiJZi+1}) and the set up of the powers $e_{i}$ we have 
$$|J_{Z_{i}}-(J_{Z_{i-1}}\cup J_{Z_{i+1}})|=|J_{Z_{i}}|-|J_{Z_{i}}\cap J_{Z_{i-1}}|-|J_{Z_{i}}\cap J_{Z_{i+1}}|\geq y_{\lfloor\frac{i}{2}\rfloor}.$$
 Therefore there are times in $J_{Z_{i}}$ that are not in the intervals $J_{Z_{i-1}}$ and $J_{Z_{i+1}}$. Putting these facts together the intervals $J_{Z_{i-1}}$, $J_{Z_{i}}$ and $J_{Z_{i+1}}$ intersect each other as in Figure \ref{fig : bddcombdiv}. Moreover by the set up of the parameters $k_{i}$ and $l_{i}$ we have $l_{i}-k_{i}\geq y_{\lfloor\frac{i}{2}\rfloor}$ and in particular $l_{i}>k_{i}$. 
 \medskip
 
 Also we can conclude that for each $i\in\mathbb{N}$ we have $l_{i}\leq k_{i+1}$. 
\medskip 

Now we are in the position to finish the construction of divergent WP geodesic rays.

Let $R={\bf m}+2M$ and $R'={\bf m}'+2M$. For each $i$ even by the subsurface coefficient bounds (\ref{eq : dWJ2}), (\ref{eq : dZjJ2}), (\ref{eq : dZi-1rs-+}), (\ref{eq : dZi+1rs-+}) and (\ref{eq : dakl}) the subsurface $Z_{i}$ has $(R,R')-$bounded combinatorics over the interval $[k_{i},l_{i}]$. For each $i$ odd by the bounds (\ref{eq : dWkl1}), (\ref{eq : dZjkl1}), (\ref{eq : dZi-1kl1}), (\ref{eq : dZi+1kl1}) and (\ref{eq : dakl}) the subsurface $Z_{i}$ has $(R,R')-$bounded combinatorics over the interval $[k_{i},l_{i}]$.

For each $i\in\mathbb{N}$ let $t_{i}\in N(s)$ for some $s\in [k_{i}+\bar{w}, l_{i}-\bar{w}]$.  As we verified above the subsurface $Z_{i}$ has $(R,R')-$bounded combinatorics over $[k_{i},l_{i}]$. Moreover $l_{i}-k_{i}\geq y_{\lfloor\frac{i}{2}\rfloor}$. Then by the set up of the constants $y_{n}$ in (\ref{eq : ei>div}) Theorem \ref{thm : inftyray} applied to the interval $[k_{i},l_{i}]$ guarantees that
$$\ell_{\partial{Z_{i}}}(r(t_{i}))\leq\epsilon_{\lfloor \frac{i}{2}\rfloor},$$
 ($\ell_{\partial{Z}}(x)=\max\{\ell_{\alpha}(x) : \alpha \in \partial{Z}\}$). 

For any $i\in\mathbb{N}$ as we saw above $l_{i}\leq k_{i+1}$ so we may choose the times $t_{i}$ and $t_{i+1}$ so that $t_{i}<t_{i+1}$. By the set up of subsurfaces $Z_{i}$ in (\ref{eq : Zi}) and the function $q$ any two consecutive subsurfaces $Z_{i}$ and $Z_{i+1}$ have a boundary curve in common. Denote $\delta_{i}=\partial{Z_{i}}\cap \partial{Z_{i+1}}$. Then by the above bound $\ell_{\delta_{i}}(r(t_{i}))\leq \epsilon_{\lfloor \frac{i}{2}\rfloor}$ and $\ell_{\delta_{i}}(r(t_{i+1}))\leq \epsilon_{\lfloor \frac{i+1}{2}\rfloor}$. Then by the convexity of length-functions and  since $\epsilon_{i}$ is decreasing we have
 $$\ell_{\delta_{i}}(r(t)) \leq \epsilon_{\lfloor \frac{i}{2}\rfloor}$$
for every $t \in [t_{i},t_{i+1}]$.

Now since the intervals $[t_{i},t_{i+1}]$ ($i\in\mathbb{N}$) cover the interval $[0,\infty)$ the domain of $r$ and $\epsilon_{\lfloor\frac{i}{2}\rfloor} \to 0$, the systole of the surfaces along the WP geodesic ray $r$ decreases and goes to 0. 

By Mumford's compactness criterion (see \cite{mumfordcpct}) every compact subset of the moduli space is contained in some $\epsilon-$thick part of the moduli space. So $\hat{r}$ the projection of $r$ to the moduli space is a divergent geodesic ray. Moreover by Proposition \ref{prop : pre1}(\ref{pre1 : minfill}), $\mu_{T}(q,e)$ is a minimal filling lamination, so the forward ending lamination of $\hat{r}$ is minimal filling. 

Note that there are uncountably many sequences $\{y_{n}\}_{n=1}^{\infty}$ which satisfy the inequality (\ref{eq : ei>div}). The corresponding WP geodesic rays are distinct because have different forward ending laminations. Therefore there are uncountably many divergent WP geodesic rays starting from a given point in the moduli space. 
\end{proof}

  \begin{figure}
\centering
\scalebox{0.2}{\includegraphics{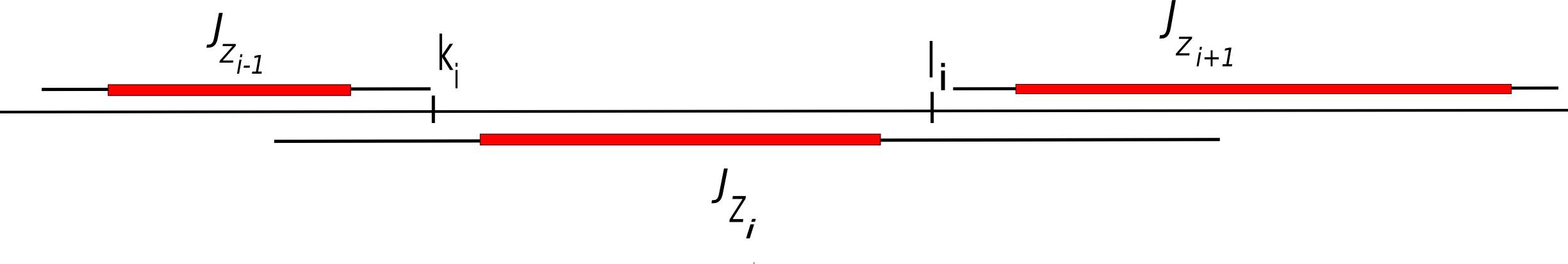}}
\caption{The intersection pattern of the intervals $J_{Z_{i-1}}$, $J_{Z_{i}}$ and $J_{Z_{i+1}}$ for $i$ odd. Over the red subinterval of each $J$ interval the corresponding domain has bounded combinatorics. The parameters $k_{i}=\max J_{Z_{i-1}}\cap J_{Z_{i}}$ and $l_{i}=\min J_{Z_{i+1}}\cap J_{Z_{i}}$.}  
\label{fig : bddcombdiv}
\end{figure}


\subsection{Closed geodesics in the thin part of the moduli space}  \label{subsec : closed}

In this subsection we construct examples of closed Weil-Petersson geodesics which stay in the thin part of the moduli space. 

\begin{thmclosed}\textnormal{ (Closed geodesics in the thin part)}
Given any compact subset of the moduli space $\mathcal{K}$, there are infinitely many closed Weil-Petersson geodesics which do not intersect $\mathcal{K}$. 
\end{thmclosed}

 \begin{proof} Let the indexed subsurfaces $X_{0},X_{1},X_{2}$ and $X_{3}$, and the indexed partial pseudo-Anosov maps $f_{0},f_{1},f_{2}$ and $f_{3}$ supported on $X_{0},X_{1},X_{2}$ and $X_{3}$ respectively be as in $\S$\ref{subsec : pre1}. Recall that $X_{0}$ and $X_{2}$ are the same subsurfaces with different indices. Also $f_{0}$ and $f_{2}$ are the same partial pseudo-Anosov maps with different indices. Let the function $q(i)\equiv i$ (mod 4) and the sequence of integers $\{e_{i}\}_{i=1}^{\infty}$ be as in $\S$\ref{subsec : pre1}. For each $i\in\mathbb{N}$ let the subsurface $Z_{i}$ be as in (\ref{eq : Zi}). Fix a marking $\mu_{I}$ so that $\base(\mu_{I})$ contains $\{\partial{X}\}_{a=0,1,2,3}=\{\alpha,\beta\}$ and let $\mu_{T}$ be as in $\S$\ref{subsec : pre1}.

Let the constants $K_{1},C_{1}$ and $E$ be as in Proposition \ref{prop : pre1}. Suppose that $|e_{i}|>E$ for all $i\in\mathbb{N}$. Then we are in the set up of the proof of Theorem \ref{thm : div} in $\S$\ref{subsec : div}. As we saw there the pair $(\mu_{I},\mu_{T})$ is ${\bf m}-$narrow where ${\bf m}$ is the constant from Proposition \ref{prop : pre1}. Let $\rho:[0,\infty]\to P(S)$ be a hierarchy path between $\mu_{I}$ and $\mu_{T}$. Let $r:[0,\infty)\to \Teich(S)$ be the corresponding WP geodesic ray with end invariant $(\mu_{I},\mu_{T})$ and prescribed itinerary as in Theorem \ref{thm : inftyray}. Let $N$ be the parameter map from Proposition \ref{prop : parmap}.
 
As in $\S$\ref{subsec : div} for each $i\in\mathbb{N}$ let $J_{Z_{i}}=[s^{-}_{i},s^{+}_{i}]$. Also for each $i$ odd let 
\begin{center}$k_{i}=\max  J_{Z_{i-1}}\cap J_{Z_{i}}$ and $l_{i}=\min J_{Z_{i+1}}\cap J_{Z_{i}}$.\end{center}
and for each $i\in\mathbb{N}$ even let 
 \begin{center}$k_{i}=\min J_{Z_{i}}$ and $l_{i}=\max  J_{Z_{i}}$.\end{center}

Then all of the subsurface coefficient bounds (\ref{eq : dWJ2}) to (\ref{eq : dakl}) in $\S$\ref{subsec : div} hold. 
 
 Let $R={\bf m}+2M$ and $R'={\bf m}'+2M$. Given $\epsilon> 0$ let
\begin{equation}\label{eq : ei>closed}y\geq 2\bar{w}({\bf m},R,R',\epsilon).\end{equation}
where $\bar{w}$ is the constant from Theorem \ref{thm : inftyray}. 

Let the constants $K\geq 1$ and $C\geq 0$ be as in $\S$\ref{subsec : div}. Define the periodic sequence $\{e_{i}\}_{i=1}^{\infty}$ with the first four terms $e_{1}=Ky+2(K^{3}y+K^{3}C+KC), e_{2}=Ky+KC, e_{3}=Ky+2(K^{3}y+K^{3}C+KC)$ and $e_{4}=Ky+KC$, satisfying $e_{i}=e_{i'}$ whenever $i\equiv i'$ (mod 4). 

The estimates  (\ref{eq : Jiebd}) and (\ref{eq : Jiobd}) for the length of $J_{Z_{i}}$ intervals hold. Then similar to $\S$\ref{subsec : div} we can show that: For each $i\in\mathbb{N}$, $l_{i}-k_{i}\geq y$ (in particular $l_{i}>k_{i}$) and $l_{i}<k_{i+1}$ (see Figure \ref{fig : bddcombdiv}). 

For any $i$ even by the subsurface coefficient bounds (\ref{eq : dWJ2}), (\ref{eq : dZjJ2}), (\ref{eq : dZi-1rs-+}), (\ref{eq : dZi+1rs-+}) and (\ref{eq : dakl}) the subsurface $Z_{i}$ has $(R,R')-$bounded combinatorics over the interval $[k_{i},l_{i}]$. For any $i$ odd by the bounds (\ref{eq : dWkl1}), (\ref{eq : dZjkl1}), (\ref{eq : dZi-1kl1}), (\ref{eq : dZi+1kl1}) and (\ref{eq : dakl}) the subsurface $Z_{i}$ has $(R,R')-$bounded combinatorics over the interval $[k_{i},l_{i}]$.

For each $i\in\mathbb{N}$, let $t_{i}\in N(s)$ where $s\in [k_{i}+\bar{w},l_{i}-\bar{w}]$. Then Theorem \ref{thm : inftyray} applied to the interval $[k_{i},l_{i}]$ guarantees that
$$\ell_{\partial{Z_{i}}}(r(t_{i})) \leq \epsilon.$$
Moreover since $l_{i}\leq k_{i+1}$ we may choose $t_{i}$ and $t_{i+1}$ so that $t_{i+1}>t_{i}$. By the set up of the subsurfaces $Z_{i}$ in (\ref{eq : Zi}) and the function $q$ any two consecutive domains $Z_{i}$ and $Z_{i+1}$ have a boundary curve in common. Denote $\delta_{i}=\partial{Z_{i}}\cap\partial{Z_{i+1}}$. Then by the above bound $\ell_{\delta_{i}}(r(t_{i}))\leq \epsilon$ and $\ell_{\delta_{i}}(r(t_{i+1}))\leq \epsilon$. Then by the convexity of length-functions we have that
 $$\ell_{\delta_{i}}(r(t)) \leq \epsilon$$
  for every $t \in [t_{i},t_{i+1}]$.
  
Now since the intervals $[t_{i},t_{i+1}]$ ($i\in\mathbb{N}$) cover $[0,\infty)$, at any time the systole of the surface along $r$ is less than $\epsilon$. Therefore, $r$ stays in the $\epsilon-$thin part of the Teichm\"{u}ller space. Consequently, $\hat{r}$ the projection of $r$ to the moduli space stays in the $\epsilon-$thin part of the moduli space. Furthermore, since $q$ is a period function and $\{e_{i}\}_{i=1}^{\infty}$ is a periodic sequence $\hat{r}$ is a closed geodesic.

By Mumford's compactness criterion (see \cite{mumfordcpct}) there is an $\epsilon_{0}>0$ such that the compact subset $\mathcal{K}\subset\mathcal{M}(S)$ is contained in the $\epsilon_{0}-$thick part of the moduli space. Therefore $\mathcal{K}$ is disjoint from the $\epsilon_{0}-$thin part of the moduli space. Let $y\in\mathbb{N}$ be such that (\ref{eq : ei>closed}) holds for $\epsilon=\epsilon_{0}$. Then our construction produces closed WP geodesics which do not intersect $\mathcal{K}$. There are infinitely many integers $y$ which satisfy the inequality (\ref{eq : ei>closed}) and consequently there are infinitely many closed geodesics which do not intersect $\mathcal{K}$. These geodesics are distinct because have different forward ending laminations.
\end{proof}

\subsection{A recurrence condition}\label{subsec : recurrence}
The following theorem is a straightforward consequence of Proposition \ref{prop : rec}.  Given $A,R$ and $R'$ positive. Let the constants $w=w(A,R)$ and $\bar{\epsilon}=\bar{\epsilon}(A,R,R')$ be from the proposition.

\begin{thm}\textnormal{ (Recurrence condition)}\label{thm : recurrence}
Let $(\mu^{-},\mu^{+})$ be an $A-$narrow pair. Let $\rho$ be a hierarchy path between $\mu^{-}$ and $\mu^{+}$. Let $[k_{i},l_{i}]$, $i\in\mathbb{N}$, be a sequence of intervals with $l_{i}-k_{i}\geq 2w$ and $l_{i}<k_{i+1}$. Furthermore suppose that over each interval $[k_{i},l_{i}]$, $S$ has $(R,R')-$bounded combinatorics. Let $r$ be a WP geodesic ray with prescribed itinerary with end invariant $(\mu^{-},\mu^{+}$). Then $\inj (r(t_{i}))\geq \frac{\bar{\epsilon}}{2}$, where $t_{i}\in N(s)$ and $s\in [k_{i}+w,l_{i}-w]$.
\end{thm}
\begin{remark}
The WP volume of $\mathcal{M}(S)$ is finite. This follows from the fact that the WP metric extends to the Deligne-Mumford compactification of the moduli space \cite{maswp}. Thus the volume of the unit tangent bundle of the moduli space $\mathcal{M}^{1}(S)$ is finite. Then the fact that the geodesic flow on $\mathcal{M}^{1}(S)$ is volume preserving and Poincar\'{e} Recurrence Theorem (see e.g. Theorem 4.1.19 in \cite[Part 4]{KatokDyn}) imply that almost every WP geodesic ray is recurrent to the $\frac{\bar{\epsilon}}{2}-$thick part of the moduli space. But our construction of WP geodesics that are recurrent to the $\frac{\bar{\epsilon}}{2}-$thick part of $\mathcal{M}(S)$ only uses the combinatorial control we developed in this paper.
\end{remark}
\begin{thm}
There are WP geodesic rays that are recurrent to the $\frac{\bar{\epsilon}}{2}-$thick part of the moduli space.
\end{thm}
\begin{proof}
Let the indexed subsurfaces $X_{0}=S$ and $X_{1}$ and the pseudo-Anosov maps $f_{0}$ and $f_{1}$ supported on $X_{0}$ and $X_{1}$, respectively be as in $\S$\ref{subsec : pre2}. Let the function $q(i)\equiv i$ (mod 2) and the sequence $\{e_{i}\}_{i=1}^{\infty}$ be as in $\S$\ref{subsec : pre2}. For each $i\in\mathbb{N}$, let the subsurface $Z_{i}$ be as in (\ref{eq2 : Zi}). Let $\mu_{I}$ be a marking so that $\base(\mu_{I})$ contains $\{\partial{X}_{a}\}_{a=0,1}=\alpha$ and $\mu_{T}$ be as in $\S$\ref{subsec : pre2}.

Let the constants $K_{1},C_{1}$ and $E$ be as in Proposition \ref{prop : pre2}. Suppose that $|e_{i}|>E$ for all $i\in\mathbb{N}$. By Proposition \ref{prop : pre2}(\ref{pre2 : ld}) we have 
\begin{equation}\label{eq : dZiITer}d_{Z_{i}}(\mu_{I},\mu_{T})\asymp_{K_{1},C_{1}} |e_{i}|.\end{equation}

By Proposition \ref{prop : pre2}(\ref{pre2 : bd}) and (\ref{pre2 : bda}) every subsurface $Z$ with $d_{Z}(\mu_{I}(q,e),\mu_{T}(q,e))> {\bf m}$ is a large subsurface. So the pair $(\mu_{I},\mu_{T})$ is ${\bf m}-$narrow.  

Let $r$ be a WP geodesic ray with the end invariant $\mu_{I}$ and $\mu_{T}$ as in Theorem \ref{thm : inftyray}. Let $N$ be the parameter map from Proposition \ref{prop : parmap}. 

For each $i\in\mathbb{N}$ even let 
\begin{center}$k_{i}=\max  J_{Z_{i-1}}$ and $l_{i}=\min J_{Z_{i+1}}$.\end{center}
 For each $i\in\mathbb{N}$ odd let 
 \begin{center}$k_{i}=\min J_{Z_{i}}$ and $l_{i}=\max  J_{Z_{i}}$. \end{center}
 \medskip

\noindent{\bf Subsurface coefficient bounds:} Here we collect some subsurface coefficient bounds.
\medskip

\noindent{\bf $i$ is odd:} 
Proposition \ref{prop : pre2}(\ref{pre2 : bd}) and (\ref{eq : nbtrack}) (no backtracking) imply that for every proper non-annular subsurface $W$ which is not in the list of subsurfaces $Z_{i}$ we have 
\begin{equation}\label{eq : rbd1}d_{W}(\rho(k_{i}),\rho(l_{i}))\leq {\bf m}+2M.\end{equation} 

By Proposition \ref{prop : pre2}(\ref{pre2 : ord}) for every odd integer $j>i$, $Z_{i}<Z_{j}$, also for every odd integer $j<i$, $Z_{j}<Z_{i}$. Then an argument similar to the one we gave to prove (\ref{eq : dZjJ2}) in $\S$\ref{subsec : div} gives us
\begin{equation}\label{eq : dZjkl1r}d_{Z_{j}}(\rho(k_{i}),\rho(l_{i}))\leq 2M+2.\end{equation}
Also since $\partial{Z}_{i}\subset \rho(k_{i})$ and $\partial{Z}_{i}\subset \rho(l_{i})$, similar to (\ref{eq : dSkl1}) we can get 
\begin{equation}\label{eq : dSklr}d_{S}(\rho(k_{i}),\rho(l_{i}))\leq 4.\end{equation}

\noindent{\bf $i$ is even:}  Proposition \ref{prop : pre2}(\ref{pre2 : bd}) and (\ref{eq : nbtrack}) (no backtracking) imply that for every non-annular subsurface $W$ which is not in the list of subsurfaces $Z_{i}$, 
\begin{equation}\label{eq : rbd2}d_{W}(\rho(k_{i}),\rho(l_{i}))\leq {\bf m}+2M.\end{equation} 

Note that $k_{i}=\max J_{Z_{i-1}}$  so $J_{Z_{i-1}}\cap [k_{i},l_{i}]=\emptyset$. Then by Theorem \ref{thm : hrpath}(\ref{h : lrbddproj}) we have the first inequality below. Moreover $l_{i}=\min J_{Z_{i+1}}$ so $J_{Z_{i+1}}\cap [k_{i},l_{i}]=\emptyset$. Then again by Theorem \ref{thm : hrpath}(\ref{h : lrbddproj}) we get the second inequality below.
\begin{eqnarray}\label{eq : dZ-+kl}
d_{Z_{i-1}}(\rho(k_{i}),\rho(l_{i}))&\leq& 2M, \;\text{and}\label{eq : dZi+-kl}\\ 
d_{Z_{i+1}}(\rho(k_{i}),\rho(l_{i}))&\leq & 2M. \nonumber
\end{eqnarray}
 For any odd integer $j$ with $j< i-1$ by Proposition \ref{prop : pre2}(\ref{pre2 : ord}), $Z_{j}<Z_{i-1}$ then by Proposition \ref{prop : ordersubsurf}, $\max J_{Z_{j}}\leq \max J_{Z_{i-1}}$ and therefore, $J_{Z_{j}}\cap[k_{i},l_{i}]=\emptyset$. Similarly for any $j$ odd with $j>i+1$ we have that $J_{Z_{j}}\cap[k_{i},l_{i}]=\emptyset$. Then by Theorem \ref{thm : hrpath}(4) we get
\begin{equation}\label{eq : dZjkl2}d_{Z_{j}}(\rho(k_{i}),\rho(l_{i}))\leq 2M.\end{equation}
\medskip

Finally for any $i\in\mathbb{N}$, by Propositions \ref{prop : pre2}(\ref{pre2 : bda}) and no backtracking property of hierarchy paths  (\ref{eq : nbtrack}), for any $\gamma\in\mathcal{C}_{0}(S)$ we have that
\begin{equation}\label{eq : rabd}d_{\gamma}(\rho(k_{i}),\rho(l_{i}))\leq {\bf m}'+2M.\end{equation} 
\medskip

\noindent{\bf The length of intervals $[k_{i},l_{i}]$:} We proceed to estimate the length of the intervals $[k_{i},l_{i}]$.
\medskip

\noindent{\bf $i$ is odd:} By the bounds (\ref{eq : rbd1}), (\ref{eq : dZjkl1r}) and (\ref{eq : dSklr}) all of the subsurface coefficients of $\rho(k_{i})$ and $\rho(l_{i})$ except that of $Z_{i}$ are bounded above by ${\bf m}+2M$ (note that for any $j\in\mathbb{N}$ even $Z_{j}=S$). Let the threshold constant in the distance formula (\ref{eq : dsf}) be ${\bf m}+2M$ and let $K_{2},C_{2}$ be the constants corresponding to this threshold constant. Then we have 
$$d(\rho(k_{i}),\rho(l_{i}))\asymp_{K_{2},C_{2}} d_{Z_{i}}(\rho(k_{i}),\rho(l_{i}))$$
Note that when $i$ is odd $J_{Z_{i}}=[k_{i},l_{i}]$. Then by (\ref{eq : nbtrack}) we have
$$d_{Z_{i}}(\rho(k_{i}),\rho(l_{i}))\asymp_{1,2M}d_{Z_{i}}(\mu_{I},\mu_{T}).$$
 Moreover, $\rho$ is a $(k,c)-$quasi-geodesic where $k\geq 1$ and $c\geq 0$ only depend on the topological type of the surface. Let $K_{3}=kK_{2}$ and $C_{3}=kC_{2}+2kK_{2}M+kc$. Then 
\begin{equation}\label{eq : klZi}l_{i}-k_{i}\asymp_{K_{3},C_{3}}d_{Z_{i}}(\mu_{I},\mu_{T}).\end{equation}

\noindent{\bf $i$ is even:} By the bound (\ref{eq : dSbdZ-+}) in the proof of Proposition \ref{prop : pre2}  we have
$$d_{S}(\partial{Z}_{i-1},\partial{Z}_{i+1})\geq \frac{1}{K'_{1}}|e_{i}|,$$
where $K'_{1}=\min\{\tau_{a}:a=0,1\}$. Then since $\partial{Z}_{i-1}\subset\rho(k_{i})$ and $\partial{Z}_{i+1}\subset\rho(l_{i})$ and $\diam_{S}(\rho(k_{i}))\leq 2$ and $\diam_{S}(\rho(l_{i}))\leq 2$ we have
\begin{equation}\label{eq : dSkl}d_{S}(\rho(k_{i}),\rho(l_{i}))\geq \frac{1}{K'_{1}}|e_{i}|-4.\end{equation}

 By the bounds (\ref{eq : rbd2}), (\ref{eq : dZi+-kl}) and (\ref{eq : dZjkl2}) all of the subsurface coefficients of the pair $\rho(k_{i})$ and $\rho(l_{i})$ except that of $S$ are bounded above by ${\bf m}+2M$. Let $K_{2},C_{2}$ be the constants corresponding to the threshold constant ${\bf m}+2M$ in the distance formula (\ref{eq : dsf}). Then we have
$$d(\rho(k_{i}),\rho(l_{i}))\asymp_{K_{2},C_{2}}d_{S}(\rho(k_{i}),\rho(l_{i})).$$
Note that $\rho$ is a $(k,c)-$quasi-geodesic where $k\geq 1$ and $c\geq 0$. Let $K'_{3}=kK_{2}$ and  $C'_{3}=kC_{2}+kc$. Then we have 
\begin{equation}\label{eq : kl>e}l_{i}-k_{i}\asymp_{K'_{3},C'_{3}} d_{S}(\rho(k_{i}),\rho(l_{i})).\end{equation}
\medskip

We proceed to set the sequence $\{e_{i}\}_{i=1}^{\infty}$ and using the above estimates finish the construction of recurrent WP geodesic rays.

Let $R={\bf m}+2M$ and $R'={\bf m}'+2M$. Let
$$y\geq 2w({\bf m},R).$$ 
Let $K'=K'_{1}K'_{3}$ and $C'=K'_{1}K'_{3}C'_{3}+4K'_{1}$. Set the sequence $\{e_{i}\}_{i=1}^{\infty}$ such that for any $i$ even $e_{i}=K'y+C'$. In the next paragraph for $i$ odd, $e_{i}$ could be any integer.

Let $i\in\mathbb{N}$ be even. By the bounds (\ref{eq : dSkl}) and (\ref{eq : kl>e}) and the set up of the powers $e_{i}$ we have $l_{i}-k_{i}\geq 2w$. Moreover by the bounds (\ref{eq : rbd2}), (\ref{eq : dZ-+kl}), (\ref{eq : dZjkl2}) and (\ref{eq : rabd}) the surface $S$ has $(R,R')-$bounded combinatorics over the interval $[k_{i},l_{i}]$. Let $t_{i}\in N(s)$ where $s\in [k_{i},l_{i}]$. Then Theorem \ref{thm : recurrence}(Recurrence Condition) guarantees that $\inj((r(t_{i})))\geq\frac{1}{2}\bar{\epsilon}({\bf m},R,R')$. Therefore $\hat{r}$ the projection of $r$ to the moduli space is recurrent to the $\frac{\bar{\epsilon}}{2}-$thick part of the moduli space. 
\medskip

Let $R={\bf m}+2M$ and $R'={\bf m}'+2M$, as before. Let $\epsilon_{n}\to 0$ as $n\to \infty$. For each $n\in\mathbb{N}$ let
$$y_{n}\geq 2\bar{w}({\bf m},R,R',\epsilon_{n}).$$
Let $K=K_{1}K_{3}$ and $C=K_{1}C_{1}+K_{1}K_{3}C_{3}$. Now suppose that in the sequence $\{e_{i}\}_{i=1}^{\infty}$ we set above for each $i$ odd, $e_{i} \geq Ky_{\lfloor\frac{i}{2}\rfloor}+C$.

Let $i\in\mathbb{N}$ be odd. The bounds (\ref{eq : klZi}) and (\ref{eq : dZiITer}) and the set up of the powers $e_{i}$ imply that $l_{i}-k_{i}>2\bar{w}$. Moreover we observe that: by the bounds (\ref{eq : rbd1}), (\ref{eq : dZjkl1r}), (\ref{eq : dSklr}) and (\ref{eq : rabd}) the subsurface $Z_{i}$ has $(R,R')-$bounded combinatorics over the interval $[k_{i},l_{i}]$ (note that for any $j\in\mathbb{N}$ even we have $Z_{j}=S$). Let $t_{i}\in N(s)$ where $s\in [k_{i}+\bar{w},l_{i}-\bar{w}]$. Then Theorem \ref{thm : inftyray} applied to the interval $[k_{i},l_{i}]$ guarantees that 
$$\ell_{\partial{Z}_{i}}(r(t_{i}))\leq \epsilon_{\lfloor\frac{i}{2}\rfloor}.$$
 Then since $\epsilon_{\lfloor\frac{i}{2}\rfloor}\to 0$ as $i\to \infty$ the ray $\hat{r}$ is not contained in any thick part of the moduli space. 
\end{proof}

\bibliographystyle{amsalpha}
\bibliography{reference}

\end{document}